\documentclass[12pt,leqno]{article}
\usepackage{amsmath,amsfonts,amssymb,amscd,amsthm,indentfirst,graphicx,mathtools,enumerate,bm}
\usepackage{extarrows}

\theoremstyle{plain}
  \newtheorem{prop}{Proposition}[section]
  \newtheorem{thm}[prop]{Theorem}
  \newtheorem{cor}[prop]{Corollary}
  \newtheorem{lem}[prop]{Lemma}
  \newtheorem*{lem*}{Lemma}
  \newtheorem*{claim*}{Claim}
\newtheorem{defn}[prop]{Definition}
  \newtheorem*{defn*}{Definition}
  
  \newtheorem*{rem*}{Remarks}

  \newtheorem*{rema*}{Remark}

  \newtheorem{rems*}{Remarks}

\def\iintl{\iint\limits}

\def\p1{P^{-1}}

\def\Th{\Theta}
\def\uTh{\underline{\Theta}}
\def\midl{{\mathop{\mid}\limits}}

\def\td{\triangledown}
\def\be{\begin{equation}}
\def\ee{\end{equation}}
\def\ul{\underline}

\def\usig{\underline{\sigma}}

\def\uP{\underline{P}}

\def\uz{\underline{z}}

\def\ol{\overline}
 
\def\la{\lambda}

\def\intl{\int\limits}

\def\om{\omega}
\def\Om{\Omega}

\def\de{{\delta}}

\def\tri{\triangle }

\def\a{\alpha}
\def\thalf{{\textstyle{\frac 12}}}
\def\un{\underline}

\def\vare{\varepsilon}

\def\k{\kappa}
\def\diag{\text{diag}}
\def\dist{\text{dist}}
\def\supp{\text{supp}}
\def\rank{\text{rank}}

\def\lub{\text{lub}}
\def\glb{\text{glb}}
\def\cala{{\mathcal{A}}}
\def\calb{\mathcal{B}}
\def\cald{{\cal{D}}}

\def\calf{{\cal{F}}}

\def\calp{{\cal{P}}}

\def\calt{{\cal{T}}}

\def\cala{{\cal{A}}}

\def\bbr{\mathbb{R}}
\def\bbs{\mathbb{S}}
\def\cals{\mathcal{S}}
\def\hcals{\hat\cals}
\def\calt{\mathcal{T}}
\def\Ga{\Gamma}

\def\Om{\Omega}
\def\eqadef{{\stackrel{{def}}{=}}}
\def\intl{\int\limits}

\def\suml{\sum\limits}

\def\range{\mathrm{range}}
\def\sgn{\mathrm{sgn}}
\def\p{\partial}

\def\tan{\hbox{tan}}

\def\pOm{\partial\Omega}

\numberwithin{equation}{section}
\def\og{\Omega\setminus\Gamma}

\date{} %

\numberwithin{figure}{section}
\newcommand{\vertiii}[1]{{\left\vert\kern-0.25ex\left\vert\kern-0.25ex\left\vert #1
  \right\vert\kern-0.25ex\right\vert\kern-0.25ex\right\vert}}
\begin{document}
\title{Systems of conservation laws in higher space dimensions}
\maketitle
\centerline{\large Michael Sever}
\centerline{Department of Mathematics}
\centerline{ The Hebrew University}
\centerline{Jerusalem, ISRAEL}

\bigskip

\begin{abstract}
The existing paradox between theory and computational experiment for weak solutions of systems of conservation laws in higher space dimensions is arguably resolved.  Apparently  successful computations are identified with underlying boundary-value problems which are well-posed only in a weakened sense.   In the absence of  hyperbolicity in particular, prescribed boundary data sufficient to determine an a priori bound for an entropy weak solution need not suffice to imply local uniqueness thereof.

In this context, fluid flow models based on stationary or self-similar reductions of Euler systems are distinguished as particularly attractive for computational investigation.

\end{abstract}

\section{Introduction}

Substantial disparity persists between analytical \cite{B,D,G,LYZ} and experimental \cite{Fl,GM,LeV,LYZ,PT} results on existence of discontinuous weak solutions of nonlinear, first-order systems of conservation laws.  For systems in more than one space dimension in particular, computations establish such, beyond reasonable doubt, in far greater generality than supported by available theory. Here we specifically exclude the special cases of a scalar conservation law, solutions everywhere continuous and of hyperbolic systems with one space dimension, for which more extensive analytical results are available \cite{Gu,K,Lu1,Lu2,LYZ,W,ZZ}.  Scalar conservation laws are nonetheless used as examples throughout.

Point of departure for this treatise is an observation that resolution of this paradox does not require an existence proof of the ostensibly computed solutions.   It will suffice to characterize a solution subset $\hcals$, containing whatever computed solutions, the existence of which implies that successful computational investigations may be undertaken determining solutions therein.  A weakened expression of uniqueness of elements of $\hcals $ for specified data then identifies the existence of nonempty $\hcals$ with a weakened form of well-posedness of the underlying problem.

The systems of interest in this context generally derive from a primitive system (such as an Euler system) hyperbolic by virtue of a strictly convex entropy extension \cite{L2}.  The solution subset $\hcals $ satisfies a reduced system determining  invariant solutions under an element of the symmetry group with which the primitive system is equipped.  Stationary and self-similar solutions are typical examples.  A problem class  is then specified by the reduced system with domain and boundary conditions of interest.

The reduced systems are generally not everywhere hyperbolic, but inherit the entropy extension and thus the familiar inequality as an entropy condition on discontinuous weak solutions.  Time-like directions for the primitive system are reflected as distinguished directions for the reduced system, determined from the entropy flux.  However, boundary conditions for initial boundary-value problems for the primitive system generally do not determine suitable boundary conditions for the reduced system.  Weak well-posedness requires specification of $P_\cals$ determining the common form of boundary conditions satisfied by elements of $\hcals$.  Local uniqueness of $z \in \hcals$ satisfying these boundary conditions is not required.

Exploiting previous results \cite{S1,S2}, here we characterize elements of $\hcals$ as within the range of a Frechet differentiable mapping of boundary data into weak solutions.  The form and regularity of elements of $\hcals$ are modestly restricted, permitting convenient linearization of the weak forms of the given primitive system and of the corresponding
reduced system.

Candidates for local uniqueness are denoted by ``unambiguous" pairs $(z, P(z))$ below, $z \in \hcals$, $P(z)$ determining the form of the requisite boundary conditions, presumably unknown a priori.  Indeed, as detailed in the following section, unambiguous pairs $(z, P)$ satisfy competing stability and admissibility conditions.  Excessive boundary data precludes stability; insufficient boundary data precludes admissibility.  A distinct paucity of unambiguous pairs $(z, P)$ is illuminated in section 2.

Three subsequent sections detail how admissibility failure of a given pair $(z, P)$ may be quantified, and of how admissible $(z,P)$ may be obtained in a limit with respect to $z$ or $P$.

The discussion is related to computational investigations in section 6.  For definiteness and as eventually required, solutions $z$ are sought as familiar vanishing dissipation limits.  With seemingly mild qualifications, it is shown that the present framework  applies, and that successful computations, with boundary conditions determined from $P_\cals$ chosen a priori, result in computed solutions $z_0$ with $(z_0, P_\cals)$ stable but generally not admissible.  A precise definition of ``successful computational investigation" is thus obtained.

As may be anticipated by comparison of the results of \cite{Ma1,Ma2} with theorem 6.1 of \cite{S1}, stability depends on sufficient regularity of the prescribed boundary data.

The main theorem 7.2 then gives sufficient conditions that given a stable $(z, P_\cals)$, by prescription of additional (likely unknown a priori) boundary data, one may achieve $(z, P(z))$ admissible and thus unambiguous.  Often if not always, such $P(z)$ is unique.

Such is employed in section 8 to provide a precise definition 8.2 of weak well-posedness, and in theorem 8.5 to express sufficient conditions that a successful computational investigation implies weak well-posedness of the underlying boundary-value problem.

With the same qualifications, of course, sufficient conditions for weak well-posedness are thus necessary for successful computational investigations.

In this context, sections 9-12 are devoted to determining sufficient conditions for weak well-posedness.  The results fall short of corroboration of the existence of the computed solutions.  In addition to the expected restrictions with regard to regularity of the prescribed data and requisite properties of distinguished directions, severe restrictions are required on the primitive system and on the regularity of computed solutions.  The attractiveness of fluid flow models in this regard is nonetheless illuminated.

However the discrepancy between these results and experimental evidence  can be largely reconciled empirically.

In section 13, a converse of theorem 8.5 is obtained.  A family of approximation schemes is introduced, depending on familiar ``numerical parameters" including a specific adopted form of dissipation.  Failure of such schemes to determine solutions is identified with either whatever solution set $\hat \cals$ incompatible with the adopted form of dissipation or else the specified boundary data incompatible with that required for weak well-posedness, assuming that weak well-posedness is possible.

\section{Review, notation and preliminaries}

We consider systems of dimension $n \ge 2$ with $m$ independent variables, not necessarily hyperbolic but equipped with an entropy extension \cite{FL,L2} and thus representation in ``symmetric" or ``gradient"  form \cite{Go,M}
\be\label{aa} \suml^{m} _{i = 1} \left(\psi_{i, z_j} (z)\right)_{x_i} = 0, \; \; \, j = 1, \dots, n.\ee

In \eqref{aa} and throughout, $x_i -, z_j- $ subscripts denote partial derivatives.

The independent variable
\be\label{aab} x\in\Om \subset \bbr^{m}, \ee
with domain $\Om$ open and connected, not necessarily simply connected or bounded, with exterior unit normal $\nu$ defined almost everywhere on the necessarily nonempty  boundary $\partial \Om$.
In \eqref{aab} and throughout, $\subset$ denotes proper subset.

The dependent variable $z$ assumes values in an open convex region $D$,
\be\label{aac} z(x) \in D \subseteq \bbr^n.\ee

A specific system \eqref{aa} is characterized by given  potential functions
\be\label{aad} \psi_i \in C^3 (D\to\bbr), \; \; \, i = 1,\dots, m,\ee
Lagrange duals of the entropy flux components $q_i$,
\be\label{aaea} q_i (\psi^\dag_{i,z}) = z\cdot \psi^\dag_{i,z} (z) - \psi_i(z), \; \; \, i = 1,\dots, m.\ee

In \eqref{aaea} and throughout, dots denote either the $l_2 (\bbr^{m})$ or $l_2(\bbr^n)$ inner product.

Use of the ``symmetric" dependent variable $z$ also facilitates exposition of the related symmetry group \cite{S4}, to be exploited in section 11 below.

For all $x \in \partial\Om$ where $\nu(x)$  is defined, a ``normal potential function" is
\be\label{aae} \left(\psi_\nu (z)\right)  (x)\eqadef \suml^{m}_{i = 1} \nu_i(x) \psi_i \left(z(x)\right).\ee

The space of $n$-vector functions on $\partial\Om$ is denoted throughout by $N$, and $\calp$ the set of projection maps on $N$, symmetric with respect to the $L_2(\partial\Om)$ inner product.

A boundary-value problem for such a system may be regarded as determined by a specific element
\be\label{aaf} P \in \calp,\ee
and a set of solutions $\cals$, ostensibly such that elements $z \in \cals$  are ostensibly determined  by partial specification of the corresponding normal boundary flux
\be\label{ab} b(z,P) = (I_n-P) \psi^\dag_{\nu,z} (z\mathop{|}\limits_{\partial\Om}),\ee
with $I_n$ the $n$ dimensional identity operator throughout.  The solution $z$ is recovered from specified $b(z,P) \in N$, using the weak form of \eqref{aa},
\be\label{ac} \intl_{\pOm} b \cdot \theta = \iintl_\Om \suml^{m}_{i = 1}\suml^n_{j = 1} \psi_{i, z_j} (z)\theta_{j, x_i}\ee
for all $\theta \in X_P$,
\be\label{aca} X_P \eqadef \{ \theta \in (C(\Om)\cap W^{1,1} (\Om))^n
  | \,  P\theta \mathop{|}\limits_{\partial\Om} = 0\}.\ee

In \eqref{ac} and throughout,  single integrals are over $(m-1)$-manifolds and double integrals over $m$-manifolds, differentials omitted, where no ambiguity arises.

As weak solutions of \eqref{ac} are not unique, traditionally $P$ is fixed and admissibility conditions imposed on elements of $\cals$, attempting to recover uniqueness.

The procedure is reversed in [S1, S2].  Examples of solutions $z$, $P(z) \in \calp$ are sought satisfying an alternative admissibility condition, that $z$ lies in the range of a Frechet differentiable map on the boundary data.

As the adopted admissibility condition requires linearization of \eqref{ac}, detailed assumptions on the form and regularity of solutions are necessary.

Throughout we  assume solutions $z$ uniformly bounded and piecewise continuous,
\be\label{ad} z \in C(\bar\Om \setminus\bar\Ga (z)\to D),\ee
overbars throughout denoting closure. Measure-valued or distribution solutions are excluded.

The associated locus of jump discontinuities $\Ga(z)$ is assumed of the form
\be\label{ada} \Ga(z) = \mathop{\cup}\limits^K_{k=1} \Ga_k(z),\ee
with disjoint $\Ga_k(z)$, each an open ($m -1$)-manifold on which a unit normal $\hat \mu =  \hat \mu_k$ is continuously defined.  The $\Ga_k$ do not intersect  the boundary $\pOm$, by convention, but may have limit  points in  $\pOm$  and may have mutual intersection limit points.

Within each $\Ga_k$, we assume coordinates $\a_1,\dots, \a_{m -1}$, with orthogonal unit vectors $\hat\a_1 (x), \dots, \hat\a_{m-1} (x), \; x (\a)$ determining an isomorphism of $Y_k\subset \bbr^{m-1}$ onto $\Ga_k$, with a bounded Jacobian determinant
\be\label{adc} \Big | \frac{\partial x (\a)}{\partial(\a_1, \dots, \a_{m-1})} \Big | \le c\ee
uniformly for $\a \in Y_k, x(\a) \in \Ga_k$.  Within each $\Ga_k$, we define $\psi_\mu (z), \psi_{\a_\ell} (z), \ell = 1, \dots, m-~1$
analogously with $\psi_\nu $ in \eqref{aae}.

In \eqref{adc} and throughout, $c$ is a (sufficiently large positive) generic constant, distinguished by subscripts  only as appropriate.

With $z$ of the form \eqref{ad}, \eqref{ada}, linearization of \eqref{ac}, \eqref{ab} relates a given Frechet derivative of $b$, at $z$ with $P$ fixed, denoted
\be\label{addz} \dot b = d b (z),\ee of the form
\begin{align}\label{add} \dot b &= (I_n-P)\dot\psi^\dag_{\nu, z} (z)\nonumber\\
&= (I_n-P) \psi_{\nu, zz} (z) \dot z \mathop{\midl}_{\partial\Om}\end{align}
to perturbations of $z, \Ga$, determined by
\begin{align}\label{ade} &\dot z : \Om\setminus \Ga \to \bbr^n\\
\label{adf} &\dot x : \Ga_k \to \bbr^{m}, \; \sigma \, \eqadef - \hat\mu_k \cdot \dot x, \; \; \, k = 1,\dots, K.\end{align}

The formal linearization of \eqref{ac} is
\be\label{afl} \intl_{\partial\Om} \dot b \cdot \theta = \iintl_\Omega \suml^{m}_{i = 1} (\psi_{i, z} (z))^{\cdot} \theta_{x_i}.\ee

In the space of measures on $\Omega$, it is assumed in \eqref{afl}  that each $\psi_{i, z} (z))^{\cdot},
 \, i = 1, \dots, m$, is of the form
\be\label{afm} (\psi_{i, z} (z))^{\cdot} = (\psi_{i,z}(z))^{\cdot}_{\Omega\setminus\Ga} +
(\psi_{i, z}(z))^{\cdot}_\Ga,\ee
with
\be\label{afn} (\psi^\dag_{i, z} (z))^{\cdot}  _{\Omega\setminus\Ga} = \psi_{i, zz} (z) \dot z\ee
 within $\Om\setminus \Ga$, and
\be\label{afo} (\psi_{i,z} (z))^{\cdot}_\Ga = - (\hat \mu \cdot \dot x) [\psi_{i, z} (z)] = \sigma [\psi_{i, z}(z) ]\ee
within $\Ga$, using \eqref{adf}.

Then for all $\theta \in X_P$, linearization of \eqref{ac}, \eqref{aca} determines a relation
\be\label{ae} \intl_{\pOm} \dot b \cdot \theta = \iintl_{\Om\setminus\Ga} \dot z \cdot R(z) \theta + \intl_\Ga \sigma S(z)\theta,\ee
with
\be\label{af} (R(z)\theta)_k = \suml^{m}_{i = 1}\suml^n_{j = 1} \psi_{i, z_jz_k} (z) \theta_{j, x_i}, \quad k = 1, \dots, n,\ee
within $\Om\setminus\Ga$, and
\be\label{afa} S(z)\theta = \suml^{m}_{i=1} \suml^n_{j = 1} [\psi_{i, z_j} (z)] \theta_{j, x_i}\ee
within $\Ga$.

In \eqref{afo}, \eqref{afa}  and throughout, brackets denote jump discontinuities of whatever functions on each $\Ga_k$, in the direction determined by $\hat \mu_k$.

On each $\Ga_k$, solutions $z$ necessarily satisfy the Rankine-Hugoniot conditions
\be\label{afb} [\psi_{ \mu, z_j} (z)] = 0, \quad j = 1,\dots, n.\ee

We use \eqref{afb} and rotational symmetry to rewrite \eqref{afa} requiring only the tangential derivatives of $\theta$,
\be\label{ag} S(z)\theta = \suml^{m -1}_{\ell = 1}\suml^n_{j = 1} [\psi_{\a_\ell, z_j} (z)] \theta_{j, \a_\ell}.\ee

By convention, we use $S(z)\theta$ given by \eqref{ag} throughout.  Once this is done, $z$ need not satisfy \eqref{aa}, and the discussion applies immediately to approximate solutions of the form \eqref{ad}.

\subsection{Solvability}

Solvability of \eqref{ae} for $(\dot z, \sigma)$ given $\dot b$ is conveniently expressed using conditions on the test space.  A suitable ``boundary norm" $\| \cdot \|_{\partial\Om, P}$, typically an $L_2$- or weighted $L_2$-norm, is placed on $X_P\mathop{|}\limits_{\partial\Om}$, which coincides with ker $P$ using \eqref{aca}.  Such determines a Banach space $B_P$
\be\label{ahd} B_P \eqadef \{ \dot b \in \ker P | \| \dot b \|_{B, P} < \infty\}\ee
with an induced norm
\be\label{ahb} \| \dot b \|_{B, P} \eqadef \underset{\theta \in X_P}{\lub}
\frac{ \intl_{\pOm}{\dot b \cdot \theta}}{\| \theta\|_{\partial\Om, P}}.\ee

For the present, indeed until theorem 6.3 and thereafter,  the regularity of $\dot b$ is limited only by \eqref{ahb}.

Introducing scalar weight functions
\be\label{aga} w:\Om\setminus \Ga \to [1, \infty),\ee
\be\label{agb} w_\Ga:\Ga \to [1,\infty),\ee
we equip $X_P$ with a seminorm, tacitly restricting $X_P$ as required,
\be\label{ah} \| \theta\|_{z, w, w_\Ga} \eqadef \left(\iintl_{\Om\setminus\Ga} w | R(z)\theta|^2  + \intl_\Ga w_\Ga (S(z)\theta)^2\right)^{1/2}.\ee

The space $H(z,P,w,w_\Ga)$, the completion of $X_P$ in the norm \eqref{ah}, is a Hilbert space with inner product
\be\label{aha} (\theta, \theta')_H\eqadef\iintl_{\Om\setminus\Ga} w(R(z)\theta) \cdot (R(z)\theta') + \intl_\Ga w_\Ga (S(z)\theta)(S(z)\theta').\ee

We seek solutions of \eqref{ae} in the class
\be\label{ava} w^{-1/2} \dot z \in L_2 (\Om\setminus \Ga)^n, \; \; \, w^{-1/2}_\Ga \sigma \in L_2 (\Ga),\ee
implicitly restricting $w, w_\Ga$ in \eqref{aga}, \eqref{agb}, satisfying \eqref{ae} for all $\theta \in H(z, P, w, w_\Ga)$.  All such solutions are of the form
\be\label{avb} \dot z = w R(z) \theta' + \dot z_\perp, \; \; \, \sigma = w_\Ga S(z) \theta'' + \sigma_\perp\ee
with $\theta', \theta'' \in H(z,  P, w, w_\Ga)$ and $\dot z_\perp, \sigma_\perp$ satisfying
\be\label{avc} \iintl_{\Om\setminus \Ga} \dot z_\perp \cdot R(z) \theta = \intl_\Ga \sigma_\perp S(z) \theta = 0\ee
for all $\theta \in H(z, P, w, w_\Ga)$.  The functions $\theta', \theta''$ determine a solution of \eqref{ae} from \eqref{avb}, \eqref{avc}  and
\be\label{avd} \intl_{\partial \Om} \dot b \cdot \theta = \iintl_{\Om\setminus\Ga} w R(z) \theta' \cdot R(z) \theta + \intl_\Ga w_\Ga (S(z)\theta'') (S(z)\theta)\ee
for all $\theta \in H(z, P, w,w_\Ga)$.

A seemingly mild, simplifying assumption is made, that \eqref{ah} determines a proper norm on $H(z, I_n, w, w_\Ga)$, that there are no nontrivial functions $\theta$  satisfying
 \be\label{agc} R(z) \theta = 0 \; \hbox{\ throughout \ } \Om\setminus\Ga,\; \,
S(z) = 0 \hbox{\ throughout \ } \Ga, \,  \theta \mathop{|}\limits_{\partial\Om} = 0.\ee

\begin{defn} A pair $(z, P)$ is stable if there exist $\| \cdot \|_{\partial\Om, P}, w, w_\Ga$ such that
\be\label{agd} \| \theta\|_{\partial\Om, P} \le c_  1 (z, P, w, w_\Ga) \| \theta\|_{z, w, w_\Ga}\ee
for all $\theta \in H(z, P, w, w_\Ga)$.
\end{defn}

In \eqref{afl}, \eqref{ae}, \eqref{avd}, using \eqref{aae}, \eqref{ab}, \eqref{aca}, \eqref{afm}, \eqref{afn}, the symmetry of $P$ and the assumed $\Ga \cap \partial\Om$, of measure zero,
\begin{align}\label{agzz}
\int\limits_{\partial\Om} \dot b \cdot \theta &= \intl_{\pOm} \big( ( I_n-P)(\psi_{\nu, z} (z))^\cdot\big) \cdot \theta\nonumber \\
&=\intl_{\pOm} \big( (I_n-P)\psi_{\nu,zz}(z) \dot z\big) \cdot\theta  \nonumber \\
&= \intl_{\pOm} \dot z \cdot \psi_{\nu,zz} (z) \theta,\end{align}
which can hold for all $\theta \in X_P$ with some $\dot z$ depending on $\dot b$ only if
\be\label{age} \dot b \in   \range \, \psi_{\nu, zz} (z)\ee
almost everywhere in $\partial\Om$.

However, \eqref{age} may be extended  at points $x \in \partial\Om$ where there exist $e_0(x) \in \bbs^n$ such that boundary values of $z$ are restricted to satisfy homogeneous conditions,
\be\label{agea} e_0 (x) \cdot \psi^\dag_{\nu,z} (z(x)) = e_0 (x) \cdot \big(\psi^\dag_{\nu,z} (z(x))\big)^\cdot = 0.\ee

The condition \eqref{agea} is reflected in \eqref{ac}, \eqref{afl} by setting
\be\label{agec} P(x) e_0 (x) = 0,\ee
and using \eqref{agec}, the condition \eqref{age} implies
\be\label{aged} (\ker P)(x) \subseteq \, \range\; \psi_{\nu, zz} (z(x)) \oplus  \hbox{span\ } \{ e_0 (x)\}\ee
for almost all $x \in \pOm$.

However in \eqref{ahb}, it suffices to consider $\theta $ satisfying, in addition to \eqref{aca},
\be\label{ageb} e_0 (x) \cdot \theta(x) = 0.\ee

Stability of $(z, P)$ implies solvability of \eqref{ae} for $\dot b \in B_P$ in the class
\be\label{ahe}\dot z \in w R(z) H(z, P, w,w_\Ga), \;  \sigma\in w_\Ga S(z) H(z,P ,w,w_\Ga).\ee

Indeed, for $(z, P)$ stable, $\theta' = \theta''$ in \eqref{avd} uniquely minimizes
$$ \tfrac 12 \iintl_{\Om\setminus \Ga} w |R(z)\theta|^2 + \tfrac 12 \intl_\Ga w_\Ga (S(z)\theta)^2 - \intl_{\pOm} \dot b \cdot \theta $$
within $H(z, P, w, w_\Ga)$.

More generally, we abbreviate
\begin{align}\label{ahf} \Phi (z, P, w_\Ga ) &\eqadef H(z,P,w,w_\Ga)\mathop{|}\limits_{\Ga}\nonumber \\
&= \{ \phi_\Ga:\Ga \to \bbr^n | \intl_\Ga w_\Ga (S(z)\phi_\Ga)^2 < \infty, P\phi_\Ga\mathop{|}\limits_{\partial\Om\cap \partial \Ga} = 0 \}.\end{align}

\begin{thm}  Assume $(z, P)$ stable.  Then for any $\| \cdot \|_{\partial\Om, P}, w, w_\Ga$ such that \eqref{agd} holds and any $\dot b \in B_P,\,  \phi_\Ga \in \Phi (z, P, w_\Ga)$
there exists a unique
\be\label{ahg} \zeta_{\dot b \phi_\Ga} \in H(z, P, w, w_\Ga)\ee
such that
\be\label{aiz} \mathop{\int}\limits_{\partial\Omega} \dot b \cdot \theta +\intl_\Ga w_\Ga (S(z) \phi_\Ga)(S(z)\theta) = (\zeta_{\dot b \phi_\Gamma}, \theta)_H\ee
 for all $\theta \in H(z, P, w, w_\Gamma)$; and
\be\label{ai} \dot z_{ \dot b \phi_\Ga}  \, \eqadef\,  wR(z)
\zeta_{\dot b \phi_\Ga},\;  \sigma_{\dot b  \phi_\Ga} \, \eqadef\,  w_\Ga S(z) (\zeta_{ \dot b  \phi_\Ga} - \phi_\Ga)\ee
satisfies \eqref{ae} with
\be\label{ahh} \iintl_{\Om\setminus\Ga} \frac{|\dot z_{\dot b \phi_\Ga} |^2}{w} + \intl_\Ga \frac{\sigma_{\dot b \phi_\Ga}^2}{w_\Ga} < \infty.\ee

Furthermore,
\be\label{aia} \| \zeta_{\dot b \phi_\Ga} \|_{z, w, w_\Ga} \le c_1 (z, P, w,w_\Ga) \| \dot b \|_{B, P} + \Bigg(\intl_\Ga w_\Ga (S(z)\phi_\Ga)^2\Bigg)^{1/2},\ee
and the choice $\phi_\Ga = 0$ uniquely minimizes $\iintl_{\Om\setminus\Ga} \tfrac 1 w |R(z)\zeta_{\dot b \phi_\Ga} |^2 + \intl_\Ga \tfrac{1}{w_\Ga} (S(z)\zeta_{\dot b \phi_\Ga} )^2$.
\end{thm}
\begin{proof}
Any solution of \eqref{ae}  for all $\theta \in H (z, P, w, w_\Ga)$ obtained from \eqref{ah} necessarily satisfies  \eqref{ahh}, and  may be expressed
\be\label{ahi} \frac{\dot z_{\dot b \phi_\Ga }}{w^{1/2}} = w^{1/2} R(z)\zeta_{\dot b \phi_\Ga} + e_\perp\ee
\be\label{ahj} \frac{\sigma_{\dot b \phi_\Ga}}{w^{1/2}_\Ga} = w^{1/2}_\Ga S(z) (\zeta_{\dot b \phi_\Ga} - \phi_\Ga) + e'_\perp\ee
with some $\zeta_{\dot b \phi_\Ga}$ in \eqref{ahg}, some $\phi_\Ga \in \Phi (z, P, w_\Ga)$, and $e_\perp, e'_\perp $ satisfying
\be\label{ahk} \iintl_{\Om\setminus\Ga} w^{1/2} e_\perp \cdot R(z) \theta = \intl_\Ga w^{1/2}_\Ga e'_\perp S(z) \theta = 0\ee
for all $\theta \in H(z, P, w, w_\Ga)$.  Thus all solutions of \eqref{ae}, \eqref{ahh} are of the form \eqref{ai}.

Setting  $\theta' = \zeta_{\dot b \phi_\Ga}, \; \theta'' = \zeta_{\dot b \phi_\Ga} - \phi_\Ga$ in \eqref{avd},  using \eqref{aha} we obtain
 \eqref{aiz} for all $\theta \in H(z, P,w,w_\Ga)$.  From \eqref{agd}, \eqref{ahf}, the left side of \eqref{aiz} determines a bounded linear functional on $H(z, P, w, w_\Ga)$, so the existence and uniqueness of $\zeta_{\dot b \phi_\Ga}$ are immediate from the Riesz theorem.

The estimate \eqref{aia} follows by setting $\theta = \zeta_{\dot b \phi_\Ga}$ in \eqref{aiz} and using \eqref{ahb}, \eqref{agd}.

A superposition principle follows from the linearity in \eqref{aiz},
\be\label{aic} \zeta_{\dot b \phi_\Ga} = \zeta_{\dot b 0} + \zeta_{0\phi_\Ga}.\ee

Using \eqref{aic} in \eqref{ai},
\be\label{aid} \dot z_{\dot b\phi_\Ga} = \dot z_{\dot b 0} + \dot z_{0\phi_\Ga}, \; \; \sigma_{\dot b \phi_\Ga } = \sigma_{
\dot b 0} +\sigma_{0\phi_\Ga}.\ee

From \eqref{ai}, \eqref{aid}, using \eqref{aic}
\begin{align}\label{aie}
\iintl_{\Om\setminus\Ga} \tfrac 1w &|\dot z_{\dot b \phi_\Ga} |^2 + \intl_\Ga \tfrac{1}{w_\Ga} \sigma^2_{\dot b \phi_\Ga}
= \iintl_{\Om\setminus\Ga} w | R(z) \zeta_{\dot b\phi_\Ga}|^2 + \intl_\Ga w_\Ga \left(S(z) (\zeta_{\dot b \phi_\Ga} - \phi_\Ga)\right)^2\nonumber \\
=& \iintl_{\Om\setminus\Ga} w|R(z) (\zeta_{\dot b 0} +\zeta_{0\phi_\Ga})|^2 + \intl_\Ga w_\Ga (S(z)\left(\zeta_{\dot b 0} + (\zeta_{0\phi_\Ga} - \phi_\Ga))\right)^2\nonumber \\
=& \iintl_{\Om\setminus\Ga} \tfrac 1w | \dot z_{\dot b 0} |^2 + \intl_\Ga \tfrac{1}{w_\Ga} \sigma^2_{\dot b 0} + \iintl_{\Om\setminus\Ga} \tfrac 1 w | \dot z_{0\phi_\Ga} |^2 + \intl_\Ga \tfrac{1}{w_\Ga} \sigma^2_{0\phi_\Ga}\nonumber \\
+ &2\Big(\iintl_{\Om\setminus\Ga} w(R(z)\zeta_{0\phi_\Ga}) \cdot (R(z)\zeta_{\dot b 0} ) + \intl_\Ga w_\Ga
\Big(S(z)(\zeta_{0\phi_\Ga} - \phi_\Ga)) (S(z)\zeta_{\dot b 0})\Big).
\end{align}


The final term in \eqref{aie} vanishes, setting $\dot b = 0$, $\theta = \zeta_{\dot b 0}$ in \eqref{aiz}, proving the final claim.
\end{proof}

In particular, from \eqref{aie}, \eqref{aia}
\be\label{aiea} \iintl_{\Om\setminus \Ga} \tfrac 1w |\dot z_{\dot b 0} |^2 + \intl_\Ga \tfrac{1}{w_\Ga} \sigma^2_{\dot b 0} \le c^2_1 (z, P, w, w_\Ga) \| \dot b \|^2_{\calb, P}.\ee

Solutions of \eqref{ae} satisfy a boundary condition on $\dot z_{\dot b\phi_\Ga}$ of the form
\be\label{aieb} \intl_{\partial\Om} \theta\cdot ( b - \psi^\dag_{\nu, z})^\cdot ) = 0\ee
for all $\theta \in H(z, P, w, w_\Ga)$. In this sense $\ker P$ determines the prescribable boundary data on $\partial\Om$.

The condition \eqref{agd} does not require upper  bounds on $w, w_\Ga$.  For
\be\label{agf} \tilde w \ge w, \; \tilde w_\Ga \ge w_\Ga\ee
pointwise, from \eqref{ah}
\be\label{agg} \| \theta\|_{z, \tilde w, \tilde w_\Ga} \ge \| \theta\|_{z, w, w_\Ga}\ee
and \eqref{agd} will hold with some
\be\label{agh} c_1(z, P, \tilde w, \tilde w_\Ga) \le c_1 (z, P, w, w_\Ga).\ee

Furthermore, from \eqref{agg}
\be\label{agi} H(z, P, \tilde w, \tilde w_\Ga) \subseteq H(z, P, w, w_\Ga),\ee
with equality only of $\tilde w /w, \tilde w_\Ga / w_\Ga$ are pointwise bounded above.
Thus solutions $(\dot z, \sigma)$ of \eqref{ae} for given $\dot b$ and all $\theta \in H(z, P, w, w_\Ga)$ also satisfy \eqref{ae} for all $\theta \in H(z, P, \tilde w, \tilde w_\Ga$).  The solution set of \eqref{ae} is thus (generally) modified but not enlarged by replacing $w, w_\Ga$ by $\tilde w, \tilde w_\Ga$.

In particular, if \eqref{agd} holds with some $w, w_\Ga$ bounded above, equivalently with
\be\label{agj} w=w_\Ga = 1,\ee
the obtained solution set by this means is maximal.

However, we anticipate $w, w_\Ga$ generally unbounded from above.  Establishment of \eqref{agd} appears to require \cite{S1} $w, w_\Ga$ sufficiently large that $\| \theta\|_{z,w,w_\Ga}$ majorizes $\theta\mathop{\mid}\limits_\Ga$, and thus $w_\Ga$ becoming unbounded in the neighborhood of limit  points in $\partial \Ga$ where $|[z]|$ and thus $|S(z)\theta|$ vanish.

\subsection{Uniqueness}

For fixed $z, P, \dot b$, solutions $(\dot z_{\dot b \phi_\Ga}, \sigma_{\dot b \phi_\Ga})$ of the form \eqref{ai} satisfying \eqref{ae} generally depend on ``artificial" functions $\phi_\Ga, w, w_\Ga$.  With $w, w_\Ga$ also fixed, uniqueness of
 $(\dot z_{\dot b \phi_\Ga}, \sigma_{\dot b \phi_\Ga})$ in the class determined by  \eqref{ahe}, independent of $\phi_\Ga \in \Phi(z, P, w_\Ga)$, is equivalent \cite{S1} to any of three statements.  The proofs are elementary and not repeated here.

(i) The homogeneous form of \eqref{ae},
\be\label{ak} \iintl_{\Om\setminus\Ga} \dot \uz \cdot R(z)\theta +\intl_\Gamma \usig S(z) \theta = 0\ee
for all $\theta \in H(z, P, w, w_\Ga)$, admits only the trivial solution within the class determined by \eqref{ahe}.

(ii) The space $H(z, P, w, w_\Ga)$ is of the form
\be\label{al} H(z, P, w, w_\Ga) = H_{\Om\setminus\Ga} (z, P, w) \oplus H_\Ga (z, P, w_\Ga)\ee
with subspaces satisfying
\be\label{ala} S(z) \theta = 0, \; \, \theta \in H_{\Om\setminus\Ga} (z, P, w)\ee
almost everywhere on $\Ga$;
\be\label{alb} R(z)\theta = 0,\;  \,  \theta \in H_\Ga (z, P, w_\Ga)\ee
within $\Om \setminus \Ga$.

(iii)  With $H_\Ga (z, P, w_\Ga)$ determined from \eqref{alb},
\be\label{alc} S(z) H(z, P, w, w_\Ga) = S(z) H_\Ga (z, P, w_\Ga)\ee
in the sense that the two spaces coincide.

\begin{defn} A pair $(z, P)$ is admissible if there exists a norm $\| \cdot \|_{\partial \Om, P}$ on $H(z,P,w,w_\Ga) \mathop{\mid}\limits_{\pOm}$ such that for any $\dot b \in B_P$ determined from \eqref{ahd}, \eqref{ahb}, any solution of the form \eqref{ai} satisfying \eqref{ae} is independent of $\phi_\Ga \in \Phi(z, P, w_\Ga)$, $w$ satisfying \eqref{aga}, $w_\Ga$ satisfying \eqref{agb}.
\end{defn}

We emphasize that admissibility of $(z, P)$ so defined does not imply the existence of such solutions or that \eqref{agd} holds for some $w, w_\Ga$.

\begin{thm}  Assume any one of the three statements (i, ii, iii) holding for some $w'$ satisfying \eqref{aga}, $w'_\Ga$ satisfying \eqref{agb}. Then $(z, P)$ is admissible.
\end{thm}
\begin{proof}
The statements (i,ii,iii)  are equivalent, so by hypothesis they all hold for $w', w'_\Ga$.  We claim this is equivalent to their holding with $w, w_\Ga$ satisfying \eqref{agj}.

From \eqref{agf}, \eqref{agi}, \eqref{aga}, \eqref{agb}, \eqref{agj}
\be\label{ama} H(z, P, w', w'_\Ga) \subseteq H(z, P, 1, 1),\ee
so if \eqref{al}, statement (ii) above holds for $H(z, P, 1,1)$, it holds for $H(z, P, w', w'_\Ga)$.

Next assume that statement (i) above holds for $H(z, P, w', w'_\Ga)$ but not for $H(z, P, 1,1)$.  Such could occur only if there exists nontrivial
\be\label{amb} \dot{\underline{z}}
 \in R(z)H(z, P, 1, 1), \; \dot {\underline{z}} \notin w' R(z) H(z, P, w', w'_\Ga)\ee
or there exists
\be\label{amc} {\underline{\sigma}} \in S(z)H(z, P, 1, 1), \;  {\underline{\sigma}} \notin w'_\Ga  S(z) H(z, P, w', w'_\Ga)\ee
satisfying \eqref{ak} for all $\theta \in H(z, P, 1,1)$.

By assumption, \eqref{al} also holds for $H(z, P, w', w'_\Ga); \; H_{\Om\setminus \Ga} (z, P, w'), H_\Ga (z, P, w'_\Ga)$ are Hilbert spaces with norms, respectively
\be\label{amd} \| \theta\|_{\Om\setminus\Ga, z, w', w'_\Ga} \eqadef (\iintl_{\Om\setminus \Ga} w' |R(z)\theta|^2)^{1/2},\ee
\be\label{ame} \| \theta\|_{\Ga, z, w', w'_\Ga} \eqadef (\intl_\Ga w'_\Ga (S(z)\theta)^2)^{1/2}.\ee

Using \eqref{ah}, \eqref{amd}, \eqref{ame} and \eqref{amb}, \eqref{amc}, we note that
\begin{equation*} \iintl_{\Om\setminus\Ga} {\underline{\dot z}} \cdot R(z)\theta\end{equation*}
determines a bounded linear functional on $H_{\Om\setminus\Ga} (z, P, w')$, and
\begin{equation*} \intl_{\Ga} \underline{\sigma} S(z)\theta\end{equation*}
a bounded linear functional on $H_\Ga(z, P, w'_\Ga)$.  Thus by application of the Riesz theorem, we have $\theta_{\dot z}$ unique in $H_{\Om\setminus\Ga} (z, P, w')$, $\theta_\sigma$ unique in $H_\Ga (z, P, w'_\Ga)$  satisfying
\be\label{amf} \iintl_{\Om\setminus\Ga} \underline{\dot z} \cdot R(z)\theta = \iintl_{\Om\setminus\Ga} w' (R(z)\theta_{\dot z})\cdot (R(z)\theta)\ee
for all $\theta\in H_{\Om\setminus\Ga} (z,P,w')$,
\be\label{amg} \intl_\Ga \underline{\sigma} S(z)\theta= \intl_\Ga w'_\Ga (S(z)\theta_\sigma) (S(z)\theta)\ee
for all $\theta \in H_\Ga (z,P,w'_\Ga)$.

Using \eqref{al}, \eqref{ala}, \eqref{alb}, each of
\eqref{amf}, \eqref{amg} holds for all $\theta \in H(z, P, w', w'_\Ga)$.  Combining \eqref{amf}, \eqref{amg},
\be\label{amh} \iintl_{\Om\setminus\Ga}  w' (R(z)\theta_{\dot z}) \cdot (R(z)\theta) + \intl_\Ga w'_\Ga(S(z) \theta_\sigma) (S(z)\theta) = 0 \ee
for all $\theta \in H(z, P, w', w'_\Ga)$.  Thus
\be\label{ami} \dot z = w' R(z)\theta_{\dot z}, \; \sigma = w'_\Ga S(z)\theta_\sigma\ee
is a nontrivial solution of \eqref{ak}, \eqref{ahe}, contradicting statement (i) for $H(z, P, w, w_\Ga)$ and completing the proof of the claim.

Thus the three statements (i, ii, iii) hold for any $w, w_\Ga$
satisfying \eqref{aga}, \eqref{agb}.
For any $\dot b \in B(z, P)$, a solution
$\left(\dot z_{\dot b} (w, w_\Ga), \sigma_{\dot b } (w, w_\Ga)\right)$ satisfying \eqref{ae}, \eqref{ahe} for all $\theta \in H(z, P, w, w_\Ga)$ is unique, and it will suffice to prove that
\be\label{amj}
\dot z_{\dot b} (w, w_\Ga) = \dot z_{\dot b} (1, 1), \; \sigma_{\dot b} (w, w_\Ga) = \sigma_{\dot b } (1, 1).\ee

Using \eqref{al}, from \eqref{ae} it follows that $\dot z_{\dot b} (w, w_\Ga)$ satisfies
\be\label{amk} \intl_{\partial\Om} \dot b \cdot \theta = \iintl_{\og               } \dot z_{\dot b} (w, w_\Ga)\cdot R(z) \theta \ee
for all $\theta \in H_{\Om\setminus\Ga}  (z,P,w)$, and similarly $\sigma_{\dot b} (w, w_\Ga)$ satisfies
\be\label{aml} \intl_{\partial\Om} \dot b\cdot  \theta = \intl_\Ga \sigma_{\dot b} (w, w_\Ga) S(z)\theta\ee
for all $\theta \in H_\Ga (z,P,w_\Ga)$.  From \eqref{ah}, \eqref{ala}, \eqref{alb}, \eqref{aga}, \eqref{agb}
\be\label{amm} H_{\og} (z, P, w) = \{ \theta \in H_{\og} (z, P, 1) | \iintl_{\og} w | R(z)\theta|^2 < \infty\},\ee
and
\be\label{amn} H_\Ga (z, P, w_\Ga) = \{ \theta\in H_\Ga (z, P, 1) | \intl_\Ga w_\Ga (S(z)\theta)^2 < \infty\}.\ee

From \eqref{amm}, it follows that \eqref{amk}  must hold for all $\theta \in H_{\og} (z, P, 1)$, which is the condition satisfied uniquely by $\dot z_{\dot b} (1, 1)$.  Thus
$ z_{\dot b} (w, w_\Ga)$ and $\dot z_{\dot b} (1, 1)$ coincide.

Similarly, \eqref{amn} implies \eqref{aml} for all $\theta \in H_\Ga (z, P, 1)$, as uniquely satisfied by $\sigma_{\dot b} (1, 1)$ , so $\sigma_{\dot b} (w, w_\Ga)$ coincides with $\sigma_{\dot b} (1, 1)$.

\end{proof}

\subsection{Simultaneous stability and admissibility}

\begin{defn} A pair $(z, P)$ is unambiguous if it is both stable and admissible.
\end{defn}

Unambiguous $(z,P)$ implies existence and uniqueness (for given $\dot b$) of $\dot z, \sigma $ of the form \eqref{ava} satisfying \eqref{ae}.

As has been discussed \cite{S2}, stability and admissibility are competing if not incompatible requirements on $\ker P$, and $(z,P)$ unambiguous  is obtained with difficulty.

In particular, for any $P, \hat P \in \calp$ such that
\be\label{ana} \ker P \subset \ker \hat P\ee
and any $w, w_\Ga$, it  follows from \eqref{aca}, \eqref{ah} that
\be\label{anb} H(z, P, w, w_\Ga) \subset H(z, \hat P, w, w_\Ga).\ee

From \eqref{anb}, if $(z, \hat P)$ is stable, so is $(z, P)$ with
\be\label{anc} c_1 (z,P,w,w_\Ga) = c_1(z, \hat P, w, w_\Ga)\ee
in \eqref{agd}.

Admissibility of $(z,P)$ requires
\be\label{ara} H_\Ga (z, 0_N, 1)\mathop{\midl}_{\Ga} = H(z,P,1,1)   \mathop{\midl}_\Ga = L_2(\Ga).\ee

From \eqref{ara}, necessarily there exists
\be\label{arb} G(\cdot, \cdot; z): L_2 (\Ga) \to H_\Ga (z,0_N, 1),\ee
such that for any $\theta_\Ga \in H_\Ga (z, 0_N, 1)$,
\be\label{arc} \theta_\Ga (x) = \intl_\Ga G(x,y; z) (S(z)\theta_\Ga)(y),\; \; \, x \in \bar \Om, \; y \in \Ga \ee
understanding integration with respect to $y$ in \eqref{arc}.

So determined, $ G (\cdot, \cdot; z)$ is generally not unique.

Using \eqref{alb}, \eqref{arc},  $\theta_\Ga \in H_\Ga (z, P, 1)$ requires
\be\label{ard} \left( R (z) G (\cdot, y; z)\right) (x) = 0, \; x \in \Om\setminus \Ga, \; y \in \Ga\ee
while \eqref{arc} obviously requires
\be\label{are} \left(S(z)G(\cdot, \cdot; z)\right) (x, y) = \delta_\Ga (x-y),\; \; x, y \in \Ga, \ee
with $\delta_\Ga$ the Dirac measure on $\Ga$.

Using \eqref{al}, \eqref{aca}, admissibility  of $(z, P)$ requires $\ker P$ sufficiently large that there exists $G(\cdot, \cdot; z)$ satisfying
\be\label{arf} \left(PG(x, y; z)\right) (x) = 0, \; \; \, x \in \partial\Om, \; y \in \Ga.\ee

\begin{thm}  Assume $P, \hat P$ satisfying \eqref{ana} and $(z,P)$ admissible.  Then $(z, \hat P)$ is admissible.
\end{thm}
\begin{proof} Using \eqref{anb}, if \eqref{arf} holds for $(z, P)$, it holds for $(z, \hat P)$.
\end{proof}

In particular, any admissible $(z, P)$ implies $(z, 0_N)$ admissible, and the existence of $\theta' \in H_{\Om \setminus\Ga} (z, 0_N, 1)$ for each $\theta \in H(z, P, 1,1)$ such that almost everywhere in $\Om\setminus\Ga$
\be\label{axa} R(z)\theta' = R(z)\theta.\ee

If $(z,0_N)$ is inadmissible, there  exists a nontrivial solution of \eqref{ak} with $P=0_N$.  Partial integration in \eqref{ak} using \eqref{af} determines such $\underline{\dot {z}}$ satisfying
\be \label{axh} \psi_{\nu,zz} (z) {\underline{\dot z}} = 0\ee
almost everywhere on $\partial\Om.$ Thus such $z$ is linearly unstable even with $\psi^\dag_{\nu,z} (z)$ fixed almost everywhere on $\partial \Om$.

An example of inadmissible $(z, 0_\aleph)$, for a scalar conservation law in the strip $0  < x_1 <1, \; x_2 \in \bbr$, is shown in figure 2.1.  Necessarily one of the discontinuities is ``entropy-violating"; the dashed curves illustrate characteristic trajectories.


\begin{figure}[h]
\begin{center}
\includegraphics[width=0.8\textwidth]{ 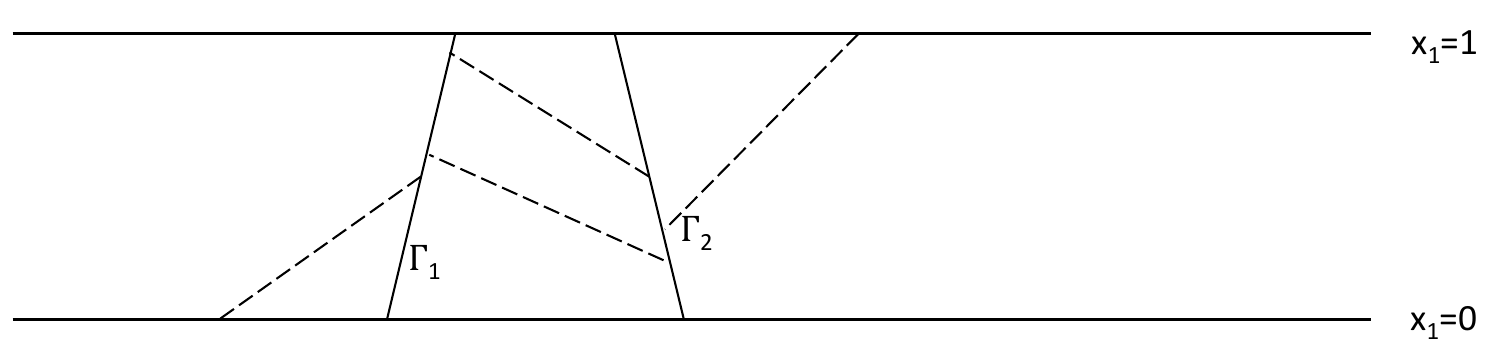}
\end{center}
\caption{
Inadmissible $(z, 0_N)$}
\label{ }
\end{figure}

Using definition 2.1, $(z,P)$ stable requires $\ker P$ sufficiently small that $G$ satisfying
\eqref{arb}, \eqref{ard}, \eqref{are}, \eqref{arf} is unique.

From \eqref{ara}, existence of the integral in \eqref{arc} requires
\be\label{arg} \| G (x, \cdot; z) \|_{L_\infty (\Ga)} \le c_{x,z}, \;  x \in \bar\Om,\ee
but not uniform boundedness with respect to $x$.  But for any $G$ satisfying a seemingly very mild additional assumption, that the condition \eqref{arb} holds with
\be\label{arh} G: L_2 (\Ga) \to L_\infty (\Om)^n,\ee
it follows that
\be\label{ari} \| \; \|G(\cdot, \cdot; z)\|_{L_\infty (\Ga)} \|_{L_\infty (\Om)} \le c_z.\ee

The conditions \eqref{arg}, \eqref{ari} are independent of stability of $(z, P)$ or uniqueness of $G(\cdot, \cdot; z)$.

For a given $z$, the condition $(z,P)$ unambiguous does not uniquely determine $P\in\calp$, but places severe restrictions thereon.

\begin{thm} Assume $(z,P)$ and $(z,P')$ both unambiguous, and that the same $\{ e_0\}$ applies for $P, P'$.

Then there exist $\hat P, \hat P', \in \calp$ such that $\ker P \subseteq \ker \hat P, \ker P' \subseteq \ker \hat P'$; $(z, \hat P)$ and $(z, \hat P')$ are both unambiguous; and $B_{\hat P}$ is isomorphic to $B_{\hat P'}$.  Furthermore the spaces $H_{\Om\setminus\Ga} (z, \hat P,\cdot), \; H_{\Om\setminus\Ga} (z, \hat P', \cdot)$, and the spaces $H_\Ga (z, \hat P, \cdot), H_\Ga (z, \hat P', \cdot)$, coincide up to functions $\theta$ satisfying $R(z) \theta = S(z)\theta = 0$. The same $\{ e_0\}$ applies to $\hat P, \hat P'$.
\end{thm}
\begin{proof}  By assumption, there exist $w, w_\Ga, \| \cdot \|_{\partial\Om, P}$ and $w', w'_\Ga, \| \cdot \|_{\partial\Om, P'}$ each satisfying \eqref{agd}.

Pointwise in $\Om\setminus\Ga$, we denote
\be\label{apa} \tilde w (x) \eqadef \hbox{ maximum\ } \left(w(x), w'(x)\right),\ee
and  pointwise within $\Ga$,
\be\label{apb} \tilde w_\Ga (x) = \hbox{ maximum\ } \left(w_\Ga(x), w'_\Ga (x)\right).\ee

Using \eqref{agf}, \eqref{agh}, the condition \eqref{agd} holds for $P, \tilde w, \tilde w_\Ga$ and for $P', \tilde w, \tilde w_\Ga$, with
\be\label{apc} c_1 (z, P,  \tilde w,  \tilde w_\Ga) \le c_1 (z, P,  w, w_\Ga), \; c_1 (z, P', \tilde w,  \tilde w_\Ga) \le c_1 (z, P',  w,  w_\Ga).\ee

By theorem 2.4, admissibility of $(z, P)$ and $(z, P')$ are not lost by use of $\tilde w, \tilde w_\Ga$.

We construct  a  linear map
\be\label{apd} A:B_P \to B_{ \hat P'}\ee
as follows.  For arbitrary given $\dot b \in B_P$, from $(z,P)$ unambiguous we find unique
\be\label{ape} \dot z \in \tilde w R(z)H_{\Om\setminus\Ga} (z, P, \tilde w), \; \, \sigma \in \tilde w _\Ga S(z) H_\Ga (z, P, \tilde w_\Ga)\ee
satisfying \eqref{ae} for all $\theta \in H(z, P, \tilde w, \tilde w_\Ga).$

Using $\dot z $ so obtained, we extend $\dot b$ to $\dot b_0 \in B_{0_N}$ satisfying
\be\label{apf}  (I_n-P) \dot b_0 = \dot b, \ee
\be\label{apg} P\dot b_0 = P\psi_{\nu, zz} (z) \dot z\midl_{\partial\Om}. \ee

Here
\be\label{apga} B_{0_N} \eqadef \, \{ \dot b_0 \in N \, | \, \dot b_0 \in \range\;  \psi_{\nu,zz} (z) \}\ee
almost everywhere in $\p\Om$.

With $\dot b_0$ so obtained, $(\dot z, \sigma)$ satisfies an extended form of \eqref{ae},
\be\label{aph} \intl_{\partial\Om} \dot b_0 \cdot \theta = \iintl_{\Om\setminus\Ga} \dot z \cdot R(z) \theta + \intl_\Ga \sigma S(z) \theta \ee
for all $\theta \in H(z, 0_N, \tilde w, \tilde w_\Ga)$.  From theorem 2.6, $(z, 0_N)$ necessarily is admissible, so $(\dot z, \sigma)$ uniquely satisfies \eqref{aph} with
\be\label{api} \dot z \in \tilde w R(z) H_{\Om\setminus\Ga} (z, 0_N, \tilde w), \; \sigma \in \tilde w_\Ga S(z) H_\Ga (z, 0_N, \tilde w_\Ga).\ee

Denote by $\ker\uP,\;  \uP \in \calp$, the subspace (possibly trivial) of $\ker P$ such that
\be\label{jaa} (I_n-P') \dot b_0 = 0 \; \; \hbox{for\ all\ } \dot b \in \ker (I_n-\uP).\ee

Then $A, \ker \hat P' $ in \eqref{apd} are determined, using \eqref{jaa}, from
\be\label{jab} \ker(I_n-\hat P') = \ker (I_n-P') \oplus \ker (I_n-\uP);\ee
\be\label{apj} (I_n-\hat P') (A\dot b) = (I_n-\hat P') \dot b_0;\ee
\be\label{jac} (I_n-P') (A\dot b) = (I_n-P') \dot b_0;\ee
\begin{align}\label{jad} (P' - \hat P') (A\dot b) &=(P'- \hat P') \dot b_0\nonumber\\
&=(I_n-\uP) \dot b_0\nonumber \\
&=(I_n-\uP)\dot b,\end{align}
using $\ker\uP$ a subspace of $\ker P$ in the last step.

From \eqref{aph}, \eqref{apj}
\be\label{apl} \intl_{\partial\Om} (A\dot b)\cdot \theta = \iintl_{\Om\setminus\Ga} \dot z \cdot R (z) \theta + \intl_\Ga \sigma S(z)\theta\ee
for all $\theta \in H (z,\hat P', \tilde w, \tilde w_\Ga)$.

So determined, the mapping $A$ is $1-1$ but not necessarily onto $B_{\hat P'}$.  However reversing the argument, we determine a linear map
\be\label{jaf} A^*: B_{ \hat P'} \to B_{\hat P}\ee
such that $\ker P \subseteq \ker \hat P$ and
\be\label{jag} A^*A = I\ee
on $B_P$.

The mapping $A^*$ is $1-1$ onto $B_{\hat P}$, so we may extend $A$ as the right inverse of $A^*$ so that \eqref{jag} holds on $B_{\hat P}$.  Such makes $A$ onto $B_{\hat P'}$ so
\be\label{jah} AA^* = I\ee
on $B_{\hat P'}$.

Next we show that stability has not been lost for $(z, \hat P), (z, \hat P')$.  Using \eqref{jad}, it will suffice to show that the map $A$ as applied to $B_P$ is bounded.

For admissible $(z,  P)$, we may set $\phi_\Gamma = 0$ in \eqref{ai}, \eqref{aia} to obtain
\begin{align} \label{apn} \Big(\iintl_{\Om\setminus\Ga} \tfrac{1}{\tilde w} |\dot z|^2 +& \intl_\Ga \tfrac{1}{\tilde w_\Ga} \sigma^2\Big)^{1/2}
= \| \zeta_{\dot b 0}\|_{z, \tilde w, \tilde w_\Ga}\nonumber \\
&\le c_1 (z, P, \tilde w, \tilde w_\Ga) \| \dot b\|_{B, P}.\end{align}

Using \eqref{ahb} for $\hat P'$ and \eqref{apl},
\begin{align}\label{apo} \|A \dot b\|_{B, \hat P'} &= \underset{\theta \in H(z, \hat P', \tilde w, \tilde w_\Ga)}{\lub}
\frac{\intl_{\pOm} A\dot b \cdot \theta}{\| \theta\|_{z, \tilde w, \tilde w_\Ga }}\nonumber \\
&=\underset{\theta \in H(z, \hat P', \tilde w, \tilde w_\Ga)}{\lub} \frac{\iintl_{\Om\setminus\Ga} \dot z \cdot  R(z)\theta + \intl_\Ga  \sigma S(z)\theta}{\| \theta\|_{z, \tilde w, \tilde w_\Ga}}\nonumber \\
&\le \Big( \iintl_{\Om\setminus\Ga} \tfrac{1}{\tilde w} |\dot z|^2 + \intl_\Ga \tfrac{1}{\tilde w_\Ga} \sigma^2\Big)^{1/2}\nonumber \\
&\le c_1 (z, P, \tilde w, \tilde w_\Ga) \; \| \dot b\|_{B, P}.\end{align}

Thus both $(z, \hat P)$ and $(z, \hat P')$  are unambiguous.

As \eqref{al} holds for both $\hat P, \hat P'$, using \eqref{apj}, \eqref{jag}, \eqref{jah},
we obtain from \eqref{apl}
\be\label{jaj} \intl_{\partial\Om} \dot b_0\cdot \theta = \iintl_{\Om\setminus\Ga} \dot z \cdot R(z)\theta\ee
for all $\theta \in H_{\Om\setminus\Ga} (z, \hat P, \tilde w) \cup H_{\Om\setminus\Ga} (z, \hat P', \tilde w)$, with
\be\label{jam} \dot z \in \{ \tilde w R(z) H_{\Om\setminus\Ga} (z, \hat P, \tilde w) \} \cap \{ \tilde w R(z) H_{\Om\setminus\Ga} (z, \hat P', \tilde w)\}\ee
unique.  Thus $H_{\Om\setminus\Ga} (z, \hat P, \tilde w)$ and $H_{\Om\setminus\Ga} (z, \hat P', \tilde w)$ coincide possibly excepting elements $\theta $ with $R(z)\theta$ vanishing identically.

Analogously, we have
\be\label{jao} \intl_{\partial\Om} \dot b_0 \cdot \theta = \intl_\Ga \sigma (S(z)\theta)\ee
for all $\theta \in H_\Ga (z, \hat P, \tilde w_\Ga ) \cup H_\Ga (z, \hat P', \tilde w_\Ga)$, with
\be\label{jap} \sigma \in \{ \tilde w_\Ga S(z) H_\Ga (z, \hat P, \tilde w_\Ga)\} \cap \{ \tilde w_\Ga S(z) H_\Ga (z, \hat P', \tilde w_\Ga)\}\ee
unique.  Thus $H_\Ga (z, \hat P, \tilde w_\Ga)$ and $H_\Ga (z, \hat P', \tilde w_\Ga)$ coincide, up to elements $\theta $ with $S(z)\theta$ vanishing identically.

This completes the proof.
\end{proof}


A sufficient condition for uniqueness of $P$ making $(z, P)$ unambiguous follows from theorem 2.7.

\begin{thm} Assume $(z, P_0)$ stable and such that
\begin{align}\label{jb} &\{ w^{1/2}  R(z)\theta | \theta \in H(z, P_0, w, w_\Ga)\} = L_2 (\Om\setminus \Ga)^n,\nonumber \\
&\{ w_\Ga^{1/2}  S(z)\theta | \theta \in H(z, P_0, w, w_\Ga)\} = L_2 (\Ga).\end{align}

Then $P$ making $(z, P)$ unambiguous and satisfying
\be\label{jba} \ker P_0 \subset \ker P\ee
is unique. In this case $z$ is locally unique, satisfying boundary data $b(z,P)$, within the set of solutions determined from \eqref{ac}, \eqref{aca}, such that \eqref{ae}, \eqref{ade}, \eqref{adf} are satisfied uniquely for all $\dot b \in B_P$.
\end{thm}
\begin{proof} Assume otherwise, some other $P'$ satisfying
\eqref{jba} and making $(z, P')$ unambiguous.

If $P, P'$ satisfy \eqref{ana} ($P'$ replacing $\hat P$), then using \eqref{jba}, \eqref{al}
\begin{align}\label{jbb}
&\{w^{1/2} R(z) \theta | \theta \in H(z, P_0, w, w_\Ga)\} \subseteq \{ w^{1/2} R(z) \theta | \theta \in H_{\Om\setminus\Ga} (z, P, w)\},\nonumber \\
&\{w_\Ga^{1/2} S(z) \theta | \theta \in H(z, P_0, w, w_\Ga)\} \subseteq \{ w_\Ga^{1/2} S(z) \theta | \theta \in H_{\Ga} (z, P, w_\Ga)\};\\
&\{w^{1/2} R(z) \theta | \theta \in H_{\Om\setminus\Ga} (z, P, w)\} \subseteq \{ w^{1/2} R(z) \theta | \theta \in H_{\Om\setminus\Ga} (z, P', w)\},\nonumber \\
\label{jbc}&\{w_\Ga^{1/2} S(z) \theta | \theta \in H_\Ga (z, P, w_\Ga)\} \subseteq \{ w_\Ga^{1/2} S(z) H_{\Ga} (z, P', w_\Ga)\};\\
\label{jbd}
&\{w^{1/2} R(z) \theta | \theta \in H_{\Om\setminus\Ga}(z, P', w)\} \subseteq L_2(\Om\setminus\Ga)^n\nonumber, \\
&\{w_\Ga^{1/2} S(z) \theta | \theta \in H_\Ga (z, P', w_\Ga)\} \subseteq L_2(\Ga).
\end{align}

If \eqref{ana} holds, one of the conditions \eqref{jbc} is necessarily with strict inequality, precluding \eqref{jb}.

Thus \eqref{jb}, \eqref{jba} could hold, with uniqueness of $P$ failing, only if theorem 2.7 applies with $\hat P = P, \hat P'= P'$, making
\begin{align}\label{jbe}
&\{w^{1/2} R(z) \theta | \theta \in H_{\Om\setminus\Ga}(z, P, w)\}\; \, \hbox{isomorphic\ to \ }  \{ w^{1/2} R(z) \theta | \theta \in H_{\Om\setminus\Ga} (z, P', w)\},\nonumber \\
&\{w_\Ga^{1/2} S(z) \theta | \theta \in H_\Ga (z, P, w_\Ga)\} \; \, \hbox{isomorphic\ to \ }   \{ w_\Ga^{1/2} S(z)  \theta | \theta\in H_\Ga (z, P', w_\Ga)\}.
\end{align}

We obtain $P''$, satisfying $\ker P'' \subset \ker P, \ker P'$, from
\be\label{jbf} \ker P''= \ker P\cap \ker P'.\ee

As $(z, P'')$ necessarily is stable, if $(z, P'')$ is also admissible then again we have a contradiction using \eqref{ana}.

For $(z, P'')$ not admissible, there is a nontrivial solution of \eqref{ak} for all $\theta \in H(z, P'',w,w_\Ga)$.  From \eqref{jbf}, \eqref{jba} (also with $P'$ replacing $P$)
\be\label{jbg} H(z, P_0, w,w_\Ga) \subseteq H(z, P'', w, w_\Ga),\ee
so there is a nontrivial solution of \eqref{ak} with $P = P_0$.  Such precludes \eqref{jb}.

If local uniqueness fails for solutions with the same boundary data and such that  \eqref{ae}, \eqref{ade}, \eqref{adf} are uniquely satisfied, from \eqref{avd} with $\dot b = 0$ there exists a nontrivial solution of
\be\label{jca} \iintl_{\partial\setminus \Ga} w R(z)\theta^{'} \cdot  R(z)\theta + \intl_\Ga w_\Ga (S(z)\theta^{''})  (S(z)\theta ) = 0\ee
for all $\theta \in H(z, P, w, w_\Ga)$, with some fixed $\theta', \theta'' \in H(z, P, w, w_\Ga)$.  But as $(z, P)$ is admissible, each term in \eqref{jca} mush vanish separately, which is impossible by inspection.
\end{proof}


Under common if not generic conditions, unambiguous $(z,P)$ is further restricted.

\begin{thm} Assume $(z, P)$ unambiguous with $z\midl_{\Om'}$ constant in an open subset $\Om'$ of $\Om$, such that $\nu$ is continuously defined on $\partial\Om' \cap \partial\Om$.

Assume that \eqref{agd} holds with
\be\label{atb} w\midl_{\Om'} = 1,\ee
and such that
\be\label{atc} \| \theta\|_{L_2(\partial\Om' \cap \partial\Om)} \le c \| \theta\|_{\partial\Om, P}\ee
for all $\theta \in H(z, P, w, w_\Ga)$.

Assume  in addition that
\be\label{atd} \| \theta\|_{L_2(\Om')} \le c \| \theta\|_{z, w, w_\Ga}.\ee

Denote by $\psi^\pm_{\nu, zz} (x)$ the nonnegative semidefinite commuting symmetric matrices satisfying
\be\label{ata} \left(I_n-P(x)\big) \psi_{\nu, zz} (z\midl_{\Om'}) \big(I_n-P(x)\right) = \psi^+_{\nu,zz} (x) - \psi^-_{\nu,zz} (x)\ee
for all $x \in \partial\Om' \cap \partial \Om$.


Then for any $\om$ compactly contained in $\partial\Om'\cap \partial\Om$ such that $P(x)$ is defined for all $x \in \partial\Om' \cap\partial \Om$,
\be\label{atz} \ker P(x) = (\range\;  \psi^+_{\nu, zz} (x) ) \oplus (\range\;  \psi^-_{\nu, zz} (x)) \oplus \mathrm{span\ } \{ e_0(x)\}\ee
for each $x \in \omega$, with $e_0(x)$ from \eqref{agea}.
\end{thm}
\begin{rem*} The statement is vacuous for $P = 0_N $ or $I_n$ within $\om$.  For this result, $z$ may be only an approximate solution of \eqref{aa}.
\end{rem*}

\begin{proof}  From \eqref{ata}, \eqref{age}, \eqref{agec}, pointwise within $\om$  it will suffice to show that
\be\label{atf}  \ker P(x) \subseteq \left(\range\; \psi^+_{\nu,z z} (x)\right) \oplus \left( \range\;  \psi^-_{\nu,zz}(x)\right) \oplus \hbox{span\ }  \{ e_0(x)\}.\ee

The conclusion is trivial if \eqref{atf} holds with equality for each $x \in \omega$.

Otherwise, it will suffice to show that $(z, P)$ is stable with $P$ satisfying \eqref{atz}, as stability would be lost for any increase of $\ker P$.

By hypothesis, we may choose a scalar function
\be\label{atg} \chi \in W^{1,\infty} (\bar \Om' \to [0,1]),\ee
 with $\bar \Om'$ the closure of $\Om'$, satisfying
\be\label{ath} \chi\midl_\om = 1\ee
and
\be\label{ati} \chi\mathop{\midl}_{\partial\Om'\cap\Om} = 0.\ee

As $\chi\theta \in H(z, P, w, w_\Ga)$, using \eqref{agd}, \eqref{atb}, \eqref{atd}, \eqref{atg},
\begin{align}\label{atj}
| \iintl_{\Om'} \chi \theta\cdot R(z) (\chi\theta) | &\le c (\| \chi\|^2_{W^{1, \infty}(\Om')} \| \theta\|^2_{L_2(\Om')}+ \|R(z)\theta\|^2_{L_2 (\Om')})\nonumber \\
&\le c \| \theta \|^2_{z, w, w_\Ga}.
\end{align}

Using $z\midl_{\Om'} $ constant, \eqref{af}, \eqref{ati}, partial integration gives
\begin{align}\label{atk}
\iintl_{\Om'} \chi \theta\cdot R(z\midl_{\Om'}) (\chi\theta) &= \tfrac 12 \intl_{\partial\Om' \cap \partial\Om} \chi\theta\cdot \psi_{\nu,zz} (z\midl_{\Om'}) (\chi\theta)\nonumber \\
&= \tfrac 12 \intl_{\partial\Om' \cap \partial\Om} \chi\theta\cdot (\psi^+_{\nu,zz} - \psi^-_{\nu,zz}) (\chi\theta),\end{align}
so from \eqref{atj}, \eqref{atk}, for all $\theta \in H(z, P, w, w_\Ga)$
\be\label{atl} | \intl_{\partial\Om' \cap \partial\Om} \chi\theta\cdot (\psi^+_{\nu,zz} - \psi^-_{\nu,zz}) (\chi\theta)|\le c \| \theta\|^2_{z, w, w_\Ga}.\ee

Simultaneously, using \eqref{atc}, \eqref{agd},
\begin{align}\label{atm} \intl_{\partial\Om' \cap \partial\Om}(\chi\theta) \cdot  (\psi^+_{\nu,zz} + \psi^-_{\nu,zz}) (\chi\theta) &\le c \|\theta\|^2_{L_2(\partial\Om'\cap\partial\Om)}\nonumber \\
&\le c \|\theta\|^2_{\partial\Om, P}\nonumber \\
&\le c \|\theta\|^2_{z, w, \om_\Ga}.\end{align}

Combining \eqref{atl}, \eqref{atm}, then using \eqref{ath}
\be\label{ato} \intl_\om \theta\cdot \psi^\pm_{\nu,zz} \theta \le c \| \theta\|^2_{z, w, w_\Ga}\ee
for all $\theta \in H(z, P, w, w_\Ga)$.

Stability of $(z, P)$ with $P$ satisfying \eqref{atz} is thus established.

\end{proof}

\section{Quantification of admissibility failure}

Tacit assumption that a given $(z, P)$ is stable makes admissibility equivalent to unambiguity.  Applying theorem 2.4, discussion of admissibility  is simplified, with no loss of generality, by assumption of $w, w_\Ga $ satisfying \eqref{agj}.  Following theorem 2.6, we assume  $(z, 0_N)$ admissible  throughout.

Given the form \eqref{ada} of $\Ga$, using \eqref{ahf}, \eqref{agj} we denote
\be\label{baa} \Phi_0 (z, P) \eqadef \{ \phi\in \Phi (z, P, 1) \cap W^{1,p} (\Ga) \, | \, \phi\midl_{\partial\Ga_k} = 0, \; \, k = 1, \dots, K\},\ee
understanding various norms on $\Ga$ obtained using the $\a$-coordinates within each $\Ga_k$ throughout, for some (finite) $p > m $ to be left unspecified throughout.

For arbitrary $\phi \in \Phi_0 (z, P), \;  \theta \in H(z, P, 1,1)$, we denote
\be\label{bab} L_0(\theta, \phi; z, P) \eqadef \iintl_{\Om\setminus\Ga} | R (z) \theta|^2 + \intl_\Ga \left(S(z)(\theta-\phi)\right)^2,\ee
\be\label{bac} M_0(\phi; z, P)\eqadef\underset{\theta \in H(z,P,1,1)}{\glb} L_0(\theta, \phi; z, P).\ee

Immediately from \eqref{bab}, \eqref{bac}, we have
\be\label{bad} M_0 (\phi; z, P) = L_0 \left( \Xi (z, P)\phi, \phi; z, P\right),\ee
with a homogeneous linear map
\be\label{bae} \Xi(z, P): \Phi_0 (z,P) \to H(z, P, 1,1)\ee
determined from
\be\label{baf} \Xi(z, P) \phi =\zeta_{0\phi} \in H(z, P, 1,1)\ee
as determined from \eqref{aiz}, \eqref{aha} with $\dot b = 0, w = w_\Ga = 1$;
\be\label{bag} \iintl_{\Om\setminus\Ga} R(z)\zeta_{0\phi} \cdot R(z)\theta +\intl_\Ga\left( S(z)(\zeta_{0\phi} - \phi)\right)\left(S(z)\theta)\right) = 0,\ee
for all $\theta \in H (z, P, 1,1)$.

Observing that $M_0 (\phi; z, P)$ is homogeneous of degree two in $\phi$, our adopted measure of inadmissibility of $(z,P)$ is a nonnegative scalar,
\be\label{bb}  Q_0 (z,P)\eqadef\underset{\phi\in\Phi_0(z,P)}{\lub} \frac{M_0(\phi; z, P)}{\| \phi\|^2_{W^{1,p}(\Ga)}}.\ee
\begin{thm} Admissibility  of $(z,P)$ is equivalent to
\be\label{bba}Q_0(z, P) = 0.\ee
\end{thm}
\begin{proof}  If $(z, P)$ is admissible, \eqref{alc} holding in particular, for all $\phi\in\Phi(z,P,1)$ we have $\Xi(z,P)\phi\in H_\Ga (z,P, 1), \; S(z)\left(\Xi(z,P)\phi\right)= S(z)\phi$ from \eqref{alc}, $M_0(\phi; z,P) = 0 $ from \eqref{bad}, \eqref{bab}, so \eqref{bba} is immediate using \eqref{bb}.  At issue is whether the converse is obtained with $\phi$ restricted to $\Phi_0 (z,P)$ given in \eqref{baa}.

Given $(z,P)$, we denote
\be\label{bbb} H_{\Ga\perp} (z,P)\eqadef\{ \theta\in H(z,P,1,1)\, | \, \intl_\Ga \left(S(z)\theta\right))\left(S(z)\theta_\Ga\right) = 0\}\ee
for all $\theta_\Ga \in H_\Ga (z,P,1)$.  Immediately from \eqref{bbb},
\be\label{bbc} H(z,P,1,1) = H_\Ga ( z,P,1)\oplus H_{\Ga\perp} (z,P),\ee
so
\be\label{bbd} \Ga=\left(\supp\;  S(z) H_\Ga(z,P,1)\right) \cup\left(\supp\;  S(z)H_{\Ga\perp}(z,P)\right).\ee

Denoting
\be\label{bbe} \Ga_E (z,P) = \Ga\setminus\supp \;  (S(z)H_{\Ga\perp} (z,P))\ee
we observe from \eqref{bbe}, \eqref{bbc} that if $M_0(\phi; z,P)$ vanishes for all $\phi\in\Phi_0 (z,P)$ with $\supp \; (S(z)\phi) \subseteq \Ga',\; \Ga'$ some open subset of $\Ga$, then $\Ga'\subseteq \Ga_E (z,P)$.

Using \eqref{bb}, \eqref{baa}, the condition \eqref{bba} suffices to establish
\be\label{bbf} \Ga_k \subseteq \Ga_E(z,P),\; \, k = 1,\dots, K\ee
implying $\Ga_E(z,P) = \Ga$ from \eqref{ada}, $\supp \; (S(z)H_{\Ga\perp} (z))$ vacuous in \eqref{bbe}.  Thus \eqref{alc} holds and $(z,P)$ is admissible.
\end{proof}

Positive $Q_0 (z,P)$ is a measure of failure of admissibility  of $(z,P)$ in the sense that for any $\phi\in\Phi_0(z,P)$, $\dot z_{0\phi}, \sigma_{0\phi}$ determined from \eqref{ai}, \eqref{aiz} with $\dot b = 0$ and appearing in \eqref{aid}, \eqref{aie}, satisfy
\begin{align}\label{bbg} \| \dot z_{0,\phi} \|^2_{L_2(\Om\setminus\Ga)} + \| \sigma_{0\phi}\|^2_{L_2(\Ga)} &= M_0(\phi; z,P)\\
&\le \| \phi\|^2_{W^{1,p}(\Ga)} Q_0 (z,P)\nonumber\end{align}
using \eqref{bab}, \eqref{bad}, \eqref{bb}.

With $\Ga$ of the form \eqref{ada}, an additivity condition holds for $Q_0(z,P)$, allowing investigation of $Q_0(z,P)$ separately within each component $\Ga_k, \; k = 1,\dots, K$.

For $k = 1, \dots, K$, denote
\be\label{bca} \Phi_k(z,P)\eqadef\{ \phi\in \Phi_0(z,P) \Big | \supp \;  \phi \subset  \Ga_k\}\ee
and
\be\label{bcb} Q_k(z,P)\eqadef\underset{\phi\in\Phi_k (z,P)}{\lub} \frac{M_0(\phi; z,P)}{\| \phi\|^2_{W^{1,p}(\Ga_k)}}.\ee
\begin{lem} For $\phi_{(k)} \in \Phi_k (z,P), \; k = 1,\dots, K'$, satisfying \eqref{bca}
\be\label{bc} M_0 ( \suml^{K'}_{k = 1} \phi_{(k)}; z,P) \le (\suml^{K'}_{k = 1} M_0 (\phi_{(k)}; z, P)^{1/2})^2.\ee
\end{lem}
\begin{proof} It  will suffice to prove \eqref{bc} for $K'=2$, as by induction it will hold for any $K'$.

Denoting
\be\label{bcd} \zeta_{(k)} \eqadef \Xi(z,P)\phi_{(k)}, k = 1,2; \; \, \zeta_{(0)} \eqadef \Xi (z,P) (\phi_{(1)} + \phi_{(2)} ),\ee
it follows from the linearity of $\Xi(z,P)$ that
\be\label{bcc} \zeta_{(0)} = \zeta_{(1)} + \zeta_{(2)}.\ee
Then from \eqref{bad}, \eqref{bab}, \eqref{bcd}, \eqref{bcc}
\begin{align}\label{bcz}
M_0(\phi_{(1)} + \phi_{(2)}; z,P) &= L_0 (\zeta_{(0)}, \phi_{(1)} + \phi_{(2)}; z,P)\nonumber \\
&= L_0 (\zeta_{(1)}+\zeta_{(2)},          \phi_{(1)} +   \phi_{(2)}; z,P)\nonumber\\
&= \iintl_{\Om\setminus\Ga} | R(z) (\zeta_{(1)}+\zeta_{(2)})|^2 + \intl_\Ga (S(z) (\zeta_{(1)}- \phi_{(1)} + \zeta_{(2)}- \phi_{(2)}))^2\nonumber\\
&\le \iintl_{\Om\setminus\Ga}
(| R(z) \zeta_{(1)}|^2+ |R(z)\zeta_{(2)}|^2 + 2 |R(z)\zeta_{(1)} \cdot R(z) \zeta_{(2)} |)\nonumber\\
+ \intl_\Ga \Big(    ((S(z)(\zeta_{(1)} - \phi_{(1)}))^2 +&(S(z) (\zeta_{(2)}-\phi_{(2)}))^2 + 2 (S(z)(\zeta_{(1)}- \phi_{(1)})) (S(z)(\zeta_{(2)} - \phi_{(2)}))   \Big)\nonumber \\
\le M_0 (\phi_{(1)}; z, P) &+M_0 (\phi_{(2)}; z,P) + 2 \iintl_{\Om\setminus\Ga} | R(z) \zeta_{(1)} | \; | R(z)\zeta_{(2)} |\nonumber \\
\qquad \qquad \quad &+2\intl_\Ga |S(z) (\zeta_{(1)} - \phi_{(1)}) | \; |S(z)(\zeta_{(2)}- \phi_{(2)})|\nonumber \\
\le M_0 (\phi_{(1)}; z,P) &+ M_0 (\phi_{(2)}; z,P) + 2M_0 (\phi_{(1)}; z, P)^{1/2} M_0(\phi_{(2)}; z, P)^{1/2}.\end{align}~\end{proof}

An immediate consequence of lemma 3.2 is the following:
\begin{thm}
 \be\label{bce} Q_0 (z,P) \le K^{(p-2)/p} \suml^K_{k=1} Q_k(z, P).\ee
\end{thm}
\begin{proof} We apply lemma 3.2 with arbitrary $\phi_{(k)} \in \Phi_k(z,P)$, $k = 1,\dots, K$, abbreviating
\be\label{bcf} \beta_k\eqadef \| \phi_{(k)} \|_{W^{1, p}(\Ga_k)}, \; \;  k = 1,\dots, K.\ee

Without loss of generality, we set
\be\label{bcg} \| \suml^K_{k=1} \phi_{(k)} \|^p_{W^{1,p}(\Ga)} = 1,\ee
so
\be\label{bch} \suml^K_{k = 1} \beta^p_k = 1, \; \, \suml^K_{k = 1} \beta^2_k \le K^{(p-2)/p}.\ee

Then from \eqref{bb}, \eqref{bcg}, \eqref{bc}, \eqref{bcb}, and $M_0$ homogeneous of degree two in $\phi$,
\begin{align}\label{bci} Q_0 (z,P) &=\underset{\phi_{(k)}}{\lub} \;  M_0 \Big(\suml^K_{k=1} \phi_{(k)}; z, P\Big)\nonumber\\
  &\le \underset{\phi_{(k)}}{\lub}\Big(\suml^K_{k=1} M_0 (\phi_{(k)}; z, P)^{1/2}\Big)^2 \nonumber\\
&= \underset{\phi_{(k)}}{\lub} \Big(\suml^K_{k=1} \beta_k M_0 (\frac{\phi_{(k)}}{\beta_k}; z, P)^{1/2}\Big)^2 \nonumber\\
&\le \Big(\suml^K_{k=1} \beta^2_k\Big) \Big(\suml^K_{k=1} Q_k(z, P)\Big)
\end{align}
which is \eqref{bce} using \eqref{bch}.
\end{proof}

For example, if \eqref{ard}, \eqref{are} can be established for all $y \in \Ga_k$, then \eqref{bbf} holds and $Q_k (z, P)$ necessarily vanishes.  Such may be anticipated when \eqref{aa} is hyperbolic in a region containing $\Ga_k$.

The restriction \eqref{baa} is further  exploited in the following.
\begin{lem} For any $k = 1,\dots, K, \; \phi_{(k)} \in \Phi_k(z,P)$, there exists
\be\label{bda} \theta_{(k)} \in H(z, P, 1,1)\cap W^{1,p} (\Om)\cap W^{1,p}(\Ga)\ee
vanishing on $\partial\Om$, and satisfying the following:
\be\label{bdb} \theta_{(k)} (x) = \phi_{(k)} (x), \; \; x \in \Ga;\ee
\be\label{bdc} \| \theta_{(k)} \|_{ W^{1,p}(\Om\setminus\Ga)} \le  c_z \| \phi_{(k)} \|^{1/p}_{L_p(\Ga_k)} \| \phi_{(k)} \|^{ 1-1/p}_{W^{1,p}(\Ga_k)};\ee
\be\label{bdd} \supp \; \theta_{(k)} \subset \{ x \in \Om | \dist \; (x, \Ga_k) \le \| \phi_{(k)} \|_{ L_p(\Ga_k) } /\| \phi_{(k)} \|_{W^{1,p} (\Ga_k)} \}.\ee
\end{lem}
\begin{proof}  The $(m-1)$-manifold $\Ga_k$ partitions $\Om$ locally.  Separately for each side of $\Ga_k$, with $Y_k$ as in obtaining \eqref{adc}, we choose an $m$-vector function $\tilde x(\a, t)$,
\be\label{bdy} \tilde x \in W^{1,\infty} : Y_k \times  (0,\gamma_k)\to \bbr^{m},\ee
abbreviating
\be\label{bd} \gamma_k\eqadef \| \phi_{(k)} \|_{L_p (\Ga_k)} / \| \phi_{(k)} \|_{W^{1,p}(\Ga_k)},\ee
and satisfying the following:
\be\label{bde} \{ \tilde x(Y_k, 0)\} = \Ga_k;\ee
\be\label{bdz}\{ \tilde x (Y_k \times (0, \gamma_k ))\} \; \, \hbox{does\ not\ intersect\ } \, \partial\Om
\ \; \hbox{or\ } \Ga;\ee
\be\label{bdf}
 \Big| \frac{\partial(\tilde x_1,\dots, \tilde x_{m})}{\partial(\alpha_1,\dots, \alpha_{m-1}, t)}\Big| \le 1,\ee
uniformly for $ \alpha \in Y_k, \; t \in (0,\gamma_k)$, here using \eqref{adc}.

Then we set
\be\label{bdg} \theta_{(k)} (\tilde x (\a, t)) \eqadef \phi_{(k)} (\tilde x (\a, 0) (1 - t/\gamma_k), \; \; \, \a \in Y_k, \, t \in [0, \gamma_k ],\ee
$\theta_{(k)} $ vanishing elsewhere.

The condition \eqref{bdd} is immediate from \eqref{bdg} and \eqref{bdb} follows using \eqref{bde}.

From \eqref{bdg}, using \eqref{bdf},
\begin{align}\label{bdh} \| \theta_{(k)}\|^p_{W^{1,p}(\Om\setminus\Ga)} &\le c_z \iintl_{Y_k \times (0,\gamma_k)} (|\theta_{(k)} |^p +|\theta_{(k), \a}|^p + | \theta_{(k), t}|^p)\nonumber\\
&\le c_z \Big(\gamma_k \intl_{Y_k}  (|\phi_{(k)} |^p + | \phi_{(k),\a} |^p) + \frac {1}{\gamma^{p-1}_k} \intl_{Y_k} |\phi_{(k)}|^p\Big)\nonumber \\
&= c_z \Big( \gamma_k \| \phi_{(k)} \|^p_{W^{1,p}(\Ga_k)} + \frac{1}{\gamma_k^{p-1}} \| \phi_{(k)} \|^p_{L_p(\Ga_k)} \Big),\end{align}
from which \eqref{bdc} is immediate using \eqref{bd}.
\end{proof}

A consequence is the following.
\begin{thm} Assume $\Ga_k$ bounded. Then $Q_k(z,P)$ determined from \eqref{bcb} is achieved within $\Phi_k(z,P)$.
\end{thm}
\begin{proof}  From \eqref{bcb}, there exists a sequence $\{ \phi_\delta\}$ as $\delta \downarrow  0$,
\be\label{bea} \phi_\de\in \Phi_k (z,P),\ee
\be\label{beb} \| \phi_\de\|_{W^{1,p}(\Ga_k)} = 1,\ee
\be\label{bec} M_0(\phi_\de; z,P) \overset{\de\downarrow 0}{\longrightarrow}      Q_k(z,P).\ee

We presume $Q_k(z,P)$ positive, as if $Q_k (z,P)$ vanishes, the conclusion is trivial.

Using the boundedness of $\Ga$ and taking a subsequence as necessary, we have
\be\label{bed} \phi_0 \in \Phi_k(z,P), \; \; \| \phi_0\|_{W^{1,p}(\Ga_k)} \le 1 \ee
such that
\be\label{bee} \phi_\de\overset{\de\downarrow 0}{\longrightarrow} \phi_0\ee
strongly in $L_p(\Ga_k)$.

Applying lemma 3.4, using \eqref{bee}, \eqref{beb}, \eqref{bed},
\be\label{bef} M_0 (\phi_\de - \phi_0;z, P) \overset{\de\downarrow 0}{\longrightarrow}  0,\ee
so by application of lemma 3.2, again using \eqref{bed},
\begin{align}\label{beg}
M_0 (\phi_\de; z,P) \overset{\de\downarrow 0}{\longrightarrow}  & M_0(\phi_0; z,P)\nonumber\\
&= Q_k (z, P)\end{align}
from \eqref{bec}.
\end{proof}

\subsection{Examples}

A useful measure of admissibility failure $Q_0 (z,P)$ will satisfy continuity properties with respect to $z,P$, and will be such that small positive $Q_0(z_\de,P_\de)$ is suggestive of an admissible limit as $\de\downarrow 0$. Here we  present three simple examples of how such can occur, using theorem 3.3.

\smallskip

Example 1:  Suppose that for $\de > 0$, a component $\Ga_{k\de} \subseteq \Ga(z_\de)$ is weak, disappearing in the limit $\de\downarrow 0$,  thus satisfying
\be\label{bga} \| [z_\de]\|_{L_\infty (\Ga_{k\de})} \overset{\de\downarrow 0}{\longrightarrow} 0.\ee

For any $\phi \in \Phi_k (z_\de, P)$, the choice $\theta = 0 $ in \eqref{bac} achieves
\be\label{bgb}
 Q_k (z_\de, P) \le c \| [z_\de]\|^2_{L_\infty (\Ga_{k\de})}.\ee

An example is a rarefaction wave approximated by a fan of entropy-violating discontinuities of strength $\delta$.
 \smallskip

 Example 2:  Suppose that a component $\Ga_{k\de}$ is small in size, again disappearing in the limit $\de\downarrow 0$, in the sense that for $\phi\in \Phi_k(z_\de, P)$ (not vanishing identically)
 \be\label{bgc} \underset{ \phi\in \Phi_k(z_\de, P)}{\lub}
  \frac{\| \phi\|_{L_p (\Ga_{k\de})}}
  {\|\phi\|_{W^{1, p}(\Ga_{k\de})}}
  \overset{\de\downarrow 0}{\longrightarrow} 0.
 \ee

Given $\phi_{(k)} \in \Phi_k(z_\de, P)$ in \eqref{bac} we choose $\theta = \theta_{(k)} $ obtained in lemma 3.4, readily obtaining
\be\label{bgd} Q_k (z_\de, P)\le c \underset{\phi\in \Phi_k(z_\de, P)}{\lub}\left( \frac{\| \phi\|_{L_p(\Ga_{k\de})}}{\| \phi\|_{W^{1,p} (\Ga_{k\de})}}\right)^{2/p},\ee
using \eqref{bdc} in particular.

A toy example, given a Cauchy problem for a scalar conservation law, is an entropy-violating discontinuity in the initial data becoming a rarefaction wave after time $\delta$.

\smallskip

Example 3:  Suppose that $(z, 0_N)$ is admissible but $(z, P_\delta)$ is not, for $\delta > 0$.  There exists $G(\cdot, \cdot; z)$ satisfying \eqref{ard}, \eqref{are}, and we presume that \eqref{arh}, \eqref{ari} hold.  However for each $\de > 0$ there exists a nonempty subset of $\partial \Om$
\be\label{bge} \Sigma_\de \eqadef \supp \{ x \in \partial\Om | P_\de G (x, \cdot; z) \neq 0\}.\ee

Suppose that as $\de$ decreases, $\ker P_\de$ increases and $\Sigma_\de $
becomes small in the sense that there exists $\chi_\de \in H^1_{loc} $ satisfying
\be\label{bgf} \chi_\de\mathop{|}\limits_\Ga = 1,\ee
\be\label{bgh} \chi_\de \mathop{\mid}\limits_{\Sigma_\de} = 0,\ee
satisfying
\be\label{bgi} \| \triangledown \chi_\de\|_{L_2(\Om)}\overset{\de\downarrow 0}{\longrightarrow} 0.\ee

Then given $\phi_{(k)} \in \Phi_k (z, P_\delta )$, in \eqref{bac} we choose $\theta$ satisfying
\be\label{bgj} \theta(x) = \chi_\de (x)\intl_\Ga G(x,y;z) (S(z)\phi_{(k)}) (y) (dy),\ee
an element of $H(z,P_\de, 1,1)$ using \eqref{bgh}, \eqref{bge}.  Using \eqref{bgf}, \eqref{are}
\be\label{bgk} S(z) \theta = S(z)\phi_{(k)}\ee
almost everywhere on $\Ga$, and using \eqref{ard},
\be\label{bgl} (R(z)\theta) = \suml^m_{i = 1} \chi_{\de, x_i} \suml^n_{j=1} \psi_{i, z_jz_l} (z) \intl_\Ga G_j (x,y;z)(S(z)\phi_{(k)} )(y),\; \;  l = 1, \dots, n.\ee

Using \eqref{ari}, \eqref{bgl} in \eqref{bab},
\be\label{bgm} M_0 (\phi_{(k)}; z, P_\de ) \le c \| \triangledown \chi_\de \|^2_{L_2(\Om)} \| S(z)\phi_{(k)}\|^2_{L_2(\Ga)}\ee
so from \eqref{bgm}, \eqref{bcb}
\be\label{bgn} Q_k(z, P_\de ) \le c \| \triangledown \chi_\de\|^2_{L_2(\Om)}.\ee

An example is a Cauchy problem with initial data omitted on $\Sigma_\delta$.

An immediate generalization of example 3 is an estimate of how $Q_0(z,P)$ depends on $P$. Proof is omitted.

\begin{lem} Assume $(z, 0_N)$ admissible  and $P, \hat P \in \calp$ satisfying \eqref{ana}.

Then for any $\chi \in W^{1,\infty} (\Om)$ such that
\be\label{bgp} \chi\midl_\Ga = 1\ee
and such that
\be\label{bgg} \theta \in H(z,\hat P, 1,1) \; \, \hbox{implies \ }  \chi \theta\in H(z,P,1,1)\ee
it follows that
\be\label{bgr} Q_0 (z,P) - Q_0(z,\hat P) \le c \| \triangledown\chi\|^2_{L_2(\Om)}.\ee
\end{lem}

\section{Extension and mollification}

Investigation of the dependence of $Q_0(z,P)$ on $z$ uses extension and mollification of the expressions for $L_0, M_0, Q_0$.

For arbitrary $P'\in \calp$, denote
\be\label{caa} Z(P')\eqadef \{ \theta: \Om \to \bbr^n \big| \;  \| \theta\|_Z < \infty, \; \; P'\theta\mathop{\mid}\limits_{\partial\Om} = 0\}\ee
with
\be\label{cab} \| \theta\|_Z \eqadef\| \theta\|_{W^{1,p} (\Om)} + \| \theta\|_{W^{1,p} (\Ga)}.\ee

We note that $Z (P')$  depends on $z$ through $\Ga$, and that using $p > m$, the elements of $Z(P')$ are continuous and $Z(P')$ is complete.

Then for $\phi \in \Phi_0 (z, P'), \theta \in Z(P'), \Ga' \subseteq \Ga, \, \vare \ge 0$,
\be\label{cac} z' \in L_\infty (\Om\to D),\ee
of the form \eqref{ad}, \eqref{ada},  we denote
\begin{align}\label{cad}
L(\theta, \phi; z', P', \Ga', \vare) &\eqadef \iintl_\Om |R(z') \theta|^2 + \intl_{\Ga'} ( S (z')(\theta - \phi))^2\nonumber \\
&+\intl_{\Ga\setminus \Ga'} ( S(z')\theta)^2 + \vare \| \theta\|^2_Z;\end{align}
\be\label{cae} M(\phi; z', P', \Ga', \vare) \eqadef \underset{\theta\in H(z',P',1,1)\cap Z(P')}{\glb} L(\theta, \phi; z', P', \Ga', \vare);\ee
\be\label{caf} Q( z', P', \Ga', \vare) \eqadef \underset{\phi\in \Phi_0(z,P')}{\lub} \frac{M(\phi; z', P', \Ga', \vare)}{\| \phi\|^2_{W^{1,p}(\Ga')}}.\ee

By inspection,
\be\label{cdd} M_0(\phi; z, P) = M(\phi; z,P,\Ga,0)\ee
for any $\phi \in\Phi_0(z,P)$, and thus
\be\label{cde} Q_0 (z, P) = Q(z, P,\Ga,0).\ee

Using theorem 3.1 and \eqref{cde}, admissible $(z,P)$ is equivalent to
\be\label{cdf} Q(z,P,\Ga,0) = 0.\ee

Setting $\theta = 0$ in \eqref{cad}, \eqref{cae}, \eqref{caf}, we have
\be\label{cdz} Q(z', P', \Ga', \vare) \le c_{z', \Ga'},\ee
independent of $P',\vare$.

Additional features of $M,Q$ so defined are readily determined by inspection.

For $\vare' < \vare$,
\be\label{cba} Q(z', P', \Ga', \vare') \le Q(z', P', \Ga', \vare).\ee

For $\Ga'' \subset \Ga'$,
\be\label{cbb} Q(z', P', \Ga'', \vare) \le Q(z', P', \Ga', \vare).\ee

For $P, \hat P$ satisfying \eqref{ana}
\be\label{cbc} Z(P) \subset Z(\hat P), \; H(z,P,1,1) \subset H(z, \hat P, 1,1),\ee
and thus
\be\label{cbd} M(\phi; z', \hat P, \Ga', \vare) \le M(\phi; z', P,
\Ga', \vare),\ee
\be\label{cbe} Q(z', \hat P, \Ga', \vare) \le Q(z', P,\Ga', \vare).\ee

As $L(\cdot, \phi; z', P', \Ga', \vare)$ is  nonnegative and convex, and $H(z,P',1,1)\cap Z(P')$ is complete, the minimum in \eqref{cae} is achieved.  There exists $\xi_\phi \in Z(P')$, (also depending on $z', \Ga', \vare)$ such that
\be\label{cbf} M(\phi; z',P',\Ga',\vare) = L (\xi_\phi, \phi; z', P', \Ga', \vare)\ee
and satisfying
\be\label{cbg} \vare \| \xi_\phi \|^2_Z \le
Q(z', P', \Ga', \vare) \|\phi\|^2_{W^{1,p} (\Ga')}.
\ee

As $\Phi_0 (z,P')$ is complete, analogously with theorem 3.5, for $\Ga'$ bounded the maximum in \eqref{caf} is achieved within $\Phi_0 (z,P')$.

For $z$ of the form \eqref{ad}, $z'$ satisfying \eqref{cac}, we denote
\be\label{cca} \dist_\Ga (z,z')\eqadef \| z-z'\|_{L_{2p/(p-2)}(\Om)} +  \| [z-z']\|_{L_{2p/(p-2)}(\Ga(z))}.\ee

We observe that while $z'$ is restricted only by \eqref{cac}, $z'$ viewed as an approximation of $z$ can approach $z$ only if $z'$  has jump discontinuities on $\Ga$ associated with $z$ in \eqref{ad}, $\Ga(z)\subseteq \Ga (z')$.

For $\theta \in Z(P')$, from \eqref{af} \eqref{ag}, \eqref{cab}, \eqref{cca}
\be\label{ccb} \| R(z) \theta - R(z') \theta\|_{L_2(\Om\setminus\Ga(z))} \le c \| \theta\|_Z \dist_\Ga(z,z');\ee
\be\label{ccc} \|S(z)\theta - S(z')\theta\|_{L_2(\Ga(z))} \le c \| \theta\|_{W^{1,p}(\Ga)} \dist_\Ga(z,z'),\ee
with constants depending on $\|z\|_{L_\infty (\Om)}, \| z'\|_{L_\infty (\Om)}.$

The use of positive $\vare $ is reflected in the following.

\begin{thm} For $\vare > 0$, $\Ga'\subseteq \Ga(z)$,
\be\label{cd} | Q(z,P', \Ga', \vare) - Q(z', P', \Ga', \vare)\big| \le c \hbox{\ dist}_\Ga(z,z')
\left(1 + \Big(\frac{Q(z',P',\Ga', \vare)}{\vare}\Big)^{1/2}\right)  \ee
with a constant depending on $\| z\|_{L_\infty(\Om)}, \| z'\|_{L_\infty (\Om)}$.
\end{thm}

\begin{proof} With $\Ga'\subseteq \Ga$ determined for $z$ in \eqref{ad},
\begin{align}\label{cda} |Q(z,P',\Ga',\vare)-& Q(z', P', \Ga',\vare)|\nonumber \\
&\le\underset{\phi}{\lub} | M(\phi; z,P',\Ga',\vare) -M(\phi; z', P', \Ga', \vare)|\end{align}
with $\phi \in \Phi_0 (z, P')$ satisfying
\be\label{cdb} \| \phi\|_{W^{1,p}(\Ga')} = 1.\ee

For any $\phi$ such that $M(\phi; z, P', \Ga', \vare) \ge M(\phi; z', P', \Ga', \vare)$, we choose
$\xi_\phi$ satisfying \eqref{cbf}, \eqref{cbg} as stated.  For any $\phi$ such that $M(\phi; z', P', \Ga', \vare) >  M(\phi; z, P', \Ga', \vare)$, we choose $\xi_\phi$ satisfying \eqref{cbf}, \eqref{cbg} with $z$ replacing $z'$.

Then from \eqref{cae}, \eqref{cad}, using \eqref{cda}, \eqref{cdb}, \eqref{ccb}, \eqref{ccc},  \eqref{cbg}
\begin{align}\label{cdc}
 &|Q(z,P',\Ga', \vare)| - Q(z', P', \Ga', \vare)|\nonumber \\
\le & \underset{\phi}{\lub} | L(\xi_\phi, \phi; z,P', \Ga', \vare) - L(\xi_\phi, \phi; z', P', \Ga', \vare)|\nonumber \\
\le & \underset{\phi}{\lub}  | \; \| R(z) \xi_\phi\|^2_{L_2(\Om)} - \| R(z') \xi_\phi\|^2_{L_2(\Om)} \nonumber \\
&\qquad \qquad \; \; + \| S(z)(\xi_\phi - \phi)\|^2_{L_2(\Ga')} - \| S (z')(\xi_\phi - \phi)\|^2_{L_2 (\Ga')}|\nonumber \\
 \le & c \; \underset{\phi}{\lub} \Big( \| R(z)\xi_\phi - R(z')\xi_\phi\|_{L_2(\Om)} + \| S(z) (\xi_\phi - \phi) - S(z') (\xi_\phi - \phi)\|_{L_2 (\Ga')}\Big)\nonumber \\
 \le & c\; \underset{\phi}{\lub} \Big(\| R(z)\xi_\phi - R(z')\xi_\phi\|_{L_2(\Om)} + \| S(z) \phi-S(z') \phi\|_{L_2(\Ga')} + \| S(z) \xi_\phi - S(z') \xi_\phi \|_{L_2(\Ga')}\Big)\nonumber \\
 \le & c\; \underset{\phi}{\lub} \frac{\dist_\Ga (z,z')}{\vare^{1/2}} (\| \xi_\phi\|_Z + \| \xi_\phi - \phi\|_{W^{1,p}(\Om)})  \nonumber \\
 \le & c\; \dist_\Ga(z,z') \underset{\phi}{\lub} (1+\| \xi_\phi \|_Z)\nonumber \\
 \le & c \; \dist_\Ga (z,z') (1+\Big(\frac{Q(z', P', \Ga', \vare)}{\vare}\Big)^{1/2}).\end{align}
\end{proof}

Thus at issue is the limit $\vare\downarrow 0$ in \eqref{cae}, \eqref{caf}.
\begin{lem} For any $\phi \in \Phi_0(z', P'), M(\phi; z', P', \Ga',\cdot)$ is continuous in $\vare$ at $\vare = 0$.
\end{lem}
\begin{proof}  Given $\phi$, we obtain $\xi_\phi$ from \eqref{cbf} with $\vare = 0$.

Then from \eqref{cae}, \eqref{cad}, for any $\vare > 0$,
\be\label{cdfa} M(\phi; z', P', \Ga', \vare) - M(\phi; z', P', \Ga', 0) \le \vare \| \xi_\phi \|^2_Z.\ee

Any condition
\be\label{cdg} \underset{\vare\downarrow 0}{\lim} M(\phi; z', P', \Ga', \vare) - M(\phi; z', P', \Ga', 0) = \eta>0\ee
is contradicted by taking $\vare > 0$ such that
\be\label{cdg} \vare \| \xi_\phi \|^2_Z < \eta.\ee
\end{proof}

\begin{lem} For $\phi, \phi' \in \Phi_0 (z, P')$
\begin{align}\label{ce}
&M(\phi + \phi'; z', P', \Ga', \vare) - M(\phi; z', P', \Ga', \vare)\nonumber \\
\le 2 &M(\phi; z', P', \Ga', \vare)^{1/2} M(\phi'; z', P', \Ga', \vare)^{1/2} +  M(\phi'; z', P', \Ga', \vare).
\end{align}
\end{lem}
\begin{proof} Given $\phi, \phi'$, we obtain $\xi_\phi, \xi_{\phi'} $ from \eqref{cbf}. Then from \eqref{cae}, \eqref{cad}
\begin{align}\label{cea}
&M(\phi + \phi'; z', P', \Ga', \vare) - M(\phi; z', P', \Ga', \vare)\nonumber \\
\le &L(\xi_\phi + \xi_{\phi'}, \phi + \phi'; z', P', \Ga', \vare) - L(\xi_\phi, \phi; z', P', \Ga', \vare)\nonumber \\
= & \| R(z') (\xi_\phi + \xi_{\phi'}) \|^2_{L_2(\Om\setminus \Ga)} - \| R (z')\xi_\phi \|^2_{L_2(\Om\setminus \Ga)} + \| S (z') (\xi_\phi - \phi + \xi_{\phi'} - \phi' \|^2_{L_2(\Ga')}\nonumber \\
 &- \|S(z') (\xi_\phi - \phi)\|^2_{L_2(\Ga')} + \| S(z') (\xi_\phi + \xi_{\phi'}) \|^2_{L_2(\Ga \setminus \Ga')} - \| S (z') \xi_\phi\|^2_{L_2(\Ga\setminus\Ga')}\nonumber \\
 \; \; &+\vare \| \xi_\phi + \xi_{\phi'} \|^2_Z - \vare \| \xi_\phi\|^2_Z.\end{align}

With obvious choice of $u_1, u_2,u_3,u_4$, the right side of \eqref{cea} is of the form
\begin{align}\label{ceb}
\suml^4_{i = 1} ( (&u'_i + u_i)^2 - u_i^2) = 2 \suml^4_{i = 1} u_i u'_i + \suml^4_{i = 1} (u'_i)^2\nonumber\\
&\le 2 (\suml^4_{i = 1} u^2_i)^{1/2} (\suml^4_{i=1} (u'_i)^2)^{1/2} + \suml^4_{i = 1} (u'_i)^2.
\end{align}
which is the right side of \eqref{ce}.
\end{proof}

\begin{thm} Assume $\Ga'$ bounded.  Then $Q(z',P',\Ga', \cdot)$  is continuous with respect to $\vare $ at $\vare = 0$.
\end{thm}
\begin{proof} If not, by application of lemma 4.2, there exists an infinite sequence of values of $\vare$ decreasing to zero, for each of which there exists
\be\label{cfa} \phi_\vare \in \Phi_0 (z, P'), \supp\, \phi_\vare \subseteq \Ga', \; \| \phi_\vare \|_{W^{1,p} (\Ga')} = 1,\ee
such that
\be\label{cfb} \lim\limits_{\vare \downarrow 0} M(\phi_\vare; z', P', \Ga', \vare) > Q(z', P', \Ga', 0).\ee

Using \eqref{cfa} and the assumption of bounded $\Ga'$, taking a subsequence as necessary, there exists
\be\label{cfc} \phi_0 \in \Phi_0 (z, P'), \; \supp\,  \phi_0 \subseteq \Ga', \; \| \phi_0\|_{W^{1,p}(\Ga')} \le 1, \ee
such that
\be\label{cfd} \phi_\vare\overset{\vare \downarrow 0}{\longrightarrow } \phi_0\ee
uniformly on $\Ga'$.

For each $\vare$, we apply lemma 3.4 for each $k$ (such that $\Ga'$ intersects $\Ga_k$) with
\be\label{cfe} \phi_{(k)} = (\phi_\vare - \phi_0)\midl_{\Ga_k},\ee
obtaining
\be\label{cff} \theta_\vare = \suml_k \theta_{(k)}\ee
so determined.  We thus have $\theta_\vare \in Z(P')$, coinciding with $\phi_\vare - \phi_0$ on $\Ga'$, vanishing on $\Ga \setminus\Ga'$, satisfying
\be\label{cfg} \| \theta_\vare \|_{W^{1,p}(\Om\setminus\Ga)} \overset{\vare \downarrow 0}{\longrightarrow } 0\ee
from \eqref{bdc}, \eqref{cfd}, \eqref{cfa}, and
\be\label{cfh} \| \theta_\vare\|_{W^{1,p}(\Om')} \le 2\ee
from \eqref{cfa}, \eqref{cfc}.

From \eqref{cad}, \eqref{caa},  using \eqref{cfg}, \eqref{cfh},
\be\label{cfi} L(\theta_\vare, \phi_\vare - \phi_0; z', P', \Ga', \vare) \overset{\vare\downarrow 0}{\longrightarrow} 0,\ee
so from \eqref{cae}
\be\label{cfj} M(\phi_\vare - \phi_0; z',P', \Ga', \vare) \overset{\vare\downarrow 0}{\longrightarrow} 0.\ee

By application of lemma 4.3, using \eqref{cfj}, and then lemma 4.2,
\begin{align}\label{cfk} \lim\limits_{\vare\downarrow 0} M(\phi_\vare; z', P', \Ga', \vare) &= \lim\limits_{\vare\downarrow 0} M(\phi_0; z', P', \Ga',\vare)\nonumber \\
&= M(\phi_0; z', P', \Ga', 0)\nonumber \\
&\le Q(z', P', \Ga', 0)\end{align}
using \eqref{caf}, \eqref{cfc} in the last step.  The conclusion \eqref{cfk} is incompatible with \eqref{cfb}.

\end{proof}

\section{Limit admissible sequences}

Section 3.1 may be interpreted as examples of sequences of inadmissible pairs $\{ (z_\de, P_\de)\}$, with admissible and potentially unambiguous  limits as $\de\downarrow 0$.  Such are generalized using the results of section 4.

The above examples and \eqref{cca} tacitly assume each $z_\de$ also of form \eqref{ad}, with some $\Ga_\de$ such that $\Ga$ coincides with $\Ga_\de$ or a subset thereof.  This is extended by assumption of each $z_\de$ also of the form \eqref{ad}, \eqref{ada}, with
\be\label{da} \Ga_\de = \mathop{\cup}\limits^{K_\de}_{k = 1} \Ga_{\de k},\ee
some $K_\de \ge K$.

For $k \le K$, we assume $\Ga_k, \Ga_{\de k}$ related by the existence of
\be\label{daa}  T_\de \in C^1(\Om\to \Om), \; \| T_\de - I_m\|_{C^1(\Om)} \overset{\de\downarrow 0}{\longrightarrow} 0\ee
such that \be\label{dab} \Ga_{\de k} = T_\de \Ga_k,\; \; k = 1,\dots, K.\ee

Thus for any bounded $\Om'\le \Om $ and any $\de > 0$,
\be\label{dac} \Ga \cap \Om'\subseteq T^{-1}_\de (\Ga_\de \cap (T_\de \Om')).\ee

We assume $z_\de$ satisfying
\be\label{dad} \| z_\de\|_{L_\infty (\Om)} \le c,\ee
with a constant independent of $\de$, and
\be\label{dae} \dist_\Ga (z, z_\de \circ T_\de)  \overset{\de\downarrow 0}{\longrightarrow} 0,\ee
with $\dist_\Ga$ defined in \eqref{cca} and
\be\label{daj} (z_\de \circ \Ga_\de) (x) \eqadef z_\de (\Ga_\de x).\ee

Finally, we assume
\be\label{daf} T_\de \in C (\partial\Om\to \partial\Om)\ee
and $P_\de$ such that there exists a limit
\be\label{dag} \lim\limits_{\de\downarrow 0} P_\de \circ T^{-1}_\de \in \calp,\ee
as determined from
\be\label{dah} (P_\de \circ T^{-1}_\de) (\theta \circ T_\de) = P_\de \theta.\ee

Then from \eqref{daa}, \eqref{dac}, \eqref{daj}, \eqref{dag}, \eqref{dah}, for any $\de > 0, \vare \ge 0$,
\begin{align}\label{dai} \lim\limits_{\de\downarrow 0} Q\Big(z_\de, P_\de, &\Ga_\de \cap (T_\de \Om'), \vare\Big)\nonumber \\
&= \lim\limits_{\de\downarrow 0} Q \Big(z_\de\circ T_\de, P_\de \circ T^{-1}_\de, T^{-1}_\de (\Ga_\de \cap (T_\de \Om')), \vare \Big),\end{align}
as the left-hand terms in \eqref{dai} correspond to the right-hand terms modulo expressions disappearing in the limit $\de\downarrow 0$ using \eqref{daa}.

\begin{defn} A sequence $\{ (z_\de, P_\de)\}$ is limit admissible if
\be\label{db} \lim\limits_{\vare\downarrow 0} (\lim\limits_{\de\downarrow 0} Q(z_\de, P_\de, \Ga_\de \cap \Om'', \vare)) = 0\ee
for any bounded $\Om'' \subseteq \Om$.
\end{defn}

Suitable limits of limit admissible sequences are admissible and potentially unambiguous.
\begin{thm}
Assume $\{ (z_\de, P_\de)\}$ limit admissible; $\{ z_\de\}$ satisfying \eqref{dad} with a limit $z$, of form \eqref{aac}, \eqref{ad}, \eqref{ada}, such that \eqref{dab}, \eqref{dae} hold; and $\{ P_\delta \}$ such that the limit \eqref{dag} exists, for some $\{ T_\delta \}$ satisfying \eqref{daa}, \eqref{daf}.

Then $(z, P)$ is admissible for any $P \in \calp$ such that
\be\label{dba} \ker  (\lim\limits_{\de\downarrow 0}  (P_\de \circ T^{-1}_\de) \subseteq \ker P.\ee
\end{thm}
\begin{proof}  For any bounded $\Om'\subseteq \Om$, we choose a bounded $\Om''$ such that
\be\label{dbb} \{T_\de x \,|\, x \in \Om'\} \subseteq \Om''\ee
for all $\de > 0$.

Then
\begin{align*} \lim\limits_{\vare\downarrow 0} ( &\lim\limits_{\de\downarrow 0} Q(z_\de, P_\de, \Ga_\de \cap \Om'', \vare))\\
&\ge \lim\limits_{\vare\downarrow 0}(\lim\limits_{\de\downarrow 0} Q(z_\de, P_\de, \Ga_\de \cap (T_\de\Om'),\vare)\end{align*}
using \eqref{cbb}, \eqref{dbb};
\[ = \lim\limits_{\vare\downarrow 0}( \lim\limits_{\de\downarrow 0} Q(z_\de\circ T_\de, P_\de \circ T^{-1}_\de, T^{-1}_\de (\Ga_\de \cap (T_\de\Om'), \vare)\]
using \eqref{dai};
\[ \ge \lim\limits_{\vare\downarrow 0}( \lim\limits_{\de\downarrow 0} Q(z_\de\circ T_\de, P_\de \circ T^{-1}_\de, \Ga\cap \Om', \vare)\]
using \eqref{cbb}, \eqref{dac};
\[ \ge \lim\limits_{\vare\downarrow 0}( \lim\limits_{\de\downarrow 0} Q(z_\de\circ T_\de, P,  \Ga\cap \Om', \vare)\]
using \eqref{cbe}, \eqref{dba}, \eqref{ana};
\[ = \lim\limits_{\vare\downarrow 0}Q(z,P, \Ga\cap\Om',\vare)\]
using theorem 4.1;
\be\label{dbc} = Q(z,P,\Ga\cap \Om', 0)\ee
using theorem 4.4.

Thus \eqref{db} implies \eqref{cdf} and $(z,P)$ admissible.

\end{proof}

Reversing the argument produces a converse statement, suggesting that given $(z,P)$ admissible, limit admissible sequences with limit $(z, P)$ are realistic expectations.

\begin{thm} Assume $(z, P)$ admissible, and a sequence $\{ (z_\de, P_\de, \Ga_\de)\}$ satisfying
\eqref{da}, \eqref{daa}, \eqref{dab}, \eqref{dad}, \eqref{dae}, \eqref{daf}, \eqref{dag}, and  \eqref{dac} with equality, and
\be\label{dc}\ker P \subseteq \lim\limits_{\bar{\de}\downarrow 0} \mathop{\cap}\limits_{\de \le \bar{\de}} \ker (P_\de \circ T^{-1}_\de).\ee

Then $\{ (z_\de, P_\de)\}$ is limit admissible.
\end{thm}
\begin{proof} Given any bounded $\Om'' \subseteq \Om$ we choose bounded $\Om'$ such that
\be\label{dca} \Om'' \subseteq T_\de \Om'\ee
for all $\de > 0$.

Then from \eqref{cdf}, using theorem 4.4
\begin{align*} 0 &= \lim\limits_{\vare \downarrow 0} Q (z, P, \Ga \cap \Om', \vare)\\
&= \lim\limits_{\vare \downarrow 0} ( \lim\limits_{\de \downarrow 0} Q(z_\de \circ T_\de, P, \Ga \cap \Om', \vare))
\end{align*}
using theorem 4.1;
\[ = \lim\limits_{\vare \downarrow 0} (\lim\limits_{\de \downarrow 0}
Q(z_\de \circ T_\de, P_\de \circ T^{-1}_\de, \Ga \cap \Om', \vare)) \]
using \eqref{cbe}, \eqref{dc};
\[ = \lim\limits_{\vare \downarrow 0} (\lim\limits_{\de \downarrow 0}
Q(z_\de \circ
 T_\de, P_\de \circ T^{-1}_\de,
 T^{-1}_\de(\Ga_\de\cap(T_\de \Om')), \vare))\]
  using \eqref{dac} with equality;
  \[=\lim\limits_{\vare\downarrow 0} ( \lim\limits_{\de\downarrow 0} Q(z_\de, P_\de, \Ga_\de \cap (T_\de \Om')), \vare))\]
  using \eqref{dai};
  \be\label{dd} =\lim\limits_{\vare\downarrow 0} ( \lim\limits_{\de\downarrow 0} Q(z_\de, P_\de, \Ga_\de \cap \Om'', \vare))\ee
  using \eqref{cbb}, \eqref{dca}, thus establishing \eqref{db}.
\end{proof}

In the special case that
\be\label{dda} z_\de = z\ee
independent of $\de$, the condition \eqref{db} can be simplified using theorem 4.4.

\begin{cor} Assume \eqref{dda}, that
\be\label{ddb} \lim\limits_{\de'\downarrow 0} \mathop{\cup}\limits_{\de < \de'} \ker P_\de \subseteq \ker P,\ee
and that for any bounded $\Om''$,
\be\label{ddc} \lim\limits_{\de\downarrow 0} Q (z, P_\de, \Ga (z) \cap \Om'', 0) = 0.\ee

Then $(z, P)$ is admissible.
\end{cor}

Use of theorem 4.1 and \eqref{cdz} to simplify \eqref{db} results in the following.

\begin{cor} Assume $P_\de = P$ independent of $\de$ in \eqref{db}, that \eqref{dad}, \eqref{dae} hold, and that for any bounded $\Om''\subseteq \Om$
\be\label{ddf} \lim\limits_{\de\downarrow 0} Q(z_\de, P, \Ga(z_\de) \cap \Om'', 0) = 0.\ee
Then $(z, P)$ is admissible.
\end{cor}

\section{Interpretation of computations}

We now relate the above discussion to the results of typical computational investigations.  Results from this section will be used in section 8 in formulating a weakened definition of well-posed problems, and in section 13 in construction of a class of approximation schemes.

A system \eqref{aa} is determined by $m, n, \psi_1,\dots, \psi_m, D$.  Without assuming hyperbolicity, a boundary-value problem (including Cauchy problems and initial-boundary value problems) is determined by additional specification of $\Om \subset \bbr^{m} $ open and connected (not necessarily  bounded or simply connected), a specific element $P_\cals\in \calp$ determining the form of the boundary conditions, $D_\cals \subset D$ open and bounded, a set of available boundary data $\tilde B_\cals \subset \ker P_\cals$, and whatever $\{ e_0\}$ in \eqref{agea}.

Posited is the existence of a corresponding solution set $\tilde \cals$, with elements
\be\label{faa} z:\Om\to D_\cals,\ee
of the form \eqref{ad}, \eqref{ada}, satisfying a suitable entropy condition, with $\Ga$ depending on $z$,  $K$ depending on $\tilde \cals$, satisfying \eqref{ac}, \eqref{aca} with $P=P_\cals$ and some  $b\in \tilde B_{P_\cals}$, available a priori.  Indeed, $\tilde \cals$ is restricted implicitly, using \eqref{ab}, by
\be\label{fab} \tilde B_{P_\cals} \subseteq \{ b(z, P_\cals), \; z \in \tilde \cals\}.\ee

The set $\tilde \cals$ is further restricted by the  adopted entropy condition,  attempting to make $\tilde \cals$ isomorphic to $\tilde B_{P_\cals}$.

For some specified higher order dissipation operator $\cald$, we seek the elements $ z \in \tilde \cals$ obtained as $h\downarrow 0$ limits of ``viscous approximations"
\be\label{fac} \tilde z_h \in C(\Om \to D_\cals)^n,\ee
satisfying
\be\label{fad} \suml^{m}_{i=1}  ( \psi^\dag_{i, z} (\tilde z_h))_{x_i} = h\cald(\tilde z_h)\ee
within $\Om$, and partially specified  boundary data, reflecting \eqref{ab}, of the form
\be\label{fae} (I_n-P_\cals) \psi^\dag_{\nu,z} (\tilde z_h\underset{\partial\Om}{\Big|} ) = b\ee
 together with \eqref{agea} (with $z = \tilde z_h$) on $\partial\Om$.

 Throughout, we assume that with $q(z)$ obtained from \eqref{aaea}, $z\in \tilde{\cals}$ so obtained satisfies the familiar entropy inequality \cite{L2}
\be\label{faf} \triangledown  \cdot  q(z) \le 0\ee
weakly within $\Om$.  This choice of entropy condition has
attractive features: it does not require hyperbolicity or precise specification of $\cald$; it is easily incorporated in discretization schemes; satisfied with equality where $z$ is continuous, it is easily verified a posteriori.

It is speculated that a well-posed problem results, in the sense that the adopted entropy condition \eqref{faf}, $P_\cals$ and $\{ e_0\}$  are such that $\tilde \cals$ is indeed isomorphic to $\tilde B_{P_\cals}$, in particular that there exists a continuous  mapping $\cala_\cals:\tilde B_{P_\cals} \to \tilde \cals$ satisfying
\be\label{fag} \cala_\cals(b(z, P_\cals)) = z\ee
for all $z \in \tilde \cals$,  using \eqref{ab}.

In the present framework, a (traditionally) well-posed problem is associated with $P_\cals$ for which $\cala_\cals $ is Frechet differentiable with respect to $b$, and such that
\be\label{pb} \tilde \cals \subseteq \{ z \, | \,(z, P_\cals) \; \hbox{unambiguous} \}\ee

A computational investigation tacitly incorporates a further assumption,  that any such  mapping $\cala_\cals$ is found as the $h \downarrow 0$ limit of computable approximations $z_h \in \cals$,
\be\label{fah} \cals \eqadef \{ z:\Om \to D_\cals\},\ee
$z$ of the form \eqref{ad}, \eqref{ada},  with $\Ga, K$ depending on $z$.

Approximation schemes
\be\label{fai} \cala_h:  \tilde B_{P_\cals} \to \cals, \; h > 0\ee
are constructed, depending on the following: \
$h$ as in \eqref{fad}, also denoting the discretization  parameter;
$\cald$  in \eqref{fad}, conveniently selected to obtain \eqref{faf} a posteriori;
\be\label{faj} U_\cald :N\times \calp \to N\ee
determining additional boundary conditions needed to recover uniqueness in \eqref{fad}, \eqref{fae}, with $\cald$ of the higher order in particular; and
\be\label{fak}\calf_{h,b} : C(C(\Om\to D_\cals)) \to \cals\ee
a convenient ``shock fitting" mapping to recover approximations $z_h \in\cals$ each of the form \eqref{ad}, with $\Ga(z_h)$ but not $K$ generally depending on $h$.

Then given $b \in \tilde B_{P_{\cals}}$, we compute $\tilde z_h(b)$ satisfying (a suitable discretization of)
\eqref{fad}, \eqref{fae}, \eqref{agea},
simultaneously with
\be\label{fal} U_\cald(\tilde z_h, P_\cals) = 0 \ee
on $\partial\Om$.  The mapping $U_\cald $ is chosen, typically depending on $\nu \cdot \nabla z$, so that \eqref{fal} is generally lost in the limit $h\downarrow 0$.

Then our approximation of $z$ satisfying \eqref{fag} is
\be\label{fam} z_h(b) = \calf_{h,b} (\tilde z_h(b)).\ee

We denote two mappings on $\tilde B_{P_\cals}$  below
\be\label{fan} \tilde \cala_h (b)\eqadef \tilde z_h(b), \; \; \cala_h(b) \eqadef z_h(b).\ee

The ``shock fitting" procedure \eqref{fam}, \eqref{fak} is not computationally necessary, in the sense that for any $ b \in \tilde B_{P_\cals}$, in a sufficiently weak topology,
\be\label{fana} \lim\limits_{h\downarrow 0} \cala_h(b) = \lim\limits_{h\downarrow 0} \tilde \cala_h (b) \,\eqadef\, \cala_0 (b).\ee

The step \eqref{fam} facilitates comparison of $z_h \in \cals$ with $z\in \tilde \cals$ satisfying \eqref{fag}, and permits discussion of stability of the approximations $(z_h, P_\cals)$ in theorem 6.4 below.

\subsection{Interpretation of obtained results}

For such schemes $\cala_h$, positive empirical results admit interpretation using the above analysis.  Motivated by \eqref{cca}, we introduce a distance function~on~$\cals$,
\begin{align}\label{ga} &\dist_\cals(z,z'; w, w_\Ga)\eqadef \underset{T_\de}{\glb} \Big( \| w^{1/2} (z-z'\circ T_\de)\|_{L_{2p(p-2)}(\Om)}\\
+&\| w^{1/2}_\Ga [z-z'\circ T_\de]\|_{L_{2p/(p-2)}(\Ga(z))} + \| T_\de - I_m\|_{C^1(\Om)} + \| T_\de - I_m\|_{L_2(\Om)}\Big)\nonumber\end{align}
for any $z, z'\in \cals$, $w, w_\Ga$ satisfying \eqref{aga}, \eqref{agb}, $T_\delta $ satisfying
\be\label{ffa} T_\de\Ga(z) \subseteq \Ga(z').\ee

Several statements are immediate from \eqref{ga}.

For $T_\delta $ satisfying \eqref{daa}, \eqref{daf}, $\theta \in Z(P)$, it follows that  $ \theta \circ T^{-1}_\de \in Z(P\circ T^{-1}_\de)$
and
\be\label{gaa} | \, \| \theta \circ T^{-1}_\de  \|_Z - \| \theta\|_Z | \, \le c( \| T_\de - I_m\|_{C^1(\Om)} + \| T_\de - I_m\|_{L_2(\Om)}) \| \theta\|_Z,\ee
\be\label{gab} | \, \| \theta \circ T^{-1}_\de \|_{\partial\Om, P \circ T^{-1}_\de} - \| \theta \|_{\partial\Om, P} | \le c \| T_\de - I_m\|_{L_\infty (\partial\Om)} \|\theta\|_{\partial\Om, P}.\ee

For $\theta \in H(z,P,w,w_\Ga) \cap Z(P)$,
\begin{align}\label{gac} \| w^{1/2} &|R(z) \theta - R(z'\circ T_\de)\theta | \, \|_{L_2 (\Om\setminus \Ga(z))}+\| w_\Ga^{1/2} (S(z)\theta - S (z'\circ T_\de)\theta)\|_{L_2 (\Ga(z))}\nonumber \\
& \le c \,\dist_\cals (z, z'; w, w_\Ga)\| \theta\|_Z,\end{align}
so
\be\label{gad} | \, \|\theta\|_{z, w,w_\Ga} - \| \theta\circ T^{-1}_\de\|_{z', w\circ T^{-1}_\de, w_\Ga \circ T^{-1}_\de} | \le c \, \dist_\cals (z,z'; w, w_\Ga) \; \| \theta \|_Z.\ee

From \eqref{gad}, for any $z, z', w, w_\Ga$ such that \eqref{ffa} holds with equality and   $\dist_\cals (z, z'; w, w_\Ga)$  is finite,
\be\label{gae} H(z, P, w, w_\Ga) \cap Z(P) = H(z'\circ T^{-1}_\de, P\circ T^{-1}_\de, w\circ T^{-1}_\de, w_\Ga \circ T^{-1}_\de) \cap Z(P \circ T^{-1}_\de).\ee


The following is well-known, restated here in present notation.

\begin{lem} For $b,w,w_\Ga$ fixed, assume that the sequence $\{ z_h\}$ obtained in \eqref{fam} satisfies \eqref{fae} and is Cauchy with respect to $\dist_\cals$, so  that there exists $z_0 \in \cals$ such that
\be\label{fb} \dist_\cals  (z_0, z_h; w, w_\Ga) \overset{h\downarrow 0}{\longrightarrow} 0.\ee

Assume further that
\be\label{fbb} h \cald (\tilde z_h) \overset{h\downarrow 0}{\rightharpoondown} 0\ee
weakly in the dual space  $X^{{}^{\ast}}_{P_\cals}$, and that for $i = 1,\dots, m$, any $ r < \infty, $
\be\label{fba} \| \psi_{i, z} (\tilde z_h) - \psi_{i, z} (z_h) \|_{L_r(\Om)} \overset{h\downarrow 0}{\longrightarrow} 0.\ee

Then $z_0 $  satisfies \eqref{ac} for all $\theta \in X_{P_\cals}$.
\end{lem}
\begin{rem*} The assumption \eqref{fba} implies  a condition on $\calf_{b, h}$ in \eqref{fak}, maintaining \eqref{fana}.

With $D_\cals$ bounded, convergence \eqref{fba} in $L_1 (\Om)$ implies convergence in $L_{p'}(\Om)$ for any finite $p'$.

For a particular case, we likely  regard  \eqref{fb}, \eqref{fbb}, \eqref{fba} as established beyond reasonable doubt by empirical results.
\end{rem*}
\begin{proof}  Using \eqref{fad}, \eqref{fae}, establishment  of \eqref{ac} is immediate from \eqref{fb}, \eqref{fbb}, \eqref{fba}.
\end{proof}

From \eqref{fb}, \eqref{ga}, \eqref{fba}, with $z_0, z_h$ satisfying \eqref{ad},
\be\label{fbe} \| [z_0] - [z_h\circ T_h]\|_{L_1(\Ga(z_0))} \overset{h\downarrow 0}{\longrightarrow} 0.\ee

For \eqref{fb}, \eqref{fba}, \eqref{fbb} regarded as established for all $b \in \tilde B_{P_\cals}$, perhaps with $\hat B_{P_\cals}$ suitably restricted, we designate
\be\label{fbc} \cala_0:\tilde B_{P_\cals} \to \hat \cals_0\ee
as determined from \eqref{fana},
\be\label{fbd} \cala_0(b) \eqadef z_0 (b), \; \, b \in \tilde B_{P_\cals}\ee
as our candidate for $\cala_\cals$ in \eqref{fag}.

In \eqref{fbc}, \eqref{fbd}, $\hat\cals_0$ is the solution subset of \eqref{ac}, \eqref{aca} so obtained, for whatever adopted approximation scheme $\cala_h$.  Existence of a well-posed problem, in the sense that \eqref{fag} holds, implies a conjecture
\be\label{fbda} \hat\cals_0 = \tilde \cals,\ee
irrespective of $\cala_h$.

In this context, identification of $\cala_0$ as $\cala_\cals$ remains uncertain; the condition \eqref{fbda} is tentative at best.
If empirical evidence establishes \eqref{fbc} beyond reasonable doubt, it does not establish uniqueness.  In particular, $\cala_h$ given in \eqref{fan} necessarily depends on $\cald$ in \eqref{fad}, $U_\cald $ in \eqref{fal}, and whatever discretization of \eqref{fad}. There is no assurance that such dependence disappears in the limit $h\downarrow 0$.  In general there is also no a priori assurance of uniqueness of the solution of \eqref{fad}, \eqref{fae}, \eqref{fal}, \eqref{agea}.

Additionally, we may recall that \eqref{fag} is not necessary for convergence of $z_h$ as $h\downarrow 0$.  The space
\be\label{fhb} \underline{\cals} \eqadef \{ z \in L_\infty (\Om)^n \, |\,z_{j, x_i}\; \hbox{locally\  bounded\ measures,\ } i = 1, \dots, m, \;   j = 1, \dots, n\}\ee
is precompact \cite{D}, lemma 15.2.1, in $L^n_{p', loc}$ for any finite $p'$.  Thus a bound
\be\label{fha} \| z_h\|_{\underline{\cals}} \le c_b\ee
independent of $h$ will result in a subsequence convergent in $L^n_{p', loc}$, presumably achieving \eqref{fb} in any given case.

Crude bounds on $z_h$ are obtained from the form of the boundary conditions \eqref{fae} and the entropy inequality \eqref{faf}. As is well-known, each entropy flux component $q_i$ is a function only of the corresponding fluxes $\psi_{i, z_1}, \dots, \psi_{i, z_n}$.

We presume $P_\cals$ chosen to achieve
\begin{align}\label{fhc}
  \nu\cdot q ( \psi^\dag_{\nu, z} (z_h))
&= \nu\cdot q ( b + P_\cals \psi^\dag_{\nu,z} (z_h))\nonumber \\
& \ge - c_b
\end{align}
independent of $h$, using \eqref{ab}. Such is expected, for example, choosing $P_\cals$ to satisfy \eqref{fdd} below \cite{S2}, throughout $\p\Om$.

Then using \eqref{aaea}, and a partial integration
\begin{align}\label{fhz} \iintl_\Om &\ul{\Theta} z_h \cdot \big( \suml^m_{i = 1} \psi^\dag_{i, z} (z_h)_{x_i}\big) = \iintl_\Om \ul{\Theta} \triangledown \cdot q (\psi^\dag_{\cdot, z} (z_h))\nonumber \\
&= \intl_{\pOm} \uTh \nu \cdot q (\psi^\dag_{\nu,z} (z_h) - \iintl_\Om\triangledown \uTh \cdot q (\psi^\dag_{\cdot, z} (z_h); \end{align}
  using \eqref{faf}, assumed applying also to $z_h$, and \eqref{fhc}  gives a bound on $z_h$, independent of $h$, of the form
\be\label{fhd} - \iintl_\Om \triangledown { \uTh} \cdot q (\psi^\dag_{\cdot, z} (z_h)) \le c_b \| {\ul{\Theta}}\|_{L_1 (\pOm)}\ee
for any nonnegative sufficiently smooth  scalar function $\underline{\Theta}$.

The condition \eqref{fhc} may be relaxed in the not infrequent if special case that $P_\cals$ is of the form
\begin{align}\label{fhca} P_\cals (x) &= \begin{cases} 0, \; \; &x\in \pOm_0\\ I_n, &x\in \pOm_I\\
I_n - &\suml_{\{ e_0\}} e_0 (x) e^\dag_0 (x), \; \; x \in \pOm_e\end{cases} \nonumber \\
\pOm &= \pOm_0 \cup \pOm_I \cup \pOm_e.\end{align}

Then in \eqref{fhz}, \eqref{fhd}, we restrict $\uTh$
\be\label{fhcb} \uTh (x) = 0, \; \; \, x \in \pOm_I \cup \pOm_e,\ee
whence
\be\label{fhcc} \theta = \uTh \tilde z_h\ee
satisfies $P_\cals \theta = 0$ throughout $\pOm$ using \eqref{fhca}, and \eqref{fhc} holds trivially within $\pOm_0$.  Further details omitted.

The dissipation included in \eqref{fad} is typically achieves
\be\label{fhe} h \| \triangledown \tilde z_h\|^2_{L_{2, loc}} \le c\ee
depending on $\| z_h\|_{L_\infty}$. In any particular case, \eqref{fhd}, \eqref{fhe} or a weakened form thereof may suffice to imply \eqref{fha}. Indeed, anticipation  of a bound \eqref{fha} may be understood as motivation for adoption of the vanishing dissipation limit in \eqref{fad}.

Under seemingly mild assumptions, the mapping $\cala_0$ is Frechet differentiable, making the above discussion applicable to $z_0$ as obtained from \eqref{fbd}.


Denoting Frechet derivatives by $d$ throughout, we assume $\cald, \calf_{h,b}, U_\cald $ as appearing in \eqref{fad}, \eqref{fak}, \eqref{fal}, respectively, Frechet differentiable with respect to $z$.

For $b \in \tilde B_{P_\cals}, \dot b \in B_{P_\cals}$ given in \eqref{ahd}, linearization of
\eqref{fad}, \eqref{fae}, \eqref{fal} determines
\be\label{fbg} \dot{\tilde z}_h = d \tilde \cala_h(b) \dot b\ee
satisfying
\be\label{fbh} \intl_{\partial\Om} \dot b \cdot \theta = \iintl_\Om \dot{\tilde z}_h \cdot R(\tilde z_h) \theta + h\iintl_\Om(d\cald(\tilde z_h) \dot{\tilde z}_h)\cdot \theta\ee
for all $\theta \in X_{P_\cals}$, with supplemental boundary conditions on $\partial \Om$,
\be\label{fbi} dU_\cald (\tilde z_h, P_\cals)\dot{\tilde z}_h = 0.\ee

Linearization of \eqref{fam}, using \eqref{fan}  determines
\be\label{fbj} d\cala_h(b) \dot b = d\calf_{b,h} (\tilde z_h) \dot{\tilde z}_h = (\dot z_h, \sigma_h)\ee
satisfying
\be\label{fbk} \iintl_\Om \dot{\tilde z}_h\cdot R(\tilde z_h) \theta - \iintl_{\Om\setminus\Ga(z_h)} \dot z_h \cdot R(z_h) \theta - \intl_{\Ga (z_h)} \sigma_h S(z_h)\theta\overset{h\downarrow 0}{\longrightarrow} 0\ee
for any $\theta \in X_{P_\cals}$.

Assuming that \eqref{fb} holds,  taking weak limits $\dot z_h\rightharpoondown \dot z_0$,  $\sigma_h \rightharpoondown \sigma_0$ as $h\downarrow 0$, we have
\be\label{fbl} d\cala_0 (b) \dot b = (\dot z_0, \sigma _0)\ee
satisfying
\begin{align}\label{fbm} \iintl_{\Om\setminus \Ga (z_0)} ( (\dot z_h \circ T_h)\cdot &R( z_h \circ T_h)\theta - \dot z_0 R(z_0) \theta)\nonumber \\
&+\intl_{\Ga(z_0)} ( (\sigma_h\circ T_h) S(z_h\circ T_h)\theta - \sigma_0 S(z_0)\theta) \overset{h\downarrow 0}{\longrightarrow} 0\end{align}
for all $\theta \in X_{P_\cals}$, with $T_h$
as determined from \eqref{fb}, \eqref{ga} (with $h$ replacing~$\delta, \,  z_0$ replacing $z, \,  z_h$ replacing $z'$).

\begin{lem} For fixed $b,  w, w_\Ga$ assume that \eqref{fb}, \eqref{fbb}, \eqref{fba} hold, and for any $\dot b \in B_{P_\cals}$, assume $\dot{\tilde z}_h \in (C\cap L_1)(\Om)$ obtained from \eqref{fbh}, \eqref{fbi} and satisfying
\be\label{fbo} h d\cald (\tilde z_h) \dot{\tilde z}_h \overset{h\downarrow 0}{\rightharpoondown} 0 \ee
weakly in $X^\ast_{P_\cals}$.

Assume $\dot z_h\in L_1 (\Om\setminus \Ga(z_h))$, $\sigma_h[z_h] \in L_1 (\Ga (z_h))$ satisfying \eqref{fbk}, and $\dot z_0 \in L_1 (\Om \setminus \Ga(z_0))$, $\sigma_0 [z_0] \in L_1(\Ga(z_0))$ satisfying \eqref{fbm}.

Then $\dot b, z_0, \dot z_0, \sigma_0$ satisfy \eqref{ae}.
\end{lem}

\begin{rem*}A familiar example in which \eqref{fbb}, \eqref{fbo} are established is
\be\label{fbp} \cald(z) = \triangle z, \; U_\cald (z, P) = P \nu \cdot \triangledown z.\ee

In this case, energy estimates for \eqref{fad}, \eqref{fae}, \eqref{fal} and for \eqref{fbh}, \eqref{fbi} establish
\[ h\| \triangledown \tilde z_h\|^2_{L_2(\Om)}, \; h\| \triangledown \dot{\tilde z}_h\|^2_{L_2(\Om)}\]
bounded independently of $h$, from which \eqref{fbb}, \eqref{fbo} follow easily.
\end{rem*}
\begin{proof}  The proof is immediate, using \eqref{fbo} in \eqref{fbh}, then \eqref{fbk}, taking the limit $h\downarrow 0$ using \eqref{fbm}.
\end{proof}

Stability of $(z_0, P_\cals)$ and $(z_h, P_\cals)$ are anticipated if not guaranteed.

Lack of ``boundary layers" in $\tilde z_h$, as implied by $z_h$ satisfying \eqref{fae}, is compelling indication that the boundary data is not over-specified and thus that $(z_0, P_\cals)$ is stable.  Sufficient requirements on $\| \cdot \|_{\partial\Om, P_\cals}, w, w_\Ga$ for stability are determined by the following.

\begin{thm}
Assume $z_0, P_\cals$ such that lemma 6.2 holds, and existence of a norm $\| \dot b \|_{B,P_\cals}$ on $B_{P_\cals}$ in \eqref{ahd}, $w$ satisfying \eqref{aga}, $w_\Ga$ satisfying \eqref{agb} such that for any $\dot b \in B_{P_\cals}$, $\dot z, \sigma$ obtained from \eqref{ae} for all $\theta \in X_{P_\cals}$ satisfy
\be\label{fbq} \iintl_{\Om \setminus \Ga(z_0)} \tfrac 1w |\dot z|^2 + \intl_{\Ga(z_0)} \tfrac{1}{w_\Ga}\sigma^2 \le c_1 (z_0, P_\cals, w, w_\Ga)^2 \| \dot b \|^2_{B, P_{\cals}}.\ee

Then $(z_0, P_\cals)$ is stable; the conditions \eqref{agd}, \eqref{ahb} hold with
\be\label{fbqa} \| \theta\|_{\pOm, P_\cals} \eqadef \, \underset{\dot b \in B_{P_\cals}}{\lub} \; \frac{\intl_{\pOm} \dot b \cdot \theta}{\| \dot b \|_{B, P_\cals}}.\ee
\end{thm}
\begin{rem*} In \eqref{fbq}, the specific choice of $\dot b$ affects the obtained values of $w, w_\Ga, c_1$.  Here the regularity of $\dot b$ becomes relevant.

Stability of $(z_0, P_\cals)$ may be inferred by solvability of \eqref{aiz}, \eqref{ai} (with $\phi_\Ga = 0$) for $\dot b$ in a finite-dimensional space $B_{P_\cals}$, subsequently obtaining $w, w_\Ga, c_1$ from \eqref{fbq} and $\| \theta\|_{\pOm, P_\cals}$ from \eqref{fbqa}.
\end{rem*}
\begin{proof} For all $\theta \in X_{P_\cals}$, $\dot b \in B_{P_\cals}$, from \eqref{ae}, \eqref{ah}, \eqref{fbq}
\begin{align}\label{fbr}  \frac{
\intl_{\pOm} \dot b \cdot \theta}{\| \dot b \|_{B, P_\cals}} &=
\frac{\iintl_{\Om\setminus\Ga(z_0)} \dot z \cdot R(z_0)\theta + \intl_{\Ga(z_0)}\sigma S(z_0)\theta}{\| \dot b \|_{B, P_\cals}}
\nonumber \\
&\le \frac{
(\iintl_{\Om\setminus\Ga(z_0)} \frac 1 w (\dot z)^2 + \intl_{\Ga(z_0)} \frac {1}{w_\Ga} \sigma^2)^{1/2} \| \theta\|_{z_0,w,w_\Ga}}
{\| \dot b \|_{B, P_\cals}}
\nonumber \\
&\le c_1 (z_0, P_\cals, w, w_\Ga) \|\theta\|_{z_0,w,w_\Ga}.\end{align}

Now \eqref{agd}, \eqref{ahb} follow from \eqref{fbr}, \eqref{fbqa}.
\end{proof}

It is anticipated that theorem 6.3 will be used to verify stability of $(z_0, P_\cals)$ and no nontrivial solutions of \eqref{agc} (with $z = z_0$) a posteriori of determination of $z_0$ from \eqref{fana}, \eqref{fbd}.  Here dependence of stability on the regularity of the boundary data arises, as anticipated from comparison of \cite{Ma1,Ma2,S1}.

We choose disjoint subspaces $H(z_0, P_\cals, w, w_\Ga; \eta),\; \,  \eta = 1,\dots$, with $w, w_\Ga$ satisfying \eqref{aga}, \eqref{agb} independent of $\eta$, such that
\be\label{xaa} H(z_0, P_\cals, w, w_\Ga) = \mathop{\oplus}^\infty_{\eta = 1} H(z_0, P_\cals, w, w_\Ga; \eta),\ee
and for $\dot b\in B_{P_\cals}$,
\be\label{xab}\dot b = \suml^\infty_{\eta = 1} \dot b_\eta, \; \; \dot b_\eta \in (H(z_0, P_\cals, w, w_\Ga; \eta)\mathop{\midl}_{\pOm} )^*,\ee
the dual space.

For judiciously chosen $w, w_\Ga$, for each value of $\eta$ we seek $\dot z_\eta,
\sigma_\eta, \dot b_\eta$ satisfying \eqref{ae}, recovering \eqref{aiea} in the form
\be\label{xac} \iintl_{\Om\setminus\Ga(z_0)} \frac 1 w |\dot z_\eta |^2 + \intl_{\Ga(z_0)} \frac{1}{w_\Ga} \sigma^2_\eta \; \eqadef\;  \| \dot b_\eta\|^2.\ee

Values of $\eta$ are restricted to those for which \eqref{xac} holds.  Failure of any such is identified with nontrivial solutions of \eqref{agc}.

From \eqref{xab}, \eqref{xac} and
\be\label{xad} \dot z = \suml_\eta \dot z_\eta, \; \; \sigma = \suml_\eta \sigma_\eta,\ee
we recover \eqref{fbq},
\begin{align}\label{xae} \iintl_{\Om\setminus \Ga(z_0)} \frac 1w |\dot z|^2 + \intl_{\Ga(z_0)} \frac{1}{w_\Ga} \sigma^2
 &\le \suml_n 2^{1+\eta} \Big(\iintl_{\Om\setminus \Ga(z_0)}\frac 1w |\dot z_\eta|^2 + \intl_{\Ga(z_0)} \frac{1}{w_\Ga} \sigma^2_\eta \Big)\nonumber \\
&= \suml_\eta 2^{1+\eta} \| \dot b_\eta\|^2\nonumber \\
&\eqadef \| \dot b \|^2_{B, P_\cals}.\end{align}

Thus $(z_0, P_\cals)$ is stable; \eqref{agd} holds with
\be\label{xaf} c_1(z_0, P_\cals, \; w, \, w_\Ga) = 1,\ee
and obtaining $\| \theta\|_{\pOm, P_\cals}$ from \eqref{fbqa}, \eqref{xae}.

The $(z_h, P_\cals)$ are then also stable.
\begin{thm} Stability of $(z_0, P_\cals)$  is equivalent to that of $(z_h, P_\cals)$ for $h$ sufficiently small.
\end{thm}
\begin{proof}  It is no loss of generality to assume $\theta \in Z(P)$ in \eqref{agd}. Then for any $z, P$, directly from \eqref{agd},
\be\label{fbs} c_1(z,P,w,w_\Ga) =
\underset{\theta}{\lub} (\| \theta \|_{\partial\Om, P} /  \| \theta\|_{z, w,w_\Ga})\ee
for $\theta \in H(z,P,w,w_\Ga) \cap Z(P)$. The conclusion now follows  from \eqref{fbs}, with $z = z_0, P = P_\cals$,  using \eqref{fb}, \eqref{ga}, \eqref{gad}, \eqref{gab}.
\end{proof}

Admissibility of $(z_0, P_\cals)$ is quite another matter.

\subsection{Cause for concern}

 If stability of $(z_0, P_\cals)$ so obtained is to be anticipated, admissibility may be elusive.  In particular, the specified boundary data \eqref{fae} may be insufficient for admissibility.

A general paucity of unambiguous $(z,P)$ is expressed above in theorems 2.7, 2.8, 2.9, and in \cite{S1,S2}.  An application of theorem 2.9 is the following.
\begin{thm} Assume $(z_0, P)$ unambiguous for some $z_0$ satisfying \eqref{ac}, \eqref{aca}, $P\in \calp$.  Assume \eqref{agd} holding with $w, w_\Ga$ such that \eqref{atb}, \eqref{atc}, \eqref{atd} hold, with $\Om'$ an open subset of $\Om, \nu$  continuously defined on an open subset $\partial\Om' \cap \partial\Om$.

Assume an  $(m -2)$-manifold, open for $m > 2$,
\be\label{fdc} \gamma_\Om  \subset \partial\Ga_k(z_0) \cap \partial\Om \cap \partial\Om',\ee
some $k$ in \eqref{ada},  such that  $\rank \;  \psi^-_{\nu,zz} (z_0)$ changes abruptly on $\gamma_\Omega$, where $\psi^{\pm}_{\nu, zz}$ are the positive and negative semidefinite parts of
$\psi_{\nu, zz}$, pointwise on $\partial\Om'$, as in \eqref{atz}, \eqref{atf} above.

Assume $P$ determined such that for some $\eta > 0$,
\be\label{fdd} (I_n-P) \psi_{\nu,zz} (z) (I_n-P) \le 0\ee
within $\partial\Om'\cap \partial\Om$, for all $z \in \cals$ satisfying
\be\label{fdf} \dist_\cals (z_0, z; w, w_\Ga) \le \eta.\ee
Then for all $\de > 0$, there exists $z_\de \in \cals$ satisfying
\be\label{fde} \dist_\cals (z_0, z_\de, w, w_\Ga) \le \delta\ee
such that $(z_\de, P )$ is not unambiguous.
\end{thm}
\begin{rem*}
The result is vacuous for $P\midl_{\partial\Om'\cap \partial \Om} = 0_N $ or $I_n$, as  occurs in Cauchy problems, as \eqref{fdc} cannot hold.  Otherwise the condition \eqref{fdc} may be expected where a ``transonic shock" $\Ga_k$, across which $R(z_0)$ changes characteristic type, reaches the boundary $\partial\Om$.

The condition \eqref{fdd}  is frequent if not universal, used in establishing \eqref{agd} or \eqref{fhc}. A nontrivial example is on the lateral boundary of initial-boundary value problems for hyperbolic systems \eqref{aa}.

The conclusion implies that unambiguous $(z_0, P)$ is exceptional, lost immediately, for example,  by inclusion of a small balance term in \eqref{aa}, \eqref{ac}.  Using \eqref{fde}, \eqref{ga}, \eqref{aa}, a partial integration establishes that $z_\de$ satisfies an approximation of \eqref{ac},
\be\label{fdg} | \iintl_\Om \suml^m_{i=1} \psi_{i,z} (z_\de) \theta_{\cdot, x_i} | \le c \de \|\theta\|_{W^{1,p}(\Om)}\ee
for any $\theta \in W^{1,p} (\Om)$ vanishing on $\partial\Om$.
\end{rem*}
\begin{proof}
Fixing a point $\hat x \in \gamma_\Om$, we have one-sided limiting values $a_\pm \in D$ at $\hat x $ from opposite sides of $\gamma_\Om$.  By hypothesis \eqref{fdc},
\be\label{fdj} \rank\, \psi^-_{\nu,zz} (a_+) \neq \rank\, \psi^-_{\nu,zz} (a_-).\ee

Given $\de> 0$, restricting $\Om'$ as necessary while maintaining $\hat x \in \partial\Om'\cap \partial\Om$, there exist $z_\pm \in \cals $ satisfying \eqref{fde}
(with $z_\pm$ replacing $z_\de$) simultaneously with
\be\label{fdm}
 z_\pm \midl_{\Om'} = a_\pm.\ee

 If either $(z_+, P)$ or $(z_-, P)$ is not unambiguous, the conclusion is immediate.

 Otherwise, either of $z_+$ or $z_-$ qualifies as $z$ in theorem 2.9, and from \eqref{ata}, \eqref{fdm}, $(z_\pm, P)$ are both unambiguous, with
 \be\label{fdn} \ker P = \range  \; \psi^-_{\nu,zz} (a_\pm) \ee
 in a neighborhood of $\hat x$.

Using \eqref{fdd} in \eqref{fdn}, in a neighborhood of $\hat x$,
\be\label{fdo} \dim \ker P = \rank\,  \psi_{\nu,zz} (a_\pm),\ee
which is incompatible with \eqref{fdj}.
\end{proof}

On this basis, we define a successful computational investigation.  The form of the approximation scheme $\cala_h$ is not specified, and no requirement of admissible solutions is included.

\begin{defn} For a given problem class determined from specification of a system \eqref{aa} in a domain $\Om$ \eqref{aab}, $P_\cals, \tilde B_{P_\cals}, \; \{ e_0\} $ \eqref{agea}, a successful computational investigation is identified a posteriori of verification of \eqref{fbc}, \eqref{fbd} with each $z_0 \in \hcals_0$ of the form \eqref{ad}, satisfying \eqref{ac}, \eqref{aca}, \eqref{agea}, \eqref{faf}, $(z_0, P_\cals)$ stable, and such that there are no nontrivial solutions of \eqref{agc} with $z = z_0$.
\end{defn}

\section{Main theorem}

If \eqref{fag}, \eqref{fbda} fail but $\cala_h$ satisfies \eqref{fb}, \eqref{fbb}, \eqref{fba}, then $\cala_0$ in \eqref{fbc}, \eqref{fbd} satisfies a weaker condition,
\begin{align}\label{han}
A_0 (b) &\in V(b, P_\cals)\nonumber \\
&\eqadef \{ z \in \hat \cals_0 \, | \, (I_n-P_\cals) \psi^\dag_{\nu, z} (z\underset{\partial\Omega}{|}) = b\}.
\end{align}

In this case, the images of $\cala_0$ generally  depend on features specific to $\cala_h$, including  the choices of $U_\cald $ and  $\cald$.

Using \eqref{ab}, elements of $V(b, P_\cals)$ may be assumed to be distinguished by  their corresponding values (in $N$) of  unspecified normal boundary flux, some of which presumably unavailable a priori, denoted
\be\label{hal} \tilde b (z, P) \eqadef P\psi^\dag_{\nu,z}(z\underset{\partial\Omega}{|})\ee
for some $P$ depending on $z$.

 We then do not expect $(z_0, P_\cals)$ to be unambiguous, but suppose that for each $z \in \tilde \cals_0$, there  exists $P(z) \in \calp$, satisfying
 \be\label{hae} \ker P_\cals \subseteq \ker P(z)\ee
 and such that $(z,P(z))$ is unambiguous.  We may then project  $\tilde b (z, P_\cals -  P(z))$ independent of features of $\cala_h$, thus relating the uncertainly in $z_0 \in V(b, P_\cals)$ to $\tilde b (z_0, P_\cals - P(z_0))$,  the boundary data which has been supplied by $\cala_h$.

 With seemingly mild qualifications,  \eqref{hae} indeed holds.
 \begin{defn} For $z\in \cals, \, P \in \calp, \, (z, P)$ is potentially admissible if $(z, 0_N)$ is admissible and if for all $\eta > 0$ there exists a bounded $\Om_\eta \subseteq \Om$ such that
 \be\label{haa} Q(z, P, \Ga(z) \cap ( \Om\setminus \Om_\eta), 0)\le \eta.\ee
\end{defn}

 For bounded $\Om$, potential admissibility reduces to $(z, 0_N)$ admissible.  Otherwise, one anticipates \eqref{haa} satisfied by $\Ga(z)\cap (\Om\setminus\Om_\eta)$ composed of weak shocks, as in \eqref{bgb}, and entropy shocks, as in \eqref{bbe}, with $G(\cdot, \cdot;  z)$ satisfying \eqref{ard}, \eqref{are}, \eqref{arf} and
 \be\label{hab} G(x,y; z) \overset{|x| \to \infty}{\longrightarrow} 0\ee
 for any $y \in \Ga (z) \cap (\Om\setminus\Om_\eta)$.
 \begin{thm}
Assume $(z, P_\cals)$ stable, \eqref{agd} holding with $w, w_\Ga$ satisfying
\be\label{hac} w\in L_{p/(p-2), loc}\ee
within $\Om$,
\be\label{had} |[z]|^2 w_\Ga \in L_{p/(p-2), loc} \ee
within $\Ga(z)$.

Assume $(z, P_\cals)$ potentially admissible but not admissible.  Then there exists $P (z) \in \calp$ satisfying \eqref{hae} such that $(z, P(z))$ is unambiguous.

 If in addition, \eqref{jb} holds with $P_0 = P_\cals$,
\begin{align}\label{kaa} &\{ w^{1/2} R(z)\theta | \theta \in H(z, P_\cals, w, w_\Ga)\} = L_2 (\Om\setminus \Ga(z))^n, \nonumber \\
 &\{ w^{1/2}_\Ga  S(z)\theta | \theta \in H(z, P_\cals, w, w_\Ga)\} = L_2 ( \Ga(z)),
  \end{align}
 then $P(z)$ is unique.

  \end{thm}
\begin{rem*}  For a given application
 utility of $z_0$ obtained from \eqref{fbd} is recovered  by restriction of
$\tilde B_{P_\cals}$ as necessary  such that theorem 7.2 holds for all $z \in \hat \cals_0$ and such that for all $b \in \tilde B_{P_\cals}$,
\be\label{hee} \underset{z \in V (b, P_\cals)}{\lub} \| \tilde b (z, P_\cals - P(z)) \|_{B, P_\cals}\ee
is acceptably  small in the context of the application.

The need for assumptions of the form \eqref{hac}, \eqref{had}, (with $p$ given in \eqref{baa} and used throughout sufficiently large) is anticipated from \eqref{fbq} with $\Om$ unbounded.

Considerable flexibility exists in choosing $\| \cdot \|_{\partial\Om, P(z)} $ to achieve \eqref{agd} with $P = P(z)$. The choice below achieves
\be\label{hap} c_1 (z, P(z), w, w_\Ga ) = c_1 (z, P_\cals, w, w_\Ga)\ee
and permits expression of moments of $\tilde b (z, P_\cals - P(z))$, at the expense of
\be\label{haq} \| \theta\|_{\partial\Om, P(z)} \neq \| ( I_n-P(z))\theta\|_{\partial\Om, 0_N}.\ee

With $z$ fixed, the notion of limit admissibility is considerably simplified; it suffices to set $\vare =0$, $z_\delta = z$ in definition 5.1, \eqref{db}.

If \eqref{hae}, \eqref{kaa} hold and theorem 7.2 applies, then
\begin{align}\label{kax} H(z, P(z), w, w_\Ga) & = H(z, P_\cals, w, w_\Ga) \oplus \{ \theta|R(z)\theta = S(z) \theta = 0,\nonumber \\
P(z) \theta &= (P_\cals - P(z)) \theta = 0\}.\end{align}

Thus for a given $\dot b$ satisfying
\be\label{kaw} (P_\cals - P(z))\dot b = 0\ee
almost everywhere on $\partial\Om$, the solution of \eqref{ae} for all $\theta \in H(z, P_\cals, w, w_\Ga)$ will coincide with that for all $\theta \in H(z, P(z), w, w_\Ga)$.

Because of \eqref{ageb}, we are tacitly assuming
\be\label{hez} \hbox{span}\, \{ e_0\} \subseteq \hbox{span} \, \{ e'_0\},\ee
where \eqref{agea}, \eqref{agec} hold with
\be\label{heza} \hbox{span}\, \{ e_0\} \subset \ker P_\cals,\;  \hbox{span} \{ e'_0\} \subset \ker P(z).\ee
\end{rem*}
\begin{proof} The proof is constructive and self-contained, but may be understood as an application of theorem 5.2.

We obtain $P(z)$ as the limit of a sequence
\be\label{hba} \{ P_{j,l}, j = 0, 1, \dots, \kappa_l, l = 1, \dots  \}\ee
satisfying
\be\label{hbb} P_{0, 1} = P_\cals, \; P_{0, l+1} = P_{\k_l , l}, l = 1, \dots.\ee
\be\label{hbc} \ker P_{j, l} \subseteq \ker P_{j +1, l}, \; j = 0, 1, \dots, \kappa_l -1, l = 1, \dots .\ee

Choosing a sequence
\be\label{hbd} \de_l \downarrow 0 \; \hbox{as\ } l \to \infty,\ee
from \eqref{haa} we have a bounded $\Om_{\de_l}$ such that
\be\label{hbe} Q(z, P_\cals, \Ga (z) \cap (\Om\setminus \Om_{\de_l}), 0 ) \le \de_l.\ee

Below we use abbreviations
\be\label{hbf}
H_{j,l} \eqadef H(z, P_{j, l},  w, w_\Ga)\ee
\be\label{hbg}
\Ga'_l \, \eqadef\,  \Ga(z)\cap \Om_{\de_l}
\ee
\be\label{hbh}
\tilde M_{j, l} (\phi_\Ga) \,\eqadef  \underset{\theta \in H_{j, l}}{\glb}\;
\frac{ \iintl_{\Om_{\de_{l}\setminus \Ga'_l}} w|R(z)\theta|^2 +\intl_{\Ga'_l} w_\Ga (S(z) (\theta-\phi_\Ga))^2 }
{ \| \phi_\Ga\|^2_{W^{1,p}(\Ga'_l)} }
\ee
with $w, w_\Ga$ independent of $j, l$ throughout.

The $\{ P_{j, l} \}$ satisfying \eqref{hbb}, \eqref{hbc} are determined recursively.  For each $j, l$, we denote
\be\label{hc} \Sigma_{j, l} \subset \{ \phi_\Ga \in \Phi_0 (z, P_{j, l})\cap \Phi (z, P_{j,l}, w_\Ga)\},\ee
$S(z)\phi_\Ga $ not vanishing identically, using \eqref{ahf}, \eqref{baa},  satisfying
\be\label{hbm} \tilde M_{j, l} (\phi_\Ga) > 0,\ee
\be\label{hcb} \supp (\phi_\Ga ) \subset \Ga'_l, \ee
\be\label{hcc} \| \phi_\Ga \|_{W^{1,p}(\Ga'_l)} = 1,\ee
and for $j \ge 1$,
\be\label{hcd} \intl_{\Ga'_l} \phi_\Ga\cdot \phi_{\Ga, i, l} = 0, \; i = 0, \dots, j-1,\ee
for previously determined $\phi_{\Ga, j, l}$.

From \eqref{had}, \eqref{hcc}, for $\phi_\Ga  \in \Sigma_{k, l}$
\be\label{hcf} \| w^{1/2}_\Ga S(z) \phi_\Ga \|_{L_2(\Ga'_l)} \le c_l\ee
independent of $j$, depending on $l$ because of $\Om_{\delta_l}$.

Denoting
\be\label{hck} \Lambda_{j, l} \eqadef
\underset{\phi_\Ga \in \Sigma_{j, l}}{\glb}
\frac{\| w^{1/2}_\Ga S(z)\phi_\Ga\|_{L_2 (\Ga'_l)}}
{\| \phi_\Ga\|_{L_2(\Ga'_l)}}
\ee
we choose $\phi_{\Ga, j,l} \in \Sigma_{j,l} $ such that
\be\label{hcka}  \frac{\| w^{1/2}_\Ga S(z)\phi_{\Ga, j,l}\|_{L_2(\Ga'_l)}}{\|\phi_{\Ga, j,l}\|_{L_2 (\Ga'_l)}} \le 2 \Lambda_{j,l} + 1,\ee
otherwise arbitrary.

For given $\phi_{\Ga, j, l}$ satisfying \eqref{hcf}, analogously with \eqref{bag} there exists a unique $\zeta'_{j,l} \in H_{j,l}$ satisfying
\be\label{hcg}  \iintl_{\Om\setminus\Ga(z)} w R(z)\zeta'_{j,l} \cdot R(z)\theta + \intl_{\Ga(z)} w_\Ga (S(z) (\zeta'_{j,l} - \phi_{\Ga, j,l})) (S(z)\theta) = 0 \ee
for all $\theta \in H_{j,l}$.

From \eqref{hcg}, \eqref{hcf}, recalling \eqref{ah},
\begin{align}\label{hch}
\| \zeta'_{j,l} \|_{z, w, w_\Ga} &\le \| w_\Ga^{1/2}  S(z) \phi_{\Ga, j,l} \|_{L_2 (\Ga'_l})\\
&\le c_l \nonumber \end{align}
independent of $j$, depending on $l$.

Using \eqref{hch}, $\zeta'_{j,l} $ determines a bounded linear functional on $H(z, I_n-P_{j, l}, w, w_\Ga)$, so by another application of the Riesz theorem there exists a unique $\xi_{j,l} \in H(z, I-P_{j, l}, w, w_\Ga)$ such that
\begin{align}\label{hci}  &\qquad
\iintl_{\Om\setminus\Ga(z)} w  R(z)\xi_{j,l}\cdot R(z)\theta + \intl_{\Ga(z)} w_\Ga (S(z)\xi_{j,l} ) (S(z)\theta)\nonumber \\
&=\iintl_{\Om\setminus\Ga(z)} w R(z)\zeta'_{j,l}\cdot  R(z) \theta+\intl_{\Ga(z)} w_\Ga ( S(z)(\zeta'_{j,l}    -\phi_{\Ga, j,l} ))
(S(z)\theta)\end{align}
for all $\theta \in H(z, I_n-P_{j, l}, w, w_\Ga)$.

We determine $P_{j+1, l}$ from
\begin{align}\label{hcj} &\ker P_{j+1,l} = \ker P_{j,l} \oplus \text{ span} \{\xi_{j,l} \}\nonumber \\
&=\ker P_{j, l} \oplus \text{ span } \{ \xi_{j,l} - \zeta'_{j, l}\}.
\end{align}

From \eqref{hbf}, \eqref{hcj}

\be\label{hcm} H_{j+1, l} = H_{j, l}\oplus \text{ span\ } \{ \xi_{j, l} - \zeta'_{j, l}\}.\ee

From \eqref{hcg}, \eqref{hci}
\be\label{kab} \iintl_{\Om\setminus\Ga(z)} w R(z) (\zeta'_{j,l} - \xi_{j, l})\cdot R(z) \theta +\intl_{\Gamma(z)} w_\Ga (S(z)(\zeta'_{j, l} - \xi_{j,l} - \phi_{\Ga, j, l})) (S(z)\theta)= 0\ee
for all $\theta \in H_{j, l} \oplus H(z, I_n-P_{j, l}, w, w_\Ga)$.

However from \eqref{hbf}
\be\label{kac} H_{j, l} \oplus H(z,I_n-P_{j, l}, w, w_\Ga) = H(z, 0_N, w, w_\Ga).\ee

From \eqref{hcm}
\be\label{kad} \zeta'_{j, l} - \xi_{j, l} \in H_{j + 1, l},\ee
and thus from \eqref{cae}, \eqref{hbh}, \eqref{kab}, \eqref{kac}, \eqref{kad},
\begin{align}\label{hcn}
\tilde M_{j+1, l} (\phi_{\Ga, j, l}) &= M(\phi_{\Ga, j, l} ; z, 0_N, \Ga'_l, 0)\nonumber\\
&= 0.
\end{align}

Indeed
\begin{align}\label{heb} &\tilde M_{j+1} (\phi_{\Ga, i, l}) = 0, \; \, i = 0, \dots, j,\nonumber\\
&\tilde M_{j+1} (\phi_{\Ga, i, l')} = 0, \; \,  i = 0, \dots, \kappa_{l'}, \; l' < l, \end{align}
and for $j = 0, \dots, \k_l-1$,
\be\label{heba} Q(z, P_{j+1, \k_l}, \Ga'_l, 0) \le Q(z, P_{j, \k_l}, \Ga'_l, 0).\ee

With $l$ fixed, from \eqref{hcd}, $\Lambda_{\cdot, l}$ is nondecreasing in $j$.  Indeed,
\be\label{hcl} \Lambda_{j,l} \overset{j\to\infty}{\longrightarrow} \infty.\ee

Otherwise, using \eqref{hcc}, a subsequence within $\Sigma_{\cdot, l} $ converges in $L_2 (\Ga'_l)$ to   some $\underline{\phi}_\Gamma  \in \Sigma_{\cdot, l} $ with $\| \underline{\phi}_\Gamma \|_{L_2 (\Ga'_l)} > 0$, which will be removed from $\Sigma_{j,l}$ for some finite $j$ by \eqref{hcd}.

From \eqref{hbf}, \eqref{hc}, \eqref{ahf}, for each $l$,
\be\label{kal} S(z) H_{0, l} \subset S(z) \text{\ span} \, \Sigma_{0, l}.\ee

By assumption $(z, 0_N)$ is admissible, \eqref{al} applies to $H(z, 0_N, w, w_\Ga)$.  Thus
\eqref{kab}, \eqref{kac}, \eqref{hcn} implies
\be\label{kam} S(z) \phi_{\Ga, j, l} = S(z) (\zeta'_{j, l} - \xi_{j, l})\mathop{\midl}_{\Ga'_l}.\ee

From \eqref{hcj}, \eqref{kal}, then \eqref{kam}
\begin{align}\label{kan}
S(z) H_{j+1, l} \mathop{\midl}\limits_{\Ga'_l}  =& S(z) (H_{j, l} \oplus \text{ span\ } \{ \zeta'_{j, l} -\xi_{j, l}\}) \mathop{\midl}\limits_{\Ga'_l}\nonumber\\
=&S(z)(H_{j,l} \mathop{\midl}\limits_{\Ga'_l} \oplus \text{ span\ } \{ \phi_{\Ga, j,l}\}).
\end{align}

From \eqref{kan}, using \eqref{kal} then \eqref{hcd}
\begin{align}\label{hcu}
  S(z)H_{j+1, l} \mathop{\midl}\limits_{\Ga'_l} &= S(z) (H_{0, l}  \mathop{\midl}\limits_{\Ga'_l} \oplus \text{ span } \, \{ \phi_{\Ga, 0, l},\dots, \phi_{\Ga,j, l}\}) \nonumber \\
&= S(z) (\Sigma_{0, l} \oplus \text{ span} \, \{ \phi_{\Ga, 0, l}, \dots, \phi_{\Ga, j, l}\}) \nonumber \\
&= S(z) (\Sigma_{j+ 1, l} \oplus \text{ span} \, \{ \phi_{\Ga, 0, l}, \dots, \phi_{\Ga, j, l}\}).
\end{align}


From \eqref{hck}, \eqref{hcl}, \eqref{hcf}
\be\label{hco} \underset{\phi_\Ga\in\Sigma_{j,l}}{\lub} \| \phi_\Ga\|_{L_2(\Ga'_l)}\overset{j \to\infty}{\longrightarrow} 0,\ee
whence from \eqref{hcc}
\be\label{hcp} \underset{\phi_\Ga\in\Sigma_{j,l}}{\lub} \| \phi_\Ga\|_{L_p(\Ga'_l)}\overset{j \to\infty}{\longrightarrow} 0.\ee

From lemma 3.4, for each $\phi_\Ga \in \Sigma_{j, l} $ we have $\theta_\phi \in H_{j, l}$ satisfying
\be\label{hcq} \theta_\phi = \phi_\Ga\ee
on $\Ga(z)$, and such that, using \eqref{hcp},
\be\label{hcr} \underset{\phi\in\Sigma_{j,l}}{\lub} \| \theta_\phi\|_{W^{1, p}(\Om)}\overset{j\to\infty}{\longrightarrow} 0.\ee

From \eqref{hac}, \eqref{hcr}
\be\label{hcs} \underset{\phi\in\Sigma_{j, l}}{\lub} \| w^{1/2} R(z) \theta_\phi \|_{L_2 (\Om\setminus\Ga(z))} \overset {j\to\infty}{\longrightarrow} 0.
\ee

Using \eqref{hbh}, \eqref{hcq}, \eqref{hcs}, \eqref{hcc}, for each fixed $l$,
\be\label{hct} \underset{\phi_\Ga\in\Sigma_{j, l}}{\lub} \tilde M_{j,l} (\phi_\Ga) \overset {j\to\infty}{\longrightarrow} 0,
\ee
so from  \eqref{heb}, \eqref{hct}, using \eqref{hbf}, \eqref{hc}, \eqref{ahf}, \eqref{baa},
\be\label{hcu} \underset{\phi_\Ga\in H_{j, l} \mathop{\midl}\limits_{\Ga'_l}}{\lub} \tilde M_{j + 1, l} (\phi_\Ga ) \overset {j\to\infty}{\longrightarrow} 0.\ee

For each $l$, we fix $\kappa_l$ such that
\be\label{hcw} \underset{\phi_\Ga \in H_{\kappa_l, l}}{\lub} \tilde M_{\kappa_{l, l}} (\phi_\Ga ) \le \de_l.\ee

Then from \eqref{hcw}, \eqref{hbh}, \eqref{caf}, using \eqref{hcc}, \eqref{aga}, \eqref{agb}
\be\label{hcx} Q(z, P_{\kappa_l, l}, \Ga'_l, 0) \le \de_l.\ee

To \eqref{hcx}, we apply an analog of theorem 3.3 with $K = 2, P = P_{\k_l, l}$. Using \eqref{cde}, \eqref{hbg}, \eqref{hbe}, \eqref{heba} we obtain
\be\label{hcy} Q(z, P_{\kappa_l, l}, \Ga(z), 0) \le c\de_l.\ee

Setting
\be\label{hd} \ker P(z) = \underset{l}{\cup}  \ker P_{\k_l, l},\ee
for any bounded $\Ga'\subseteq \Ga(z)$, using \eqref{cbe}, \eqref{hd}, \eqref{hcy}, \eqref{hbd}
\be\label{hda} Q(z,P(z), \Ga', 0) = 0.\ee
As $\Ga'$ is arbitrary, \eqref{hda} implies
\be\label{hdb} Q(z, P(z), \Ga(z), 0) = 0,
\ee
and $(z, P(z))$ is admissible from corollary 5.4,  \eqref{cdf}.

It thus suffices to determine a suitable $\| \cdot \|_{\partial\Om, P(z)}$ such that \eqref{agd} holds with $P = P(z)$. From \eqref{hci}, \eqref{hcg},
\be\label{hdc} \iintl_{\Om\setminus \Ga(z)} wR(z) \xi_{j,l} \cdot R(z)\theta + \intl_{\Ga(z)} w_\Ga (S(z) \xi_{j, l}) (S(z)\theta) = 0\ee
for all $\theta \in H_{j, l}$. In particular, \eqref{hdc} holds for all $\theta \in H(z, P_\cals, w, w_\Ga)$ and implies the $\{ \xi_{j, l}\}$ orthogonal,
\be\label{hdd} \iintl_{\Om\setminus\Ga(z)} w R(z) \xi_{j, l} \cdot R(z)\xi_{j', l'} + \intl_{\Ga(z)} w_\Ga (S(z)\xi_{j, l})(S(z)\xi_{j', l'}) = 0\ee
for all $(j, l) \neq (j', l')$.

For any $\theta \in H(z, P(z), w, w_\Ga)$, from \eqref{hbb}, \eqref{hcj}, \eqref{hd}
on $\partial \Om$
\begin{align}\label{hde} \theta &= (I_n-P(z))\theta\nonumber\\
&=(I_n-P_\cals)\theta + \suml_{j, l} a_{j, l} \xi_{j, l}\nonumber\\
&=\theta_\cals + \suml_{j, l} a_{j, l }\xi_{j, l}, \; \theta_\cals \eqadef (I_n-P_\cals) \theta \in H(z, P_\cals, w, w_\Ga),
\end{align}
with constants $a_{j, l}$ and $\suml_{j, l} $ abbreviating $\suml^\infty_{l=1} \suml^{\kappa_l}_{j = 0}  $ such that $\xi_{j, l}$ does not vanish identically. Necessarily \eqref{hde} holds also within $\Om$, so using
\eqref{hdc}, \eqref{hdd},
\be\label{hdf} \|\theta\|_{z, w, w_\Ga}= ( \|\theta_\cals\|^2_{z, w, w_\Ga} +  \suml_{j, l} a^2_{j, l}\| \xi_{j,l}\|^2_{z, w, w_\Ga})^{1/2}.\ee

Using \eqref{agd} with $P = P_\cals$, the choice
\be\label{hdg} \| \theta\|_{\partial \Om, P(z)} = ( \| \theta_\cals\|^2_{\partial\Om, P_\cals} +
c_1 (z, P_\cals, w, w_\Ga)^2\suml_{j, l} a^2_{j, l} \| \xi_{j, l}\|^2_{z, w, w_\Ga})^{1/2}\ee
satisfies \eqref{agd} with $P = P(z)$, recovering \eqref{hap}, \eqref{haq}.

Uniqueness of $P(z)$ follows from theorem 2.8, using \eqref{kaa}.
\end{proof}

The corresponding expression for $\| \cdot \|_{B, P(z)}$ is obtained from \eqref{ahb}, \eqref{hdg}.  For $\theta$ of the form \eqref{hde},
\be\label{hdh} \intl_{\partial\Om} \dot b \cdot \theta = \intl_{\partial\Om} \dot b \cdot \theta_\cals + \suml_{j, l} a_{j, l } \int_{\partial\Om} \dot b \cdot \xi_{j, l},\ee
and
\begin{align} \label{hdi}
| \intl_{\partial\Om}\dot b \cdot \theta | \le & \left(
\frac{(\intl_{\partial\Om}\dot b \cdot \theta_\cals)^2}
{\| \theta_\cals \|^2_{\partial\Om, P_\cals}}
 + \frac{1}
 {c_1(z, P_\cals, w,w_\Ga)^2}
 \suml_{j, l}
 \frac{(\intl_{\partial\Om}\dot b \cdot \xi_{j, l})^2}
 {\| \xi_{j,l} \|^2_{z, w, w_\Ga}}\right)^{1/2} \nonumber \\
&\times \Big( \| \theta_\cals \|^2_{\partial \Om, P_\cals} + c_1 (z, P_\cals, w, w_\Ga) \suml_{j, l} a^2_{jl} \|\xi_{j, l} \|^2_{z, w, w_\Ga}\Big)^{1/2}\nonumber\\
\le& \| \dot b \|_{B, P(z)} \| \theta\|_{\partial\Om, P(z)},
\end{align}
thus identifying
\be\label{hdj}
\| \dot b \|_{B, P(z)} =
\left( \underset{\theta_\cals \in H(z, P_\cals, w, w_\Ga)}{\lub} \frac{(\intl_{\partial\Om}\dot b \cdot \theta_\cals)^2}
{\| \theta_\cals \|^2_{\partial\Om, P_\cals}}
 +
 \frac{1}
 {c_1(z, P_\cals, w, w_\Ga)^2}
 \suml_{j, l}
 \frac{(\intl_{\partial\Om}\dot b \cdot \xi_{j, l})^2}{\| \xi_{j, l}\|^2_{z, w, w_\Ga}}\right)^{1/2}.
 \ee

Finally, while $P(z)$ is not readily accessible using \eqref{hd}, moments of $\tilde b (z, P_\cals - P(z))$ may suffice to determine if the expressions \eqref{hee} are acceptably small.

Such may be determined by choosing  $\phi_\Ga \in \Phi_0(z, P_\cals) \cap \Phi(z, P_\cals, w_\Ga)$, thus satisfying \eqref{hcf}, a linear combination of the $\phi_{\Ga, j, l}$,  and such that
\be\label{hdm} M_0 (\phi_\Ga; z, P_\cals) > 0.\ee

Analogously with \eqref{hcg}, for any given $\phi_\Ga$ there exists a unique $\zeta'_\phi \in H(z, P_\cals, w, w_\Ga)$ satisfying
\be\label{hdt} \iintl_{\Om\setminus\Ga(z)} w R(z)\zeta'_\phi \cdot R(z)\theta + \intl_{\Ga(z)} w_\Ga (S(z)(\zeta'_\phi - \phi_\Ga)) (S(z)\theta)=0 \ee
for all $\theta \in H(z, P_\cals, w, w_\Ga)$.

As $(z, P(z))$ is admissible, analogously with \eqref{hci} there exists a unique $\xi_\phi$,
\be\label{hdu} \xi_\phi \in H(z, P(z), w, w_\Ga),\ee
satisfying
\begin{align}\label{hdp}
&\iintl_{\Om\setminus\Ga (z)} w R(z) \xi_\phi \cdot R(z) \theta + \intl_{\Ga(z)} w_\Ga(S(z)\xi_\phi) (S(z)\theta)\nonumber \\
= &\iintl_{\Om\setminus\Ga (z)} wR(z) \zeta'_\phi \cdot R(z)\theta + \intl_{\Ga(z)} w_\Ga
(S(z)(\zeta'_\phi - \phi_\Ga)) (S(z)\theta)\end{align}
for all $\theta \in H(z, P(z), w, w_\Ga)$.

From \eqref{hdm}, \eqref{hdt}, the right side of \eqref{hdp} does not vanish (for every $\theta$) and thus
\be\label{hdr} \xi_\phi \notin H(z, P_\cals, w, w_\Ga).\ee

Thus from \eqref{hal}, \eqref{hdr}, we obtain  nontrivial moments of $\tilde  b(z, P_\cals - P(z))$ with $\xi_\phi$ so obtained using \eqref{hdt}, \eqref{hdp},
\begin{align}\label{hds} \intl_{\partial\Om} \tilde b (z, P_\cals - P(z))\cdot \xi_\phi &= \intl_{\partial\Om} ((P_\cals - P(z)) \psi^\dag_{\nu,z} (z) \cdot \xi_\phi\nonumber\\
&= \intl_{\partial\Om} \psi^\dag_{\nu, z} (z) \cdot (P_\cals - P(z)) \xi_\phi\nonumber \\
&= \intl_{\partial\Om} \psi^\dag_{\nu, z} (z) \cdot P_\cals \xi_\phi.\end{align}

\section{Weakly well-posed problems}

 We now weaken the condition of well-posedness.  For a given $P_\cals$, the condition \eqref{pb} is relaxed to
 \be\label{paa} \tilde \cals \subseteq \hat \cals\ee
 \be\label{pa} \hcals\,  \eqadef \,  \{ z \, |\,  (z, P(z)) \hbox{\ unambiguous}\}\ee
 for some $P(z)$ satisfying \eqref{hae}.

 Correspondingly, \eqref{fbda} is relaxed to a conjecture,
 \be\label{pab} \hcals_0 \subseteq \hcals.\ee

A toy  example shows that $\hat \cals$ is too large to be identified as $\tilde \cals$, and suggests how suitable elements of $\tilde \cals$ are obtained.

We consider Burgers equation
\be\label{pe}  z_{x_1} + \left(\tfrac{z^2}{2}\right)_{x_2}= 0 \ee
in the strip $0 < x_1 < 1$, with $P_\cals $ as for Cauchy problems
\be\label{pf} P_\cals (x_1, x_2) =
\begin{cases} 0,\,  &x_1 = 0\\ I_1, \,  &x_1 = 1,\end{cases}
\ee
and data $b$ in \eqref{ab}
\be\label{pg} b (0, x_2) =
\begin{cases} -1,\,  &x_2 < 0\\ \;\; 1, \,  &x_2  > 0,\end{cases}\ee
\be\label{pga} b(1, x_2) = 0.\ee

The subset of $\hat \cals$ corresponding to $b$ given in \eqref{pg}, \eqref{pga} includes a one parameter family $\{ z_\la, 0 \le \la\le 1\}$, given by
\be\label{ph} z_\la (x_1, x_2) = \begin{cases} &-1, x_2 \le  (x_1 - \la), x_1 > \la\\
&x_2 / (x_1 - \la),\,  x_1 > \la, \; |x_2| < x_1 - \la\\
&1, \, x_2 \ge x_1 - \la, \; x_1 > \la. \end{cases},\ee
corresponding to Example 2, \eqref{bgc}, \eqref{bgd} above.

Clearly only $z_\la  $ with $\la = 0$ is to be included in $\tilde \cals$.

%
\begin{figure}[h]
\begin{center}
\includegraphics[width=0.8\textwidth]{  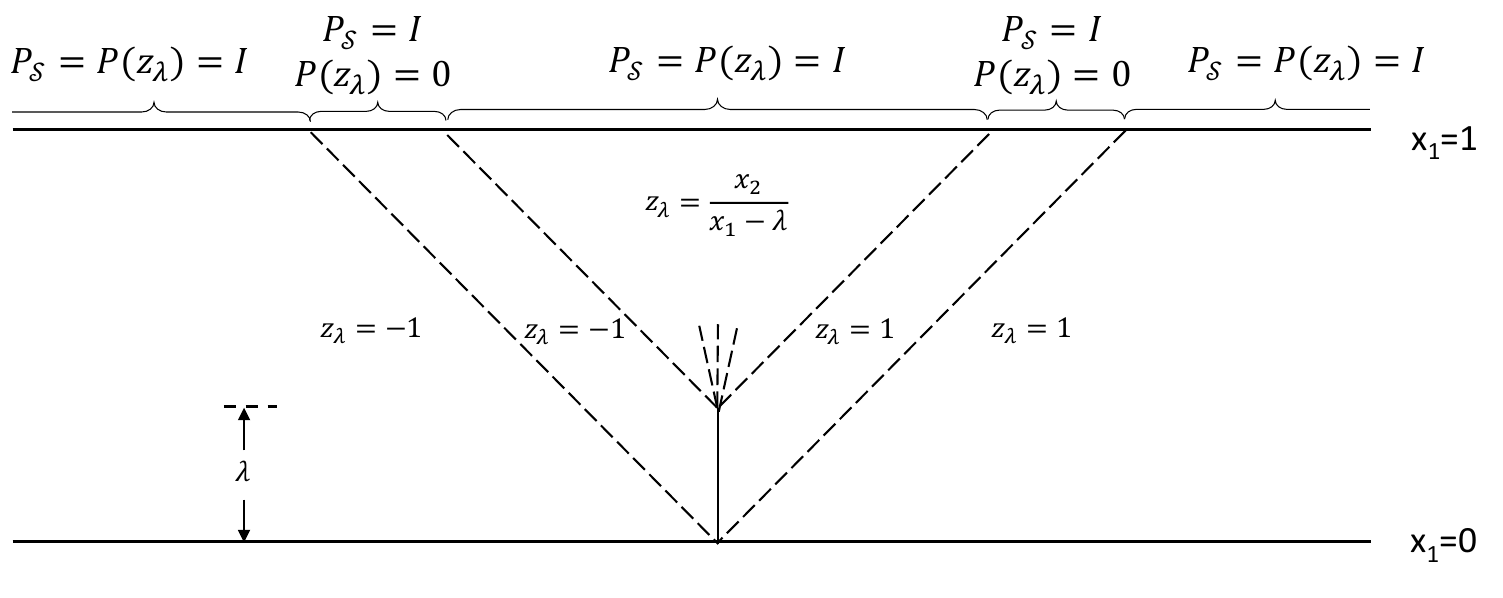 }
\end{center}
\caption{ Solution $z_\lambda \in \hat \cals, \; z_\lambda \notin \tilde \cals$}
\label{ }
\end{figure}

 For  $\hat\cals$ given in \eqref{pa}, we denote
an extension  of $V(b, P_\cals)$ determined from \eqref{han}, \eqref{pab}, \eqref{pa}

\be\label{pk} \hat V(b) \eqadef \{ z \in \hat\cals | b(z, P_\cals) = b\}, \; \, b \in \tilde B_{P_\cals},\ee
and $P_\cap  (b)$ determined from
\be\label{pka} \ker P_\cap (b) \eqadef \underset{z \in \hat V(b)}{\cap} \ker P(z).\ee

From \eqref{hae}, \eqref{pka}, for any $b\in \tilde B_{P_\cals}$,
\be\label{pkb} \ker P_\cals \subseteq \ker P_\cap (b).\ee

The set $\hat V(b)$ admits useful partitioning.
\begin{lem} Assume theorem 7.2 holding. Then for  any $z \in \hat V(b), b \in \tilde B_{P_\cals}$,

\begin{align}\label{pl}
\hat V(b) &= V_u(z) \cup V_= (z) \cup V_\subset (z);\\
\label{pla}
V_u(z) &\eqadef \{ z'\in \hat V(b) \, \big| \, (z, P(z')) \; \hbox{is\ not\ stable}\};\\
\label{plb}
V_=(z) &\eqadef\{ z'\in \hat V(b) \,\big| \, P(z') = P(z)\};\\
\label{plc}
V_\subset (z) &\eqadef \{ z'\in \hat V(b)\,\big|\, \ker P(z') \subset \ker P(z)\}.
\end{align}

The sets $V_u(z), V_= (z), V_\subset(z)$ are nonintersecting.
\end{lem}
\begin{rem*} In \eqref{pla}, \eqref{plb}, \eqref{plc}, $P(z), P(z')$ are as determined from theorem 7.2.

The generic case for $z, z' \in \hat V(b)$ is neither   $\left(z, P(z')\right)$ nor $\left(z', P(z)\right)$ stable or admissible.
For $z'\in V_\subset(z), \, \left(z', P(z)
\right)$ is admissible but not stable and $\left(z, P(z')\right)$ stable but not admissible.  Both $\left(z', P(z)\right)$ and $\left(z, P(z')\right)$ are stable only if $z'\in V_= (z)$.  The set $V_= (z)$ may or may not be the singleton $\{ z\}$.
\end{rem*}
\begin{proof}  Assume $\left(z, P(z')\right)$ stable and $z'\notin V_= (z)$, so $\left(z, P(z')\right)$ is not admissible.  With $z \in \hat \cals$, we may apply theorem 7.2 to $z$ with $P_\cals$ replaced throughout by $P(z')$ obtaining $(z, P'')$ unambiguous with
\be\label{pld} \ker P_\cals \subseteq \ker P(z') \subset \ker P''.\ee

By uniqueness, $P''$ must coincide with $P(z)$, establishing \eqref{plc}.
\end{proof}

For the example \eqref{ph},
\be\label{pm} P(z_\la) (x_1,x_2) = \begin{cases} 0, \, x_1 = 0 \\
I_1, \,  x_1 = 1, \,  |x_2|< 1-\la\\
0, \,  x_1 = 1, \,   1-\la < |x_2| < 1\\
I_1, \,  x_1 = 1, \,  |x_2| > 1\end{cases};\ee
and
\be\label{pma}
V_u (z_\la) = \{ z_{\la'} \, |\, \la' > \la\},\ee
\be\label{pmb} V_= (z_\la) = \{ z_\la\},\ee
\be\label{pmc} V_\subset (z_\la) = \{ z_{\la'} \, |\,\la' < \la\}.\ee

From \eqref{ph}, \eqref{pmc}, it follows  that $V_\subset (z_\la)$ is empty for $\la = 0$, corresponding to the candidate solution for $\tilde\cals$.  Such is not accidental.

In general, nonempty $V_\subset(z)$ is interpreted that $z\notin \tilde \cals$, that $z'\in V_\subset (z)$ is in some sense preferable to $z$.

Specifically, from \eqref{pmc}, in \eqref{bb}
\be\label{pmd} Q_0 (z', P) \le Q_0(z, P)\ee
for any $P$ such that $\ker P_\cals \subseteq  \ker P$, and with strict inequality for $P = P(z')$.  Also from \eqref{pmd}
\be\label{pme} \range\;  (P_\cals - P(z'))\subset \range\; (P_\cals - P(z))\ee
implying less ``under-specified" data for $z'$ than for $z$ using \eqref{hal}.

On this basis we postulate $\tilde \cals$, obtaining \eqref{fag} with
\be\label{pn} \tilde \cals \eqadef \{ z \in \hat V(b), b \in \tilde B_{P_\cals}\, | \, V_\subset (z) = \emptyset\}.\ee

Immediately from \eqref{pn}, \eqref{pka}, \eqref{pkb} we have $z \in\tilde S$ if
\be\label{pnc} P(z) = P_\cap (b),\; \, z \in \hat V(b),\ee
recovering \eqref{pb} in the special case that \eqref{pnc} holds, and \eqref{pkb} holds with equality, for all $b \in \tilde B_\cals$.
Otherwise, $\tilde b (z, P_\cals - P(z))$ in \eqref{hal}  is nonempty and uniqueness of $z \in  \hcals $ for given $b \in \tilde B_{P_\cals}$ is lost, even if $V_= (z) = \{ z\}$.\\

Weak well-posedness is crudely identified with existence of a Frechet differentiable map of (boundary) data into weak solutions.  The precise if tentative definition given here applies to a system of the form \eqref{aa} in a given domain $\Om$, admitting weak solutions $z$ of the form \eqref{ad}, necessarily satisfying \eqref{afb}.

The required expressions for $R(z), S(z)$ in \eqref{ae} are then obtained from \eqref{af}, \eqref{ag} respectively.

Extension of this definition, as discussed in section 13 below, is anticipated.

\begin{defn}
A weakly well-posed boundary-value problem is determined by existence of a nontrivial set $\hat\cals$ of weak solutions and a projection map $P_\cals  \in \calp$ determining boundary conditions on $\partial\Om$, with the following properties.

Elements $z \in \hat \cals$ satisfy the entropy inequality \eqref{faf}.

The set $ B_{P_\cals}$ obtained from \eqref{ahd},  \eqref{ahb} with $P = P_\cals$ is nontrivial.

For all $z \in \hat \cals$ there exists $P(z)$ satisfying \eqref{hae} such that $(z, P(z))$ is unambiguous.

For all $z\in \hcals$, there are no nontrivial solutions of \eqref{agc}.
\end{defn}

Not included in definition 8.2 is tacit assumption of existence of $\cala_\cals$ satisfying \eqref{fag} with $\tilde \cals$ satisfying \eqref{paa}, \eqref{pa}.

It remains to show that \eqref{pn} is compatible with \eqref{fab}, that $\tilde B_{P_\cals}$ is not unduly   restricted by adoption of \eqref{pn}. It suffices that each $\hat V(b)$ contains suitable limit points.
\begin{thm} For arbitrary fixed $b\in \tilde B_{P_\cals}$, assume a sequence $\{ z_\de\}$ as $\de\downarrow 0$ within $\hat V(b)$, satisfying
\be\label{pp} \underset{\de}{\cap} \ker P(z_\de) = \ker {P}_\cap (b)\ee
given in \eqref{pka}, \eqref{pkb}.

Assume the existence of $w, w_\Ga$ satisfying \eqref{aga}, \eqref{agb}
 and $z_0 \in \hat V(b)$ such that using \eqref{ga},
 \be\label{pq} \dist_\cals (z_0, z_\de, w, w_\Ga)\overset{\de\downarrow 0}{\longrightarrow} 0,\ee
and $\| \cdot \|_{\partial\Om, P(z_\de)}$ such that \eqref{agd} holds for each $z_\de$ with $c_1 (z_\de, P(z_\de), w \circ T^{-1}_\de, w_\Ga \circ T^{-1}_\de)$ bounded uniformly with respect to $\de$.

Then $z_0 \in \tilde \cals$ as determined by \eqref{pn}.
\end{thm}

\begin{rem*} The assumption \eqref{pp} is slightly stronger than necessary for simplicity.  Throughout $T_\de$ is determined from \eqref{pq}, \eqref{ga}.

If theorem 8.3 applies to all $b\in \tilde B_{P_\cals}$, then
\be\label{pr} \{ b(z, P_\cals) \, | \, z \in\tilde \cals\} = \{ b(z, P_\cals) \, | \, z \in \hat\cals\};\ee
adoption of \eqref{pn} does not reduce $\tilde B_{P_\cals}$. No examples where such fails are known.
\end{rem*}
\begin{proof} It will suffice to show that $(z_0, P_\cap (b))$ is unambiguous, establishing
\be\label{pra}
 P(z_0) = P_\cap(b),\ee
from which $V_\subset (z_0)$ empty will follow from \eqref{plc}, \eqref{pka}.

Stability of $(z_0, P_\cap (b))$ follows by assumption \eqref{pq}, using theorem 6.4.

The proof is completed with the following lemma, expressing admissibility of a limit of admissible sequences.
\end{proof}
\begin{lem} Assume a sequence $\{(z_\de, P_\de)\}$, each admissible, satisfying \eqref{pq} and such that the limit
\be\label{psa} \ker P_0 = \mathop{\cap}\limits_\de \ker P_\de\ee
exists.

Then $(z_0, P_0)$ is admissible.
\end{lem}
\begin{proof} We apply lemma 3.4 to $z_0, P_0$.  Using \eqref{pq}, \eqref{psa} it follows that the conclusions \eqref{bda}, \eqref{bdb}, \eqref{bdc}, \eqref{bdd} hold for any sufficiently small $\de$, without loss of generality uniformly with respect to $\de$.

Thus theorem 4.4 holds uniformly with respect to $\de$.  For any bounded $\Om''$, the sequence $\{ Q(z_\de, P_\de, \Ga(z_\de) \cap \Om'', \cdot)\}$ is equicontinuous in $\vare$ at $\vare = 0$.  With each $(z_\de, P_\de)$ admissible, it follows that the sequence $\{ (z_\de, P_\de)\} $ is limit admissible, satisfying \eqref{db}.

Then application of theorem 5.2 gives $(z_0, P_0)$ admissible, as \eqref{pq} implies $T_\de \to I_m$ on $\partial\Om$ as $\de \downarrow 0$ and   \eqref{psa} implies \eqref{dba} (with equality).
\end{proof}

Using definitions 6.6, 8.2, a corollary of theorem 7.2 is the following.

\begin{thm} Assume a successful computational investigation with $\Om$ bounded, such that for each $z_0 \in \hcals_0, \; \, (z_0, 0_N)$ is admissible and \eqref{hac}, \eqref{had}, \eqref{kaa} hold.

Then $(\hcals_0, P_\cals)$ determine as weakly well-posed problem.
\end{thm}
\begin{proof} With $\Om$ bounded, presumably required for computational investigation, each $(z_0, P_\cals)$ is potentially admissible  as per definition 7.1.  The conclusion
now follows from definition 8.2, using theorem 7.2.
\end{proof}

A posteriori of a successful computational investigation, admissible $(z_0, 0_N)$ is anticipated.
\begin{lem} For $z_0 \in\hat\cals_0$, assume there is no nontrivial $f\in L_2(\Om)^n$ simultaneously satisfying
\be\label{xba} R^\dag (z_0) f = 0\ee
almost everywhere in $\Om\setminus\Ga(z_0)$;
\be\label{xbb} \psi_{\nu,zz} (z_0) f = 0 \ee
almost everywhere in $\pOm$; and
\be\label{xbc} \intl_{\Ga(z_0)}[f\cdot \psi_{\mu,zz} (z_0) \theta ] = 0, \; \; \theta \in H_{\Om\setminus\Ga} (z_0, 0_N, 1)\ee
using \eqref{afb}, \eqref{adc}, \eqref{ala}.

Then $(z_0, 0_N)$ is admissible.
\end{lem}

\begin{rema*} The example in figure 2.1 is decidedly pathological, depending on a scalar conservation law and incompatible with \eqref{faf}.
\end{rema*}
\begin{proof} Failure of $(z_0, 0_N)$ admissible implies nontrivial $f$ satisfying
\be\label{xbd} \iintl_{\Om\setminus\Ga(z_0)} f \cdot R(z_0) \theta = 0, \; \;  \, \theta \in H_{\Om\setminus\Ga} (z_0, 0_N, 1).\ee

Partial integration of \eqref{xbd} using \eqref{af} determining $f$ satisfying \eqref{xba}, \eqref{xbb}, \eqref{xbc}.
\end{proof}

Given $z_0$, from \eqref{xbc} by inspection
\be\label{xbe} [ \psi_{\mu,zz} (z_0) f] = 0\ee
almost everywhere in $\Ga(z)$ is anticipated, and that the boundary of the support of any such $f$ cannot lie within the interior of $\Om$, contradicting \eqref{xbb}.

\section{Sufficient conditions for weak well-posedness}

The next four sections are devoted to determination of sufficient conditions for weak well-posedness as per definition 8.2.

Of necessity, an abuse of notation is made, using $a, h, \beta, \gamma, \delta, \vare, \zeta, \eta, \chi, \kappa, \xi$, \linebreak $  \tau, \omega, \Sigma, E, Y$  differently below than previously, typically with different subscripts. Use of $z,\theta, \psi, B, \cals, P, \Om, \nu, \mu$ below is consistent with the above.

In this context, we regard successful computational investigations as having established \eqref{fbc}, \eqref{fbd}  beyond reasonable doubt, with a (nontrivial) empirically determined solution set $\hat\cals_0$ and corresponding  boundary data
\be\label{raa} \tilde B_{P_\cals} = \big\{ (I_n-P_\cals) \psi^\dag_{\nu,z} (z), \;  z \in \hat\cals_0 \big\}\ee
replacing \eqref{fab} using \eqref{fbc}, \eqref{fbd}, with some $P_\cals$ independent of $z \in \hcals_0$.

We also assume each $z\in \hcals_0$ such that each $(z, P_\cals)$ is stable,
\be\label{rast} \hcals_0 \subset \{ z\, |\,  (z, P_\cals) \; \hbox{stable}\}.\ee

The decisive condition for weak well-posedness is then \eqref{pab}, \eqref{pa}.  We thus seek to identify  problem classes with nontrivial  projected solution sets $\hat\cals$ such that $(z,P(z))$ is unambiguous for all $z \in \hat\cals$, for some $P(z)$ satisfying \eqref{hae}.

For  this purpose, a problem class is determined by a specified system of equations in a domain $\Om$, required local structure of weak solutions $z \in \hat \cals$, the required boundary conditions determined by $P_\cals$, and conditions on the boundary data $b \in \tilde B_{P_\cals}$ from \eqref{raa}, tacitly anticipating \eqref{pab}.

Below  identification of sufficient conditions for weak well-posedness is extended over four sections, successively determining various restrictions on the problem class details.

Theorems 9.1, 9.2 express sufficient conditions on $z \in \hcals$ to relate the required boundary estimate $\| \theta \|_{\pOm, P(z)}$  to interior estimates, $\| \theta\|_{L_2(\Ga(z))}$ in particular.
 Theorem 9.3 below expresses abstract sufficient conditions for $(z, P)$ unambiguous.  Anticipating application of theorem 7.2, with suitable restrictions on the regularity of the boundary data, sufficient condition for $(z, P_\cals)$ stable and  $(z, 0_N)$ admissible are expressed in theorems 9.4, 9.5 respectively.

Then theorem 10.4 expresses sufficient structural conditions on elements $z \in\hat\cals$, in particular a local entropy condition on whatever discontinuities $\Ga(z)$, such that the assumptions of theorems 9.4, 9.5 hold.  Here it is required that each $z \in \hat\cals$ be equipped with a distinguished direction determined from $e(\cdot; z) \in \bbs^m$ almost everywhere within $\Om$, reminiscent of time-like directions for hyperbolic systems.

Section 11 then relates the assumptions of theorem 10.4 to conditions on the given system \eqref{aa}.  Sufficient conditions are obtained, including hyperbolicity of the system \eqref{aa}, that  the required entropy condition is implied by  the familiar inequality \eqref{faf}.
 The assumption of hyperbolicity of the underlying system is subsequently  escaped by association of $\hat\cals$ with a reduced system obtained from a given hyperbolic ``primitive" system \eqref{aa} using the associated symmetry group.

 Construction of the required $e(\cdot ; z)$ is then  materially simplified by assumption of further increased regularity of the boundary data.

Assembling these results in section 12, we obtain problem class details for which weak well-posedness appears at least realistic.  The attractiveness of fluid flow models in this context is made explicit.

In the absence of suitable existence theorems (for example for systems \eqref{aa} with $m > 2$ or not everywhere hyperbolic), existence of nontrivial $\hat\cals $ must depend at present  on ``verification" a posteriori of computational investigations.  In this context, in section 13  we introduce and discuss a class of approximation schemes, the success of which is directly related to sufficient conditions for weak well-posedness.   Application thereof is beyond the scope of the present work.

\subsection{Abstract conditions for stability and admissibility}

Judicious choice of $\| \cdot \|_{\partial\Om, P}$ reduces stability of $(z,P)$, for $z$ of the form \eqref{ad}, \eqref{ada}, to a priori estimates for weak solutions $\theta $ of
\be\label{rad} w^{1/2} R(z)\theta = f\ee
within $\Om\setminus \Ga(z), \; \; f \in L_2(\Om\setminus \Ga (z))^n$ otherwise arbitrary;
\be\label{rae} w^{1/2}_\Ga S(z) \theta = g\ee
within $\Ga(z), g \in L_2 (\Ga (z))$ otherwise arbitrary, satisfying
\be\label{raf} P \theta = 0\ee
almost everywhere within $\partial\Om$.

\begin{thm}
 Assume $P_0 (z) $ piecewise continuous, uniformly inverse bounded on its range, and satisfying \eqref{aged} with $P = P_0(z)$.   Assume
\be\label{pya} z \in W^{1, \infty}(\Omega_\kappa)\ee
for some $\kappa > 0$, with
\be\label{pyb} \Omega_\kappa \eqadef (\Omega\setminus\Ga (z)) \cap \{ x | dist (x, \partial\Om) < \kappa \}.\ee

Assume $w, w_\Ga, c_{z, w, w_\Ga}$ such that for all $\theta \in H(z, P_0(z), w, w_\Ga)$
\be\label{pyc} \| \theta\|_{L_2(\Omega)} \le c_{z, w, w_\Ga} \| \theta\|_{z, w, w_\Ga}.\ee

Then $(z, P_0(z) )$ is stable.
\end{thm}

\begin{proof}  We choose a scalar function $\chi \in H^1 (\Om)\cap L_\infty (\Om)$ satisfying
\be\label{pyd} \supp \; \chi \subset \Om_\kappa,\ee
and
\be\label{pyf} \chi (x) \le \, \hbox{minimum\ }  (1, \dist (x, \Ga(z)), \; \; \, x \in \partial\Omega.\ee

We designate
\be\label{pye} \| \theta\|_{\partial\Om, P_0 (z)} \eqadef\,  \underset{\theta'}{\lub}\,  \frac{\intl_{\partial\Om} \chi (I_n - P_0(z)) \theta \cdot \psi_{\nu, zz} (z)\theta'}{\| \theta'\|_{H^1(\Om)}}
\ee
with $\theta'$ restricted by
\be\label{pyg} \supp\;  \theta' \subset \Om_\kappa\ee

Then from \eqref{pye}, \eqref{pyf}, after partial integration and subsequent use of
\eqref{af}, \eqref{ah}, \eqref{pyc}, \eqref{pyg}, \eqref{pya}
\begin{align}\label{pyj}
\| \theta\|_{\partial\Om, P_0(z)} &\le\underset{\theta'}{\lub}  \intl_{\partial\Om} \theta \cdot  \psi_{\nu, zz} (z) \theta'/\| \theta'\|_{H^1(\Om)} \nonumber \\
&=\underset{\theta'}{\lub} \Big(\iintl_{\Om_\kappa} R(z) \theta \cdot \theta' - \iintl_{\Om_\kappa} \theta \cdot R^\dag(z) \theta' - \iintl_{\Om_k} \theta \cdot
\Big(\suml^{m}_{i = 1} (\psi_{i,zz} (z))_{x_i}\Big) \theta' \Big)/\| \theta'\|_{H^1(\Om_\kappa)}\nonumber \\
&\le c \; \underset{\theta'}{\lub} \Big(\| \theta\|_{z, 1, 1} \| \theta' \|_{L_2 (\Om_\kappa)} + \| \theta \|_{L_2(\Om_\kappa)} \| \theta'\|_{H^1(\Om_\kappa)} + \| \theta\|_{L_2(\Om_\k)} \| \theta'\|_{L_2 (\Om_\k)}\Big)/\| \theta'\|_{H^1(\Om_\kappa)} \nonumber  \\
&\le c \| \theta \|_{z, w, w_\Ga}.
\end{align}
which implies \eqref{agd}, with $c$ depending on $\| z\|_{W^{1, \infty} (\Om_\k)}, \| \chi\|_{H^1(\Om)}$.
\end{proof}

Alternatively, the assumption \eqref{pya}, \eqref{pyb} may be weakened with a stronger form of \eqref{pyc}.

\begin{thm} Assume $z \in H^1(\Omega\setminus \Ga(z))$, and $w, w_\Ga, c_{z, w, w_\Ga}$
such that for all $\theta \in H(z, P(z), w, w_\Ga)$,
\be\label{pxm} \| \theta/ w_0\|_{L_2 (\Om)} + \| \theta\|_{L_2(\Ga(z))} \le c_{z, w, w_\Ga}\| \theta\|_{z, w,w_\Ga}, \ee
for some $w_0: \Om \to [1, \infty)$.

Assume \eqref{aged} with $P = P(z)$.

Then $(z, P(z))$ is stable.
\end{thm}
\begin{rema*} The condition \eqref{pxm} precludes nontrivial $\theta$ satisfying \eqref{agc}.
\end{rema*}
\begin{proof}  Using \eqref{aged}, \eqref{aga},  we denote
\be\label{pxn} \| \theta\|_{\partial \Om, P(z)} \eqadef \, \underset{\theta' \in H(z, 0_N, w, w_\Ga)}{\lub} \, \frac{\intl_{\partial\Om} \theta' \cdot \psi_{\nu, zz} (z)\theta}
{\| \theta'\|_{L_2 (\Om)}
+ \| w_0 R^\dag (z) \theta'
\|_{L_2 (\Om\setminus \Ga(z))} + \| \theta'\|_{L_2(\Ga(z))}}.\ee

Then using \eqref{af} and partial integration, for some $c_z$ for arbitrary $\theta'\in H(z, 0_N, w, w_\Ga)$
\begin{align}\label{pxo} |\intl_{\partial\Om} \theta' &\cdot \psi_{\nu, zz} (z) \theta | \le | \iintl_{\Om\setminus\Ga(z)} \theta' \cdot R(z) \theta | + |\iintl_{\Om\setminus\Ga(z)} \theta \cdot R^\dag (z) \theta'| + c_z\intl_{\Ga(z)} |\theta| |\theta'|\nonumber \\
\le & \|  \theta' \|_{L_2 (\Om\setminus \Ga(z))} \| \theta\|_{z, w, w_\Ga}
 + \| \theta/w_0\|_{L_2(\Om)} \| w_0 R^\dag (z) \theta'\|_{L_2 (\Om\setminus\Ga(z))}\nonumber \\
\qquad \quad &+ c_z\|\theta\|_{L_2(\Ga(z))} \| \theta'\|_{L_2 (\Ga(z))}\nonumber \\
\le& c_z(\| \theta\|_{z, w, w_\Ga} + \| \theta/w_0\|_{L_2 (\Om)} + \| \theta\|_{L_2 (\Ga(z))}) ( \|  \theta'\|_{L_2 (\Om)}\nonumber \\
\qquad \quad &+ \| w_0 R^\dag (z) \theta' \|_{L_2(\Om\setminus \Ga(z))} + \| \theta'\|_{L_2 (\Ga (z))})\nonumber\\
\le& c_z (1+c_{z, w,w_\Ga})\| \theta\|_{z, w, w_\Ga} (\|  \theta'\|_{L_2 (\Om)} + \| w_0 R^\dag (z) \theta'\|_{L_2(\Om\setminus\Ga(z))} + \| \theta'\|_{L_2(\Ga(z))})\end{align}
using \eqref{pxm}.
\end{proof}

Sufficient conditions for $(z, P(z))$ unambiguous are now obtained,  employing  theorem 9.2 with
\be\label{rafb} w_0 = 1.\ee

A partitioning of $\Om$ is introduced, depending on a unit vector
\be\label{saa} e(\cdot ; z) \in C(\Om\setminus\Ga^I(z) \to \bbs^m),\ee
with
\be\label{saaz} \Ga^I (z) \eqadef  \mathop{\cup}\limits_{k \neq k'} \partial\Ga_k (z) \cap \partial\Ga_{k'} (z) \ee
composed of the $(m-2)$-manifolds corresponding to ``interaction points" where $\hat\mu$ is not continuous.

Analogously with \eqref{aae}, such $e(\cdot; z)$ determines
\be\label{saf} \psi_e (y)(x) \eqadef \suml^m_{i = 1} e_i(x; z) \psi_i (y)\ee
\be\label{sag} \psi_{e, zz} (y)(x)\eqadef \, \suml^m_{i = 1} e_i(x; z) \psi_{i, zz} (y), \; \; x \in \Om\setminus \Ga(z), \; \, y \in D,\ee
and
\be\label{rak} \triangle (z; e) \eqadef\, \{ x \in \Om\setminus \Ga (z) | \psi_{e, zz} (z(x)) (x) \; \hbox{ is\ singular}\}.\ee

We assume $z, e $ such that $\triangle (z; e)$ is a finite union of disjoint $(m-1)$-manifolds (familiar as ``sonic manifolds")
\be\label{rap} \triangle (z; e) = \mathop{\cup}\limits^{K'}_{k' = 1} \triangle_{k'} (z; e),\ee
within each of which there exists a continuous unit normal $\hat \mu_{k'}$.

From \eqref{rak}, $\Ga(z) \cap \triangle (z; e)$ is empty; by convention, $\partial\Ga(z) \cap \triangle (z; e)$ and $\Ga(z) \cap \partial\triangle (z; e)$ are also empty.  But nonempty $\partial\Ga (z) \cap \partial\triangle (z; e)$ is anticipated.

Then there exist open, connected $\Om_l(z), \; l = 1, \dots, L_\cals$,  such that
\be\label{rag} \hbox{measure\ } \left( \Om\setminus ( \mathop{\cup}\limits^{L_\cals}_{l = 1} \Om_l (z)) \right) = 0 \ee
in $\bbr^m$, with boundaries $\partial\Om_l(z)$ satisfying
\be\label{rah} \partial\Om_l(z) = (\partial\Om_l (z) \cap \partial\Om) \mathop{\cup}\limits_{l' \neq l} (\partial \Om_l (z) \cap \partial\Om_{l'} (z)), \ee
\be\label{rai} \Gamma(z) \cup \triangle (z; e) \subseteq \mathop{\cup}\limits_{l \neq l'} (\partial\Om_l(z) \cap \partial\Om_{l'} (z)).\ee

An exterior unit normal $\nu_l$ is defined almost everywhere on each $\partial\Om_l$, necessarily satisfying
\be\label{raha} \nu_l (x) = \nu (x), \; \; x \in \partial\Om\ee
and
\be\label{rahb}
\nu_l(x) = - \nu_{l'} (x), \; \; x \in \partial\Om_l \cap \partial\Om_{l'}.\ee

In particular, for each $k = 1, \dots, K$ in \eqref{ada} there exist $l_k, l'_k$ such that
\be\label{raj} \Ga_k(z) = \partial\Om_{l_k} (z) \cap \partial\Om_{l'_k} (z).\ee

Similarly, for $k' = K+1, \dots, K'$, $K'$ independent of $z \in \hat \cals$, each segment $\triangle_{k'}$ satisfies
\be\label{raja} \triangle_{k'} (z;e) = \partial\Om_{l_{k'}} (z) \cap \partial\Om_{l'_{k'}} (z).\ee

\begin{thm}
Assume $z\in H^1(\Om\setminus\Ga(z))$, $w$ from \eqref{aga}, $w_\Ga$ from \eqref{agb}, $w_0$ from \eqref{pxm}, $P(z)$, existence of mappings
\be\label{rba} \land(z): H(z, P(z), w, w_\Ga) \mathop|\limits_{\Om_l} \to \bbr^{n\times n}, \; \; \, l = 1, \dots, L_\cals,\ee
\be\label{rbe} \Sigma(z): H(z, P(z), w, w_\Ga) \mathop{|}\limits_{\Ga_k} \to \bbr, \; \; \, k= 1, \dots, K,\ee
and a symmetric $n$-matrix valued function
\be\label{rbh} B_l (z) \in L_\infty (\partial\Om_l \to \bbr^{n\times n}), \; \; \, l = 1,\dots, L_\cals\ee
such that the following hold:
 \be\label{rbx}  w^{1/2}  \ge | \land (z) |;\ee
\be\label{rbc}
\land (z) \tilde H \text{ is\ dense\ in\ }  L_2(\Om\setminus\Ga (z))^n\ee
with $\tilde H \subset H_{\Om\setminus\Ga} (z, P(z), w) $ using \eqref{ala},
\be\label{rbd}
\tilde H \eqadef \{ \theta\in  H (z, P(z), w, w_\Ga) | \theta = 0 \; \text{within \ }  \Ga (z)\};\ee
\be\label{rbf} \Sigma(z) H(z, P(z), w, w_\Ga) \; \, \text{is\ dense\ in \ } L_2 (\Ga(z));\ee
\be\label{rby} w^{1/2}_\Ga \ge | \Sigma (z)|.\ee

Assume in addition that for all $\theta \in H(z, P(z), w, w_\Ga)$
\be\label{rbl} \intl_{\partial\Om_l \cap \partial\Om_{l'} } \theta\cdot (B_{l_{k'}} (z) + B_{l'_{k'}}(z)) \theta = 0, \; \, k' = 1,\dots, K',\ee
for all $l, l'$ such that $\partial\Om_{l} (z) \cap \partial\Om_{l'} (z) \not\subset   \Ga (z);$
\be\label{rbi} \theta\cdot B_l(z)\theta \ge 0, \; l = 1, \dots, L_\cals,\ee
almost everywhere within $\partial\Om_l\cap \partial\Om$, and such that the following hold:  for each $l = 1, \dots L_\cals$,
    \be\label{rbj} c\| w^{1/2} R(z) \theta\|_{L_2(\Om_l)} \|  \theta \|_{L_2 (\Om_l)}\ge \| \theta \|^2_{L_2(\Om_l)} + \intl_{\partial\Om_l} \theta\cdot B_l(z)\theta;\ee
and with some constant $c_\Ga$ (independent of $\theta$)
\be\label{rbk}\|\theta\|^2_{L_2 (\Ga(z))} \le c_\Ga\suml^K_{k = 1}\, \intl_{\Ga_k(z)} \theta\cdot (B_{l_k} (z) + B_{l'_{k}}(z))\theta + c\intl_{\Ga (z)}  (S(z) \theta)(\Sigma(z)\theta).\ee

Then $(z, P(z))$ is unambiguous, with $\| \cdot\|_{\partial\Om, P(z)}$ obtained from
\eqref{pxn}, \eqref{rafb}.
\end{thm}

\begin{rem*}

The matrix $B_l(z)$ does not have to be positive semidefinite within $\partial\Om_l(z)\cap \partial\Om_{l'} (z)$.

An entropy condition on the discontinuities is implicit  in \eqref{rbk},  \eqref{rbf}.

The required $P(z)$ is determined  implicitly from \eqref{rbi}.

 Theorem 9.3 can be proved in this generality  for some everywhere hyperbolic systems and with $m = 2$.  There are two reasons for this.  One is the need of a suitable basis for $H(z, P(z), w, w_\Ga)$ to obtain suitable expressions for $R(z)\theta$.  (For hyperbolic systems with $m = 2$, the characteristic eigenvectors suffice.)  The second problem is that $S(z)\theta$ is scalar-valued, whereas $S(z)$ in \eqref{ag} includes $m-1$ independent $\alpha$-derivatives, generally precluding  \eqref{rbk} for $m > 2$.
\end{rem*}

 \begin{proof}
 Summing \eqref{rbj} with respect to $l,$ using \eqref{rag}, \eqref{rah}, \eqref{rai}, \eqref{raj}, \eqref{rbi}, \eqref{rbl}, \eqref{rbx}, we obtain
 \begin{align}\label{rca} \|\theta\|^2_{L_2(\Om)} &+ \suml^K_{k = 1} \, \intl_{\Ga_k} \theta \cdot (B_{l_k} + B_{l'_k} )\theta\nonumber\\
\le &c \suml^{L_\cals}_{l = 1}\| w^{1/2} R(z)\theta\|_{L_2(\Om_l)} \| w^{-1/2} \land (z) \theta\|_{L_2(\Om_l)}  \nonumber \\
\le & c\| w^{1/2} R(z) \theta\|_{L_2 (\Om)} \| \theta \|_{L_2(\Om)\setminus \Ga(z))}\nonumber \\
\le & c  \| w^{1/2} R(z) \theta\|^2_{L_2 \Om\setminus \Ga(z)} \| + \thalf \| \theta \|^2_{L_2(\Om\setminus\Ga(z))}\end{align}

 Combining \eqref{rca}, \eqref{rbk}, \eqref{rby},
 \begin{align}\label{rcb} \tfrac 12\| \theta\|^2_{L_2(\Om\setminus\Ga(z))} &+ \frac{1}{c_\Ga} \|  \theta\|^2_{L_2(\Ga)}
 \le c \|w^{1/2}R(z) \theta \|^2_{L_2(\Om\setminus\Ga(z))} + c \intl_\Ga (S (z)  \theta)(\Sigma(z)\theta)\nonumber \\
\le & c\| w^{1/2} R(z)\theta\|^2_{L_2(\Om)\setminus\Ga(z))} + c \| w^{1/2}_\Ga S(z)\theta\|_{L_2(\Ga)} \|\theta\|_{L_2(\Ga)}\nonumber \\
 \le & c\| w^{1/2} R(z)\theta\|^2_{L_2(\Om)\setminus\Ga(z))} + c \| w^{1/2}_\Ga S(z)\theta\|^2_{L_2(\Ga)}+ \frac {1}{2c_\Ga}  \|\theta\|^2_{L_2(\Ga)}\nonumber \\
 \le & c\| \theta\|^2_{z,w,w_\Ga} + \frac{1}{2c_\Ga} \| \theta \|^2_{L_2(\Ga)}.\end{align}

 Stability of $(z, P(z))$ follows from \eqref{rcb}, \eqref{pxm}.

 Admissibility and thus $(z, P(z))$ unambiguous now follows from existence of a solution of \eqref{rad}, \eqref{rae}, \eqref{raf}.   Here we employ a variant of the familiar Lax-Milgram theorem.

 For $\theta \in H (z, P(z), w, w_\Ga)$, we denote a mapping on $H(z, P(z), w, w_\Ga)$
 \be\label{rdc} M_Y (\theta) \eqadef\begin{cases} \land(z)\theta  &\hbox{within\ } \Om\setminus\Ga (z)\\
 \Sigma (z) \theta &\hbox{within\ } \Ga(z),\end{cases}\ee
 which is $n$-vector valued within $\Om\setminus \Ga(z)$ and scalar valued within $\Ga(z)$.  Denote a Hilbert space
 \be\label{rcc} H_\phi \eqadef \{ M_Y\theta, \; \theta \in H(z, P(z), w, w_\Ga)\}\ee
 equipped with inner product
 \be\label{rcca} ((\phi, \phi'))\, \eqadef\,\iintl_{\Om\setminus\Ga(z)} \phi\cdot \phi' + \frac{1}{c_\Ga} \intl_{\Ga(z)} \phi\phi',\ee
 the constant $c_\Ga$ from \eqref{rbk}.

 We denote a bilinear form on $H(z, P(z), w, w_\Ga) \times H_\phi$
 \be\label{rcd} Y(\theta, \phi)  \eqadef\, \iintl_{\Om\setminus\Ga (z)} w^{1/2}  R(z)\theta\cdot \phi + \frac{1}{c_\Ga} \intl_{\Ga(z)} ( w^{1/2}_\Ga  S(z)\theta) \phi.\ee

 By inspection, using \eqref{ah}, \eqref{rcca}
 \be\label{rce} |Y(\theta,\phi)| \le c \|\theta\|_{z, w, w_\Ga} \| \phi\|_{H_\phi},\ee
 so for any $\theta \in H(z, P(z), w, w_\Ga)$ there exists a unique $\gamma_\theta \in H_\phi$ such that
 \be\label{rcf} Y(\theta, \phi) = ((\gamma_\theta, \phi))\ee
 for all $\phi \in H_\phi$.

 For $f, g$ as in \eqref{rad}, \eqref{rae}, we denote a bounded linear functional on $H_\phi$
 \be\label{rcg} \tilde Y_{f, g} (\phi) \, \eqadef \iintl_{\Om\setminus \Ga(z)} f \cdot \phi + \frac {1}{c_\Ga} \intl_{\Ga(z)} g \phi,\ee
 and there exists a unique $\xi_{f, g} \in H_\phi$  such that for all $\phi \in H_\phi$,
 \be\label{rch} \tilde Y_{f, g} (\phi) = ((\xi_{f,g}, \phi)).\ee

  We claim that for any such $f, g$, there exists $\theta' \in H(z, P(z), w, w_\Ga)$ such that $\gamma_{\theta'}$ coincides with $\xi_{f,g}$,
  \be\label{rda} (( \gamma_{\theta'} - \xi_{f, g}, \phi)) = 0\ee
  for all $\phi \in H_\phi$.

  Should \eqref{rda} fail, there exists nonzero $f, g$ such that for all
  $\theta \in H(z, P(z), w,w_\Ga)$,
  \begin{align}\label{rdb} 0 &= ((\gamma_\theta, \xi_{f, g} ))\nonumber \\
  &= Y(\theta, \xi_{f, g)})\end{align}
  using \eqref{rcf}.

  From \eqref{rcd}, using \eqref{rdc}, \eqref{rbj}, \eqref{rbk}, \eqref{rbi}, \eqref{rbl},
  \be\label{rdd} Y(\theta, M_Y \theta) \ge \| \theta \|^2_{L_2(\Om)} + \frac{1}{c_\Ga} \| \theta\|^2_{L_2(\Ga(z))},\ee
  so as $M_Y$ is onto $H_\phi$ by definition, $\theta'' \in H (z, P(z),w,w_\Ga)$ satisfying
  \be\label{rde} M_Y \theta'' = \xi_{f,g},\ee
  which contradicts \eqref{rdb} and establishes \eqref{rda}.

  From \eqref{rda}, \eqref{rcca}, using \eqref{rcd}, \eqref{rcf} with $\theta = \theta'$ and \eqref{rcg}, \eqref{rch}, for all $\phi \in H_\phi$
  \be\label{rdf} \iintl_{\Om\setminus \Ga(z)} (w^{1\setminus 2} R(z) \theta' - f)\cdot \phi + \frac{1}{c_\Ga} \int (w^{1\setminus 2}_\Ga S(z) \theta' - g)\phi = 0.\ee

  Choosing
  \be\label{rdg} \phi \in M_Y \tilde H,\ee
  using \eqref{rbc}, \eqref{rbd}, \eqref{rdc}, \ the statement \eqref{rad} (with $\theta = \theta'$) follows from \eqref{rdf}. Then \eqref{rae} follows from \eqref{rdf} using \eqref{rbf}.
 \end{proof}

\subsection{Boundary data regularity} 
Use of theorem  9.3 is made  at the expense of required increased regularity in the prescribed data $\dot b$.

For $z\in \hcals, \, z \in H^1 (\Om \setminus\Ga(z)), \, P(z)$ as in theorems 9.2, 9.3 with \eqref{rafb}, we shall construct an alternative test space $H' (z, P(z))$ satisfying
\be\label{rha} H'(z, P(z)) \hbox{ \ dense \ in the space \ } \{ \theta \in (C^1(\Om)\cap C (\ol{\Om}))^n | P(z)\theta\mathop{\mid}\limits_{\pOm} = 0 \}\ee
with respect to the norm $W^{1,\infty} (\Om)$.

The space $H'(z, P(z))$ has structure
\be\label{rfa} H'(z, P(z)) = \mathop{\oplus}\limits^\infty_{\eta=1} \; \hat H_\eta,\ee
with each
\be\label{rhb} \hat H_\eta = \mathop{\oplus}\limits^\eta_{\eta'=1} \, \hat H'_{\eta'}.\ee

The spaces $\hat H'_\eta$ satisfy an orthogonality condition
\be\label{rhc} (\theta'_{\eta'}, \theta''_{\eta''})_H = 0, \, \, \theta'_{\eta'} \in \hat H'_{\eta'}, \theta ''_{\eta''} \in \hat H'_{\eta''}, \; \; \eta' \neq \eta''\ee
using \eqref{aha} with
\be \label{rhd} w = w_\Ga = 1.\ee

In sections 10, 11, we shall construct the spaces $\hat H'_\eta$, depending on $z, P(z)$, such that for all $\theta_\eta \in \hat H'_\eta$ and each $\eta$,
\be\label{pxmr} \| \theta_\eta \|_{L_2(\Om\setminus \Ga(z))} + \| \theta_\eta\|_{L_2(\Ga(z))} \le c_{z, w_\eta, w_{\Ga\eta}} \| \theta_\eta\|_{z, w_\eta, w_{\Ga \eta}}\ee
using the norm \eqref{ah} (with $w_\eta, w_{\Ga \eta}$ replacing $w, w_\Ga$) with some constants $c_{z, w_\eta, w_{\Ga\eta}}$ and suitable functions
\begin{align}\label{rhe} & w_\eta: \Om\setminus\Ga(z) \to [1, \| w_\eta\|_{L_\infty}] \nonumber \\
& w_{\Ga\eta}: \Ga (z) \to   [1, \| w_{\Ga \eta}\|_{L_\infty}].\end{align}

In \eqref{pxmr}, \eqref{rhe}, we anticipate
\be\label{rhf} c_{z, w_\eta, w_{\Ga\eta}}, \| w_\eta \|_{L_\infty}, \, \, \| w_{\Ga \eta}\|_{L_\infty} \overset{\eta \to\infty}{\longrightarrow} \infty,\ee
indeed with
\be\label{rfg} w_\eta (\cdot) \, \overset{\eta \to\infty}{\longrightarrow} \infty\ee
almost everywhere in $\Om\setminus\Ga(z)$, and
\be\label{rfh} w_{\Ga \eta}(\cdot) \overset{\eta \to\infty}{\longrightarrow} \infty\ee
almost everywhere in $\Ga(z)$.

For $w = 1$ or $w_{\eta}, w_\Ga = 1$ or $w_{\Ga\eta}$, denote
\be\label{rja} \hat H'_\eta (z, P(z), w, w_\Ga) \, \eqadef\; \hbox{ completion of \ } H'_\eta\ee
with respect to the norm $\| \cdot\|_{z, w, w_\Ga}$ in \eqref{ah}.

From \eqref{rhe}, notwithstanding \eqref{rhf},
\be\label{rjb} \hat H'_\eta (z, P(z), w_\eta, w_{\Ga\eta}) \, =   \hat H'_\eta (z, P(z), 1, 1).\ee

In \eqref{ahd}, \eqref{ahb}, \eqref{ae}, we assume $\dot b$ of the form
\be\label{rfna} \dot b = \suml^\infty_{\eta= 1} \dot b_\eta,\ee
\be\label{rjc} \intl_{\pOm} (\dot b - \dot b_\eta) \cdot \theta_\eta = 0, \; \; \theta_\eta \in \hat H'_\eta,\ee
\be\label{rjca} \intl_{\pOm} \dot b_\eta  \cdot \theta_{\eta'} = 0, \; \; \theta_{\eta'} \in \hat H'_{\eta'}, \;\eta' \neq \eta.\ee

We denote constants
\be\label{rjd} c'_\eta \, \eqadef\, \hbox{maximum \ } (\| w_\eta\|^{1/2}_{L_\infty}, \; \| w_{\Ga\eta} \|^{1/2}_{L_\infty}).\ee

Then for any $\dot b_\eta \in B_{P(z)}$, using \eqref{ahb}, then \eqref{pxn}, \eqref{pxo}, \eqref{rafb}, \eqref{rja}, then \eqref{rjb}, \eqref{rjd},
\begin{align}\label{rjda} | \intl_{\pOm} \dot b_\eta \cdot \theta_\eta | &\le \| \dot b_\eta \|_{\calb, P(z)} \| \theta_\eta\|_{\pOm, P(z)}\nonumber \\
&\le c_z (1+c_{z, w_\eta, w_{\Ga\eta}}) \| \dot b_\eta\|_{\calb, P(z)} \| \theta_\eta\|_{z, w_\eta, w_{\Ga\eta}}\nonumber \\
&\le c_z c'_\eta (1+c_{z, w_\eta, w_{\Ga \eta}}) \| \dot b_\eta \|_{\calb, P(z)} \| \theta_\eta\|_{z, 1, 1}\end{align}
for all $\theta_\eta \in \hat H'_\eta (z, P(z), 1, 1)$.

Using \eqref{rjda}, \eqref{rhd}, by application of the Riesz theorem, for given $\dot b_\eta$ there exists a unique
\be\label{rjdb}\zeta'_\eta \in \hat H'_\eta (z, P, 1, 1), \; \; \|  \zeta'_\eta\|_{z, 1, 1} \le c_zc'_\eta (1+c_{z, w_\eta, w_{\Ga eta}}) \| \dot b _\eta\|_{\calb, P(z)}\ee
satisfying
\begin{align}\label{rjdc} \intl_{\pOm} \dot b_\eta \cdot \theta &= (\zeta'_\eta, \theta)_H\nonumber \\
&= \iintl_{\Om\setminus\Ga(z)} \dot z'_\eta \cdot R(z)\theta + \intl_{\Ga(z)} \sigma'_\eta S(z)\theta\end{align}
for all $\theta \in \hat H'_\eta (z, P(z), 1,1)$, with
\be\label{rjdd} \dot z'_\eta \, \eqadef \, R(z)\zeta'_\eta, \; \; \sigma'_\eta = S(z)\zeta'_\eta.\ee

Using \eqref{rhc}, \eqref{rhd}, \eqref{rfa}, \eqref{rhb}, it follows that \eqref{rjdc}, \eqref{rjdd} holds for all $\theta \in H'(z, P(z))$, thus for all $\theta \in H(z, P(z), 1, 1)$ using \eqref{rha}.

Thus using \eqref{rfna} and identifying
\be\label{rka} \dot z = \suml^\infty_{\eta = 1} \dot z'_\eta, \; \; \, \sigma = \suml^\infty_{\eta = 1} \sigma'_\eta,\ee
we satisfy \eqref{ae} formally for all $\theta \in H(z, P(z), 1, 1)$.

We thus seek conditions on $\dot b $ such that the sums in \eqref{rfna}, \eqref{rka} are defined.

From \eqref{rfa}, \eqref{rhb}
\be\label{rje}  \theta = \suml^\infty_{\eta = 1} \theta_\eta, \; \; \theta \in H'(z,  P(z)), \; \; \theta_\eta \in \hat H'_\eta, \; \; \eta = 1,2, \dots ,\ee
and from \eqref{rhc}, \eqref{rhd}, $\| \theta\|^2_{z, 1, 1} = \suml^\infty_{\eta = 1} \| \theta_\eta \|^2_{z, 1, 1}$, so
\be\label{rjf} \| \theta\|_{z, 1, 1} \le \suml^\infty_{\eta = 1} 2^{(1+\eta)/2} \| \theta_\eta\|_{z, 1,1}.\ee

From \eqref{rjda}, for $\theta_\eta \in \hat H'_\eta$ and $\| \theta_\eta\|_{\pOm, P(z)}$ from \eqref{pxn}, \eqref{rafb}
\be\label{rjg} \frac{\| \theta_\eta\|_{\pOm, P(z)}}{c'_\eta (1+c_{z, w_\eta, w_{\Ga\eta}})} \le c_z \| \theta_\eta \|_{z,1,1},\ee
and for all $\theta \in H' (z, P(z))$, using \eqref{rjf},
\begin{align}\label{rjh} \| \theta\|_{\pOm, P(z), H} \, &\eqadef \, \suml^\infty_{\eta = 1}
\frac{
\| \theta_\eta\|_{\pOm, P(z)}
}
{
2^{(1+\eta)/2} c'_\eta (1+c_{z, w_\eta, w_{\Ga \eta}})
}\nonumber \\
&\le c_z \|\theta\|_{z, 1,1}.\end{align}

Using \eqref{rha}, the estimate \eqref{rjh} extends immediately to $\theta \in H(z, P(z),1,1)$ and establishes $(z, P(z))$ stable, identifying
\be\label{rji} c_1 (z, P(z), 1,1) = c_z\ee
in \eqref{agd}.

In \eqref{ahb}, we use \eqref{rjh}, \eqref{rje}, \eqref{rfna}, \eqref{rjc}, \eqref{rjca} to obtain
\begin{align}\label{rjk}
\| \dot b\|_{B, P(z) } & = \underset{\theta \in H'(z, P(z))}{\lub}\; \, \frac{\intl_{\pOm} \dot b \cdot \theta}{\| \theta \|_{\pOm, P(z), H}} \nonumber\\
&  = \underset{\theta \in H'(z, P(z))} {\lub} \;
\frac{
\suml^\infty_{\eta=1}  \; \intl_{\pOm}\dot b_\eta \cdot \theta_\eta}
{\suml^\infty_{\eta=1}\, \| \theta_\eta\|_{\pOm,P(z)} / (2^{(1+\eta)/2}
c'_\eta (1+c_{z, w_\eta, w_{\Ga\eta}}))}.\end{align}

The need of small $\dot b_\eta$ for large $\eta$ is obvious in \eqref{rjk}, using \eqref{rhf}, \eqref{rjd}.

Thus we will determine sufficient  conditions below such that \eqref{rha}, \eqref{rhe}, \eqref{pxmr} hold.  Using theorem 9.3, the results are summarized as follows.


\begin{thm}
Assume spaces $\hat H_\eta$ satisfying \eqref{rha}, \eqref{rhb}, \eqref{rhc}, \eqref{rhd}, $\land_\eta, \Sigma_\eta, B_{l\eta}, w_\eta, w_{\Ga\eta}$ of forms \eqref{rba}, \eqref{rbe}, \eqref{rbh},\eqref{rhe}, \eqref{aga}, \eqref{agb}, respectively, satisfying
\be\label{rfv} c \| w^{1/2}_\eta R(z)\theta_\eta\|_{L_2(\Om_l)}\|\theta_\eta\|_{L_2(\Om_l)} \ge \| \theta_\eta\|^2_{L_2(\Om_l)} + \intl_{\partial\Om_l} \theta_\eta, B_{l \eta} \theta_\eta,\ee
for each $\Om_l$;
\begin{align}\label{rfva} \| \theta_\eta\|_{L_2(\Ga(z))} \le c_\Ga &\suml^K_{k=1} \, \intl_{\Ga_k(z)} \theta_\eta \cdot (B_{l_k\eta} + B_{l'_k\eta}) \theta_{\eta }\nonumber \\
+& c \intl_{\Ga (z)} (S(z) \theta_\eta) (\Sigma_\eta(z)\theta_\eta);\end{align}
\be\label{rfvb} \intl_{\partial \Om_l (z) \cap \partial\Om_{l'} (z)} \theta_\eta \cdot (B_{l\eta} + B_{l'\eta} ) \theta_\eta = 0,\ee
for all $l, l'$ such that $\partial\Om_l(z)\cap \partial\Om_{l'} (z) \not\subset \Ga (z)$;

\be\label{rfw} \theta_\eta \cdot B_{l\eta} (z) \theta_\eta \ge 0,\ee
almost everywhere in $\partial\Om$;
\be\label{rfx} w^{1/2}_\eta \ge |\land_\eta(z)| ,\ee
\be\label{rfxa} w^{1/2}_{\Ga\eta} \ge | \Sigma_\eta(z)|.\ee

Then with the qualification that $B_{P(z)}$ is constrained to satisfy \eqref{rjk}, $(z, P(z))$  is stable.
\end{thm}

Additionally, from the proof of theorem 9.4, we have the following.
\begin{thm}\label{9.5} Assume the conditions of theorem 9.4.  Then $(z, 0_N)$ is admissible.
\end{thm}
\begin{proof} Denote
\be\label{rfy} \land (z) \theta \, \eqadef \suml^\infty_{\eta = 1} \land_\eta(z) \theta_\eta,\ee
\be\label{rfya}  \Sigma (z) \theta \, \eqadef  \suml^\infty_{\eta = 1} \Sigma_\eta(z) \theta_\eta.\ee

Then \eqref{rbc}, \eqref{rbf} could fail only because of boundary conditions \eqref{raf}  on $\theta, \theta_\eta$ within $\partial\Om$.
\end{proof}

\section{Distinguished directions and an entropy condition}

For each $z \in \hat\cals$, introducing a distinguished direction $e(\cdot; z)$ in \eqref{saa} and spaces $ H'(z_1 P(z)), \hat H_\eta, \hat H'_\eta$
in \eqref{rha}, \eqref{rfa}, \eqref{rhb}, \eqref{rhc}, the problematic  assumptions \eqref{rbl}, \eqref{rbi}, \eqref{rbj}, \eqref{rbk} of theorem 9.3 are replaced by problematic assumptions \eqref{rfv},  \eqref{rfva}, \eqref{rfvb},  \eqref{rfw} in theorems 9.4, 9.5.  We next characterize some circumstances, not necessarily uniquely, for which the latter  conditions hold.

Such incorporates additional conditions on $e(\cdot; z)$, with implied restrictions on the location of the inhomogeneous boundary data in \eqref{raa}.

Using $e(\cdot; z)$ so determined, a sufficient example of the anticipated entropy condition on $\Ga(z)$ is formulated.

With $e(\cdot; z)$ of regularity \eqref{saa}, \eqref{saaz}, sonic manifolds $\triangle (z; e)$ are generally anticipated, satisfying \eqref{rak}, \eqref{rap}.  A compatibility condition  on such $\triangle (z; e)$ is introduced, depending on $e(\cdot; z)$ and on the adopted entropy condition for $\Ga(z)$.

The results are summarized in theorem 10.4 below.

In \eqref{saa} and throughout, $e(\cdot;z)$ depends on $z \in \hat \cals$, and does not need to be unique or uniformly continuous.  However, we do assume a continuous extension to $\partial\Om$ almost everywhere.

Any such $e(\cdot;z)$ necessarily also  partitions the boundary $\partial\Om$ into upstream, lateral, and downstream segments,
\be\label{sab} \partial\Om = \partial\Om^U \cup \partial\Om^L\cup \partial\Om^D,\ee
possibly excluding a set of measure zero within $\partial\Om$, such that
\be\label{sac} e(x;z) \cdot \nu(x) < 0 (\hbox{resp.} = 0, > 0),\ee
respectively, for
\be\label{saca} x \in \partial\Om^U (\hbox{resp.}\;  \partial\Om^L, \partial\Om^D).\ee

Analogously for each $\Om_l, \; l = 1, \dots, L_\cals $ satisfying \eqref{rag},
\be\label{sad} \partial \Om_l = \partial\Om^U_l \cup \partial \Om^L_l \cup \partial\Om^D_l,\ee
\be\label{sae} e (x;z) \cdot \nu_l (x) < 0 ( \hbox{resp.}\! = 0, > 0)\ee
respectively for
\be\label{saea}  x \in \partial\Om^U_l (\hbox{resp.\ } \partial\Om^L_l, \partial\Om^D_l).\ee

Within each $\Ga_k(z), \; k = 1, \cdot, K$, without loss of generality the coordinates $\a_l, l = 1, \dots, m - 1$, from \eqref{adc}, are determined  with unit vectors $\hat
\a_l$ satisfying
\be\label{sbb} \hat\a_l(x) \cdot e (x; z) = 0, \; l = 2, \dots, m - 1,\ee
implicitly distinguishing $\hat\a_1$.

The coordinate $\a_1$ is assumed such that  a seemingly mild condition holds almost everywhere on $\Ga (z)$,
\be\label{sbba} \Big | \left( \frac{[\psi_{\a_1, z}]}{|[z]|} \right)_{\a_1} \Big | \le c\ee
with a constant depending on $z\in \hcals$.

Several additional conditions on $e(\cdot; z)$ are also required.

First and foremost, as implied by \eqref{rak},
\be\label{sbk} \psi_{e,zz} (x)\;  \hbox{nonsingular,\ } x \in \Om\setminus (\Ga(z)\cup \triangle(z;e)).\ee

Additionally, the trajectories $x(t)$ determined from
\be\label{sba} x_t(t) = e (x(t); z), \; t > 0, \; x(0) \in\partial\Om^U\ee
reach $\partial\Om^D$ in uniformly bounded time and cover $\Om$, possibly excepting a set of measure zero.

Analogously, trajectories $\a(t)$ within $\Ga(z)$, determined from
\be\label{sbc} \a_t (x(\a(t))= \hat\a_1 (x(\a(t)), \; t > 0\ee
\be\label{sbd} x(\a(0)) \in \Pi(z) \subset \partial\Om\ee
with $x(\a(\cdot))$ continuous at the common endpoints
 $ \Ga^I(z) $ in \eqref{saaz}  cover $\Ga(z)$, for some
\be\label{sbe} \Pi(z) \subseteq \partial\Ga (z) \cap \partial\Om.\ee

Additionally, using \eqref{sad}, \eqref{sae}, it is assumed that the conditions \eqref{rai}, \eqref{raj}, \eqref{rak} hold with each
\be\label{sbf} \partial\Om_l \cap \partial\Om_{l'} =  \partial\Om^D_{l} \cap \partial\Om^U_{l'}\ee
or else
\be\label{sbfa} \partial\Om_l \cap \partial\Om_{l'} = \partial\Om^L_l \cap \partial\Om^L_{l'},\ee
where $z$ is continuous and $\psi_{e, zz} (z) $ nonsingular.

In cases where \eqref{sbf} holds, necessarily either for some $k$
\be\label{sbg}\partial\Om^D_{l_k} \cap \partial\Om^U_{l'_k} = \Ga_k, \; \, l' < l\ee
\be\label{sbh} e(\cdot; z) \cdot \hat \mu_k \ge c > 0, \; \, \hat \mu_k = \nu_{l} = -\nu_{l'},\ee
or else for some $k'$,
\be\label{sbi} \partial\Om^D_{l_{k'}} \cap \partial\Om^U_{{l'}_{k'}} = \triangle_{k'}, \; \; l' < l,
\ee
\be\label{sbj} e(\cdot ; z) = \hat\mu_{k'}, \; \, \hat\mu_{k'} = \nu_l = - \nu_{l'}.\ee

\begin{defn}
An element $z\in\hat\cals$ is equipped with a distinguished direction $e(\cdot; z)$ satisfying\eqref{saa} if the conditions \eqref{sbk}, \eqref{sba}, \eqref{sbf}, \eqref{sbfa}, \eqref{sbg}, \eqref{sbh}, \eqref{sbi}, \eqref{sbj} hold.
\end{defn}
Restricting $\hat \cals$ as necessary, without loss of generality we assume that the topological structure of $\Ga(z), \triangle(z;e)$, whether or not \eqref{sbe} holds with equality, and the linear ordering of $\{ \Om_l\}$ such that\eqref{sbg} holds will be independent of $z \in \hat \cals$.

For $z$ so equipped, the required entropy condition on $\Ga(z)$ is formulated assuming existence of a suitable uniformly bounded positive definite symmetric $n$-matrix function $M(\cdot; z)$ within $\Om\setminus\Ga (z)$.  It is anticipated, as in the examples following, that
\be\label{sca}M(x;z) = \psi_{0,zz} (z(x)), \; \, x \in \Om\setminus \Ga (z)\ee
with some $\psi_0$ uniformly convex within $D$ associated with the given system \eqref{aa}.

Then from \eqref{sag}, pointwise within $\Om\setminus \Ga (z)$, we have $\lambda_j (x) \in \bbr, v_j(x) \in \bbr^n$ satisfying
\be\label{scb} \psi_{e,zz} (z(x)) v_j (x) = \lambda_j (x)M(x; z) v_j(x)\ee
\be\label{scc} v_{j'(x)} \cdot M(x; z) v_j(x) = \delta_{jj'}, \; \, j, j' = 1, \dots, n\ee
employing the Kroneker delta.  On each $\partial\Om_l$, one-sided  limiting values from within $\Om_l$ are denoted $\lambda_{jl}, v_{jl}, M_l(\cdot ; z)$ below.

\begin{defn}
A discontinuity segment $\Ga_k$  satisfies the entropy condition if \eqref{sbg} holds and  if there exists $M(\cdot ; z)$ and $  j_k \in \{ 1, \dots, n\}$ independent of $x \in\Ga_k$ such that either
\be\label{sce}\lambda_{j_{k}l'} < 0 < \lambda_{j_kl},\ee
and \eqref{sbh} holds, or else there exists $J_k\subset\{ 1, \dots, n\}$,  such that
\be\label{scf} \lambda_{jl} = \lambda_{j l'} = 0, \; \, j \in J_k,\ee
and
\be\label{scfa} e(\cdot ; z) = \hat\mu_k, \; \; j_k \in J_k.\ee

In each case, it is required that
\be\label{scg} | \lambda_{jl}|, \, | \lambda_{jl'} | \ge c > 0, \; \sgn \lambda_{jl} = \sgn \lambda_{jl'}, \; j \neq j_k \; \hbox{or\ } j \notin J_k,\ee
respectively, and
\be\label{sch} \Big| \frac{[\psi_{e,z}]}{|[z]|} M_{l'} (\cdot;z) v_{j_{k}l'} \Big| \ge c > 0,\ee
with $j_k$ such that \eqref{sce} holds or for some $j_k \in J_k$ where \eqref{scf}, \eqref{scfa} hold.

\end{defn}
We note that the conditions \eqref{scg}, \eqref{sch} are generally true for discontinuities of modest strength.  The conditions \eqref{sce}, \eqref{scf} are familiar for hyperbolic systems, corresponding to nondegenerate or linearly degenerate characteristic fields, and are compatible with the entropy inequality \eqref{faf}.

A complementary condition is expressed for segments $\triangle_{k'}$.  From \eqref{rak}, \eqref{rap}, \eqref{sbk}, \eqref{scb}, for each $\triangle_{k'}$ necessarily there exists $j_{k'} \in \{ 1, \dots, n\}$ independent of $x \in \triangle_{k'}$ and of $z \in \hcals$ such that
\be\label{sci}\lambda_{j_{k'}} (x) = 0, \; \; x \in \triangle_{k'} (z; e),\ee
tacitly assuming $e(\cdot; z)$ such that $j_{k'}$ is single-valued.

\begin{defn} 
A ``sonic manifold" $\triangle_{k'}$ satisfying \eqref{sbi}, \eqref{sci} satisfies the compatibility condition if
\be\label{scj} e(\cdot; z) \cdot \triangledown \lambda_{j_{k'}} \ge c > 0\ee
almost everywhere on $\triangle_{k'}$; if either
\be\label{sck} \partial\Om^U_{l_{k'}} \subset \partial\Om\ee
or else
\be\label{scl} \partial\Om^U_{l_{k'}} = \Ga_k,\ee
where \eqref{sce} holds (with $l'$ replaced by $l_{k'}$ and $l$ by some $l''$) and with

\be\label{scz} j_k = j_{k'}\ee
in \eqref{sce}; and if either
\be\label{scm} \partial\Om^D_{l'_{k'}} \subset \partial\Om\ee
or else
\be\label{scn} \partial\Om^D_{l'_{k'}} = \Ga_k\ee
where \eqref{sce}, \eqref{scz} (with $l$ replaced by $l_{k'}$ and $l'$ by some $l''$) hold.
\end{defn}

Using these definitions, we establish the following.
\begin{thm}
Assume $z \in \hat\cals$ equipped with a distinguished direction $e(\cdot; z)$ satisfying \eqref{saa}, each $\Ga_k(z)$ satisfying the entropy condition and each $\triangle_{k'} (z;e)$ satisfying the compatibility condition.

Assume $P(z)$ such that
\be\label{sdi} \frac{[\psi_{\a_1,z}]^\dag}{|[z]|} \in \range\,  P(z)  \ee
 almost everywhere within $\Pi(z)$ as appearing in \eqref{sbd}, \eqref{sbe}, and that \eqref{sbba} holds.

 Within each $\Om_l(z;e)$, assume $z, M(\cdot ; z)$ sufficiently smooth that \eqref{scb}, \eqref{scc} hold with
 \be\label{sda}  \lambda_j \in W^{1,\infty} (\Om_l), \; \, v_j \in W^{1,\infty} (\Om_l)^n, \; \; \, j = 1,\dots, n.\ee
 Then there exist $\hat H_\eta, \land_\eta(z), \Sigma_\eta(z), B_{l\eta} (z), w_\eta, w_{\Ga\eta}$
 satisfying \eqref{rha}, \eqref{rfa}, \eqref{rhb}, \eqref{rhc}, \eqref{rhe}, \eqref{rfx}, \eqref{rfxa} such that \eqref{rfv}, \eqref{rfva}, \eqref{rfvb}, \eqref{rfw}  hold.
 \end{thm}

\begin{proof}  Using \eqref{saa}, \eqref{sba}, pointwise throughout $\Om $ we rotate the coordinates $x_i, \, i = 1,\dots, m$,  obtaining $x'_i, \, i = 1,\dots, m $ satisfying the following conditions:
\be\label{sdb}(\triangledown_x  x'_1)^\dag = (\triangledown_x x'_1 e (\cdot ; x)) e(\cdot, z) \ee
with
\be\label{sdba} 0 < c_e \le \triangledown_x x'_1 e(\cdot ; z) \le c'_e\ee
with some constants $c_e, c'_e$ depending on $z$, throughout $\Om\setminus \Ga (z)$;
on each segment $\Ga_k$ where \eqref{sce}, \eqref{sbh} hold, there exists a continuous function $\beta_{\Ga k}$ such that
\be\label{sdc} x'_1 (x) = \beta_{\Ga k} (x'_2(x),\dots, x'_{m} (x));\ee
on each segment $\Ga_k$ where \eqref{scf}, \eqref{scfa} hold, there exists a constant $\beta_k$ such that
\be\label{sdca} x'_1 (x) = \beta_k;\ee
on each $\triangle_{k'}$, where \eqref{sbj} holds, there exists a constant $\beta_{k'}$ such that
\be\label{sdd} x'_1(x) = \beta_{k'};\ee
and the obtained Jacobian determinant is uniformly positive and bounded below,
\be\label{sde} 0 < c \le \frac{\partial(x_1,\dots, x_{m})}{\partial(x'_1,\dots, x'_{m})}.\ee

The coordinates $x'_1, \dots, x'_{m}$ will be employed throughout within each $\Om_l$.

Within each $\partial\Om_l$, excluding those segments $\partial\Om_l\cap\partial\Om^L$, we further restrict the coordinates $\a_1,\dots, \a_{m -1}$, satisfying \eqref{adc}, \eqref{sbb} within each $\Ga_k$, with unit vectors $\hat\a_i$ satisfying
\be\label{sdea} \hat\a_i \cdot \nu_l = 0, \; \; i = 1, \dots, m -1.\ee

By convention, using \eqref{sdb}, \eqref{sac}, \eqref{sbh}, \eqref{scfa}, \eqref{sbj}, the Jacobian determinant
\be\label{sdeb} \omega \eqadef \frac{\partial(x'_2, \dots, x'_{m})}
{ \partial(\a_1,\dots, \a_{m-1}) }\ee
is positive and bounded.  Indeed, within $\Ga_k$ where \eqref{sce}, \eqref{sbh} hold, we shall have
\be\label{sdec} 0 < c \le \om (x) \le c \ee
and within $\Ga_k$ where \eqref{scf}, \eqref{scfa} hold and within $\triangle_{k'}$ where \eqref{sbj} holds, without loss of generality we set
\be\label{sded} \om(x) = 1.\ee

Using \eqref{sdb}, within each $\Om_l$ we denote slices
\be\label{sdg} \Om_{l\beta} \eqadef \{ x' \in \Om_l | x'_1 = \beta \}.\ee

Using \eqref{sbb}, within each $\Ga_k$ we denote slices
\be\label{sdh} \Ga_{k\beta} \eqadef \{ x'\in \Ga_k | \a_1 (x') = \beta \}. \ee

Next we determine the spaces $\hat H_\eta$.  We choose a nested sequence $\{ E_\eta\}, \eta = 1, \dots$,  of finite-dimensional spaces in the variables $x'_2, \dots, x'_{m}$, becoming dense in smooth functions as $\eta \to \infty$.

Elements $\theta_\eta \in \hat H_\eta$ introduced in \eqref{rfa}  are expanded using \eqref{scb}, \eqref{scc} throughout $\Om\setminus \Ga (z)$
\be\label{sea} \theta_\eta (x') = \suml^n_{j=1} a_j(x') v_j(x')\ee
with the scalar coefficients $a_j(x')$ satisfying
\be\label{seb} a_j(x'_1,\dots) \in E_\eta.\ee

In \eqref{sea}, \eqref{seb} and throughout, subscript $\eta$ is dropped from $a_j$ as no ambiguity arises.

The spaces $E_\eta$ satisfy an inverse assumption, are required such that \eqref{rhb}, \eqref{rhc} hold, and such that using \eqref{seb}, \eqref{sdg},
\be\label{sec} \| \theta_{\eta, x'_i} \|_{L_2(\Om_{l\beta})} \le c_\eta \| \theta_\eta\|_{L_2 (\Om_{l\beta})}, \; \;  i = 2, \dots, m,\ee
with generic constants $c_\eta$ satisfying
\be\label{sed} c_\eta \xrightarrow[]{\eta\to\infty} \infty.\ee

Then using \eqref{sdb}, denoting
\begin{align}\label{see} R_e(z)\theta_\eta &\eqadef \psi_{e, zz}( e(\cdot; z) \cdot \triangledown_x \theta_\eta)\nonumber \\
&= \psi_{e, zz} \theta_{\eta, x'_1}/ ( \td_x x'_1  e(\cdot; z)),
\end{align}
from \eqref{af}, \eqref{sec},
\be\label{sef} \| (R(z) - R_e(z)) \theta_\eta\|_{L_2(\Om_{l\beta})} \le c_\eta \| \theta_\eta\|_{L_2(\Om_{l\beta})}.\ee

Within each segment $\Ga_k $ satisfying \eqref{sbg}, \eqref{sbh}, we use limiting values from \eqref{scb}, \eqref{scc}, and analogously with \eqref{sea}, for $x' \in\Ga_k$, $\a = \a(x')$,
\begin{align}\label{seg} \theta_\eta (x') &= \suml^n_{j=1} a_{jl_k} (\a) v_{jl_{k}} (\a)\nonumber \\
&= \suml^n_{j=1} a_{jl'_k} (\a) v_{jl'_{k}} (\a).\end{align}

Now from \eqref{sbb}, \eqref{sdb}, \eqref{seb}
\be\label{seh} a_{jl_k} (\a_1,\dots), a_{jl'_k} (\a_1, \dots) \in E_\eta, \; \; \, j = 1,\dots, n,\ee
and using \eqref{sdh},
\be\label{sei} \| \theta_{\eta,\a_i} \|_{L_2(\Ga_{k\beta})} \le c_\eta \|\theta_\eta\|_{L_2(\Ga_{k\beta})}, \; \, i = 2, \dots, m -1.\ee

Then denoting
\be\label{sej} S_e(z)\theta_\eta \eqadef [\psi_{\a_1, z}]  \theta_{\eta, \a_1},\ee
from \eqref{ag}, \eqref{sei}, \eqref{sej}
\be\label{sek} \| \frac{1}{|[z]|} (S(z) - S_e(z)) \theta_\eta\|_{L_2 (\Ga_{k\beta})} \le c_\eta \| \theta_\eta\|_{L_2(\Ga_{k\beta})}.\ee

An additional restriction is placed on the elements $\theta_\eta $ within neighborhoods of $\Ga_k$ where \eqref{scf} holds and neighborhoods of $\triangle_{k'}$, where \eqref{sci} holds.

We choose a sequence
\be\label{sel} \vare_\eta \downarrow 0 \; \hbox{\ as \ } \,  \eta \to \infty\ee
and construct $\hat\lambda_j (x'), \; \; j = 1, \dots, n$, depending on $\eta$, such that
\be\label{sem} | \hat\lambda_j (x') | \ge \delta_\eta,\ee
for a sequence
\be\label{sen} \delta_\eta \downarrow 0\; \hbox{\ as \ }\,  \eta\to \infty.\ee

Then the coefficients $a_j$ in \eqref{sea}  are further restricted so that either
\be\label{seo} a_j (x')= 0 \; \hbox{ or\ }  \hat\lambda_j (x')=\lambda_j (x'),\ee
with \eqref{sem}, \eqref{seo} holding for all $j = 1, \dots, n$ and throughout $\Om$, including limiting values on $\Ga_k$ where \eqref{sce} holds.

Within a neighborhood of each $\Ga_k$ where \eqref{sbg}, \eqref{scf}, \eqref{sdca} hold, we satisfy \eqref{seo} with values
\be\label{sep} \hat\lambda_j (x') = \lambda_j(x'), \; \; \, j \notin J_k,\ee
and for $j \in J_k$,
\be\label{seq} \hat\lambda_j (x') = \lambda_{j}(x'), \; \; |x'_1 - \beta_k| \ge \vare_\eta,\ee

\be\label{ser} a_{j} (x') = 0, \;\,   |x'_1 - \beta_k| < \vare_\eta,\ee
choosing  positive $\delta_\eta$ sufficiently small that \eqref{sem} holds, using \eqref{sbk}.

Analogously, within a neighborhood of each $\triangle_{k'}$ where \eqref{sbi}, \eqref{sdd}, \eqref{sci}  hold, we assign values
\be\label{ses}\hat\lambda_j (x') = \lambda_j (x'), \; \; \, j \neq j_{k'}, \ee
\be\label{set} \hat\lambda_{j_{k'}} (x') = \lambda_{j_{k'}} (x'), \; \; |x'_1 - \beta_{k'}| \ge \vare_\eta,\ee
\be\label{seu} a_{j_{k'}} (x') = 0, \; |x'_1 - \beta_{k'}| < \vare_\eta,\ee
reducing $\delta_\eta$ as necessary  to maintain \eqref{sem}.

We now address the conditions \eqref{rfv}, \eqref{rfx}.  From \eqref{see}, $R_e(z)$ is an ordinary differential operator on $\hat H_\eta$, for which familiar ``energy estimates" apply.  Within arbitrary $\Om_l$, using \eqref{sea} we denote
\be\label{sfa} (\land_\eta \theta_\eta) (x') \, \eqadef - \suml^n_{j = 1} \rho_{j \eta} (x') a_j(x') v_j(x'),\ee
introducing bounded scalar functions $\rho_{j\eta}$ to be determined.

Then within each slice $\Om_{l\beta}$, from \eqref{sdg}, \eqref{see}, \eqref{sea}, \eqref{sfa}, using \eqref{scb}, \eqref{scc}, then \eqref{seo},
\begin{align}\label{sfb}
(\td_x x'_1 e(\cdot; z))\land_\eta\theta_\eta&\cdot R_e (z) \theta_\eta = - \Big( \suml^n_{j=1} \rho_{j\eta} a_jv_j\Big) \cdot \psi_{e, zz} \Big( \suml^n_{j'=1}  a_{j'}v_{j'}\Big)_{x'_1}\nonumber \\
&= -  \suml^n_{j=1} \rho_{j\eta} \lambda_ja_ja_{j, x'_1} -  \suml^n_{j,j'=1} \rho_{j\eta} \lambda_ja_ja_{j'} (M(\cdot ; z) v_j \cdot v_{j',x'_1})\nonumber \\
&=-\suml^n_{j=1} \rho_{j\eta} \hat\lambda_ja_ja_{j, x'_1} - \suml^n_{j,j'=1} \rho_{j\eta} \hat \lambda_ja_ja_{j'}  (M(\cdot ; z) v_j \cdot v_{j',x'_1})\nonumber \\
&=\Big( - \thalf \suml^n_{j=1} \rho_{j\eta} \hat \lambda_j a^2_j\Big)_{x'_1} + \thalf \suml^n_{j=1} (\rho_{j\eta} \hat\lambda_j)_{x'_1} a^2_j\nonumber \\
&\qquad \qquad \qquad -\suml^n_{j, j'=1} \rho_{j\eta} \hat \lambda_j a_j a_{j'} (M(\cdot; z) v_j \cdot v_{j', x'_1}). \end{align}

Denote, for any $\beta$ such that $\Om_{l\beta}$ is nonempty,
\be\label{sfc} \gamma_\eta (\beta) = \underset{j = 1, \dots, n}{\hbox{maximum }} \| \rho_{j\eta} \hat\lambda_j\|_{L_\infty (\Om_{l\beta})}\ee
whence from \eqref{sem}, \eqref{sfa}, \eqref{sea}, \eqref{scc}
\be\label{sfca} \| \land_\eta\|_{L_{\infty}(\Om_{l\beta})} \le c \frac{\gamma_\eta(\beta)}{\delta_\eta}.\ee

For the final right-hand term in \eqref{sfb}, using \eqref{sda}, \eqref{sfc}, \eqref{sea}, there is a constant $c_v$ such that
\be\label{sfd} \mid \intl_{\Om_{l\beta}} \suml^n_{j,j' = 1} \rho_{j\eta} \hat\lambda_j a_ja_{j'} (M(\cdot ; z) v_j\cdot v_{j', x'_1}) \mid \le c_v\gamma_\eta(\beta) \| \theta_\eta\|^2_{L_2(\Om_{ l\beta})}.\ee

Next we use \eqref{sef}, \eqref{sfca}, \eqref{sdb}  to obtain
\be\label{sfe} \mid\intl_{\Om_{l\beta}} ((\td_x x'_1 e(\cdot; z))  \land_\eta\theta_\eta\cdot (R(z) - R_e(z)) \theta_\eta \, \mid \, \le c_\eta\frac{\gamma_\eta (\beta)}{\delta_\eta} \| \theta_\eta\|^2_{L_2(\Om_{l\beta})}.\ee

The $\rho_{j\eta}$ are chosen such that each $\rho_{j\eta} \hat \lambda_j$ is nonnegative and  to satisfy, within each $\Om_{l}$,
\be\label{sff} (\rho_{j\eta} \hat\lambda_j)_{x'_1} \ge (2 (c_v + \frac{c_\eta}{\delta_\eta} ) \gamma_\eta (x'_1) + 2)/c_e,\ee
using \eqref{sdb}, \eqref{sdba}  which is compatible with \eqref{sfc}, using \eqref{sem}.  Then from
\eqref{sfb}, \eqref{sfd}, \eqref{sfe}, \eqref{sff},
\be\label{sfg} \intl_{\Om_{l\beta}}(\td_x x'_1 e(\cdot; z)) \land_\eta\theta_\eta \cdot R(z) \theta_\eta \ge - \thalf \intl_{\Om_{l\beta}} \Big(\suml^n_{j=1} \rho_{j\eta} \hat\lambda_j a^2_j\Big)_{x'_1} + \| \theta_\eta \|^2_{L_2(\Om_{l\beta})}.\ee

We  integrate \eqref{sfg} with respect to $\beta$, over the values of $\beta $ such that $\Om_{l \beta}$ is nonempty  using \eqref{sba}.  Using \eqref{sdba}, \eqref{sde}, \eqref{rfx}, we have
\begin{align}\label{sfh} \int d \beta \intl_{\Om_{l\beta}} (\td_x x'_1 e(\cdot; z)) \land_\eta \theta_\eta \cdot R(z) \theta_\eta &\le \int d\beta \| \theta_\eta\|_{L_2 (\Om_{l\beta})} \| \land_\eta\|_{L_\infty (\Om_{l\beta})} \| R(z)\theta_\eta\|_{L_2(\Om_{l\beta})}\nonumber\\
&\le \int d\beta \| \theta_\eta\|_{L_2(\Om_{l\beta})} \| w^{1/2}_\eta R(z) \theta_\eta\|_{L_2 (\Om_{l\beta})}\nonumber \\
&\le c \| \theta_\eta\|_{L_2(\Om_l)} \| w_{\eta}^{1/2}  R(z) \theta_\eta \|_{L_2(\Om_l)},\end{align}
the left-hand term in \eqref{rfv}.

Again using \eqref{sde}
\be\label{sfj} \int d \beta \|\theta_\eta\|^2_{L_2(\Om_{l\beta})} \ge c \| \theta_\eta\|^2_{L_2(\Om_l)}.\ee

Using \eqref{sdb}, \eqref{sdeb}
\begin{align}\label{sfk} \int d\beta\intl_{\Om_{l\beta}} \Big(\suml^n_{j = 1} \rho_{j\eta} \hat\lambda_j a^2_j\Big)_{x'_1} &= \int dx'_1 \intl_{\Om_{lx'_1}} \Big(\suml^n_{j=1} \rho_{j\eta} \hat\lambda_j a^2_j\Big)_{x'_1} \; \; dx'_2 \dots dx'_{m}\nonumber \\
&= \intl_{\partial\Om_l} \omega (\td_x x'_1 \nu_l) \suml^n_{j=1} \rho_{j\eta l} \hat\lambda_{jl} a^2_{jl}\end{align}
using limiting values $\rho_{j\eta l},  \hat\lambda_{jl}$, and $a_{jl}$ from \eqref{seg}, \eqref{seo}, \eqref{sfa}.

Now \eqref{rfv}  follows from \eqref{sfh}, \eqref{sfj}, \eqref{sfk}, identifying
\be\label{sfl} \theta_\eta \cdot B_{l\eta} \theta_\eta \, \eqadef - \thalf \omega (\td_x x'_1 \nu_l) \suml^n_{j = 1} \rho_{j\eta l} \hat \lambda_{jl} a^2_{jl}\ee
almost everywhere on $\partial\Om_l$.

To address the conditions \eqref{rfva}, \eqref{rfxa}, we make \eqref{sfc}, \eqref{sfd} precise  specifying within each $\Om_l$
\be\label{sga} (\rho_{j\eta} \hat\lambda_j) (x') \, \eqadef \begin{cases} \k_{l\eta} \tau_\eta(x'_1) (1-\chi_l (x')), \; \; \, j = j_k, \partial \Om^U_l = \Ga_k;\\ \k_{l\eta} \tau_\eta (x'_1), \; \; \hbox{otherwise}.\end{cases}\ee

In \eqref{sga}, $j_k, \Ga_k$ are from \eqref{sbg}, \eqref{sce}, \eqref{scf}, \eqref{sch}; we introduce constants
\be\label{sgb} \k_{l\eta} \in \bbr_+, \; \; l = 1, \dots, L_\cals\ee
to be determined; the scalar functions $\tau_\eta $ will be chosen satisfying
\be\label{sgc} \tau_{\eta, x'_1} = 2(c_v+\frac{c_\eta}{\delta_\eta}) \tau_\eta + 2\ee
throughout $\Om$, with $\tau_\eta = 1$ at a suitable (unimportant) value of $x'_1$.

The scalar functions $\chi_l$, independent  of $\eta$, are determined satisfying
\be\label{sgd} \chi_{l, x'_1} \le 0 \ee
within $\Om_l$ and boundary values
\be\label{sge} \chi_l(x') = \begin{cases} 1, \; \; x' \in\partial\Om^U_l\\ 0, \; \; x' \in \partial\Om^D_l\end{cases}.\ee

The constants $\k_{l\eta}$ are chosen to satisfy
\be\label{sgf} \k_{1\eta} = 1\ee
\be\label{sgg} \k_{l\eta} = \k_{l'\eta}\ee
for $l, l'$ satisfying \eqref{sbi}, recovering \eqref{rbl}.

Then in \eqref{rfva}, for any $l,l'$ in \eqref{sbf}, using  \eqref{sfl}, \eqref{sga}, \eqref{sae}, \eqref{saea},  within $\partial\Om^D_l\cap \partial \Om^U_{l'}$
\be\label{sgh} \theta_\eta \cdot (B_{l\eta} + B_{l'\eta}) \theta_\eta = \thalf\om (\td_x x'_1\nu_l) \tau_\eta \suml^n_{j = 1} (\k_{l'\eta} a^2_{jl'} - \k_{l\eta} a^2_{jl}).\ee

 Using \eqref{sgg}, \eqref{seg}, \eqref{scc}, the right side of \eqref{sgh} vanishes for $l,l'$ satisfying  \eqref{sbi}  as required in \eqref{rfvb}.  For $l, l'$ satisfying
 \eqref{sbg}, using \eqref{sga}, \eqref{sge}
 \be\label{sgi} \theta_\eta\cdot (B_{l\eta} + B_{l'\eta}) \theta_\eta = \thalf \om (\td_x x'_1 \nu_l ) \tau_\eta \Big( \k_{l'\eta}
 \suml_{j\neq j_k} a^2_{jl'} - \k_{l\eta} \suml^n_{j=1} a^2_{jl}\Big).\ee

 From \eqref{seg}, \eqref{sch}, there is a constant $c_2$ such that pointwise within $\partial\Om^D_l \cap \partial\Om^U_{l'}$
 \be\label{sgn} |\theta_\eta|^2 \le c_2\Big( \Big( \frac{[\psi_{\a_1, z}]\theta_\eta}{|[z]|}\Big)^2  + \suml_{j \neq j_k} a^2_{jl'}\Big).\ee

 Analogous to \eqref{sfa}, in \eqref{rfva} we choose
 \be\label{sgj} \Sigma_\eta (z) \theta_\eta \, \eqadef\, - \frac{\rho^\Ga_\eta}{|[z]|^2}\; [\psi_{\a_1, z}] \theta_\eta\ee
 with scalar functions $\rho^\Ga_\eta (\a_1)$, constant within each $\Ga_{k\beta}$ to be determined.

 Then from \eqref{sej}, \eqref{sgj}, using \eqref{sbc}, \eqref{sbd}, \eqref{sbe},  after a partial integration using \eqref{sdi}, then using \eqref{sbba},
\begin{align}\label{sgk} \intl_{\Ga(z)} (\Sigma_\eta(z)\theta_\eta ) &(S_e(z)\theta_\eta) \ge\thalf \intl_{\Ga(z)} \rho^\Ga_{\eta, \a_1} \Big( \frac{[\psi_{\a_1,z}]\theta_\eta}{|[z]|}\Big)^2 \nonumber \\
&+   \intl_{\Ga(z)} \rho^\Ga_\eta \Big( \frac{[\psi_{\a_1,z}]\theta_\eta}{|[z]|}\Big)\;
\Big(\Big( \frac{[\psi_{\a_1,z}]}{|[z]|}\Big)_{\a_1} \theta_\eta\Big) \nonumber \\
\ge &\thalf \intl_{\Ga(z)} \rho^\Ga_{\eta, \a_1} \Big( \frac{[\psi_{\a_1,z}]\theta_\eta}{|[z]|}\Big)^2 - c \intl_{\Ga(z)} \rho^\Ga_\eta \Big| \frac{[\psi_{\a_1,z}]\theta_\eta}{|[z]|}\Big| |\theta_\eta|\nonumber \\
\ge &\thalf \intl_{\Ga(z)} \rho^\Ga_{\eta, \a_1} \Big( \frac{[\psi_{\a_1,z}]\theta_\eta}{|[z]|}\Big)^2 - c_3 \Big( \|(\rho^\Ga_\eta)^{1/2} \Big(\frac{[\psi_{\a_1,z}]\theta_\eta}{|[z]|}\Big)\|^2_{L_2(\Ga(z))} \nonumber \\
&  \qquad -\| (\rho^\Ga_\eta)^{1/2} \theta_\eta \|^2_{L_2(\Ga(z))}\Big),\end{align}
with a suitable constant $c_3$, independent of $\eta$.

From \eqref{sek}, \eqref{sgj}, \eqref{sdh}
\begin{align}\label{sgl}
| \intl_{\Ga(z)} (\Sigma_\eta (z)\theta_\eta)&((S(z) - S_e(z)) \theta_\eta)|\nonumber \\
&\le c_4 c_\eta \| (\rho^\Ga_\eta)^{1/2} \Big( \frac{[\psi_{\a_1,z}]\theta_\eta}{|[z]|}\Big) \|_{L_2(\Ga(z))} \| (\rho^\Ga_\eta)^{1/2} \theta_\eta\|_{L_2(\Ga(z))}\nonumber \\
&\le \frac{c_4c_\eta}{2} \; \Big( \| (\rho^\Ga_\eta)^{1/2}  \frac{[\psi_{\a_1,z}]\theta_\eta}{|[z]|}\Big) \|^2_{L_2(\Ga(z))} + \|
(\rho^\Ga_\eta)^{1/2} \theta_\eta\|^2_{L_2(\Ga(z))}\Big)\end{align}
 with suitable constant $c_4$.

From \eqref{sgk}, \eqref{sgl}
\begin{align}\label{sgm} \intl_{\Ga(z)} (\Sigma_\eta (z) \theta_\eta )& (S(z)\theta_\eta) \ge \thalf \intl_{\Ga(z)} \rho^\Ga_{\eta, \a_1}
 \Big( \frac{[\psi_{\a_1,z}]\theta_\eta}{|[z]|}\Big)^2\nonumber \\
 &-(c_3 + \frac{c_4c_\eta}{2}) \Big( \| (\rho^\Ga_\eta)^{1/2} \Big(  \frac{[\psi_{\a_1,z}]\theta_\eta}{|[z]|}\Big) \|^2_{L_2(\Ga(z))} + \| (\rho^\Ga_\eta)^{1/2} \theta_\eta\|^2_{L_2(\Ga(z))}\Big).\end{align}

 In \eqref{sgm}, using
 \eqref{sgn},
 \begin{align}\label{sgo} \| (\rho^\Ga_\eta)^{1/2}\theta_\eta &\|^2_{L_2(\Ga(z))} \le  c_2 \| (\rho^\Ga_\eta)^{1/2}
  \Big(\frac{[\psi_{\a_1,z}] \theta_\eta}{|[z]|}\Big) \|^2_{L_2(\Ga(z))}\nonumber \\
&+ c_2\suml^K_{k = 1}\intl_{\Ga_k} \rho^\Ga_\eta \suml_{j\neq j_k} a^2_{jl'}.\end{align}


From \eqref{seg}, there is a constant $c_5$ such that from \eqref{sgh}, using \eqref{sgn}
\begin{align}\label{sgp}
\intl_{\Ga_k} \theta_\eta&\cdot (B_{l\eta} + B_{l'\eta}) \theta_\eta \ge \thalf\intl_{\Ga_k} \om  (\td_x x'_1\nu_l)\tau_\eta (\k_{l'\eta} \suml_{j = j_k} a^2_{jl'} - \k_{l\eta} c_5 |\theta_\eta |^2)\nonumber \\
&\ge \thalf \intl_{\Ga_k} \om (\triangledown_x x'_1 \nu_l) \tau_\eta \Big( ( \k_{l'\eta} - c_2c_5 \k_{l\eta}) \suml_{j \neq j_k} a^2_{1l'} - c_2c_5 \k _{l\eta} \big( \frac{[\psi_{\a_1, z}]\theta_\eta}{|[z]|}\big)^2\Big).\end{align}

Combining \eqref{sgm}, \eqref{sgp}, using \eqref{sgo}

\begin{align}\label{sgq} &\intl_{\Ga(z)} (\Sigma_\eta (z) \theta_\eta) (S(z) \theta_\eta)
+ \suml^K_{k = 1} \intl_{\Ga_k} \theta_\eta \cdot (B_{l\eta} + B_{l'\eta}) \theta_\eta \nonumber \\
\ge & \suml^K_{k = 1} \intl_{\Ga_k} \Big( \frac{[\psi_{\a_1, z}]\theta_\eta}{|[z]|}\Big)^2 \Big( \frac{\rho^\Ga_{\eta, \alpha_1}}{2} - (c_3 + \frac{c_4 c_\eta}{2})
c_2 \rho^\Ga_\eta - \frac{c_2 c_5 \k_{l\eta} \om (\td_x x'_1 \nu_l)\tau_\eta}{2}\Big)\nonumber \\
&+ \suml^K_{k = 1} \intl_{\Ga_k} (\suml_{j\neq j_k} a^2_{jl'}) \Big( \frac{(\k_{l'\eta} - c_2 c_5 \k_{l\eta}) \om (\triangledown_x x'_1 \nu_l)\tau_\eta}{2} - (c_3 + \frac{c_4 c_\eta}{2} )c_2 \rho^\Ga_\eta\Big).
\end{align}

The conclusion \eqref{rfva} follows from \eqref{sgq}, successively choosing $\rho^\Ga_\eta$ to satisfy
\be\label{sgr} \frac{\rho^\Ga_{\eta,\a_1}}{2} = (c_3 + \frac{c_4 c_\eta}{2}) c_2 \rho
^\Ga_\eta +
\frac{c_2c_5 \k_{l\eta} \om (\td_x x'_1 \nu_l) \tau_\eta}{2} + 1\ee
and then, recursively,
\be\label{sgs}  \k_{l'\eta} \ge c_2 c_5 \k_{l\eta} + 2 \, \frac{(c_3 + \frac{c_4 c_\eta}{2}) c_2\rho^\Ga_\eta}{\om (\td_x x'_1 \nu_l) \tau_\eta} +1
\ee
using the ordering \eqref{sbg}, the starting value \eqref{sgf}, and positivity of the denominator in \eqref{sgs} from \eqref{sbh}, \eqref{sdec},  \eqref{sded}, \eqref{sgc}, and again applying \eqref{sgn}.

The conclusions \eqref{rfvb}, \eqref{rfw} follow readily from \eqref{sfl}, using \eqref{sdb}, \eqref{sdeb}, \eqref{sga}, \eqref{sgh}. \end{proof}

With both $\land_\eta (z)$ and $\Sigma_\eta (z)\; L_\infty$-bounded, depending on $\eta$, the conclusion of theorem  9.5, which is independent of $e(\cdot; z)$, is obtained by inspection.

The boundary segments $\partial\Om^U, \partial\Om^L, \partial\Om^D$ depend on $z$ through the limiting values of $e(\cdot; z)$, as expressed in \eqref{sab}, \eqref{sac}, \eqref{saca}. Thus by inspection of \eqref{pxn},  compatible with \eqref{aged}, we have $(z, P(z,e))$ stable, with $P(z,e)$ satisfying, for almost all $x \in \partial\Om$, any $z \in \hat\cals$,
\be\label{rfyb} \ker P(z, e)(x) = \begin{cases} \range \; \,  \psi_{\nu,zz} (z(x)), \; \; x \in \partial\Om^U(z)\cup \partial\Om^L(z)\\ 0, \; \; \, x \in \partial\Om^D(z)\end{cases}  \oplus \; \hbox{span\ } \{ e_0(x)\}.\ee

A nontrivial set of prescribable boundary data, as needed for weak well-posedness, is then obtained from
\eqref{raa}, \eqref{agea}, subject to regularity requirements as detailed in subsection 9.2, specifically
\eqref{rjk}, for $\dot b  \in B_{P(z)}$.

\section{Underlying systems and the entropy inequality}

Next  we characterize some underlying systems \eqref{aa} for which the assumptions of theorem 10.4 can be satisfied with suitable domain $\Om$.

Initially we consider the special case where the system \eqref{aa} is everywhere uniformly hyperbolic.  For definiteness, we take $x_m$ as the time-like independent variable, assuming existence of a constant $c_D$ independent of $y \in D$ such that in the sense of symmetric $n$-matrices,
\be\label{tab} c_D \psi_{m, zz} (y) \ge I_n.\ee

We denote by $\un{\hat \cals}$ the corresponding set of weak solutions $z$ of the form \eqref{ad}, \eqref{ada}, \eqref{aac}, with whatever boundary data.  For $z \in \un{\hat \cals}$,
additional conditions on the system \eqref{aa}, imply the entropy condition as expressed in definition 10.2 equivalent to the familiar inequality \eqref{faf}.

Weak well-posedness will then be discussed with solution sets
\be\label{taa} \hat \cals = \hat \cals_T \subset \un{\hat \cals},\ee
with elements $z = z_{\hat\cals_{T}}$ invariant under symmetries $T$ with which the system \eqref{aa}
is associated.  Such $z$ satisfy \eqref{aa}
 simultaneously with a reduced system with fewer independent variables, which typically is not everywhere hyperbolic and often not of the form \eqref{aa}, to be discussed in subsection 11.1.

Thus regarding hyperbolic \eqref{aa} as an underlying ``primitive system", weak well-posedness may be  obtained without assuming hyperbolicity for such  reduced systems.  In this context, the initial assumption of hyperbolic \eqref{aa} is no loss of generality, as any given \eqref{aa} determines stationary solutions, with respect to $x_0$, of a related hyperbolic system with $m+1$ independent variables, corresponding to $T$ translation with respect to $x_0$,
\be\label{tabc} \suml^m_{i = 0} (\psi_{i, z_j} (z))_{x_i} = 0, \; \; \, j = 1, \dots, n,\ee
with $\psi_0 $ uniformly convex throughout $D$ (as in \eqref{sca} above).

For a given system \eqref{aa} satisfying \eqref{tab}, any unit vector
\be\label{tac} \hat\tau \in \bbs^m, \; \; \hat \tau \cdot \hat x_m = 0,\ee
determines a potential function analogously with \eqref{saf}
\be\label{tad} \psi_\tau (y) \eqadef \suml^{m-1}_{i = 1} \hat \tau_i \psi_i(y), \; \; \, y \in D,\ee
and a reduced hyperbolic system of dimension $n$ with two independent variables, time-like $x_m $ and $\tau \eqadef \hat \tau \cdot x$,
\be\label{tae} \big( \psi_{m, z} (z)\big)_{x_m} + \big( \psi_{\tau,z} (z)\big)_\tau = 0.\ee

In the special case $m = 2$, the systems \eqref{aa}, \eqref{tae} coincide.  For larger $m$, the solution subset of $\un{\hat\cals}$ with $z$ independent of all $x_\perp$ satisfying
\be\label{taf} \hat x_\perp \cdot \hat x_m = \hat x_\perp \cdot \hat\tau = 0\ee
satisfies \eqref{tae}.

For any $y \in D$, the real  characteristic speeds $\xi_j (y)$ and corresponding eigenvectors $\tilde v_j (y) $ associated with \eqref{tae} satisfy
\be\label{tag} \psi_{\tau, zz} (y) \tilde v_j (y) = \xi_j (y) \psi_{m,zz}(y) \tilde v_j (y), \; \; \, j = 1, \dots, n,\ee
which also holds for one-sided limiting values $y_l, y_{l'}, \xi_{jl}, \xi_{jl'}, \tilde v_{jl}, \tilde v_{jl'}$ at an arbitrary point in $\Ga_k(z)$ where \eqref{sbg} holds ($z$ satisfying \eqref{aa}, not necessarily \eqref{tae}, and dropping the $k$-subscript from $l, l'$).

We determine $\hat\tau$ satisfying \eqref{tac} and $\phi_\mu$  pointwise, such that almost everywhere  in $\Ga_k(z)$
\be\label{tah} \hat \mu_k = \sin\phi_\mu \;  \hat x_m + \cos \phi_\mu \; \hat \tau\ee
with, by convention,
\be\label{tai} 0 \le \hat \mu_k\cdot \hat x_m < 1.\ee

For any  fixed $\hat \mu$, denote by $y_l, y_{l'} \in D$ two values satisfying the Rankine-Hugoniot condition \eqref{afb}, otherwise arbitrary. Then with $ \phi_\mu$ obtained from \eqref{tah} (with $\hat \mu_k = \hat \mu$), the piecewise constant solution
\be\label{taj} \hat z (x_m, \tau ) = \begin{cases} y_l, &\tau > \hat s x_m\\ y_{l'} &\tau < \hat s x_m\end{cases} \ee
\be\label{tak} \hat s \eqadef - \tan \; \phi_\mu\ee
satisfies both \eqref{aa} and \eqref{tae}, using \eqref{taf}.

It follows that for any $z \in {\ul\hcals}$,  at any point in $\Ga_k(z)$, with $\hat\tau$ such that $\hat \mu_k, \hat x_m$ and $\hat \tau$ are coplanar, the Rankine-Hugoniot
conditions \eqref{afb} and the entropy inequalities  \eqref{faf} (weakly in the space of measures) for \eqref{aa} and for \eqref{tae} coincide.

Pointwise within each $\Ga_k, \; \; k = 1, \dots, K$, we denote functions
\be\label{taka} \hat e_k (\cdot ; z) \in C (\Ga_k(z) \to \bbs^m)\ee
such that $\hat e_k(\cdot; z), \hat x_m, \hat \mu_k$ are coplanar.  We denote by $\hat \la_j, \hat v_j, \; \; j = 1, \dots, n$, the values obtained from \eqref{scb}, \eqref{scc} with $\hat e_k (\cdot ; z)$ replacing $e(\cdot; z)$ and $\psi_{m, zz}$ replacing $M$,
\be\label{tal} \psi_{\hat e, zz} (z(x)) \hat v_j (x) = \hat\lambda_j(x) \psi_{m, zz} (z(x)) \hat v_j (x), \; \; \, j = 1, \dots, n,\ee
\be\label{tam} \hat v_{j'} (x) \cdot \psi_{m, zz} (z(x)) \hat v_j (x) = \delta_{jj'}, \; \; \, j, j' = 1, \dots, n.\ee

Limiting values on $\Ga_k$ (with $l, l'$ as appearing in \eqref{sbg} or in \eqref{taj} are denoted $\hat \la_{jl}, \hat v_{jl}, \hat \la_{jl'}, \hat v_{jl'}$.

Using this notation, the entropy condition of definition 10.2 is related to \eqref{faf}.

\begin{lem}
Assume a system \eqref{aa} such that \eqref{tab} holds, and a solution $z \in \ul{\hcals}$.

With $\hat\tau, \phi_\mu$ obtained from \eqref{tah}, \eqref{tai} pointwise within each $\Ga_k(z)$, assume the system \eqref{aa} such that each characteristic field $j = 1,\dots, n$ for the reduced system \eqref{tae} is either genuinely nonlinear or linearly degenerate in an open neighborhood of $D$ containing both $y_l, y_{l'}$, the one-sided limiting values of $z$.

 With limiting values $\xi_{jl}, \xi_{jl'}, \tilde v_{jl}, \tilde v_{jl'}, \; j = 1, \dots, n$ obtained from \eqref{tag} assume that either
 \be\label{tba} \xi_{jl} = \xi_{j'l}, \; \; \, j \neq j'\ee
 or
 \be\label{tbb}  \xi_{jl'} = \xi_{j'l'}, \; \; \, j \neq j'\ee
 occurs only if both \eqref{tba}, \eqref{tbb} hold and characteristic fields $j, j'$ are both linearly degenerate.

 Then for modest values of $\| \hat\mu - e (\cdot; z)\|_{L_\infty (\Ga(z))} $ and  $\|[z]\|_{L_{\infty (\Ga(z))}}$, the entropy condition in definition 10.2 with the specific choice
 \be\label{tbd} M(x, z) = \psi_{m, zz}(z(x)).\ee
  coincides with the inequality \eqref{faf}.

 \end{lem}

 \begin{rem*} The assumptions of lemma 11.1 hold for (rotationally symmetric) fluid flow models, with only mild restrictions on the given equation of state. The assumption
 \eqref{tba}, \eqref{tbb} is a weakened expression of strict hyperbolicity for the system \eqref{tae}.
 \end{rem*}
 \begin{proof} Pointwise within each $\Ga_k(z)$, we obtain $\hat \tau \in \bbs^m$ and $\phi_\mu$ from \eqref{tah}.  Then using the given $e(\cdot; z)$ continuous almost everywhere on $\Ga(z)$, we determine  $\hat e_k (\cdot ; z) \in \bbs^m$ (almost everywhere) from a projection
 \be\label{tbch} \hat e_k(\cdot ; z) = (e(\cdot; z)\cdot \hat x_m)\hat x_m + (e(\cdot; z) \cdot \hat \mu_k)\hat \mu_k\ee
 using \eqref{tai}.

 From \eqref{tah}, \eqref{tbch}, the three unit vectors $\hat e_k (\cdot ; z), \hat x_m,\hat \mu_k$ are coplanar, and we obtain $\phi_e$ satisfying
 \be\label{tbc} \hat e_k(\cdot ; z) = \sin \phi_e \; \hat x_m + \cos \phi_e\; \hat \tau.\ee

Using \eqref{raj}, one-sided limiting values of $\hat \la_j, \hat v_j$ on $\Ga_k(z)$ are denoted $\hat \la_{jl}, \hat\la_{jl'}, \hat v_{jl}, \hat v_{jl'}$.

Comparison of \eqref{tal}, \eqref{tag} using \eqref{tbc}
 then establishes limiting values on $\Gamma_k(z)$ satisfying
 \be\label{tbe} \hat \lambda_{jl} = \sin\phi_e + \xi_{jl} \cos\phi_e, \; \; \, \hat v_{jl} =  v_{jl}\ee
 and identical results with $l'$ replacing $l$.

 By assumption, each $\Ga_k (z)$ is either associated with a genuinely nonlinear characteristic field $j_k$ or a nonempty $J_k$ such that each characteristic field $j \in J_k$ is linearly degenerate.

 For the case of genuinely nonlinear field $j_k$, the strict entropy inequality \eqref{faf} is equivalent, at least for modest $|[z]|$, to the familiar comparison of propagation speed  and limiting characteristic speeds,
 \be\label{tbf} \xi_{j_k l'} < \hat s < \xi_{j_kl}\ee
 in present notation, using \eqref{tag}, \eqref{tak}.

  By assumption, using \eqref{tbch}, we have modest values of $|\hat e_k(\cdot; z) - \hat\mu_k|$, implying modest values of $|\phi_e - \phi_\mu|$ using \eqref{tah}, \eqref{tbc}.

  With $|\phi_e - \phi_\mu|$ sufficiently modest, using \eqref{tbe}, \eqref{tak}, the condition \eqref{tbf} with strict inequality becomes equivalent to \eqref{sce}.

For linearly degenerate field $j_k \in J_k $ in \eqref{scf}, the condition  \eqref{faf} with equality is equivalent to
\be\label{tbg} \xi_{j_kl'} = \hat s = \xi_{j_kl}.\ee

From \eqref{tah},  \eqref{scfa}, \eqref{tbch}, \eqref{tbc} necessarily $\phi_e, \phi_\mu$ coincide, and using \eqref{tbe}, \eqref{tak}, the condition \eqref{tbg} is equivalent to
\eqref{scf}.

The condition \eqref{scg} holds by assumption, and \eqref{sch}  for modest $|[z]|$, using \eqref{tbd}.
\end{proof}

Assembling the results of theorems 7.2, 9.2, 9.3, 9.4, 9.5, 10.4 and lemma 11.1, we obtain arguably realistic sufficient conditions for identification of weakly well-posed boundary-value problems using definition 8.2.  The proofs of the following two theorems are sketched, emphasizing the salient features thereof, as the arguments are largely repetitive of previous results.

\begin{thm}
Assume a hyperbolic underlying system \eqref{aa} for which \eqref{tab} holds, such that each characteristic field of reduced systems \eqref{tae} is either genuinely nonlinear or linearly degenerate, and such that the conditions \eqref{tba}, \eqref{tbb} occur only simultaneously with fields $j, j'$ linearly degenerate.

Assume a weak solution $z\in\un{\hat\cals}$ of the form \eqref{ad}, \eqref{ada}, \eqref{aac} with
\be\label{tca} z \in W^{1,\infty} (\Om\setminus\Ga(z)),\ee
satisfying the entropy inequality \eqref{faf} (weakly in the space of measures on $\Om$),  satisfying \eqref{haa} (in the case of unbounded $\Om$), and satisfying \eqref{sbba}.

Assume $|[z]| $ sufficiently modest throughout $\Ga(z)$, with $\Ga(z)$ such that \eqref{sbc}, \eqref{sbd}, \eqref{sbe} hold.

Assume $z$ equipped with a distinguished direction $e(\cdot; z)$ as per definition 10.1, that
 $|e(\cdot; z) - \hat \mu_k|$ is sufficiently modest
  within each $\Ga_k(z)$, and such that \eqref{scfa} holds and  each  $\triangle_{k'}$ satisfies the compatibility condition given in definition 10.3.

Then there exists $P(z)$ satisfying
\be\label{tcb} \ker P(z;e)\subseteq \ker P(z)\ee
as obtained from \eqref{rfyb}, such that $(z,P(z))$ is unambiguous as per definition 2.5.
\end{thm}
\begin{rem*} The prescribed boundary data $\dot b$ is obtained from \eqref{rjk}, \eqref{ahd}.
\end{rem*}

 \begin{proof}

As needed in the proof of theorem 10.4, the condition \eqref{sda} follows from \eqref{tca}; the entropy condition in definition 10.2 follows by appeal to lemma 11.1.

 Then for each single value of $\eta$, theorems 7.2, 9.2, 9.3, 9.4, 9.5, 10.4 are reviewed and revised with $H(z, P, w, w_\Ga)$ replaced by $\hat H_\eta$ as introduced in \eqref{rfa}, \eqref{rhb}, \eqref{rhc}, and $\dot b$ replaced by $\dot b_\eta$ as introduced in \eqref{rfna}, \eqref{rjc}, \eqref{rjca}. The distinguished direction $e(\cdot; z)$ in definition 10.1 and the boundary segments $\pOm^U, \pOm^L, \pOm^D$ introduced in \eqref{sab}, \eqref{sac} are understood as independent of $\eta$.

 For each of these theorems, the conclusions survive with replacement of the terms stable, admissible, unambiguous, and potentially admissible, as defined in definitions 2.1, 2.3, 2.5, 7.1, respectively, replaced by analogously defined $\eta$-stable, $\eta$-admissible, $\eta$-unambiguous, and $\eta$-potentially admissible.

 In particular, for each value of $\eta$, with $P(z;e)$ obtained from \eqref{rfyb}, $(z, P(z;e))$ is $\eta$-stable by application of (so revised) theorems 9.2, 9.3, 9.4, 10.4; $(z, 0_N)$ is $\eta$-admissible by application of theorem 9.5.

 Then by application of theorem 7.2, there exists $P_\eta(z)$ satisfying
 \be\label{tcca} \ker P(z;e) \subseteq \ker P_\eta (z) \subset \hat H_\eta \mathop{\midl}_{\pOm}\ee
 such that $(z, P_\eta (z))$ is $\eta$-unambiguous.

Then using \eqref{rfna}, \eqref{rka}  the required $P(z)$, satisfying \eqref{tcb} is then obtained from
\be\label{tcc} \ker P(z) = \mathop{\oplus}\limits_\eta \ker P_\eta (z).\ee

\end{proof}

The assumptions of theorem 11.2 can be materially simplified, at the expense of increased regularity of the boundary data. In particular, the distinguished direction, $e(\cdot; z)$ above, may be determined locally, piecewise continuous, effectively avoiding the assumption \eqref{saa}.

\begin{thm}
Assume $P(z;e)$ satisfying \eqref{rfyb} with given boundary segments $\pOm^U, \pOm^L, \pOm^D$ satisfying \eqref{sab}.

Then for sufficiently smooth $\dot b \in B_{P_\cals}$, the assumptions of $z \in \ul{\hcals}$ equipped with $e(\cdot ; z)$ as per definition 10.1 and that  $\triangle(z)$ satisfies the compatibility condition as per definition 10.3 may be dropped from the assumptions of theorem 11.2.

\end{thm}

\begin{rem*} The conditions \eqref{rjk}, \eqref{ahd} on $\dot b$ no longer suffice; the precise conditions are left unspecified.

The hyperbolicity condition \eqref{tab} is used explicitly in the proof following.
\end{rem*}
\begin{proof} Theorems 9.3, 9.4, 9.5, 10.4 may be applied piecewise with $e(\cdot ; z)$ replaced by $\tilde e (\cdot ; z)$, constructed as follows.

We introduce another partitioning of $\Om$, with a finite number of disjoint open sets $\tilde \Om^i, \; \;  i = 1, \dots$, (depending on $z$) such that
\be\label{tza} \Om = \big(\mathop{\cup}\limits_i \tilde \Om^i\big) \cup \big(\mathop{\cup}\limits_{i < i'} \partial\tilde \Om^i \cap \partial\tilde \Om^{i'}\big); \ee
\be\label{tzb} \hbox{measure\ } (\partial\tilde \Om^i \cap \Ga(z)) = 0 \; \; \, (\hbox{in\ } \bbr^{m-1});\ee
\be\label{tzc} \Ga^I(z) \subset \mathop{\cup}\limits_{i' < i} (\partial\tilde \Om^i \cap \partial\tilde \Om^{i'})\ee
in \eqref{saaz}.  On each $\partial\tilde \Om^i$, the outward unit normal is denoted $\tilde \nu^i$, coinciding with $\nu$ within $\partial \tilde \Om^i \cap \partial\Om$. We shall apply theorems 9.3, 9.4, 9.5, 10.4 and lemma 11.1, adopted as discussed in the proof of theorem 11.2, separately within each $\tilde \Om^i$.

Assuming $z \in \un{\hat{\cals}}$ of the form \eqref{ad}, \eqref{ada} satisfying \eqref{tca}, within each $\tilde \Om^i$ there exist  scalar functions $\tilde \Theta^i$, satisfying
\be\label{tzd} \tilde\Theta^i \in C (\tilde\Om^i) \cap W^{1, \infty} (\tilde\Om^i \setminus \Ga (z)),\ee
\be\label{tze} \tilde\Theta^i_{x_m} \ge 0, \ee
\be\label{tzf} \tilde \Theta^i\mathop{\mid}\limits_{\tilde \Om^i \cap \Ga (z)} = \hbox{ \ constant},\ee
and additional conditions detailed below.

Given the $\tilde\Theta^i$,  we choose nonnegative constants $\tilde c^i, \; \, i = 1, \dots$, to be determined,  such that within each $\tilde \Om^i$,
\be\label{tda} \tilde e (\cdot; z) \eqadef \; \frac{\triangledown (\tilde\Theta^i + \tilde c^i x_m)}{|\triangledown (\tilde
\Theta^i + \tilde c^i x_m)|}\ee
satisfies
\be\label{tzg} \| \psi^{-1}_{\tilde e, zz} (z)\|_{L_\infty (\tilde \Om^i)} \le c\ee
for each $i = 1, \dots$.  Such occurs, for example,  with $\tilde e (\cdot; z)$ time-like within each $\tilde \Om$, using \eqref{tab}.

Additional mild requirements, analogous to \eqref{sad}, \eqref{sae}, \eqref{saea}, \eqref{sba}, \eqref{sbf}, are placed on $\tilde e (\cdot ; z), \tilde \Om^i$.

Denote
\be\label{tzia} \partial\tilde \Om^{iU} (\hbox{resp.\ } \partial\tilde \Om^{iD}) \eqadef \{ x \in \partial\tilde \Om^i\, |\, \tilde e (x; z)\cdot \tilde \nu^i(x) < 0 \,(\hbox{resp. } > 0)\}\ee
with
\be\label{tzib} \partial\tilde\Om^i = \partial\tilde \Om^{iU} \cup \partial\tilde \Om^{iD}\ee
modulo a set of measure zero necessarily including the set $\{ x \in \partial\tilde \Om^i | \tilde e(x; z) \cdot \tilde \nu^i(x) = 0 \}$.

We assume trajectories
\be\label{tzic} x_t(t) = \tilde e (x(t); z)\ee
from $\partial\tilde\Om^{iU} $ to $\partial\tilde \Om^{iD} $ in finite time, filling $\tilde \Om^i$.

Throughout $\Om$, we require the $\tilde \Om^i$ ordered so that
\be\label{tzid} \partial\tilde \Om^i\cap \partial\tilde \Om^{i'} \subseteq \partial\tilde\Om^{iD} \cap \partial \tilde\Om^{i'U}, \; \; i' < i\ee
for all $i' \neq i$.

We emphasize that the limiting values of  $\tilde \Theta^i$ and $\tilde e (\cdot; z)$ do not generally agree on each $\partial\tilde \Om^i \cap \partial\tilde \Om^{i'}$.

From \eqref{tzf}, \eqref{tda}, \eqref{tai}, necessarily within each $\Ga_k(z)$,
\be\label{tzie} \tilde e (\cdot ; z) \in \mathrm{span \ } \{ \hat \mu_k, \hat x_m\}.\ee

Indeed, in lemma 11.1 we need $|\tilde e (\cdot; z) - \hat \mu_k|$ modest throughout each $\Ga_k$, to be achieved with modest values of $\tilde c^i$ in \eqref{tda}.

Lemma 11.1 is now applied with $\tilde e (\cdot; z)$ replacing $e(\cdot;z)$ in \eqref{tbch}.  Using \eqref{tzie}, the condition \eqref{scfa} is maintained.

 Within each $\tilde \Om^i$, the  conditions \eqref{tal}, \eqref{tam} are replaced by
 \be\label{tci} \psi_{\tilde e, zz} \tilde v_j = \tilde \lambda_j \psi_{m,zz} (z)
 \tilde v_j, \; \; \, j = 1,\dots,n;\ee
\be\label{tcia} \tilde v_{j'} \cdot \psi_{m, zz} (z) \tilde v_j = \de_{jj'},\ee
with limiting values $\tilde \la_{jl}, \tilde v_{jl}$ on each segment of $\tilde \Om^i\cap \partial\Om_l \cap \Ga(z) $ recalling \eqref{rag}.

Judicious choice of $\tilde \Theta^i, \tilde c^i, \tilde e (\cdot; z)$ accomplishes the following, details omitted:
\be\label{tch} |\tilde \lambda_j| \ge c > 0, \; \; \, j = 1, \dots, n\ee
within each $\tilde \Om^i$, implying \eqref{tzg};
\be\label{tcf} \tilde e (\cdot; z) \cdot \hat x_m \ge c > 0\ee
within each $\tilde \Om^i$;
\be\label{tcg} \tilde e (\cdot ; z) \cdot \hat \mu_k \ge c > 0\ee
almost everywhere within each $\Ga_k \cap \tilde \Om^i$;
\be\label{tck} \tilde \la_{j_k l'} < 0 < \tilde \la_{j_kl}\ee
replacing \eqref{sce} on each genuinely nonlinear segment $\Ga_k$;
\be\label{tcj} \tilde \la_{jl} = \tilde \la_{jl'}, \; \; \, j \in J_k\ee
replacing \eqref{scf} within each linearly degenerate segment $\Ga_k$.

From \eqref{tch}, \eqref{rak},
\be\label{tcl} \tri (z; \tilde e) = \emptyset \ee
making definition 10.3 vacuous with $\tilde e(\cdot ; z)$ replacing $e(\cdot; z)$ within each $\tilde\Om^i$. From \eqref{tda}, \eqref{tzf}  $\tilde e (\cdot ; z), \hat \mu_k, \hat x_m$ will be coplanar within each $\Ga_k(z)$.

The conditions \eqref{scf}, \eqref{scfa} are needed in lemma 11.1, but \eqref{tcj}, \eqref{tcg} suffice for the proof of theorem 10.4.  Thus the proofs of theorems 9.3, 9.4, 9.5, 10.4 survive replacement of $\Om  $ by $\tilde \Om^i$ and $e(\cdot ; z)$ by $\tilde e (\cdot; z)$.

In particular, partitioning $\partial\tilde \Om^i$
as in  \eqref{tzia}, \eqref{tzib},  replacing $e(\cdot; z)$ by $\tilde e(\cdot; z)$, and obtaining $\tilde B_\eta^i$ on each $\partial\tilde \Om^i \cap \Om_l$ from \eqref{sfl} with $\hat\la_{jl}$ replaced by $\tilde \la_j$, we obtain an estimate
\begin{align}\label{tcm} & \intl_{\partial \tilde \Om^{iU}} \theta_\eta \cdot \tilde B_\eta^i \theta_\eta + \| \theta_\eta\|^2_{L_2(\tilde \Om^i)} + \| \theta_\eta\|^2_{L_2(\tilde \Om^i \cap \Ga (z))}\nonumber \\
&\le \Big| \intl_{\partial\tilde\Om^{iD}} \theta_\eta \cdot \tilde B_\eta^i \theta_\eta \Big| + c_\eta \big( \| w^{1/2}_\eta R(z) \theta_\eta\|^2_{L_2(\tilde \Om_i)} +  \| w_{\Ga\eta} S(z) \theta_\eta\|^2_{L_2(\tilde \Om^i \cap \Ga(z))}\big)\end{align}
with $\eta$ as introduced in \eqref{rfa}.

However such involves replacement of \eqref{saa} by the weaker condition \eqref{tda}, \eqref{tzd}.

Using the partitioning \eqref{tza}, the  outward unit normal $\tilde \nu^i$ on $\partial\tilde\Om^i$ satisfies, almost everywhere in $\partial\tilde\Om^i$,
\be\label{tdba} \tilde \nu^i = \begin{cases} \; \; \nu(x), \; \;\; \, x \in\partial\Om\\ - \tilde \nu^{i'} (x), \; \, x \in \partial\tilde \Om^i \cap \partial\tilde \Om^{i'}\end{cases}.\ee

For each $\tilde \Om^i$, the boundary segments $\partial\tilde \Om^{iU},  \partial\tilde\Om^{iD}$ are determined analogously with \eqref{sad}, \eqref{sae}, \eqref{saea}, using \eqref{tzia}, \eqref{tzib}.

In particular,  from \eqref{sab}, \eqref{sac}, \eqref{saca},  necessarily
\be\label{tdc} \partial\Om^U \subseteq \mathop{\cup}\limits_i \partial \tilde \Om^{iU}.\ee

We choose the $\tilde\Om^i$ and $\tilde e (\cdot; z)$ such that for each $i$,

\be\label{tdd} \{ x \in  \partial\tilde \Om^{i} | \tilde e (x; z)\cdot \tilde \nu^i (x) = 0 \}  \subset \partial\Om^L\cup \partial\Om^D,\ee
and
\be\label{tde} \partial\Om^{iD} \subseteq \partial\Om^D \cup \big(\mathop{\cup}\limits_{i' < i} \partial\tilde\Om^{i'U}\big).\ee

From \eqref{tda}, the limiting values of $\tilde e(\cdot; z)$ on $\partial\tilde \Om^i$ are such that in general
\be\label{tdf} \tilde e^i (x;z) \neq \tilde e^{i'} (x; z), \; \; x \in \partial\tilde\Om^i \cap \partial\tilde\Om^{i'}.\ee

Nonetheless, from \eqref{tch}, \eqref{tde}, \eqref{tza}, in \eqref{tcm}, with suitable constants $c_\eta$,
\be\label{tdg} \big| \intl_{\partial\tilde\Om^{iD} } \theta_\eta \cdot \tilde B^i_\eta \theta_\eta \big| \le c_\eta\suml_{i' < i} \intl_{\partial\tilde \Om^{i'U}\cap \partial\tilde \Om^i} \theta_\eta \cdot \tilde B^{i'}_\eta \theta_\eta.\ee

Thus applying \eqref{tcm} recursively, using \eqref{tdg}, \eqref{tdc}, we have for each~$\eta$,
\be\label{tdh} \intl_{\partial\Om^U} \theta_\eta \cdot \tilde  B^\cdot_\eta \theta_\eta \le c_\eta \suml_i \left(\| w^{1/2}_\eta R(z) \theta_\eta\|^2_{L_2(\tilde \Om_i)} + \|w^{1/2}_{\Ga\eta} S(z) \theta_\eta
\|^2_{L_2(\tilde \Om^i \cap \Ga(z))}\right)\ee
which implies stability of $(z, P(z))$ satisfying \eqref{tcb}.

\end{proof}

However, because of the need  of the constants $c_\eta$ in \eqref{tdg}, we have lost control of the required regularity  of $\dot b\in B_{P(z;\tilde e)}$.  In general  \eqref{rjk}, \eqref{ahd} no longer suffice.

Using theorem 11.2 or 11.3, weak well posedness will be discussed below  for solution sets $\hat{\ul \cals}$ such that $P_{\hat{\ul \cals}} $ as obtained from \eqref{rfyb} is nontrivial, $P_{\hat{\ul \cals}}$ satisfying
\be\label{tdi} \ker P_{\hat{\ul \cals}} (x) \subseteq
\begin{cases}
\mathop{\cap}\limits_{z \in \hat{\ul \cals}} \range\,  \psi_{\nu,zz} (z(x)), \; \; x \in \mathop{\cap}\limits_{z \in \hat{\ul \cals}}\big( \partial\Om^U(z;\tilde e(z))\cup \partial \Om^L (z; \tilde e (z))\big)\\ 0, \; \; \hbox{otherwise}
\end{cases} \oplus \; \mathrm{span \ } \{ e_0 (x)\} \ee
and with the understanding that elements of $B_{P_{\hat{\ul \cals}}} $ are of regularity as required by theorem 11.2 or 11.3.

\subsection{Reduced systems}

The underlying system for a weakly well-posed problem need not coincide with a ``primitive" hyperbolic system of the form \eqref{aa}, satisfying the assumptions of lemma 11.1, provided that the given solution set $\hcals$ is common to both systems.  Such permits underlying systems which are not everywhere hyperbolic, and need not be of the form \eqref{aa}, obtained as reduced systems for elements $T$ of the symmetry group $\calt$ with which the primitive system is equipped.

Such $T$ are one-parameter Lie groups determined by parameters
\be\label{tea} T_x \in M^{m\times m}, \; \; x_T \in \bbr^m, \; \; T_z \in M^{n\times n}, \; \; z_D \in \bbr^n,\ee
\be\label{teaa} T_x x_T = 0, \; \, T_z z_D = 0,\ee
determining infinitesimals $x^*, y^*$
\be\label{teb}x^* = T_x x+x_T, \; \, x \in \Om; \; \; y^* = T_z y + z_D, \; y \in D.\ee

Familiar examples are translation, with
\be\label{tee} T_x = T_z = z_D = 0, \; \, x_T \neq 0,\ee
for which
\be\label{tga} \Big(\suml^m_{i = 1} (\psi^\dag_{i, z} (z))_{x_i}\Big)^* = 0,\ee
and scaling in $x$, with
\be\label{tef} T_x = I_m, \; \, x_T = T_z = z_D = 0,\ee
for which
\be\label{tgb} \Big( \suml^m_{i = 1} (\psi_{i, z} (z))_{x_i} \Big)^* = -  \suml^m_{i = 1} (\psi^\dag_{i, z} (z))_{x_i}.\ee

Examples such as rotation symmetry correspond \cite{S4} to
\be\label{tgc} x_T = x_D = 0, \; \, T_x, T_z \hbox{\ antisymmetric,}\ee
such that
\be\label{tec} (\psi(y))^* = T_x \psi(y), \; \; \, y \in D\ee
resulting in
\be\label{tgd} \Big(\suml^m_{i = 1} (\psi^\dag_{i, z} (z))_{x_i} \Big)^* = - T_z \Big(  \suml^m_{i = 1} (\psi^\dag_{i, z} (z))_{x_i}\Big).\ee

In each case \eqref{tga}, \eqref{tgb}, \eqref{tgd} the solution set is invariant under~$T$.

Each $x \in \Om$ is determined by coordinates $r(x) \in \bbr^{m-1}$, scalar $t(x)$, with $r(x)$ invariant under $T$,
\begin{align}\label{tei}0 &= r(x)^*\nonumber \\
&=r_x(x) x^*\nonumber \\
&=r_x(x) (T_x x + x_T).
\end{align}

Correspondingly, for each $x \in \Om$ there exists an invariant subset $\{ y_T (x)\} \subset~D$.
\be\label{tek} y_T (x)^* = 0.\ee

For translation \eqref{tee}, by inspection
\be\label{ter} r(x)\cdot x_T = 0,\ee
\be\label{terb} t(x) = x\cdot x_T, \; \; x \in \Om\ee
\be\label{tera} y_T(x) =y, \; \; y \in D.\ee

For scaling \eqref{tef}, with
\be\label{tesa} t(x) = x_m\ee
by convention,
\be\label{tes} r_i (x) = x_i / x_m, \; \; i = 1, \dots, m-1\ee
\eqref{tei} holds with
\be\label{tesb} r_x(x)= \frac{1}{x_m} \big( I_{m-1} {\vdots}  - r(x)\big)\ee
and \eqref{tera} again holds.

For $T$ such that $T_z$ is nonzero and antisymmetric, $x_D = 0$, we satisfy \eqref{tek} at each $x \in \Om$ with
\be\label{tej} \{ y_T(x)\} = \{ T_D (x) D\}, \; \; \, T_D(x) \in M^{n\times n}.\ee

From \eqref{tek},  \eqref{tej}, \eqref{teb}, necessarily for all $x \in \Om, y \in D$,
\begin{align}\label{tejz} 0 &= (T_D (x)y)^*\nonumber \\
&= (T_D (x))^* y + T_D (x) T_zy;\end{align}
here setting
\be\label{tejy} t(x)^* = 1\ee
by convention and using $T_z$ antisymmetric, explicitly
\be\label{tejx} T_D (x) = exp (-t(x) T_z).\ee

Compatibility of \eqref{tejx}, \eqref{tejy} with \eqref{teb},
\be\label{tejw} (T_D (x))^* = T_{D,x} (x) (T_x x+x_T)\ee
requires
\be\label{tejv} t_x(x) T_x x = 1,\; \; t_x(x) x_T=0.\ee

An example is rotation symmetry in the $x_1, x_2$ plane with $n = m = 3$,
\be\label{tha} T_x = T_z = \begin{pmatrix} 0 &-1 &0\\1 &0 &0\\ 0 &0 &0\end{pmatrix}, \; \; \, x_T = x_D = 0;\ee
\be\label{thb} r (x) = {|r|(x)\choose x_3}, \; \; |r|(x) \, \eqadef\, (x^2_1 + x^2_2)^{1/2};\ee
\be\label{thc} t(x) = \tan^{-1} (x_2 / x_1);\ee
\be\label{thd} T_D (x) =  \begin{pmatrix} \cos t  &\sin t &0\\ - \sin t  &\cos t &0\\ 0 &0 &1\end{pmatrix}.\ee

For $T$ such that \eqref{tera} holds, in \eqref{tej} we have
\be\label{the} T_D(x) = I_n, \; \; x \in \Om.\ee

From \eqref{tek}, invariant solution sets  $z_T \in \hcals_T$ satisfying
\be\label{tfa} (z_T (x))^* = 0 \ee
almost everywhere in $\Om$, are obtained with
\be\label{tfb} z_T (x) \in \{ y_T (x)\}.\ee

With $T_D $ obtained from \eqref{tejx} or \eqref{the}, such $z_T$ satisfy a reduced system

\begin{align}\label{tfc}
0 &= \suml^m_{i = 1} (\psi_{i, z_k} (z_T))_{x_i}\nonumber \\
& =  \suml^m_{i = 1}  \suml^n_{j = 1} (\psi_{i, z_{T,j}} (z_T) T_{D, jk})_{x_i}\nonumber \\
&= \suml^m_{i = 1} \Big(\suml^{m-1}_{l = 1} \suml^n_{j = 1}
 (\psi_{i, z_{T, j}}(z_T))_{r_l} T_{D, jk} \; r_{l, x_i}\Big.,\; \; \,  k = 1, \dots, n. \end{align}

For translations \eqref{tee}, with \eqref{tesa} by convention, \eqref{tfc} is just the stationary form of \eqref{aa},
\be\label{tfd} \suml^{m-1}_{i = 1} (\psi_{i, z_{T,k}})_{x_i} = 0, \;\; k = 1, \dots, n.\ee

For scaling \eqref{tef}, using \eqref{tesb}, \eqref{tfc} is the familiar self-similar form of \eqref{aa},
\be\label{tfg} \frac{1}{x_m} \Big( \suml^{m-1}_{l = 1}\big((\psi_{l, z_{T, k}} (z_T)\big)_{r_l} - r_l (\psi_{m, z_{T, k}}(z_T))_{r_l})\Big) = 0, \; \, k = 1, \dots, n.\ee

For rotations $\big($\eqref{tha}, \eqref{thb}, \eqref{thc}, \eqref{thd}$\big)$, \eqref{tfc} becomes the expected
\be\label{tff} (\psi_{3, z_{T, k}} (z_T))_{x_3} + \frac{1}{|r|(x)} (x_1 \psi_{1, z_{T,k}} (z_T) + x_2 \psi_{2, z_{T, k}} (z_T))_{|r|(x)} = 0,\;  k = 1, \dots, n.\ee

\subsection{Compatible domains}

Solutions $z_T$ satisfying \eqref{tfc}, extended as per \eqref{tfa}, \eqref{tfb}, satisfy \eqref{aa} locally.  Compatible domains $\Om^r,\;  \Om $ such that solutions of \eqref{tfc} for $r \in \Om^r$ are solutions of \eqref{aa} for $x (r, t) \in \Om$, are obtained satisfying
\be\label{tfe} \Om = \{ x(r, t)\, | \, r(x(r, t)) \in \Om^r, \; \; \, t \in \Om^t \, \eqadef \, (\ul{t}, \ol{t}(r))\},\ee
with a convention
\be\label{trb} r(x(r, \ul{t}) = r, \; \; \, r \in \Om^r.\ee

The boundaries $\pOm^r, \pOm$ are then related by
\be\label{tfea} \pOm = (\pOm^r \times \Om^t) \cup (\Om^r\times \pOm^t)\ee
with obvious abuse of notation
\be\label{trc} \pOm^r \times \Om^t\, \eqadef \, \{ x(r, t) \, | \, r (x(r, t)) \in \pOm^r, \; \; \, t \in \Om^t\},\ee
\be\label{trd} \Om^r \times \pOm^t\, \eqadef \, \{ x(r, \ul{t}) \, | \, r \in \Om^r\} \cup \{ x(r, \ol{t}(r)) \, | \, x(r, \ol{t}(r) - 0) \in \Om\}.\ee

This choice of $\Om, \pOm$ permits a precise statement that boundary conditions for \eqref{tfc} on $\pOm^r$ are satisfied on $\pOm^r\times \Om^t$.  It is required that no points $x(t)$ determined from
\be\label{tra} x_t(x(t)) = x^* (x(t)), \; \; t > \ul{t}\ee
using \eqref{teb}, \eqref{tejy} traverse the boundary segment $\pOm^r\times \Om^t$ as $t$ varies.  Using \eqref{trc}, \eqref{teb}, this condition is
\begin{align} \label{uth} \nu (x) \cdot x^*(x) &=\nu(x)\cdot (T_x x + x_T)\nonumber\\
&=0, \; \; x \in \pOm^r \times \Om^t.\end{align}

Within $\pOm^r \times \Om^t$,
\be\label{trf} \nu (x, t) = {\nu_\perp (r)\choose 0}\ee
\be\label{trg} \nu_\perp (r) \cdot r = 0, \; \; r \in \pOm^r.\ee

If \eqref{uth}, \eqref{trf}, \eqref{trg} fails, then solutions of \eqref{tfc} within $\Om^r$ determine solutions of \eqref{aa} only within a subset $\ul{\Om}$
\be\label{trh} \ul{\Om} = \{ x(t) \in \Om\}\ee
with $x(t)$ determined from \eqref{tra} with
\be\label{tri} r(x(\ul{t})) \in \Om^r,\ee
limiting ${\ol t}(r)  -\ul{t}$ in particular.

\section{Problem class details}

Based on the results of sections 9-11, we characterize problem classes for which weak well-posedness is realistic if by no means guaranteed.

There are three essential sources of restrictions on such problem classes.

A reduced system \eqref{tfc} is regarded as the underlying system of physical interest, related to a ``primitive" system \eqref{aa} by a symmetry $T$ as discussed in subsection 11.1.  Weak well-posedness, as has been discussed for weak solutions of \eqref{aa}, in a domain $\Om$, may be applied to weak solutions of \eqref{tfc} in a domain $\Om^r$ by making the solution set $\hcals$, as appearing in definition 8.2, isomorphic to a subset of the weak solution set for \eqref{aa} in $\Om$.

An entropy condition is required on the elements of $\hcals$. At least for systems \eqref{aa} satisfying the assumptions of lemma 11.1, the familiar inequality \eqref{faf} will suffice.

There are competing requirements on $P_\cals$.  The underlying assumption of solutions in the range of a Frechet differentiable map $\cala $ on the boundary data implies solutions $z = \cala(b)$ majorized by $b \in \tilde B_{P_\cals}$.  Such requires $\ker P_\cals$ sufficiently large, for example such that there are no nontrivial solutions of \eqref{agc} with $\theta \in H(z, P_\cals, \cdot, \cdot)$, and may additionally imply restrictions in $D$ and the form of $q(\psi_{\cdot, z}(z))$, $z \in \hat\cals$.  As against this, stability  of $(z, P_\cals)$ for all $z \in \hat \cals$, as discussed in section 10, requires $\ker P_\cals$ sufficiently small.

Elements $z_{\hcals} \in \hcals$ satisfying \eqref{tfc}, \eqref{aa} simultaneously are determined with domains $\Om^r, \Om$ satisfying \eqref{tfe}, boundaries $\pOm^r, \pOm$ satisfying \eqref{tfea}.

 No a priori boundary conditions are applied on the boundary segments $\Om^r \times \pOm^t$, on solutions $z_{\hat\cals} $ or on the space $X_{P_\cals}$.  By convention,
 \be\label{uha} P_\cals (x) = 0, \;\, e_0(x) = 0, \;\, b(x) = \psi^\dag_{\nu(x), z} (z_{\hat{\cals}} (x)), \; \; x \in \Om^r \times \pOm^t.\ee

At any point
\be\label{uca} x \in \pOm^r \times \Om^t, \; \; r(x) \in \pOm^r\ee
the unit normals $\nu(x)$ to $\pOm$  $\nu_T (x)$ to $\pOm^r$ are related by
\be\label{ucaa} \nu_T(x) = r_x(x) \nu(x), \ee
and the ``normal potential" function in (2.6), from which boundary data is expressed assumes the form
\be\label{ucba} \psi_\nu(z)(x) = \nu_T (x) \cdot (r_x(x) \psi(z)).\ee

For simplicity of notation, below we abbreviate
\be\label{uhb} E(z) \, \eqadef\, \suml^m_{i = 1} (\psi_{i, z}^\dag (z))_{x_i}\ee
equivalently obtained from \eqref{tfc}.  Solutions $z$ correspond to $E(z)$ (weakly) vanishing.

Using \eqref{tga}, \eqref{tgb}, \eqref{tgd} we identify
\begin{equation*}T_E \in M^{n\times n}\end{equation*}
such that for any $z$ of the form \eqref{tfa}
\be\label{uhc} E(z)^* = T_E E(z).\ee

For any fixed $t\in \Om^t$, the weak form of \eqref{tfc} follows from
\be\label{uhka} \iintl_{\Om^r} dr \, \omega_T E(z) \cdot \theta = 0, \; \;  \theta \in X_{P_\cals, T};\ee
\begin{align}\label{uhd} X_{P_\cals, T}\, \eqadef\,  &\{ \theta = \theta(r), \; r \in \Om^r, \; \theta \in (C(\Om^r)\cap W^{1, 1} (\Om^r))^n,\nonumber \\
&P_\cals \theta\mathop{\mid}\limits_{\pOm^r} = 0 \};
\end{align}
\be\label{uhj} \om_T \, \eqadef \; \frac{\partial(x_1,\dots, x_m)}{\partial (r(x), t(x))}.\ee

Such is expressed with boundary conditions
\be\label{uhkb} b = (I_n-P_\cals) \psi^\dag_{\nu_T, z} (z) \ee
\be\label{uhkc} \psi_{\nu_T, z} (z) e_0 = 0\ee
almost everywhere on $\pOm^r$.

We extend $\theta, P_\cals, b, e_0, \om_T$ so that \eqref{uhka} is independent of $t \in \Om^t$, thus equivalent to the weak form of \eqref{aa}, at least where \eqref{uth} holds;
\be\label{uhk} \iintl_\Om dx\,  E(z) \cdot \theta = \intl_{\Om^t} dt \iintl_{\Om^r} dr \, \om_T E(z) \cdot \theta = 0, \; \, \theta \in X_{P_\cals}.\ee

For translation or scaling symmetries, where $T_z = 0, z_D = 0, T_x $ is either $0$ or $I_m$, and $T_E = 0 $ or $-I_m$, such is easily accomplished with
\be\label{uhkd} z, \theta, \psi_{\nu_T, z}, P_\cals, b, e_0, \om_T \; \; \hbox{ independent \ of \ } t,\ee
and using \eqref{ucaa}
\be\label{uhke} \nu_T = \nu\ee
within $\pOm^r\times \Om^t$.

 Symmetries satisfying \eqref{tgc}, \eqref{tec} correspond to
\be\label{uhkz} x^* = T_x x, z^* = T_z z, T_E = T_z, \, \nu* = T_x \nu, \; r^*_x = - r_x T^\dag_x = r_x T_x, \; \om^*_T = 0,\ee
and \eqref{uhka} is achieved with
\begin{align}\label{uhky} &\theta^* = T_z \theta, \, b^* = T_z b, \; e^*_0 = - T_z e_0\nonumber \\
&(I_n-P_\cals)^* = T_z (I_n-P_\cals) - (I_n - P_\cals) T_z;\end{align}
details omitted.

As outlined in section 6, we seek solutions $z_{\hcals} \in \hcals$ as familiar vanishing dissipation limits,
\be\label{uma} \tilde z_h\overset{h \downarrow 0}{\longrightarrow} z_{\hcals}\ee
with $\tilde z_h$, also of the form \eqref{tfa}, \eqref{tfb} satisfying suitably discretized \eqref{fad}, \eqref{fac}, \eqref{fae}.

A dissipation term $\cald$ is required compatible with \eqref{uhc}, satisfying
\be\label{uja} (\cald(\tilde z_h))^* = T_E \cald (\tilde z_h),\ee
satisfying \eqref{fbb} and a one-sided condition
\be\label{ujc} \underset{h \downarrow 0 }{\lim\sup}\; \; h \tilde z_h \cdot \cald (\tilde z_h) \le 0 \ee
in the sense of distributions, pursuant to obtaining \eqref{faf} for limits \eqref{uma}.

An a priori bound on solutions $z_{\hat\cals}$, depending on the corresponding boundary data, may be obtained by application of \eqref{fhd}.

For $\tilde z_h, z_{\hcals}$ satisfying \eqref{tfb}, from \eqref{aaea}, \eqref{tec},
\begin{align}\label{ujd} q(\psi_{\cdot, z} (\tilde z_h))^* & = T_x q(\psi_{\cdot, z} (\tilde z_h)),\nonumber \\
q(\psi_{\cdot, z} ( z_{\hcals}))^* & = T_x q(\psi_{\cdot, z} ( z_{\hcals})).\end{align}

For $z$ satisfying \eqref{tfa}, \eqref{tfb}, denoting
\be\label{uje} q_r (z, x) \,  \eqadef \, r_x (x) q (\psi_{\cdot, z} (z)),\ee
for $E(z)$ given in \eqref{uhb}, in the sense of distributions,
\begin{align}\label{ujea} E(z) \cdot z &=\triangledown_x\cdot q (\psi_{\cdot, z} (z))\nonumber \\
&=\triangledown_r \cdot q_r(z, \cdot) - \suml^m_{i = 1} \Big( \suml^{m-1}_{j = 1} \, (r_{x, ji})_{x_j} q_i (\psi_{i, z} (z)\Big).\end{align}

Restricting $D$ as necessary, we assume
\be\label{ujf} |y| \le c + c \,| \, q_r(y, x)\, |, \; \; y \in D, \, x \in \Om.\ee

Mirroring \eqref{fhc}, we assume $\ker P_\cals$ sufficiently large that almost everywhere on $\pOm^r$,
\be\label{xca} \nu_T \cdot q_r (z_{\hcals}, \cdot) \ge - c_b.\ee

Finally we assume that the entropy flow is approximately potential in the sense that
\be\label{ujh} \| q_r (z_{\hcals}, \cdot)\|_{L_\infty (\Om^r)} \le - c \; \,\underset{\Theta}{\lub} \iintl_{\Om^r} \triangledown_r \Theta q_r(z_{\hcals}, \cdot)\ee
nonnegative
\be\label{ujg}  \Theta \in W^{1,\infty} (\Om^r), \; \; \, \| \Theta\|_{L_\infty (\Om^r)} \le 1.\ee

Then from \eqref{ujh},  \eqref{ujg},  \eqref{xca},  \eqref{ujf},  \eqref{uma},  \eqref{fae},
\be\label{ujj} \| z_{\hcals} \|_{L_\infty (\Om^r)} \le c \| b \|_{L_1 (\Om^r)}.\ee

As against this, stable $(z_{\hcals}, P_\cals)$ for all $z_{\hcals} \in \hcals$ is required for weak well posedness in definition 8.2.  Here we assume \eqref{sab} applied to $\pOm^r$, obtaining $\pOm^{rU},\,  \pOm^{rL}, \,  \pOm^{rD}$ independent of $z_{\hcals} \in \hcals$.
Comparing \eqref{sac} or \eqref{tdd}  with \eqref{fhc}, such is anticipated with $e(\cdot; z_{\hcals})$ or $\tilde e (\cdot; z_{\hcals})$ roughly parallel to $q_r(z_{\hcals}, \cdot)$.

Restricting $\hcals$ as necessary, we take $\tilde e (\cdot; z_{\hcals})$ independent of $t, z_{\hcals} \in \hcals$  (satisfying \eqref{tfb}) within $\Om^r$.  Then from \eqref{tzia}, \eqref{tzib}, modulo a set of measure zero
\be\label{uada} \pOm = \partial\tilde \Om^U \cup \partial\tilde \Om^D;\ee
\be\label{uadb} \partial\tilde\Om^U (resp.\, \partial\tilde \Om^D) \, \eqadef\, \{ x\in \partial\Om\, | \, \nu(x)\cdot \tilde e(x, \cdot) < 0 (resp.\,  > 0)\}.\ee

Theorems 10.4, 11.2, 11.3 are now applied with $\tilde e (\cdot; z)$ replacing $e(\cdot; z)$ within $ \pOm$.  Stability of $(z_{\hcals}, P)$ follows for any $P$ such that almost everywhere in $\pOm$,
\be\label{uad} \ker P \subseteq \ker P(z_{\hcals}, \tilde e (\cdot ; z_{\hcals}))\ee
as obtained from \eqref{rfyb} (with $\tilde e $ replacing $e$).

From \eqref{rfyb}, \eqref{uad}, we obtain an upper bound on $\ker P_\cals$, competing with \eqref{xca},
\be\label{ujk} \ker P_\cals (x) \subseteq
\begin{cases}
\underset{z_{\hcals} \in \hcals}{\cap} \range \; \psi_{\nu, zz} (z_{\hcals} (x)) \oplus  \hbox{span\ } \{ e_0 (x)\}, \; \; x \in \partial\Om^{rU} \cup \pOm^{rL}; \\
\hbox {\ span\ }  \{ e_0 (x)\},  \; \; \,  x \in \pOm^{rD}.
\end{cases}\ee

    \subsection{Fluid flow}

The familiar examples of \eqref{aa} for which the above conditions can be satisfied are Euler systems, with $x_m$ the time-like variable and $D$ mildly restricted so that \eqref{tab} holds.  Here for simplicity we restrict attention to a single isentropic fluid, for which
\be\label{uba}    n = m,\ee
typically 3 or 4.  We assume $T$ either translation \eqref{tee} or scaling \eqref{tef}; in either case \eqref{tesa}, \eqref{the} apply.  The reduced system \eqref{tfc} is thus either \eqref{tfd} or else \eqref{tfg}.

For either case, using \eqref{uba},  setting
\be\label{ubb} z \, \eqadef \, \begin{pmatrix} u_1\\ \vdots\\ u_{n-1} \\ z_n\end{pmatrix};\ee
the dependent variable $u$, of dimension $n-1$,  is identified as the fluid velocity, and \eqref{aa} is obtained with
\be\label{ubc} \psi_l(z) = u_l \psi_m (z), \; \; \, l = 1, \dots, m-1,\ee
identifying $\psi_m (z)$ as the fluid pressure.

A given equation of state determines
\be\label{ubd} \psi_m(z) = \psi_m(z_n + \tfrac 12 |u|^2),\ee
with $z_n + \tfrac 12 |u|^2$ identified as the fluid
enthalpy,
positive by convention.

Phase space $D$ is restricted so that for all $H = H ( y) = y_n + \frac 12 \suml^{n-1}_{i = 1} y^2_i, \; \; y \in D$,
\be\label{ube} \psi'_m (H), \psi''_m (H) > 0\ee
\be\label{ubf} \psi'''_m (H) \neq 3 (\psi''_m(H))^2 / \psi'_m(H)\ee
with primes denoting derivatives on $\psi_m$ with respect to the enthalpy $H$  throughout.   From \eqref{ube}, $\psi'_m(z)$ is identified as the fluid density and
\be\label{ubg} c_s (z) = (\psi'_m (z) /\psi''_m (z))^{1/2}\ee
as the sound speed.

From \eqref{aaea}, \eqref{ubb}, \eqref{ubc}, the entropy flux components are
\be\label{ucc} q_i (\psi^\dagger_{i, z} (z)) = u_i (z_n + |u|^2) \psi'_m(z), \; \; \, i = 1, \dots, m-1;\ee
\be\label{ucd} q_m(\psi^\dag_{m, z}(z)) = (z_n + |u|^2) \psi'_m(z) - \psi_m(z).\ee

The rotation symmetry apparent in \eqref{ubb}, \eqref{ubc}, facilitates satisfying the assumptions of lemma 11.1.  With $z_{\hcals}$ independent of $x_m$, for $k=1, \dots, K$, any $x \in \Ga_k(z_{\hcals})$, in \eqref{tac}, \eqref{tad}, \eqref{tae}, we have
\be\label{ubh} \hat\tau(x) = \hat\mu_k(x) \; \eqadef\, {\mu_T(x)\choose 0}, \; \; \, \mu_T(x) \in \bbs^{m-1}.\ee

The condition \eqref{ubf}, typically satisfied with $\psi'''_m$ negative, suffices to determine that independently of $z_{\hcals}, \hat\mu_k$, characteristic fields $1, n$ are genuinely nonlinear and $2, \dots, n-1$ are linearly degenerate . Details omitted.

Thus lemma 11.1 applies to both stationary and self-similar problems.  The different expressions  \eqref{ter}, \eqref{tes} nonetheless result in material difference in subsequent treatment of the two problem classes.

The forms of \ $\Ga(x_{\hcals}), \Ga_{\hcals}$ are determined by pointwise solution of the Rankine-Hugoniot relation \eqref{afb}.  On $\Ga_k(z_{\hcals})$, using \eqref{ubh}, \eqref{tei}, \eqref{ter}, \eqref{tes}, \eqref{ubc},
\begin{align}\label{ubi}
0 = [\psi_{\hat{\mu}_{k},z} (z)]^\dag
&= [\mu_T \cdot r_x(\cdot) \psi^\dag_z(z)]\nonumber \\
&=
\begin{cases} [\psi^\dag_{\mu_{T},z} (z)], \; \; \; \hbox{\ for\ stationary\ problems\ } \\
 [\psi^\dag_{\mu_T ,z }] - (r \cdot \mu_T) [\psi^\dag_{m,z} (z)], \; \; \; \hbox{\ for\ self-similar problems},
 \end{cases} \end{align}
 with
 \be\label{ubib} \psi_{\mu_T} (z) \; \eqadef\; (\mu_T \cdot u) \psi_m(z),\ee
 with $\mu_T \cdot u$ identified as the normal velocity component pointwise on $\Ga_k (z_{\hcals})$ using \eqref{ubh}.

 The solutions of \eqref{ubi}  are familiar.  Shocks correspond to $\Ga_k(z_{\hcals})$ associated with either genuinely nonlinear characteristic field 1 or $n$, and satisfy
 \be\label{ubia} [\mu_T \cdot u], [\psi_m(z)] \neq 0, \; \; \, [u-\mu_T (\mu_T \cdot u)] = 0.\ee

 In stationary problems, such necessarily occur with ``supersonic" limiting velocities satisfying
 \be\label{ubj} | \mu_T \cdot u | > c_s(z)\ee
 for at least one side of the discontinuity segment $\Ga_k(z_{\hcals})$.  However in self-similar problems, \eqref{ubj} is replaced by
 \be\label{ubja} |\mu_T \cdot (u - r)| > c_s (z),\ee
 again details omitted.

For the isentropic system \eqref{ubb}, \eqref{ubc}, there are no ``contact discontinuities."
 Tangential discontinuities, with
\be\label{ubk} [\mu_T \cdot u], \;  [\psi_m(z)] = 0, \; \; [u-\mu_T (\mu_T \cdot u)] \neq 0\ee
correspond to $\Ga_k(z_{\hcals})$ associated with linearly degenerate characteristic fields $2, \dots, n-1$.  In stationary problems, such occur with limiting values
\be\label{ubl} \mu_T\cdot u = 0\ee
on both sides of $\Ga_k (z_{\hcals})$.  In self-similar problems, such occur with
\be\label{ubm} \mu_T\cdot (u-r) = 0\ee
on both sides of $\Ga_k(z_{\hcals})$.

Tangential discontinuities are ``unphysical" in the sense that they may be incompatible with  \eqref{fad}, depending on the adopted dissipation term.  For example, a commonly adopted dissipation in \eqref{fad}, anticipated in \eqref{fbp},
\be\label{ubn} \cald_k(z) = \suml^{m-1}_{l = 1} z_{k,r_l r_l}, \; \; \, k =1, \cdots, n,\ee
satisfies \eqref{ujc} and assures \eqref{fhe}, but is such that if \eqref{fad} holds, for positive $\Th$ in an open neighborhood intersecting a tangential discontinuity
\be\label{ubo} \iint \Th \om_T \tilde z_h \cdot \cald(\tilde z_h) \overset{ h \downarrow 0}{\longrightarrow}- \infty.\ee

A ``milder" form of dissipation
\be\label{ubp} \cald_k(z) = \begin{cases} \Big( \suml^{m-1}_{l = 1} u_{l, r_l}\Big)_{r_k}, \; \; \, k = 1, \cdots,  n-1\\
\; \; \; 0,  \quad k = n \end{cases}\ee
satisfies \eqref{ujc} and avoids \eqref{ubo} at the expense of potential failure of \eqref{fhe}.

For either problem class, exclusion of tangential discontinuities complicates determination of any $\Ga_{\hcals}$  for problem classes in which the set of interior interaction points
\be\label{ubr} x \in \Ga^I (z_{\hcals}) \cap \Om\ee
is nonempty, recalling \eqref{saaz}.  Such $x$ cannot be rarefaction centers if \eqref{tca} holds.  ``Triple points" $x$ are precluded by requirement of satisfying the Rankine-Hugoniot conditions, in a neighborhood of such $x$, for either problem class.  Thus in the absence of rarefaction centers and of  tangential discontinuities, any such points $x$ are necessarily limit points of four segments $\Ga_k(z_{\hcals})$, two 1-shocks and two $n$-shocks.

Remaining details are different for the two problem classes considered.

For stationary problems, adopting \eqref{tesa} for definiteness we have
\be\label{tesc} r_i = x_i, \; \; i = 1,\dots, m-1\ee
\be\label{tesd} r_x = (I_{m-1} \vdots \, 0) \in M^{(m-1)\times m}\ee
and in \eqref{uhj}
\be\label{uta} \om_T = 1.\ee

The condition \eqref{uhc} follows from \eqref{tga}.

From \eqref{ucba}, \eqref{tesd}, \eqref{ubc}, for any $r \in \pOm^r$,
\be\label{ucbb} \psi_\nu (z) (r) = (\nu_T (r) \cdot u(r)) \psi_m(z(r)).\ee

Differentiating \eqref{ucbb}, in \eqref{fae}, \eqref{raa},
\be\label{uda} \psi^\dag_{\nu,z} (z) = {\nu_T\choose 0} \psi_m(z) + (\nu_T\cdot u) \psi'_m (z) {u\choose 1},\ee
and in \eqref{rfyb}, \eqref{ujk}
\begin{align}\label{udc}
&\psi_{\nu, zz} (z) = \psi'_m(z) \Big( {\nu_T \choose 0} \Big. (u^\dag \; \; 1) + {u\choose 1} (\nu^\dag_T \;  \; 0)\nonumber \\
+& (\nu_T \cdot u)  \Big.\; \diag \;  (1, \dots, 1, 0)\Big) + (\nu_T\cdot u) \psi''_m(z) {u\choose 1} (u^\dag \; \, 1).\end{align}

In view of \eqref{uda},  boundary conditions on $\partial\Om^r\times \Om^t$ are expressed with the assumption that almost everywhere on $\partial\Om^r, \sgn (\nu_T \cdot u)$ is independent of $z_{\hcals} \in \hcals$, restricting $\hcals$ as necessary.  Then modulo a set of measure zero,
\be\label{udb} \partial\Om^r = \partial\Om^r_- \cup \partial\Om^r_0 \cup \partial\Om^r_+,\ee
with
\be\label{udba} \nu_T(r)\cdot u(r) < 0, \; \; \, r \in \partial\Om^r_-,\ee
\be\label{udbb} \nu_T(r)\cdot u(r) = 0, \; \; \, r \in \partial\Om^r_0,\ee
\be\label{udbc} \nu_T(r)\cdot u(r) > 0, \; \; \, r \in \partial\Om^r_+.\ee

Within $\partial\Om^r_\pm$, using \eqref{udc}, \eqref{ubg}, $\psi_{\nu,zz} (z)$ is nonsingular, implying no restriction from \eqref{ujk} on $\ker P_{\cals}$, except at ``sonic" points, where
\be\label{udd} (\nu_T \cdot u) = c_s(z),\ee
as obtained from \eqref{ubg}.

We tacitly assume $\hat\cals$ such that \eqref{udd} holds at most on a set of measure zero for any $z \in \hat\cals$.

From \eqref{uda}, \eqref{fae}, at any point in $\partial\Om^r$ where \eqref{udd} does not hold, prescribable inhomogeneous boundary data $b$ in \eqref{uhkb} includes
\be\label{udda} \psi_m(z) + (\nu_T \cdot u)^2 \psi'_m(z), \; \; \hbox{for\ } {\nu_T\choose 0} \in \ker P_\cals;\ee
\be\label{uddb} (\nu_T \cdot u) (\nu_{T\perp} \cdot u) \psi'_m (z), \; \; \hbox{ for\ } {\nu_{T\perp}\choose 0} \in \ker P_\cals, \; \; \nu_{T\perp} \cdot  \nu_T=0;
\ee
\be\label{uddc}(\nu_T\cdot u) \psi'_m(z), \; \, \hbox{ \ for\ } \begin{pmatrix} 0\\ \vdots\\ 0 \\ 1\end{pmatrix} \in \ker P_\cals.\ee
In particular, using \eqref{ube}, wherever  $P_\cals = 0$ all of $z$ is prescribable.

Within $\partial\Om^r_0$, by inspection of \eqref{udc}, using \eqref{uddb},
\be\label{ude} \range\; \,  \psi_{\nu,zz} = \hbox{ span\ }  \Big\{ {\nu_T\choose 0}, \; {u\choose 1} \Big\}.\ee

Using \eqref{ujk}, \eqref{ude}, \eqref{uda}, \eqref{fae},  here prescribable inhomogeneous boundary data  is limited to
\be\label{udea} \psi_m(z), \; \hbox{wherever \ } {\nu_T\choose 0} \in \ker P_\cals.\ee

However, \eqref{udbb} is itself a condition of the form \eqref{agea}, with
\be\label{udeb} e_0 = \begin{pmatrix} 0\\ \vdots\\ 0 \\ 1\end{pmatrix}\ee
obtained using \eqref{uda}, \eqref{ube}.  Here \eqref{agec} applies, so  necessarily within $\pOm^r_0$
\be\label{udec} e_0 \in \ker P_\cals.\ee

Additional  conditions on $z_{\hcals}, P_{\cals}$ suffice to satisfy \eqref{fhc}.

For stationary problems, from \eqref{uje}, \eqref{ucc}, \eqref{ucd}, \eqref{ter},
\begin{align}\label{uce} q_{r,l} (z) &= q_l(\psi_{\cdot, z}(z))\nonumber \\
&= u_l(z_n+ |u|^2) \psi'_m(z), \; \; \, l = 1, \dots, m-1.\end{align}

Thus in \eqref{fhc}, almost everywhere within $\partial\Om^r$,
\be\label{ucea} \nu_T \cdot  q_r (z) = (\nu_T \cdot u) (z_n + |u|^2) \psi'_m (z).\ee

Using positive fluid density and enthalpy, use of \eqref{udba}, \eqref{udbb}, \eqref{udbc} in \eqref{ucea}, it follows that \eqref{fhc} is nontrivial only in $\partial\Om^r_-$, and can be satisfied only by prescribing all of $z$, necessarily with
\be\label{udf} P_\cals\mathop{\mid}\limits_{\partial\Om^r_-} = 0\ee
almost everywhere.

From \eqref{udf}, \eqref{udec}, using \eqref{uada}, \eqref{uadb}, \eqref{uad}, such requires $\tilde e (\cdot; z_{\hcals}), z_{\hcals} \in \hcals$ constructed such that (modulo a set of measure zero) for all $z_{\hcals} \in \hcals$
\be\label{udfa} (\partial\Om^r_- \cup \partial\Om^r_0) \times \Om^t\subseteq \{ x \, | \, \tilde e (x; z_{\hcals}) \cdot \nu (x) < 0\}.\ee

As no prescribed data on $\partial\Om^r_+$ is required to satisfy \eqref{fhc}, we anticipate \eqref{udfa} with equality and
\be\label{udfb} \partial \Om^r_+ \times \Om^t = \{ x \, | \, \tilde e (x; z_{\hcals} ) \cdot \nu (x) > 0\}.\ee

Comparing \eqref{udfa}, \eqref{udfb} with \eqref{tzia}, \eqref{tzib}, using \eqref{udba}, \eqref{udbb}, \eqref{udbc}, a seemingly mild requirement on $\tilde e (\cdot; z_{\hcals})$ suffices in this context,
\be\label{udga} \sgn (\tilde e (x; z_{\hcals}) \cdot \nu (x) ) = \sgn (u(r(x)) \cdot \nu_T (r(x)), \; \, x \in \partial\Om^r_{\pm} \times \Om^t;\ee
\be\label{udgb} \tilde e (x; z_{\hcals} ) \cdot \nu(x) < 0, \; \; \, x \in \partial \Om^r_0 \times \Om^t.\ee

The condition \eqref{ujh} is recovered using \eqref{uje}, \eqref{ucc}, obtaining
\be\label{udh} -\iintl_{\Om^r} \, \suml^{m-1}_{l = 1}\;  \Theta_{r_l} q_{r,l} (z_{\hcals}) = - \iintl_{\Om^r} \, \suml^{m-1}_{l = 1} \big(\Theta_{r_l} u_l\big) (z_n + |u|^2) \psi'_m (z_{\hcals}).\ee

A seemingly mild assumption is made, that for any $z_{\hcals} \in \hcals$, there exist trajectories
\be\label{udi} r_\tau (r(\tau)) = u(r(\tau)), \; \; \, r(\tau) \in \Om^r/\Ga(z_{\hcals}), \; \; \, \tau > 0, \ee
\be\label{udj} r(0) \in \partial\Om^r_-, \ee
continuous on $\Ga(z_{\hcals})$, filling $\Om^r$.

Then in \eqref{ujh}, it suffices to choose a set $\{ \Theta\},\;  \Theta$ nonnegative and satisfying \eqref{ujg} with
\be\label{udk} \suml^{m-1}_{l = 1} \, \Theta_{r_l} (r) u_l(r) = - \xi_q (r) \, \Theta (r), \; \, \xi_q (r) > 0, \; \; \, r \in \Om^r/\Ga(z_{\hcals}),\ee
continuous on $\Ga(z_{\hcals})$ with ``initial data"
\be\label{udl} \Theta\mathop{\mid}\limits_{\partial \Om^r_-} \ge 0,\ee
for a judiciously chosen set $\{ \xi_q\}$.  Further details omitted.

Self-similar problem classes correspond to $T$ obtained from \eqref{tef}, $r$ from \eqref{tes}, retaining \eqref{tesa}, \eqref{the}.  With \eqref{tes} replacing \eqref{tesc} and \eqref{tesb} replacing \eqref{tesd},
\be\label{utc} \om_T = x^{m-1}_m\ee
replaces \eqref{uta} in \eqref{uhj}.

From \eqref{tgb}, by inspection
\be\label{utb} T_E = - I_n\ee
in \eqref{uhc}, and \eqref{uhkd} applies.

Euler systems admit Galilean symmetry, of the form \eqref{teb} with
\be\label{ute} u^* = - s_T, x^*_i = - s_{T,i} x_m, \; \; \, i = 1,\dots, m-1, s_T \in \bbs^{m-1}\ee
retaining \eqref{tesa}.  As a result, results from self-similar problem classes based on primitive Euler systems are obtained from those for stationary problem classes by replacement
\be\label{utd} u(r) \to u(r) - r \ee
in the leading terms, with appropriate ``lower order" adjustments.  Such is already apparent by juxtaposition of \eqref{tfd} with \eqref{tfg} or \eqref{ubja} with \eqref{ubj}.

Using \eqref{tesb}, \eqref{uca}, for self-similar problems \eqref{ucaa} holds,
and  the expression \eqref{ucba} is
\be\label{uec} \psi_{\nu} (z) (x) = \frac{1}{x_m} (\nu_T(r) \cdot (u(r) - r)) \psi_m(z).\ee

Differentiating \eqref{uec}, for self-similar problems \eqref{uda}, \eqref{udc} become, respectively,
\be\label{ued} \psi^\dag_{\nu, z}(z) = \frac{1}{x_m}\Bigg({\nu_T \choose 0} \psi_m(z) + (\nu_T \cdot (u-r)) \psi'_m(z) {u \choose 1}\Bigg) ;\ee
\begin{align}\label{uee}& \psi_{\nu,zz}(z)  = \frac{1}{x_m}( \psi'_m (z) \Bigg( {\nu_T\choose 0}  (u^\dag \; \; \, 1) + {u\choose 1}  (\nu^\dag_T\; \; \; 0)\nonumber \\
+  (\nu_T &\cdot (u-r)) \,  \diag (1, \dots, 1, 0)
 + (\nu_T \cdot (u-r)) \psi''_m(z) {u\choose 1} (u^\dag \; \; \; 1)\Bigg).\end{align}

For self-similar problems, using \eqref{tesb}, \eqref{ucc}, \eqref{ucd} in \eqref{uje}, the entropy flux is
\begin{align}\label{uef} q_r(z,x)_l &= \frac{1}{x_m} (q_l(\psi^\dag_{\cdot, z}(z)) - r_lq_m(\psi^\dag_{\cdot, z}(z))\nonumber \\
&= \frac{1}{x_m} \big((u_l-r_l) (z_n + |u|^2) \psi'_m(z) + r_l\psi_m(z)\big), \; \; l = 1, \dots, m-1.\end{align}

The form of $\Om^r$ and of the applied boundary conditions are significantly restricted for self-similar problems in comparison with stationary problems by the condition \eqref{uth}.

For stationary problems, using \eqref{teb}, \eqref{tee} it suffices that
\be\label{utj} x_T \cdot  \nu(x) = 0, \; \; x \in \pOm^r\times \Om^t,\ee
but for self-similar problems, using \eqref{teb}, \eqref{tef},
\be\label{utl} \nu(x) \cdot x = 0, \; \; \, x \in \pOm^r\times \Om^t\ee
is required.  Then using \eqref{ucaa}, \eqref{tesb}, in \eqref{utl}, necessarily
\be\label{ufz} \nu_T (r) \cdot r = 0, \; \; r \in \pOm^r\ee
independently of $x_m$, implying $\Om^r$ a cone  in $\bbr^{m-1}$, unbounded.

Self-similar problems are of particular interest as they admit interpretation as generalized  Riemann problems.  Unbounded $\Om^r$ is appropriate in this context, but unsuitable for computations Thus we shall forego \eqref{ufz} on an artificial boundary  $\pOm^r_A \subset \pOm^r$,  anticipating a limit
\be\label{ufza} ``\partial\Om^r_A \to \infty".\ee

Thus we take disjoint segments
\be\label{ufa} \partial\Om^r = \partial\Om^r_A \cup \partial\Om^r_L \cup \partial\Om^r_O\ee
with \eqref{ufz} holding within $\partial\Om^r_L$ and
\be\label{ufaa} \partial\Om^r_O\,  \eqadef\,  \{ r = 0\}.\ee

For definiteness and simplicity, in view of \eqref{ufza}
 here we take
 \be\label{ufab} \nu_T (r) \cdot r = |r|\, \eqadef\,
  \ol{r}, \; \; r \in \partial\Om^r_A,\ee
  and \eqref{ufza} becomes simply the limit $\ol{r} \to \infty$.

  For $\ol{r}$ sufficiently large, depending on $\hcals$, using \eqref{ube} it follows that within $\partial\Om^r_A, \, \psi_{\nu,zz} (z_{\hcals})$ is strictly negative definite.  (For large $|r|$, the system \eqref{tfg} is hyperbolic, with $-r$ time-like.) From the assumed convexity \eqref{tab} and  \eqref{aaea}, there exists a constant $c_{\hcals}$, again depending on $\hcals$, such that within $\Om^r/\Ga(z_{\hcals})$,
  \be\label{ufb} c_{\hcals} q_m(z_{\hcals}) > |\psi_{\cdot, z}(z_{\hcals} )|.\ee

  It follows from \eqref{uef}, \eqref{ufab}, \eqref{ufb} that for $\ol{r}$ sufficiently large, depending on $\hcals$, within $\partial\Om^r_A$, independently of $z_{\hcals} \in \hcals$,
  \be\label{ufc} \nu_T (r) \cdot q_r(z_{\hcals}, \cdot ) < 0.\ee

  Thus \eqref{fhc} is nontrivial throughout $\partial\Om^r_A$,
  which requires
  \be\label{ufd} P_{\cals}  \mathop{\mid}\limits_{\partial\Om^r_A} = 0,\ee
  almost everywhere, analogously with \eqref{udf}.

  Then from \eqref{ufd}, \eqref{uada}, \eqref{uadb} necessarily
  \be\label{ufe} \partial\Om^r_A \subseteq \partial\tilde \Om^U,\ee
  and from \eqref{ucaa}, \eqref{ufab}, within $\partial\Om^r_A$,
  \be\label{uff} \tilde e (r; z_{\hcals} )\cdot  \nu_T(r) < 0, \; \; |r| = \ol{r}.\ee

  From \eqref{ufz}, using \eqref{ued}, \eqref{uee}, within $\partial\Om^r_L $ the boundary conditions for self-similar problems coincide with those obtained above for stationary problems.

  However the prescribable boundary data on $\partial\Om^r_L$ is severely limited.

From \eqref{uada}
\be\label{ufg} \p\Om^r_L  = (\p\Om^r_L \cap \p\tilde \Om^U) \cup  (\p\Om^r_L \cap \p\tilde \Om^D).\ee

From \eqref{uadb}, \eqref{uff}, \eqref{tzic}, $\p\Om^r_L \cap \p\tilde \Om^D$ cannot be empty; typically
\be\label{ufh} \p\Om^r_L\subseteq \p\tilde \Om^D.\ee

Within $\p\Om^r_L\cap \p\tilde \Om^D$, from \eqref{uad}, \eqref{rfyb}
\be\label{ufi} \ker P_\cals = \hbox{span\ } \{ e_0(\cdot)\}.\ee

Typically
\be\label{ufj} \p\Om^r_L\subseteq \p\Om^r_0\ee
satisfying \eqref{udbb}, with $e_0$ obtained from \eqref{udeb}.

Recovery of \eqref{fhd} from \eqref{fhc} for self-similar problems is entirely analogous to that above for stationary problems.
The bound \eqref{ujh} is obtained with $q_r (z_{\hcals}, \cdot)$ obtained from \eqref{uef}.

Analogously with \eqref{udi}, \eqref{udj}, we assume $\hcals$ such that for any $z_{\hcals} \in \hcals$, there exist trajectories
\be\label{ugb} r_\tau(\tau) =
q_r(z_{\hcals} (r(\tau)),  \; \,   r(\tau) \in \Om^r/\Ga(z_{\hcals}),\; \; \, \tau > 0\ee
continuous on $\Ga(z_{\hcals}) $ and with
\be\label{ugc} r(0) \in \p\Om^r_A,\ee
filling $\Om^r$.

The condition \eqref{udi} tacitly assumes addition of a constant $m$-vector to $q$ so that $\mu \cdot q$ does not vanish or change sign on $\Ga(z)$.

Then analogous to \eqref{udk}, \eqref{udl}, we choose nonnegative $\Th$ satisfying
\be\label{ugd} \suml^{m-1}_{l = 1} \, \Th_{r_l}
q_{r, l} (z_{\hcals}, \cdot)
 = - \xi_q\, \Th,\; \; \, r \in \Om^r\setminus \Ga(z_{\hcals})\ee
continuous on $\Ga(z_{\hcals})$, with
\be\label{uge} \Th \mathop{\mid}\limits_{\p\Om^r_A} > 0\ee
and judiciously chosen $\xi_q = \xi_q(r) > 0$.  Further details omitted.

In contrast to stationary problems, with the form of $\Om^r$ restricted by \eqref{ufa}, distinct self-similar problem classes are determined largely by details in the prescribed boundary data.  Such is reflected in distinct $\Ga_{\hcals}$.

A familiar example is reflection of an incident shock by a wedge.  Weak regular reflection is obtained with $\hcals$ such that $\Ga^I(z_{\hcals}) \cap \Om$ is empty; the resulting $\Ga_{\hcals}$ is well-known.  With the same $\Om^r$, a stronger incident shock makes regular reflection impossible.  For irregular (Mach) reflection, $\Ga^I(z_{\hcals}) \cap~\Om$ cannot be empty; the resulting $\Ga_{\hcals}$ remains uncertain.  Indeed, the existence of a weakly well-posed problem class corresponding to irregular reflection remains uncertain.

\section{Loose ends}

Conjectured association of successful computational investigations with weakly well-posed underlying problems is obviously dependent on choice of  problem class.  At a minimum, specified details of a problem class must be sufficient to permit computational investigations to be undertaken, with success or failure thereof unambiguous a posteriori.  Simultaneously, whatever definition of weak well-posedness must be applicable.

Theorem 8.5 above determines that successful computational investigation, as per definition 6.6, implies weak well-posedness, as per definition 8.2, with perhaps surprisingly minimal additional requirements.  It suffices essentially that obtained solutions $z$ satisfy the entropy inequality \eqref{faf} and that the boundary data $\dot b$ be of sufficient regularity to permit a posteriori determination of stability of $(z, P_\cals)$,  using theorem 6.3 and \eqref{xaa}, \eqref{xab}, \eqref{xac}, \eqref{xad}, \eqref{xae}.

Within this framework, subsequent determination of sufficient conditions for weak well-posedness in sections 9-12 decisively fails to corroborate theorem 8.5 with available empirical evidence.

In this context, the problem class as detailed in definition 6.6 is significantly if superficially changed.  A reduced system \eqref{tfc} in a domain $\Om^r$, related to a primitive system \eqref{aa} as discussed in subsection 11.1, is identified as the underlying problem of interest, with conditions that the set of weak solutions of \eqref{tfc} is isomorphic to a subset of the weak solutions of \eqref{aa}.

In addition to the anticipated requirements of bounded $\Om^r$, the entropy inequality \eqref{faf} and sufficient regularity of the boundary data $\dot b$, various conditions have been required in the argument:  existence of distinguished directions $e(\cdot; z)$ or $\tilde e (\cdot; z); \, P_\cals $ restricted by \eqref{ujk}; solutions $z$ of increased regularity \eqref{tca}; and the primitive system \eqref{aa} hyperbolic and satisfying the assumptions of lemma 11.1.

A conundrum is evident.  Perhaps the adopted problem class is too large.  Perhaps the empirical evidence is misleading, based on special cases such as fluid flow.  In any event, the necessity  of the determined sufficient conditions for weak well-posedness is an open question.

A brief review of the discussion in the following subsection suggests that two of the additional required conditions, the existence of $\tilde  e(\cdot; z)$ and $P_\cals$ satisfying \eqref{ujk}, are likely satisfied in practice.  As against this, alternatives in the above analysis appear to be limited and unattractive.

Resolution of the conundrum is achieved, in the absence of predictability of successful computational investigations, by a combination of restriction of the adopted problem class and extension of definitions 6.6, 8.2.  An amended problem class is constructed within which theorem 8.5 holds simultaneously with a converse theorem 13.1.

The proof of theorem 13.1 is constructive.  A class of  approximation schemes is constructed, of the basic form adopted in section 6 but tailored to the amended problem class.  Failure of such schemes to determine solutions is directly related to failure of weak well-posedness.  Precise conditions for weak well posedness can thus be determined a posteriori from subsequent computational experiments.

The conditions of bounded $\Om^r$ and of solutions of regularity \eqref{tca} are subsequently relaxed by extension of the definitions of successful computations and of weak well-posedness to include convergent sequences of solutions in a sequence of domains.

\subsection{Review}

The obtained conditions for weak well-posedness follow from a series of strategic decisions, made largely for expedience  throughout this treatise and thus subject to review and potential revision.

The initial such decision is to restrict attention to systems of conservation laws equipped with an entropy extension, and to solutions in the range of a Frechet differentiable map \eqref{fag} of the boundary data and satisfying the entropy inequality \eqref{faf}.  Such is motivated  by previous work and by simple examples, including that illustrated in figure 2.1, applies in any number of dimensions and is inherited by reduced systems \eqref{tfc}.

Existence of an entropy extension implies the primitive system expressible in the form \eqref{aa}.  Frechet differentiability requires linearization of the weak form thereof \eqref{ae}.  Here such is obtained with solutions $z$ of the form \eqref{ad}, \eqref{ada}, satisfying \eqref{ac}, \eqref{aca}, the Rankine-Hugoniot conditions \eqref{afb} in particular.

Satisfying \eqref{ac}, \eqref{aca}, such solutions are presumed bounded in the space $\ul{\cals} $ \eqref{fhb}, with norm
\be\label{xdb} \| z \|_{\ul{\cals}} \, \eqadef \, \| z \|_{(\Om)} + \underset{\theta, i}{\lub} \iintl_\Om z\cdot \theta_{x_i} \ee
\be\label{xdd} \theta \in (C(\Om) \cap W^{1,1} (\Om))^n, \;  \, \| \theta\|_\cals \le 1, \; \;  i = 1, \dots, m,\ee
using the norm \eqref{xdc}.

The entropy inequality implies a weaker bound on solutions $z$, depending on the directionality of the entropy flux  \eqref{ujh}, \eqref{ujf} \eqref{fhc}, \eqref{fhd}, a repeatedly arising  issue.

Attention is now restricted to approximation schemes \eqref{fai} seeking solutions as vanishing dissipation limits, by discretization of systems such as \eqref{fad}.  Choosing a subsequence as necessary, a bounded sequence of $\tilde z_h \in \cals$,
\be\label{xdc} \| \tilde z_h\|_\cals \, \eqadef \, \| \tilde z_h\|_{L_\infty(\Om)} + \| \tilde z_h\|_{W^{1,1}(\Om)}\ee
determines
\be\label{xf} \tilde z_h \overset {h\downarrow 0}{\longrightarrow} z\ee
almost everywhere in $\Om$.

For definiteness, we restrict attention to a dissipation term $\cald$ of a form suggested by \eqref{ubn}, \eqref{ubp},
\be\label{xeg} \iintl_{\Om^r} \cald (\tilde z_h) \cdot \theta = - \iintl_{\Om^r} \suml^{m-1}_{l = 1} (M_{\cald, \cals} \tilde z_{h, r_l})\cdot (M_{\cald, \cals} \theta_{r_l}),\ee
$\theta \in W^{1, \infty} (\Om^r)^n$, each $M_{\cald, \cals}  \in M^{n\times n}$ which may depend, boundedly, on $\tilde z_h, x$.  Supplemental boundary conditions of the form \eqref{fal} are not required.

The condition \eqref{fbb} is immediate from \eqref{xeg}, and choosing $\theta = \Th \tilde z_h$ with nonnegative $\Th$, we recover \eqref{ujc} in the form
\be\label{xef} \iintl_{\Om^r } \Th \tilde z_h\cdot \cald (\tilde z_h) \le c \| \Th \|_{W^{1,\infty}} \| \tilde z_h\|^2_\cals\ee
using \eqref{xdc}.  Thus $z$  obtained from \eqref{fad} in the limit \eqref{xf} necessarily  satisfies \eqref{faf}.

Definitions 6.6 of a successful computational investigation and 8.2 of a weakly well-posed problem permit determination of conditions under which they are equivalent.  Bounded $\Om$ is required for computations, and simplifies definition 7.1.

A posteriori determination  of stable $(z_0, P_\cals)$ depends on regularity of the boundary data, introduced in \eqref{xaa}.  Such is not unexpected, comparing \cite{Ma1,Ma2} with theorem 6.1 of \cite{S1}.

The discussion of sections 9-12 is based on three strategic decisions.  Theorem 9.2 is used to relate stability of $(z, P(z))$ to a priori estimates for solutions of \eqref{rad}, \eqref{rae}, precluding nontrivial solutions of \eqref{agc}.

Such estimates are obtained from \eqref{rbj}, \eqref{rbk}, subsequently \eqref{rfv}, \eqref{rfva}, \eqref{rfvb}, \eqref{rfw}, at the expense of material restrictions.  Regularity of the boundary data is introduced in \eqref{rfa} and discussed in subsection 9.2.  Distinguished directions, introduced in \eqref{saa}, satisfying \eqref{sbk}, \eqref{tzg} in particular, are discussed in section 10.  Increased regularity of $z$ is required in \eqref{tca}.

Additionally, a requisite entropy condition on the discontinuity locus is expressed in definition 10.2, subsequently reduced to the inequality \eqref{faf} using lemma 11.1.  A corresponding compatibility condition on $z$,  expressed in definition 10.3,  is avoided by appeal to theorem 11.3, at the expense of increased regularity of the boundary data.

The restriction to hyperbolic systems is relaxed by introduction of the reduced systems in subsection 11.1.

The need for assumptions on the regularity of the boundary data is regarded as unavoidable, as such is anticipated and does not appear in sections 2-5, 7, 8.

Requisite conditions on distinguished directions $\tilde e(\cdot; z)$ introduced in \eqref{tda} follow from an expression of directionality of the entropy flux.  Within each $\tilde \Om^i$ in \eqref{tza} we assume $\tau \in C(\tilde \Om^i)$, constant within $\tilde \Om^i\cap \Ga (z)$, determined from trajectories satisfying
\be\label{xd} r_\tau (\tau) = q_r(z,x(r(\tau), x_m)), \; \; \, r(\tau) \in \tilde \Om^i\setminus\Ga (z)\ee
with $q_r$ from \eqref{uje}.  For self-similar problems, from \eqref{tfg}, notwithstanding \eqref{faf} with equality, within $\tilde \Om^i\setminus \Ga (z), \; q_r$ satisfies
\be\label{xded} \suml^{m-1}_{l = 1} q_{r, l} (z, \cdot )_{r_l} = - (m-1) (z \cdot \psi^\dag_{m, z} (z) - \psi_m(z))\ee
generally nonvanishing.

Then within each $\tilde \Om^i$, in \eqref{tzd}, \eqref{tze}, \eqref{tzf} we choose $\tilde \Th^i$
 a function of $\tau$.  The decisive condition \eqref{tzg} now follows from \eqref{aaea} \cite{M}
\be\label{xde} (\psi_{\tilde e, zz} (z))^{-1} = \tilde e\cdot q_{\psi_{\tilde e, z},\psi_{\tilde e,z}} (\psi_{\tilde e, z}).\ee

Choosing $\{ \tilde \Om^i\}$ judiciously and exploiting an arbitrary additive constant in $q$, we anticipate \eqref{tzia}, \eqref{tzib} compatible with
\be\label{xdea} \tilde \nu^i \cdot q_r (z, \cdot) \mathop{\midl}_{\partial\tilde\Om^{iU}} < 0 < \tilde \nu^i \cdot q_r (z, \cdot ) \mathop{\midl}_{\partial\tilde\Om^{iD}}.\ee
Then from \eqref{xdea}, \eqref{tdc}, \eqref{tde}, the condition \eqref{ujk} and possibly \eqref{fhca} follow from
\be\label{xdec} \supp (I_n-P_\cals) \subset \{ \nu_T\cdot q_r (z, \cdot) < 0 \} \cup \supp \{ e_0\}.\ee

Within each $\tilde \Om^i$, in \eqref{fhz} we take $\ul{\Th} = \tilde \Th^i$ obtaining
\be\label{xdeb} - q (\psi_{\cdot, z} (z)) \cdot( \triangledown \tilde \Th^i)^\dag \ge \frac 1c |q_r(z, \cdot)|\ee
recovering \eqref{fhd} recursively using \eqref{tzid}.

The competing conditions \eqref{xdec}, \eqref{xca} severely restrict the form of the prescribable boundary data.

The requirements of bounded $\Om, \Om^r$ and of $z$ satisfying \eqref{tca} will be weakened.  Left unanswered is whether primitive systems satisfying the assumptions of lemma 11.1 are required.

\subsection{Amended problem classes}

Details of an amended problem class include necessary information for a computational investigation to proceed, an underlying system \eqref{tfc}, bounded domain $\Om^r$, and boundary conditions determined from $P_\cals, \tilde B_{P_\cals}, \{ e_0\}$  satisfying \eqref{ab}, \eqref{agea}, likely selected on ``physical grounds".

Restrictions are needed so that the discussion of weak well-posedness applies. The given system \eqref{tfc} is the reduced form of a primitive system of the form \eqref{aa} using a given symmetry $T$, with domain $\Om$ obtained from \eqref{tfe}, extending the boundary conditions as per \eqref{uha}.

Here we shall assume for simplicity that $T$ is either translation \eqref{tee} or scaling \eqref{tef}, so \eqref{tfc} is either \eqref{tfd} or \eqref{tfg}.  In either case \eqref{uhkd} holds.

Solutions $z$ are required of the form \eqref{ad}, \eqref{ada}, satisfying \eqref{ac}, \eqref{aca}.  Tacitly, in view of \eqref{ac}, \eqref{aca},  we assume $z\in \hcals$ satisfying
\be\label{xea}\| z\|_{\ul\cals} \le c_{\hcals}\ee
using the norm \eqref{xdb}.

Anticipating solutions $z$ as limits of approximations \eqref{xf}, an amended problem class includes specification of a set $\cals$ of approximate solutions, an extension of the set introduced in \eqref{fah}, \eqref{fai}.  Indeed, a successful computational investigation will be conditioned on
\be\label{xda} \hcals \subset \cals.\ee

Such requires elements $z \in \cals$ also of the form \eqref{ad}, \eqref{ada}.  We require elements of $\cals$ satisfying \eqref{xea}, the Rankine-Hugoniot conditions \eqref{afb}, the boundary conditions
\be\label{xeb} \{ (I_n - P_\cals) \psi^\dag_{\nu_T, z} (z), \; \; z \in \cals \} = \tilde B_{P_\cals}\ee
simultaneously with \eqref{agea}, and with $\Ga(z), z \in \cals$, of a prescribed topological form $\Ga_{\hcals}$.

The amended problem class specifically anticipates \eqref{xf} as a vanishing dissipation limit.  A specified dissipation term $\cald$ of the form \eqref{xeg},  \eqref{xef} is required such that the elements of $\cals $ admit viscous structure, an extension of the familiar ``viscous profiles".

Specifically, for all $z \in \cals$, $\vare > 0$, there exists
\be\label{xfa} \hat z_\vare \in W^{1, \infty} (\Om^r)^n\ee
satisfying \eqref{agea},
\be\label{xfb} \| \hat z_\vare \|_\cals \le c_z,\ee
independent of $\vare$, using the norm \eqref{xdc};
\be\label{xfc} \hat z_\vare \overset{\vare\downarrow 0}{\longrightarrow} z,\ee
pointwise in $\Om^r\setminus \Ga (z)$; and
\begin{align}\label{xfd}&\underset{\theta}{\lub} \, \frac{1}{\| \theta \|_\cals}\, \iintl_{\Om^r}(E(\hat z_\vare)-\vare \cald (\hat z_\vare) - E(z)) \cdot \theta \overset{\vare\downarrow 0 }{\longrightarrow}
 0, \nonumber \\
 &\theta \in W^{1,\infty} (\Om^r)^n, \; \; P_\cals \theta \mathop{\midl}_{\pOm^r} = 0 .\end{align}

 Using \eqref{tfd}, \eqref{tfg}, \eqref{xeg}, \eqref{uhb}, we write out \eqref{xfd} for the two cases:
 \begin{align}\label{xfda} \underset{\theta}{\lub} \, \frac{1}{\| \theta \|_\cals}&\iintl_{\Om^r} \Big( \suml^{m-1}_{i = 1} (\psi_{i,z} (z) - \psi_{i, z} (\hat z_\vare)) \theta_{x_i}\nonumber \\
+ & \vare\suml^{m-1}_{i = 1} (M_{\cald, i} \hat z_{\vare, x_i})\cdot (M_{\cald, i} \theta_{x_i})\Big) \overset{\varepsilon\downarrow 0}{\longrightarrow} 0\end{align}
for stationary problems;
\begin{align}\label{xfdb}  \underset{\theta}{\lub} \, \frac{1}{\| \theta \|_\cals}&\iintl_{\Om^r} \Big( \suml^{m-1}_{l = 1} \big((\psi_{l,z}
(z) - \psi_{l, z} (\hat z_\vare)) \theta_{r_l}- (\psi_{m, z} (z) - \psi_{m, z} (\hat z_\vare)\big) r_l \theta)_{r_l})\nonumber \\
+ & \vare\suml^{m-1}_{l = 1} (M_{\cald, l} \hat z_{\vare, r_l})\cdot (M_{\cald,l} \theta_{r_{l}})\Big) \overset{\varepsilon\downarrow 0}{\longrightarrow} 0\end{align}
for self-similar problems.

The condition \eqref{xfd}  potentially restricts $\cald $ given $\Ga_{\hcals}$, as illustrated in \eqref{ubo} above.  In any neighborhood where $z\in W^{1,\infty}$, \eqref{xfd} is satisfied with $\hat z_\vare = z$, using \eqref{xeg}.  Indeed, the strength of this condition results from the appearance of $\| \theta\|_\cals$ as opposed to $\| \theta\|_{W^{1,\infty}}$.   Approximations $z \in \cals$ satisfying the Rankine-Hugoniot condition \eqref{afb} is essential in constructing $\hat z_\vare$.

Specifically not required is the primitive system \eqref{aa} hyperbolic or satisfying the assumptions of lema 11.1.

Conditions \eqref{tca}, \eqref{xd}, \eqref{xdea} for elements of $\cals$  in principle  are optional.  Restrictions on $\cals$ may jeopardize \eqref{xda}, but expedite construction of $\hat z_\vare$ satisfying \eqref{xfd} and may be helpful in computations.

\subsection{A converse theorem-failed computations}

\begin{thm} Assume an amended problem class as discussed in subsection 13.2 and a solution set $\hat \cals$, satisfying \eqref{xea}, \eqref{xda}, such that $(\hcals, P_\cals)$ determines a weakly well-posed problem as per definition 8.2.

Then a successful computational investigation as per definition 6.6 can be achieved.

\end{thm}
\begin{proof} An approximation scheme $\cala_{\de, \vare, h}$  of the form discussed in section 6 is constructed, but depending on three quantifiers: the dissipation parameter $\vare$ is \eqref{xfd} and ``mesh parameters" $\delta, h$.    For $\de > 0$, we introduce finite-dimensional ``test spaces"
\be\label{xha} X_\de \subset X_{P_\cals} \cap (W^{1,\infty} (\Om^r))^n\ee
becoming dense in $X_{P_\cals} \cap (W^{1,1} (\Om^r))^n$ as $\de \downarrow 0 $ (with respect to the norm $\| \cdot \|_\cals)$.

For $h > 0$, we introduce ``trial spaces"
\be\label{xhb} \tilde X_h \subset W^{1, \infty} (\Om^r)^n\ee
with elements satisfying \eqref{agea} and becoming dense as $h \downarrow 0$ (again with respect to the norm $\| \cdot \|_\cals$).

Further approximation properties of $X_\delta, \tilde X_h$ are not required; piecewise linear elements on simplexes will suffice.

Using $X_\de, \tilde X_h$, a regularized form of \eqref{tfc}  is discretized.  A sequence $\{ \tilde z_h\}$ is sought satisfying
\be\label{xhd} \lim\limits_{\de\downarrow 0}\, (\lim\limits_{\vare\downarrow 0}\, (\lim\limits_{h \downarrow 0}\,
\iintl_{\Om^r} (E (\tilde z_h) - \vare \cald (\tilde z_h)) \cdot \theta_\de))) = 0,\ee
for all $\theta_\de \in X_\de$.

For arbitrary $b \in \tilde B_{P_\cals}$, using \eqref{tfd}, \eqref{tfg}, \eqref{xeg}, \eqref{uhb}, the system \eqref{xhd} is
\begin{align}\label{xhda}
\lim\limits_{\de\downarrow 0}\, &(\lim\limits_{\vare\downarrow 0}\,  (\lim\limits_{h \downarrow 0}\, (\intl_{\pOm^r} b \cdot \theta_\de - \iintl_{\Om^r} \suml^{m-1}_{i = 1}
\psi_{i,z} (\tilde z_h)\theta_\de\nonumber \\
&+ \vare \iintl_{\Om^r} \suml^{m-1}_{i = 1} (M_{\cald, i} \tilde z_{h, x_i}) \cdot (M_{\cald, i} \theta_{\de, x_i})))) = 0, \; \; \, \theta_\de \in X_\de
\end{align}
for stationary problems;

\begin{align}\label{xhdb} \lim\limits_{\de\downarrow 0}\, &(\lim\limits_{\vare\downarrow 0}\, (\lim\limits_{h \downarrow 0}\, \big( \intl_{\pOm^r} b \cdot \theta_\de
- \iintl_{\Om^r} \big(\suml^{m-1}_{l = 1}
\psi_{l,z} (\tilde z_h)\theta_{\de, r_l}- \psi_{m,z} (\tilde z_h)(r_l\theta_\de)_{r_l}\big)\nonumber \\
&+ \vare \iintl_{\Om^r} \suml^{m-1}_{l = 1} (M_{\cald, l} \tilde z_{h, x_l}) \cdot (M_{\cald, l} \theta_{\de, x_l})\big))) = 0, \; \; \, \theta_\de \in X_\de\end{align}
 for self similar problems.

It suffices to prove that if \eqref{xda} holds, for any
\be\label{xhea} z \in \hcals,\ee
there exists a sequence $\tilde z_h \in \tilde X_h$,
\be\label{xhdz} \tilde z_h = \tilde z_{\de, \vare_\de, h_\de}, \; \; \vare_{\de}, h_{\de}\,  \xlongrightarrow{\de\downarrow 0}\,  0\ee
such that \eqref{xf} holds.

From \eqref{xhea}, \eqref{xda}

\be\label{xhe} z \in \cals\ee
and there exists $\hat z_\vare$ satisfying \eqref{xfa}, \eqref{xfb}, \eqref{xfc}, \eqref{xfd}.

From \eqref{xeb}, \eqref{xhe}, there exists $b \in \tilde B_{P_\cals}$ such that
\be\label{xhc} b = (I_n - P_\cals)\, \psi^\dag_{\nu_T, z} (z)\ee
almost everywhere in $\pOm^r$.

We write out the proof for the case of stationary solutions \eqref{xhda}, that for self-similar solutions being entirely similar.

For any $\de, \vare, h > 0, \; \, \theta \in X_{P_\cals} \cap W^{1,1} (\Om^r)^n, \; \theta_\de \in X_\de, \; \tilde z_h \in \tilde X_h, \; \hat z_\vare$ satisfying \eqref{xfa}, \eqref{xfb}, \eqref{xfd}
\begin{align}\label{xhf} \intl_{\pOm^r} b \cdot \theta - \iintl_{\Om^r}\,& \suml^{m-1}_{i = 1} \psi_{i,z} (z) \theta_{x_i} - \Big( \intl_{\pOm^r} b \cdot \theta_\de - \iintl_{\Om^r} \big(\suml^{m-1}_{i = 1} \psi_{i, z} (\tilde z_h) \theta_{\de, x_i} \nonumber \\
&+\vare\cald (\tilde z_h)\cdot \theta_\de \big) \Big) = T_1 + T_2 + T_3\end{align}
\be\label{xhg}  T_1 \, \eqadef\, \intl_{\pOm^r} b \cdot (\theta-\theta_\de) - \iintl_{\Om^r} \, \suml^{m-1}_{i = 1} \, \psi_{i, z} (z) (\theta - \theta_\de)_{x_i} \ee
\begin{align}\label{xhh}
T_2 \, \eqadef \, \iintl_{\Om^r} &\big( \suml^{m-1}_{i = 1} \psi_{i, z} (z)  \theta_{\de,x_i}- \big( \suml^{m-1}_{i = 1} \psi_{i,z} (\hat z_\vare)\theta_{\de, x_i}\nonumber \\
&+ \vare \cald(\hat z_\vare) \cdot \theta_\de\big)\big) \end{align}
\begin{align}\label{xhj}T_3 \, \eqadef\,  &\iintl_\Om \big( \suml^{m-1}_{i = 1} (\psi_{i, z} (\hat z_\vare) - \psi_{i, z} (\tilde z_h)) \theta_{\de, x_i}\nonumber \\
&+\vare (\cald (\hat z_\vare) - \cald (\tilde z_h)) \cdot \theta_\de \big).
\end{align}

For arbitrary $\theta$, using \eqref{xha} in \eqref{xhg}, $|T_1|$ is made arbitrarily small for sufficiently small $\de$, determining $\theta_\de \in W^{1,\infty} (\Om^r)$.

Then using \eqref{xfda}, \eqref{xeg} in \eqref{xhh}, $|T_2|$ is made arbitrarily small by sufficiently small $\vare$, depending on $z$.

With $\hat z_\vare, \theta_\de$ thus determined, from \eqref{xhj}, using \eqref{xfc}, \eqref{xhe}, $|T_3|$ is made arbitrarily small by sufficiently small $h$.

Thus \eqref{xhf} implies \eqref{xf}.

The assumption of weak well-posedness implies $(z, P_\cals)$ stable as required in definition 6.6.
\end{proof}

Sufficient conditions for weak well-posedness may be inferred, a posteriori of a successful computation, by appeal to theorem 8.5.  Correspondingly, theorem 13.1 indirectly relates necessary conditions for weak well-posedness to failure of computational investigations.

From theorem 13.1, failure of computational investigations using approximation schemes $\cala_{\de, \vare, h}$ necessarily arises from one of four causes:  lack of weak well-posedness of the underlying problem; failure of \eqref{xea}; failure of \eqref{xda}; or imprudent choice of the requisite discretization parameters, including the spaces $X_\de$ \eqref{xha}, $\tilde X_h $ \eqref{xhb}, the quantifiers $\de, \vare, h$ \eqref{xhdz} or an algorithm for (approximate) solution of the discrete systems obtained from \eqref{xhda} or \eqref{xhdb}.

Symptoms of such failure include inability to solve the discrete systems (to sufficient accuracy), failure of the limit \eqref{xf}, and obtaining $(z, P_\cals)$ not stable, using definition 6.6.

Assuming that \eqref{xea}, \eqref{xda} hold, the failure symptoms are readily understood.  The requisite conditions on the discretization parameters for theorem 13.1 are quite mild, save for a condition on the adopted solution algorithm for the discrete systems, requiring \be\label{xhja} \tilde z_h \approx \hat z_\vare\ee
to make $|T_3|$ small un \eqref{xhj}.  From \eqref{xhja}, we infer in particular
\be\label{xhjb} \| \tilde z_h\|_\cals \le c \| \hat z_\vare \|_\cals \le c \| z\|_{\ul{\cals}}\ee
for some $z \in \hcals$.

Such suffices for solution of the discrete systems; \eqref{xhjb}, \eqref{xea} and assures existence of a limit \eqref{xf}, taking a subsequence as necessary.

Stability of  obtained $(z, P_\cals)$ is determined a posteriori, using theorem 6.3 and \eqref{xaa}, \eqref{xab}, \eqref{xac}, \eqref{xad}, \eqref{xae}, \eqref{xaf}. In view of \eqref{axh}, failure thereof is regarded as pathological, but is related to the regularity of the data $b$ by \eqref{addz}.  Sufficient regularity of  the prescribed data is regarded as a necessary condition for weak well-posedness.

If weak well-posedness and \eqref{xea} hold, from \eqref{xeb} failure of \eqref{xda} can result only from incompatibility of the adopted dissipation term $\cald$ with the given $\Ga_{\hcals}$, appearing as failure of \eqref{xfa}, \eqref{xfb}, \eqref{xfc}, \eqref{xfd}.  A possible source of such is discussed in \eqref{ubn}, \eqref{ubo}, \eqref{ubp} above.

Remaining is the possibility of weak well-posedness with failure of \eqref{xea}.  As the solution set $\hat\cals$ is not known a priori, for some $b \in \tilde B_{P_\cals}$ it is possible that \eqref{pk} holds with some
\be\label{xka} \hat V (b) \subset \hat \cals,\ee
\be\label{xkb} \underset{z \in \hat V(b)}{\lub}\; \| z \|_{\ul{\cals}} = \infty,\ee
elements of $\hat V(b)$ satisfying the entropy condition \eqref{faf}.  This is the phenomenon referred to throughout as solutions not majorized by the prescribable boundary data, arguably pathological.

In view of the vanishing dissipation framework employed throughout, we have tacitly associated \eqref{xkb} with a stronger condition
\be\label{xkc}  \underset{z \in \hat V(b)}{\lub}\;\| z \|_{L_\infty (\Om^r)} = \infty.\ee

The condition \eqref{xkc} may be precluded a priori by imposition of conditions \eqref{xca}, \eqref{ujh}, \eqref{ujg} on the set $\cals$, thus obtaining \eqref{ujj}.

\subsection{Two remaining issues}

Unbounded domains, with nontrivial asymptotic conditions imposed on the corresponding solution sets, are frequent if not universal features of problem classes of interest meriting computational investigation.  Approximation schemes  include boundary conditions on artificial boundaries, as in \eqref{ufza}, leaving open the issue of the accuracy thereof.

Additionally, some problem classes of interest anticipate solution sets within which \eqref{tca} fails locally, in any open neighborhood containing some
\be\label{xpa} \Ga_{LS} (z) \subset \Om^r, \; \; z \in \hcals.\ee

Typically such occurs at ``interaction points"
\be\label{xpaa} \big\{ x (\Ga_{LS} (z), \ul{t} )\big\} \; \subseteq \Ga^I(z)\ee
as introduced in \eqref{saaz}, using notation from \eqref{tfe}, or at limit points of $\Ga(z)$, where
\be\label{xpab} | [z]| (r)
\xrightarrow{r\to\Ga_{LS}(z)} 0.\ee

Approximation schemes typically ignore this issue, in the selection of a dissipation term in particular, with possibly ambiguous results as noted above \eqref{ubn}, \eqref{ubp}.

Modifications of the discussion are required in either case.  Definition 6.6 of a successful computational investigation obviously does not apply if $\Om, \Om^r$ are unbounded.  If \eqref{tca} fails, failure of \eqref{hac}, \eqref{had} in a neighborhood of $\Ga_{LS}(z)$ is possible, precluding use of theorem 7.2 in the proof of theorem 8.5.  In addition, the conditions \eqref{xfa}, \eqref{xfb}, \eqref{xfc}, \eqref{xfd} in the specification of $\cals$  are unrealistic, undermining theorem 13.1.

Theorems 8.5, 13.1 may be extended to accommodate either phenomenon.  Retaining the notation of \eqref{ufza}, the given domain $\Om^r$ is regarded as a domain limit as $\ol{r}\uparrow \infty$,
\be\label{xpb} \Om^r = \underset{\ol{r}}{\cup}  \Om^r(\ol{r})\ee
with $\{ \Om^r (\ol r)\} $ a nested sequence of bounded domains, within each of which theorems 8.5, 13.1 apply with solution set $\hcals (\ol r)$ and boundary conditions determined by $P_\cals (\ol r)$.

Mirroring \eqref{ufa}, for each finite $\ol r$ the boundary $\pOm^r(\ol r)$ is of the form
\be\label{xpc} \pOm^r(\ol r ) = ( \pOm^r(\ol r) \cap \pOm^r) \cup \pOm^r_A (\ol r),\ee
with artificial boundary
\be\label{xpd} \pOm^r_A (\ol r) = \pOm^r(\ol r) \cap \Om^r.\ee

Even if $\Om^r$ is unbounded, $\pOm^r$ may be bounded, for example fluid flow past a given obstacle.  If so, then \eqref{xpc} holds with
\be\label{xpe} \pOm^r\subset \pOm^r(\ol r)\ee
for each $\ol r$.

On physical grounds, we assume given $T, \Om^r$ such that \eqref{uth} holds within $\pOm^r$.  Such cannot be expected for $\pOm^r_A (\ol r)$.  Then solutions of \eqref{tfc} within $\Om^r(\ol r)$ satisfy \eqref{aa} within $\ul{\Om} (\ol r)$, obtained from \eqref{trh}, \eqref{tri} with an additional restriction in the event of local singularities \eqref{xpa},
\be\label{xpf} \dist (\ul{\Om} ((\ol r), \{ x (t), \ul{t} < t < \ol{t} (\ol r), \; \; x (\ul{t}) \in \{ x (\Ga_{LS} (z), \ul{t} \} \} ) > 0\ee
with $x(t)$ obtained from  \eqref{tra}.

The condition \eqref{tfe} is required compatible with
\be\label{xpg} \Om \subseteq \underset{\ol r}{\cup} \ul{\Om} (\ol r)\ee
a seemingly mild restriction.

Boundary conditions \eqref{uha} are retained.

Then for each $b \in \tilde B_{P_\cals} $ and for each finite $\ol r$, within $\Om^r(\ol r)$ solutions
\be\label{xph} z(\cdot; \ol r) \in \hcals (\ol r)\ee
are ostensibly determined from approximation schemes, for example by replacement of $\Om^r$ by $\Om^r(\ol r)$ in \eqref{xhd} and following.

A successful computational investigation depends on existence of a limit solution in all of $\Om^r$,
\be\label{xpi} z (\cdot; \infty) \in \hcals\ee
satisfying
\be\label{xpj} \dist_{\cals, \ol r'} (z(\cdot; \ol r), \; \, z (\cdot; \infty) \overset{\ol r \uparrow \infty}{\longrightarrow} 0 \ee
with $\dist_{\cals, \ol r'}$ obtained from \eqref{ga} as restricted to $\Om^r(\ol r')$.

Definition 8.2 of weak well-posedness extends immediately, using solution sets \eqref{xpi}.

Extension of theorem 8.5, using \eqref{xpj}, then depends on sufficient detail in amended problem classes and sufficient conditions for successful computational investigations that weak well-posedness for each finite $\ol r$ implies weak well-posedness in the limit $\ol r \to \infty$.

In this context we recall that definitions 2.1, 2.3 of stability and admissibility of pairs $(z, P)$ require neither a bounded domain nor regularity  \eqref{tca}.  Stability in the limit $\ol r \to \infty$ is addressed using theorem 6.3, which implicitly quantifies approach to stability failure as the $w, w_\Ga$ required in \eqref{xae} become unbounded.

Limit admissibility has been discussed in section 5, and is implicitly related to a domain limit in \eqref{haa}, \eqref{hbe}.

Extension of theorem 13.1 follows from extension of \eqref{xda} to finite $\ol r$, requiring approximation sets satisfying
\be\label{xpk} \hcals (\ol r) \subset \cals (\ol r)\ee
for each $\ol r$, with $\hcals(\ol r) $ in \eqref{xph}.

Remaining details are different for unbounded domains than for existence of local singularities.  While these phenomena may obviously occur simultaneously, notation is simplified below by separate discussions.  The discussions  will be somewhat abbreviated in both cases.

The distinguishing feature of problem classes with unbounded $\pOm^r$ is the asymptotic conditions imposed on solutions \eqref{xpi}.

Specified details of an amended problem class with unbounded $\Om^r$ include a $(m-2)$-manifold
\be\label{xkka} \Phi_\Om \subset \bbr^{m-1},\ee
with unit normal $\nu_{T, \infty}$ almost everywhere, satisfying
\be\label{xkaa} \nu_{T,\infty} (r') \cdot r' > 0, \; \; \, r' \in \Phi_\Om,\ee
such that for each $\ol r$, \eqref{xpc}, \eqref{xpd} hold with
\be\label{xkkb} \pOm^r_A (\ol r) = \{ r \, | \, \frac{r}{\ol r} \in \Phi_\Om\}.\ee

The specified $P_\cals, \tilde B_{P_\cals}$ are of the form
\be\label{xkkc} P_\cals = {P_{\cals, \pOm^r} \choose P_{\cals, \infty}}\ee
\be\label{xkd} \tilde B_{P_\cals} = \bigg\{ {b_{\pOm^r}\choose b_\infty}\bigg\} \ee
with $P_{\cals, \pOm^r}$ and each $b_{\pOm^r}$ specified on (all of) $\pOm^r$ and $P_{\cals, \infty}$, each $b_\infty$ specified (almost everywhere) in $\Phi_\Om$.

Solutions $z(\cdot; \ol r)$ in \eqref{xph} are sought, satisfying whichever form \eqref{tfd} or \eqref{tfg} of \eqref{tfc}, weakly in $\Om^r (\ol r)$, using
\be\label{xkea} P_{\cals, \ol r} (r) \, \eqadef \,
\begin{cases} P_{\cals, \pOm^r} (r), \; \; \, r \in \pOm^r(\ol r) \cap \pOm^r\\
P_{\cals, \infty} \big( \frac{r}{\ol r}\big), \; \;\;\, r \in \pOm^r_A (\ol r),\end{cases}\ee
and satisfying the entropy condition \eqref{faf}.

Using \eqref{xkd}, \eqref{xkea}, \eqref{xkaa}, almost everywhere in $\pOm^r (\ol r)$, boundary data is of the form
\begin{align}\label{xke}
b_{\ol r} (r) &= (I_n - P_{\cals, \ol r} (r)) \; \psi^\dag_{\nu_T (r), z} (z(r; \ol r))\nonumber \\
&=
\begin{cases} b_{\pOm^r} (r), \; \; r \in \pOm^r(\ol r) \cap \pOm^r\\
b_\infty (\frac{r}{\ol r}), \; \;\; \,  r \in \pOm^r_A (\ol r).\end{cases}
\end{align}

Elements $\theta $ of the test space $X_{P_{\cals}, \ol{r}}$ satisfy
\begin{align}\label{xkf}
  0 &= P_{\cals, {\ol r}} (r) \theta (r)\nonumber\\
  &= \begin{cases}
   &P_{\cals, \pOm^r} (r) \theta (r), \; \; \,  r \in \pOm^r (\ol r) \cap \pOm^r \\
  &P_{\cals, \infty} \left( \frac {r}{\ol r}\right) \theta(r), \quad r \in \pOm^r_A (\ol r)\, .
  \end{cases}\end{align}

  The limit solutions \eqref{xpi}, \eqref{xpj} satisfy \eqref{xke}, \eqref{xkf} within $\pOm^r$ and asymptotic conditions
  \be\label{xkg} \bigl( I_n - P_{\cals, \infty} \left( \frac {r}{\ol r}\right) \bigr) \psi^\dag_{\nu_T, \infty} (z(r; \infty))
  \xrightarrow{{\ol r}\to \infty} \, b_\infty \left( \frac{r}{\ol r}\right)\ee
  for almost all fixed $r/\ol r \in \Phi_\Om$.

  Correspondingly, elements of the test space $X_{P_\cals, \infty} $ satisfy
  \be\label{xkh} P_{\cals, \infty} \left(\frac {r}{\ol r}\right) \theta (r) \xrightarrow{\ol r \to \infty} 0 \ee
  almost everywhere in $\Phi_\Om$.

  In the extension of theorem 8.5 to problem classes with unbounded $\Om^r$, the essential step is inference of stability of $(z(\cdot; \infty), P_\cals)$ from that of $(z(\cdot; \ol r), P_{\cals, \ol r})$ for each $\ol r$.  Here the approximation of asymptotic conditions by applied boundary conditions on the artificial boundaries is decisive.

  For each $\ol r$, stability of $(z(\cdot ; \ol r), \, P_{\cals, \ol r})$ is determined a posteriori using theorem 6.3, with some judiciously chosen $\dot b \in \calb_{P_\cals}$ of the form
  \be\label{xla} \dot b = {\dot b_{\pOm^r}\choose \dot b_\infty}\ee
  satisfying
  \begin{align}\label{xlaa} &P_{\cals, \pOm^r} (r) \dot b_{\pOm^r} (r) = 0, \; \; \, r \in \pOm^r,\nonumber \\
  &P_{\cals,\infty} (r') \dot b_\infty (r') = 0,\,\, \;   \; \; r' \in \Phi_\Om.\end{align}

  Then abbreviating
  \be\label{xlab} H_{\ol r} \, \eqadef\, H(z(\cdot; \ol r), P_{\cals, \ol r}, \; w(\ol r), w_\Ga(\ol r)),\ee
  \be\label{xlac} \| \theta\|_{\ol r} \, \eqadef\, \| \theta\|_{z(\cdot; \ol r), w(\ol r), w_\Ga (\ol r)},\ee
  the norm \eqref{ah} within $\Om^r(\ol r)$ for $\theta \in H_{\ol r}$,
  stability is verified using \eqref{xae} (within $\Om^r(\ol r))$, determination of $w(\ol r), w_\Ga (\ol r), \| \dot b \|_{B, P_\cals, \ol r}$ such that
  \begin{align}\label{xlb} \underset{\theta \in H_{\ol r}}{\glb} \; &(\tfrac 12 \| \theta\|^2_{\ol r} - \intl_{\pOm^r(\ol r) \cap \pOm^r} \!\!\!\!\theta \cdot \dot b_{\pOm^r}\nonumber \\
  &- \intl_{\pOm^r_A (\ol r)} \theta \cdot \dot b_\infty = - \tfrac 12 \| \dot b \|^2_{B, P_\cals, \ol r}.\end{align}

In addition to existence of the limit \eqref{xpj}, a successful computation for an amended problem class with unbounded $\Om^r$ requires several additional conditions.

The conditions \eqref{xlb} are required such that
\begin{align}\label{xlc} w(\infty) &\, \eqadef \, \underset{\ol r \uparrow \infty }{\lub}\,  w (\ol r) \in L_{\infty, loc},\nonumber \\
w_\Ga (\infty) \, &\, \eqadef\, \underset{\ol r \uparrow \infty }{\lub}\,  w_\Ga (\ol r) \in L_{\infty, loc},\end{align}
and existence of a limit
\be\label{xld} \| \dot b \|_{B, P_\cals, \infty} = \underset{\ol r \uparrow \infty }{\lub}\, \| \dot b \|_{ B, P_\cals, \ol r}\ee
for $\dot b \in B_{P_\cals}$.

Each $(z(\cdot; \ol r), 0_N)$ is admissible, perhaps as inferred using lemma 8.6.

The limit solution $z(\cdot; \infty)$ satisfies \eqref{haa}, anticipated as determined by inspection.

The functions $w(\infty), w_\Ga(\infty)$ will necessarily be unbounded as $|r|\uparrow \infty$.

Using \eqref{xpj}, \eqref{gae}, \eqref{xlc}, \eqref{xkf}, \eqref{xkh},  in the limit $\ol r\uparrow \infty$ the space $H_{\ol r}$ coincides, within any bounded subset of $\Om^r$, with the space $H(z(\cdot; \infty), \\ P_\cals, w(\infty), w_\Ga(\infty))$.  Passing to the limit $\ol r \uparrow \infty$ in \eqref{xlb}, using \eqref{xld}, we obtain
\begin{align}\label{xle} \tfrac 12 \| \theta\|^2_\infty  &- \intl_{\pOm^r} b_{\pOm^r}\cdot \theta \nonumber \\
-& \lim\limits_{\ol r \uparrow \infty, r/\ol r \in \Phi_\Om} \intl_{\Phi_\Om} \dot b_\infty \left(\frac{r}{\ol r}\right) \cdot \theta (r) \ge - \tfrac 12 \| \dot b \|_{B, P_\cals, \infty}\end{align}
for all $\theta \in H(z(\cdot; \infty), P_\cals, w (\infty), w_\Ga(\infty))$, so $(z(\cdot; \infty), P_\cals)$ is stable.

Each $(z(\cdot; \ol r), 0_N)$ admissible implies $(z(\cdot; \infty), 0_N) $ admissible, using \eqref{xpj} and definition 2.3.

Thus $(z(\cdot; \infty), P_\cals)$ is potentially admissible as per definition 7.1. Using \eqref{xlc}, theorem 7.2 now determines existence of $P(z(\cdot; \infty))$ such that $(z(\cdot; \infty), P(z(\cdot; \infty)))$ is unambiguous.  Thus theorem 8.5 is extended to problem classes with unbounded $\Om^r$.

Theorem 7.2 may also be applied to each $z(\cdot; \ol r)$, to get unambiguous $(z(\cdot; \ol r), P(z(\cdot; \ol r)))$ (in $\Om^r(\ol r))$.  Then identification of
\be\label{xlf} \Om_{\de_l} = \Om^r(\ol r)\ee
in \eqref{hbe} and use of \eqref{xpj} determines that within any bounded region of $\pOm^r$ and within $\Phi_\Om$, in the limit $\ol r \uparrow \infty$ $P(z(\cdot; \ol r))$ coincides with $P(z(\cdot; \infty))$.

Extension of theorem 13.1 to problem classes with unbounded $\Om^r$ is straightforward using \eqref{xpk}.

Definitions 6.6 and 8.2, theorems 8.5 and 13.1 largely survive failure of regularity \eqref{tca} on a set $\Ga_{LS}$, of measure zero in $\bbr^{m-1}$, in \eqref{xpa}.  For theorem 8.5,  it suffices that \eqref{hac}, \eqref{had} survive, perhaps uncertain where \eqref{xpab} applies.

For theorem 13.1, the essential hypothesis \eqref{xda} depends on suitable details in the amended problem class.  Where \eqref{xpaa} holds, in particular, the conditions \eqref{xfc}, \eqref{xfd} are unrealistic.

A domain limit \eqref{xpb} is again employed, \eqref{xpc}, \eqref{xpd}, \eqref{xpe} holding with
\be\label{xma}
\Om^r_A (\ol r; z) \, \eqadef\, \{ r \in \Om^r\, | \, \dist (r, \Ga_{LS}(z)) < 1/\ol r\}
\ee
\be\label{xmb} \pOm^r_A (\ol r; z) \, \eqadef\,\{ r \in \Om^r \, | \, \dist (r, \Ga_{LS}(z)) = 1/\ol r \}
\ee
\be\label{xmba} \Om^r(\ol r; z) \, \eqadef\, \{ r \in \Om^r \, | \, \dist (r, \Ga_{LS} (z)) > 1/\ol r \}
\ee
depending on $z \in \cals$.

In the specification of $\cals$, the conditions \eqref{xfa}, \eqref{xfb}, \eqref{xfc}, \eqref{xfd} are revised as follows:  for all $z \in\cals$, finite $\ol r, \, \vare$ sufficiently small depending on $\ol r$, there exists $\hat z_{\vare, \ol r} \in W^{1,\infty} (\Om^r)^n$, satisfying \eqref{agea},
\be\label{xfar} \| \hat z_{\vare, \ol r} \|_{W^{1,\infty} (\Om^r)} \le \tfrac{c}{\vare};\ee
\be\label{xfbr} \| \hat z_{\vare, \ol r} \|_\cals \le c_z\ee
independent of $\vare, \ol r$;
\be\label{xfcr} \lim\limits_{\ol r \uparrow \infty} \Big( \lim\limits_{\vare\downarrow 0} \hat z_{\vare, \ol r}\Big) = z\ee
pointwise in $\Om^r\setminus \Ga(z)$;
\begin{align}\label{xfdr} &\mathop{\lub}\limits_\theta \, \frac{1}{\| \theta\|_\cals} \iintl_{\Om^r(\ol r; z)} (E(\hat z_{\vare, \ol r}) - \vare \cald (\hat z_{\vare, \ol r}) - E(z)) \cdot \theta \xrightarrow{\vare \downarrow 0} 0,\nonumber \\
&\theta \in W^{1,\infty} (\Om^r(\ol r; z))^n, \; \; P_\cals \theta\mathop{\mid}\limits_{\pOm^r} = 0
\end{align}
using \eqref{xmba}. In particular, the $\hat z_{\vare, \ol r}$ satisfy bounds \eqref{xfar}, \eqref{xfbr}, but need not approximate $z$ within $\Om^r_A (\ol r; z)$.

Modification of the proof of theorem 13.1 is straightforward.  Solutions $z\in \hat\cals$ are obtained as limits of $\tilde z_h$ satisfying \eqref{xhda} or \eqref{xhdb} in the limit $\ol r \uparrow \infty$.  Mirroring \eqref{xhdz}, in \eqref{xhf}, \eqref{xhg}, \eqref{xhh}, \eqref{xhj} we replace $\hat z_\vare$ by $\hat z_{\vare, \ol r}$,
\be\label{xhdy} \tilde z_h = \tilde z_{h_{\ol r}, \de_{\ol r}, \vare_{\ol r}, \ol r}\ee
successively taking $\de_{\ol r}$ sufficiently small depending on $\ol r$; $\vare_{\ol r}$ sufficiently small depending on $\ol r, \de_{\ol r}; h_{\ol r}$ sufficiently small depending on $\ol r, \de_{\ol r}, \vare_{\ol r}$ as $\ol r \uparrow \infty$. Further details omitted.

\end{document}